\documentclass[reqno,11pt]{amsart}

\pdfoutput1

\usepackage[utf8]{inputenc}
\usepackage{pigpen}
\usepackage{multicol}

\DeclareFontEncoding{LS1}{}{}
\DeclareFontSubstitution{LS1}{stix2}{m}{n}
\DeclareSymbolFont{symbols4}      {LS1}{stix2bb}   {m}{it}
\DeclareMathSymbol{\smblkdiamond}{\mathord}{symbols4}{"E4}
\newcommand{\blackdiamond}{\mathbin{\smblkdiamond}}

\usepackage[bbgreekl]{mathbbol} 
\DeclareSymbolFontAlphabet{\mathbb}{AMSb} 

\usepackage{enumerate}
\usepackage{amsmath}
\usepackage{amssymb}
\usepackage{mathtools}

\usepackage[hidelinks]{hyperref}
\makeatletter
\renewcommand{\eqref}{%
             \@ifstar
                  \eqrefStar%
                  \eqrefNoStar%
}
\newcommand{\eqrefStar}[1]{\textup{\tagform@{\ref*{#1}}}}
\newcommand{\eqrefNoStar}[1]{\textup{\tagform@{\ref{#1}}}}
\makeatother

\usepackage{mathrsfs}

\usepackage{thmtools}
\usepackage{thm-restate}

\usepackage{tikz}
\usetikzlibrary{decorations.markings}

\usepackage{tikz-cd}

\newcommand\mnote[1]{} 

\newcommand{\acomment}[1]{} 

\newcounter{thecounter}
\numberwithin{thecounter}{section}
\newtheorem{lemma}[thecounter]{Lemma}
\newtheorem{prop}[thecounter]{Proposition}
\newtheorem{thrm}[thecounter]{Theorem}
\newtheorem{cor}[thecounter]{Corollary}

\newtheorem{question}[thecounter]{Question}

\newtheorem{Th}{Theorem}

\theoremstyle{definition}

\newtheorem{example}[thecounter]{Example}

\newtheorem{rem}[thecounter]{Remark}

\newtheorem{notation}[thecounter]{Notation}
\newtheorem{defn}[thecounter]{Definition}

\newtheorem{constr}[thecounter]{Construction}

\newtheorem{terminology}[thecounter]{Terminology}

\numberwithin{equation}{section}

\newcommand{\map}{\operatorname{map}}

\newcommand{\Fibre}{\operatorname{Fib}} 

\DeclareMathOperator{\ho}{Ho}

\DeclareMathOperator*{\colim}{colim}
\newcommand{\aut}{\mathrm{aut}}
\newcommand{\Aut}{\mathrm{Aut}}
\newcommand{\Rep}{\operatorname{Rep}}

\newcommand{\Out}{\mathrm{Out}}
\newcommand{\End}{\operatorname{End}}
\newcommand{\rk}{\operatorname{rk}}

\newcommand{\Hom}{\operatorname{Hom}}
\DeclareMathOperator{\im}{Im}

\let\ker\undefined
\DeclareMathOperator{\ker}{Ker}

\newcommand{\Tor}{\operatorname{Tor}}
\newcommand{\id}{\mathrm{id}}

\newcommand{\diag}{\operatorname{diag}}

\newcommand{\holim}{\operatorname{holim}}
\newcommand{\hocolim}{\operatorname{hocolim}}

\newcommand{\lcom}{\hat{{}_\ell}}
\newcommand{\threecom}{\hat{{}_3}}
\newcommand{\fivecom}{\hat{{}_5}}
\newcommand{\twocom}{\hat{{}_2}}

\newcommand{\GL}{\mathrm{GL}}
\newcommand{\SO}{\mathrm{SO}}

\newcommand{\SU}{\mathrm{SU}}
\newcommand{\Sp}{\mathrm{Sp}}

\newcommand{\Spin}{\mathrm{Spin}}

\newcommand{\Q}{{\mathbb {Q}}}
\newcommand{\R}{{\mathbb {R}}}
\newcommand{\F}{{\mathbb {F}}}
\newcommand{\Z}{{\mathbb {Z}}}

\newcommand{\T}{{\mathbf {T}}}
\newcommand{\CP}{{\mathbb{C}P}}

\newcommand{\Sq}{\operatorname{Sq}}

\newcommand{\cZ}{{\mathcal{Z}}}

\newcommand{\beq}{\begin{eqnarray*}}
\newcommand{\eeq}{\end{eqnarray*}}

\newcommand{\tuborg}{\left\{\begin{array}{ll}}
\newcommand{\sluttuborg}{\end{array}\right.}

\newfont{\bm}{msbm10}

\newcommand{\Ob}{\operatorname{Ob}}

\usepackage[all,cmtip]{xy}

\newcommand{\cW}{{\mathcal W}}

\newcommand{\mynote}[1]{}

\newcommand{\oldphi}{\phi}

\let\oldphi\phi
\let\phi\varphi

\newcommand{\G}{{\mathbf G}}

\def\co{\colon\thinspace}

\newcommand{\smCat}{\mathbf{smCat}}

\newcommand{\Spectra}{\mathbf{Sp}}
\newcommand{\hpSpectra}{\mathbf{hpSp}}

\newcommand{\hpC}{\hp\calC}

\newcommand{\pT}{p\calT} 

\newcommand{\hpMod}{\mathbf{hpMod}}
\newcommand{\Mod}{\mathbf{Mod}}

\newcommand{\grMod}{\mathbf{grMod}}

\newcommand{\SpectralSequences}{\mathbf{SS}}

\newcommand{\Ex}{\ensuremath{\mathbb{E}\mathbf{x}}}

\newcommand{\shiftedH}[1]{{}^{#1}\bbH}

\newcommand{\pt}{\mathrm{pt}}
\newcommand{\ev}{\mathrm{ev}}

\newcommand{\cupprod}{\smallsmile} 

\DeclareMathOperator{\gr}{gr}

\newcommand{\incl}{\hookrightarrow}

\newcommand{\tensor}{\otimes}
\newcommand{\smashprod}{\wedge}
\newcommand{\extsmashprod}{\mathbin{\bar{\wedge}}}
\newcommand{\exttensor}{\mathbin{\bar{\tensor}}}

\newcommand{\stringprod}{\circ}

\newcommand{\tildestringprod}{\mathbin{\tilde{\stringprod}}}

\newcommand{\isom}{\cong}
\newcommand{\homot}{\simeq}

\newcommand{\homeom}{\approx}

\newcommand{\suspension}{\Sigma}
\newcommand{\loops}{\Omega}

\newcommand{\op}{\mathrm{op}}
\newcommand{\fw}{\mathrm{fw}}
\newcommand{\fop}{\mathrm{fop}}
\newcommand{\dfop}{\mathrm{u}}
\newcommand{\cart}{\mathrm{cart}}

\newcommand{\hp}{\mathbf{hp}}
\newcommand{\pr}{\mathrm{pr}}

\DeclareSymbolFont{bbold}{U}{bbold}{m}{n}
\DeclareSymbolFontAlphabet{\mathbbold}{bbold}
\newcommand{\bbOne}{\mathbbold{1}}

\makeatletter
\def\slashedarrowfill@#1#2#3#4#5{%
  $\m@th\thickmuskip0mu\medmuskip\thickmuskip\thinmuskip\thickmuskip
   \relax#5#1\mkern-7mu%
   \cleaders\hbox{$#5\mkern-2mu#2\mkern-2mu$}\hfill
   \mathclap{#3}\mathclap{#2}%
   \cleaders\hbox{$#5\mkern-2mu#2\mkern-2mu$}\hfill
   \mkern-7mu#4$%
}
\def\rightslashedarrowfill@{%
  \slashedarrowfill@\relbar\relbar\mapstochar\rightarrow}
\newcommand\xhto[2][]{%
  \ext@arrow 1579{\rightslashedarrowfill@}{#1}{#2}}
\makeatother

\makeatletter
\def\circarrowfill@#1#2#3#4#5{%
  $\m@th\thickmuskip0mu\medmuskip\thickmuskip\thinmuskip\thickmuskip
   \relax#5#1\mkern-7mu%
   \cleaders\hbox{$#5\mkern-2mu#2\mkern-2mu$}\hfill
   \mathclap{#3}\mathclap{#2}%
   \cleaders\hbox{$#5\mkern-2mu#2\mkern-2mu$}\hfill
   \mkern-7mu#4$%
}
\def\rightcircarrowfill@{%
  \circarrowfill@\relbar\relbar{\mkern1.8mu\circ}\rightarrow}
\newcommand\xoto[2][]{%
  \ext@arrow 1579{\rightcircarrowfill@}{#1}{#2}}
\makeatother

\newcommand{\oto}{\xoto{}}
\newcommand{\longoto}{\xoto{\quad}}

\def\circsym{%
  \mathchoice%
	{\raisebox{-2.75pt}[0pt][0pt]{$\displaystyle{\circ}$}}
	{\raisebox{-2.75pt}[0pt][0pt]{$\textstyle{\circ}$}}
    {\raisebox{-2.0pt}[0pt][0pt]{$\scriptstyle{\circ}$}}
    {\raisebox{-1.5pt}[0pt][0pt]{$\scriptscriptstyle{\circ}$}}
}
\def\circdec{\mathclap{\circsym}}

\newcommand{\xto}{\xrightarrow}
\newcommand{\xot}{\xleftarrow}

\newcommand{\longto}{\xto{\quad}}

\newcommand{\longincl}{\xhookrightarrow{\quad}}

\newcommand{\bbC}{\mathbb{C}}
\newcommand{\bbD}{\mathbb{D}}
\newcommand{\bbE}{\mathbb{E}}

\newcommand{\bbH}{\mathbb{H}}

\newcommand{\bbN}{\mathbb{N}}

\newcommand{\bbR}{\mathbb{R}}

\newcommand{\calA}{\mathcal{A}}

\newcommand{\calC}{\mathcal{C}}
\newcommand{\calD}{\mathcal{D}}

\newcommand{\calF}{\mathcal{F}}

\newcommand{\calH}{\mathcal{H}}

\newcommand{\calL}{\mathcal{L}}

\newcommand{\calP}{\mathcal{P}}

\newcommand{\calT}{\mathcal{T}}

\newcommand{\calV}{\mathcal{V}}
\newcommand{\calW}{\mathcal{W}}

\newcommand{\fH}{\mathfrak{H}}

\newcommand{\fS}{\mathfrak{S}}

\newcommand{\sk}{\operatorname{sk}}

\DeclareMathSymbol{\mathinvertedexclamationmark}{\mathord}{operators}{'074}
\DeclareMathSymbol{\mathexclamationmark}{\mathord}{operators}{'041}
\makeatletter
\newcommand{\raisedmathinvertedexclamationmark}{%
  \mathord{\mathpalette\raised@mathinvertedexclamationmark\relax}%
}
\newcommand{\raised@mathinvertedexclamationmark}[2]{%
  \raisebox{\depth}{$\m@th#1\mathinvertedexclamationmark$}%
}
\makeatother

\newcommand{\muet}{\bbtau'}

\newcommand{\BtGq}{B\, {}^{\tau}\!G(q)} 

\newcommand{\pb}{\ar@{}[dr]|(0.33){\text{\pigpenfont J}}}
\newcommand{\pbop}{\ar@{}[dl]|(0.33){\text{\pigpenfont L}}}

\allowdisplaybreaks

\newcounter{saveenumi}

\usepackage{xr}
\usepackage[
	style=alphabetic,
	backend=bibtex,
	maxnames=5,
	maxcitenames=5,
	doi=false,
    isbn=false,
    url=false,
    maxalphanames=4,
    minalphanames=3,
    maxbibnames=99,
]{biblatex}
\renewbibmacro{in:}{}

\begin{document}
\title{String topology of finite groups of Lie type}

\date{\today}

\author{Jesper Grodal}
\author{Anssi Lahtinen}
\thanks{%
    \begin{tabular*}{0.97\textwidth}{@{}c@{}@{\extracolsep{\fill }}p{0.90\textwidth}@{}}
    \raisebox{-0.66\height}{%
        \begin{tikzpicture}
        	[
            	y=0.80pt, 
    			x=0.8pt, 
    			yscale=-1, 
    			inner sep=0pt, 
    			outer sep=0pt, 
    			scale=0.12
			]
            \definecolor{c003399}{RGB}{0,51,153}
            \definecolor{cffcc00}{RGB}{255,204,0}
            \begin{scope}[shift={(0,-872.36218)}]
            	\path[shift={(0,872.36218)},fill=c003399,nonzero rule] 
					(0.0000,0.0000) rectangle (270.0000,180.0000);
            	\foreach \myshift in 
                	{
    					(0,812.36218), 
						(0,932.36218), 
                		(60.0,872.36218), 
						(-60.0,872.36218), 
                		(30.0,820.36218), 
						(-30.0,820.36218),
                		(30.0,924.36218), 
						(-30.0,924.36218),
                		(-52.0,842.36218), 
						(52.0,842.36218), 
                		(52.0,902.36218), 
						(-52.0,902.36218)
					}
                    \path[shift=\myshift,fill=cffcc00,nonzero rule] 
                    	(135.0000,80.0000) -- 
    					(137.2453,86.9096) -- 
    					(144.5106,86.9098) -- 
    					(138.6330,91.1804) -- 
    					(140.8778,98.0902) -- 
    					(135.0000,93.8200) -- 
    					(129.1222,98.0902) -- 
    					(131.3670,91.1804) -- 
    					(125.4894,86.9098) -- 
    					(132.7547,86.9096) -- 
						cycle;
            \end{scope}
        \end{tikzpicture}%
    }
    &
Supported by the Danish National Research Foundation grants DNRF92 and DNRF151
and by the European Union's Horizon 2020 research and 
innovation programme under grant agreements No 800616 and 682922.
    \end{tabular*}\nopunct%
}

\address[J. Grodal]{Department of Mathematical Sciences, University of
Copenhagen, Denmark}
\email{jg@math.ku.dk}
\address[A. Lahtinen]{Copenhagen, Denmark}
\email{anssi@anssilahtinen.net}

\subjclass[2020]{%
20J06 
(Primary)
20D06,
55R35,
55P50 	
(Secondary)}

\begin{abstract}
We show that the mod $\ell$ cohomology of any finite group of Lie type
in characteristic $p \neq \ell$
admits the structure of a module over the mod $\ell$ cohomology of the free loop space
of the classifying space $BG$ of the 
corresponding compact Lie group $G$, via ring and module structures
constructed from string topology, \`a la Chas--Sullivan.
If a certain class in the homology of the
finite group of Lie type, arising from the fundamental class of $G$, is nontrivial, then this module structure is
free of rank one, providing a highly structured isomorphism
between the two cohomologies.
We verify the nontriviality of the class in a range of
cases, including all simply connected untwisted classical groups over $\F_q$,
with $q$ congruent to $1$ mod $\ell$. We also show how to deal with twistings and avoid
the congruence condition by replacing $BG$ by a certain $\ell$--compact fixed point group depending on the order of $q$ mod
$\ell$, without changing the finite group. With this modification, we know of no examples where the
class is trivial, raising the possibility of a general structural answer to an open question of Tezuka, who speculated 
about the existence of an isomorphism between the two cohomology  rings.

\end{abstract}

\maketitle

\setcounter{tocdepth}{2}
\tableofcontents

\section{Introduction}
The mod $\ell$ cohomology ring of a finite group of Lie type
$\G(\F_q)$ over a finite field $\F_q$ of characteristic $p\neq \ell$
occurs in many parts of mathematics,
from representation theory to $K$--theory. For large primes $\ell$,
the cohomology ring
was calculated by Quillen
\cite[\S2]{quillen71icm}.
When in addition $q\equiv 1$ mod $\ell$,
it can also be observed, rather mysteriously, to coincide with a different object,
namely the mod $\ell$ cohomology
ring of the free loop space
$LB \G(\bbC) = \map(S^1,B\G(\bbC))$,
or the homotopy equivalent classifying space $BL\G(\bbC)$
of the loop group $L\G(\bbC)$  studied 
e.g.\ in \cite{PS86} (taking either smooth or continuous loops \cite[Thm.~4.6]{stacey09}).
Here $\G$ is the underlying split reductive group scheme over $\Z$
and $\G(\bbC)$ denotes the complex points of $\G$ with the analytic
topology.
The abstract isomorphism
between $H^\ast(B\G(\F_q);\F_\ell)$ and $H^\ast(LB\G(\bbC);\F_\ell)$
arises 
as the consequence of the collapse of two spectral sequences with isomorphic $E_2$ pages
(both isomorphic to $\F_\ell[x_1, \ldots,x_r] \otimes
\Lambda_{\F_\ell}(y_1,\ldots,y_r) $ with $|x_i|=2d_i$ and $|y_i|=2d_i-1$, for $d_i$'s the
degrees of the root system of $\G$ \cite[\S3.7,Table~1]{humphreys90}).

When $\ell$ is small, more specifically when $\ell$ is a
torsion prime for $\G$, both 
$H^*(B\G(\F_q);\F_\ell)$ and $H^*(L B\G(\bbC);\F_\ell)$
remain of considerable interest but become very difficult to compute, 
and are in general unknown.
Calculations have also revealed isomorphic rings, even in the presence of $\ell$--torsion in $\G$, see
\cite{quillen72}, \cite{quillen71spin}, \cite{FP78},
\cite{kleinerman82}, \cite{MT91}, \cite{KK93}. 
Indeed, in an unpublished note \cite{tezuka98}, Tezuka asked
if the two cohomologies  \emph{always} agree, as long as $q \equiv 1$ mod $\ell$ 
(or $1$ mod $4$ in the case $\ell =2$).
Further calculations
supporting this ``Tezuka conjecture'' have been worked out in 
\cite{KMT00}, \cite{KK10}, \cite{KMN06}, \cite{KTY12},
\cite{Kameko-spin}, \cite{KajiMod2total} but still without pointing to any direct
structural way of relating the two objects.
The underlying spaces are certainly not homotopy
equivalent in any sense, as even the most basic cases show: For $\T$ a
one-dimensional torus, $B\T(\F_q) \homot B\Z/(q-1)$ 
is a rationally trivial space depending heavily on
$q$, whereas $L B\T(\bbC) \homot S^1 \times \CP^\infty$
is rationally nontrivial and independent of $q$.

\smallskip

The goal of this paper is to use string topology 
to establish a general structural relationship between 
$H^*(LB\G(\bbC);\F_\ell)$ and $H^*(B\G(\F_q);\F_\ell)$:
we show that $H^*(B\G(\F_q);\F_\ell)$ carries a natural module
structure over
$H^*(LB\G(\bbC);\F_\ell)$, equipped with a
Chas--Sullivan-type string product, compatible with much
additional structure. Furthermore the non-vanishing
of the image of the $d$--dimensional manifold fundamental class of the maximal compact subgroup of $\G(\bbC)$ in $H_d(B\G(\F_q);\F_\ell)$
implies that the module structure is free of rank $1$, 
which in turn yields an isomorphism between the two
cohomologies that respects a large
amount of extra structure.
We prove the non-vanishing of the image of the fundamental class for
a range of cases, including all simply connected non-exceptional groups, as long as $q$ satisfies the aforementioned congruence
condition. In fact we also show how to avoid the congruence condition
by instead modifying $\G$,
depending on the congruence of $q$ modulo $\ell$, without modifying
$\G(\F_q)$, in a technical sense that we explain below. In this formulation we know of no case where the
class vanishes, raising the question whether it is always
nonzero.

\smallskip

We work towards stating our theorems in more detail. With this aim we will first recall the homotopy theory of Lie groups at a prime $\ell$, i.e., $\ell$--compact groups, and their relationship to the $\ell$-local structure of finite groups of Lie type, followed by a recap on Chas--Sullivan string products. 
\subsubsection*{Lie groups and $\ell$--compact groups}\label{subsec:intro-Lie-lcg}
An $\ell$--compact group is an $\ell$--complete pointed connected space $BG$ whose based loop space $G = \Omega BG$ has
finite mod $\ell$ cohomology. 
Any compact
connected Lie group $K$ has an associated $\ell$--compact group,
obtained as the $\F_\ell$--homology localization $BK\lcom$.
For $\G$ a connected reductive algebraic group, there exists a homotopy equivalence
$BK \xrightarrow{\homot} B\G(\bbC)$ where $K$ is a maximal
compact subgroup of the complex algebraic group $\G(\bbC)$ (see e.g.\ \cite[\S8.1]{AGMV08}), so 
$B\G(\bbC)\lcom$ is a connected $\ell$--compact group as well.
An $\ell$--compact group $BG$ is called connected
if $G$ as a space is,  semisimple if in addition $\pi_1(G)$ is
finite, and the
dimension of $BG$ is defined as the degree $d$ of the 
top nontrivial mod $\ell$ homology group of $G$.

It turns out that $\ell$--compact groups admit a classification much like the classification of compact connected Lie groups, but everywhere replacing ordinary root data over $\Z$ with $\Z_\ell$--root data. More precisely,
the classification of  \cite{AGMV08,AG09} states that connected
$\ell$--compact groups, up to isomorphism, are in
one-to-one correspondence with root data $\bbD$ over the $\ell$--adic
integers $\Z_\ell$, up to isomorphism, and $\Out(BG) \isom \Out(\bbD_G)$. Here $\Out(BG)$
denotes the outer automorphism group of the $\ell$--compact group $BG$, 
i.e., free homotopy classes of 
self-homotopy equivalences of $BG$, and $\Out(\bbD_G)$ denotes the outer
automorphism group of the root datum $\bbD_G$ of $BG$
(see Appendix~\ref{app:pcg} and the survey \cite{grodal10}).
The process described above of passing from a compact Lie group, or complex algebraic group, to the associated $\ell$--compact group corresponds on the level of root data to tensoring with $\Z_\ell$.
Since the group of units $\Z_\ell^\times$ is
uncountable, every $\ell$--compact group has uncountably many outer
automorphisms given by multiplication by  $q \in
\Z_\ell^\times$ on $\bbD_G$. These correspond to  ``unstable Adams
operations'' $\psi^q$ on $BG$ which extend the classical operations
\cite{AM76, JMO92, JMO95}, and should be thought of as
``$q$--th power Frobenius maps'' in a sense that will be made precise below.

\subsubsection*{Finite groups of Lie type from $\ell$--compact groups}
Fundamental to our construction of the module structure 
on the cohomology $H^*(B\G(\F_q);\F_\ell)$ of a finite group of Lie type
is that, up to homotopy equivalence,
the space $B\G(\F_q)\lcom$ may be realized as a space of paths
in the $\ell$--compact group $BG = B\G(\bbC)\lcom$, as we will 
now explain. We start by recalling the definition of a 
general finite group
of Lie type. Let $\G$ be a connected split reductive algebraic group scheme
over $\Z$ with $\bar \F_p$--rational points $\G(\bar \F_p)$, and 
let $\sigma$ be a Steinberg endomorphism, i.e., an endomorphism of
$\G(\bar \F_p)$ as an algebraic group over $\bar \F_p$, which, when raised to some power,
becomes a standard Frobenius map $\psi^q\co \G(\bar\F_p) \to \G(\bar\F_p)$
induced by the $q$--th power map on $\bar \F_p$. 
A finite group of Lie type is a group (necessarily finite)
which arises as the fixed points
$\G(\bar \F_p)^\sigma$ for some such $\G$ and $\sigma$;
important examples are of course given by
the ``untwisted case'' where $\sigma =
\psi^q$ and $\G(\bar \F_p)^{\psi^q} = \G( \F_q)$. 
The classical groups
$\GL_n(\F_q)$, $\Sp_n(\F_q)$, etc.\ are 
examples of finite groups of Lie type; 
see e.g.\ \cite[\S22.1]{MT11} for more information. 

By a theorem of Friedlander--Mislin \cite[Thm.~1.4]{FM84} 
(generalizing work of Quillen \cite{quillen72}), there is a homotopy equivalence
\begin{equation}\label{lie}
	BG = B\G(\bbC)\lcom \xleftarrow{\ \homot\ } (B\G(\bar \F_p))\lcom
\end{equation}
for $\ell \neq p$
relating characteristic $p$ to characteristic $0$. 
Combining this with  another theorem of Quillen and Friedlander,
we obtain homotopy equivalences
\begin{equation}\label{lietype}
    (B\G(\bar \F_p)^\sigma)\lcom 
    \xrightarrow{\ \homot\ } 
    (B\G(\bar\F_p)\lcom)^{h\sigma} 
    \xrightarrow{\ \homot\ } 
    BG^{h\sigma}
\end{equation}
relating actual fixed points to homotopy fixed points
where $\sigma$ is a Steinberg endomorphism and
we have continued to write $\sigma$ for the self-equivalences
of $B\G(\bar\F_p)\lcom$ and $BG$ induced by $\sigma$.
See \cite[Thm.~2.9]{friedlander76},
\cite[Thm.~12.2]{FriedlanderEtaleHomotopy}, and also \cite[Thm.~3.1]{BrotoMoellerOliver}.
In particular,
in this picture the Frobenius map $\psi^q$ of $\G(\bar\F_p)$
corresponds to the map $\psi^q$ of $BG$ mentioned above.

In \eqref{lietype}, and throughout, by the homotopy fixed point space $X^{h\sigma}$ of a
self-map $\sigma\co X \to X$ we mean
the space
$X^{h\sigma} = \{ \alpha\co I \to X\ |\ \sigma\alpha(1) = \alpha(0)\}
$,
a subspace of the mapping space $X^I$. It also identifies with the
homotopy pullback
\begin{equation}
\vcenter{\xymatrix{ X^{h\sigma} \ar[r] \ar[d]\pb & X \ar[d]^{\Delta}\\
    X \ar[r]^-{(1,\sigma)} & X \times X}}
\end{equation}
  See Proposition~\ref{prop:fixed-point} for more information.

With the above dictionary in place, 
the rest of the paper is formulated in terms of
homotopy fixed points on $\ell$--compact groups.
Via \eqref{lie} and \eqref{lietype}, our results then imply
results about finite groups of Lie type in any characteristic $p \neq
\ell$. The space $BG^{h\sigma}$ is also quite interesting
for general $\ell$--compact groups
$BG$ and self-maps $\sigma$
not arising from algebraic groups and Steinberg endomorphisms, 
and is often known to be an ``exotic'' $\ell$--local
finite group, in the sense of \cite{BLO03jams}, although 
such results so far build on
case-by-case considerations (see \cite[Thm.~4.5]{LO02} \cite[Thm.~A]{BM07}).
We also note that for $\sigma = \psi^1 = \id$,
the  homotopy fixed point space $BG^{h\sigma}$ agrees with 
the free loop space $LBG$.
Thus, in view of the equivalence
$BG^{h\psi^q} \homot B \G( \F_q)\lcom$ afforded by equivalences 
\eqref{lie} and \eqref{lietype}, we may think of the 
free loop space $LBG$ as a ``finite group of Lie type over the field of one element.''

\subsubsection*{String module structure}
Chas and Sullivan 
\cite{ChasSullivan} and later authors, see e.g.\ \cite{Sullivan04, CG04},
observed that, for $X$ a closed oriented manifold, $H^*(LX)$ and
$H_*(LX)$ carry additional ``string'' products and coproducts. 
These structures are constructed, roughly speaking,
by reversing the direction of one of the two horizontal maps in the diagram
\begin{equation}\label{eq:stringtop101}
\xymatrix@R-15pt{
 \map(S^1 \amalg S^1,X) \ar@{=}[d] & \map(S^1 \vee S^1,
 X) \ar[r]\ar[l] \ar@{=}[d]& \map(S^1,X) \ar@{=}[d]\\
 LX \times LX & LX \times_X LX & LX }
 \end{equation}
by an umkehr map (also called degree-shifting transfer map,
wrong-way map, or ``integration along the fiber map''). Here the first map is induced by joining the circles at the basepoint and the
second map is induced by the pinch map. Furthermore,
versions of these constructions, allowing $X$ to be an orbifold,
Borel construction, Gorenstein space, stack, or classifying space, 
have been considered by a number of authors, see e.g.\ 
\cite{LUX08, FelixThomas, BGNX07, BGNX12, GS08, GW08, ChataurMenichi,HL15}.
In particular, Chataur and Menichi
\cite{ChataurMenichi} constructed a "string product" $\stringprod$ on the shifted cohomology $\bbH^\ast(LBG;\F_\ell) = H^{*+d}(LBG;\,\F_\ell)$ (putting non-trivial cohomology between degree $-d$ and infinity) which is associative and commutative, and turns out also to be unital. 
(See Theorem~\ref{thm:cmcomparison2} and Remark~\ref{rk:unitality} for a summary and
generalization of previous results.) 
The product should be thought
of as mixing the cup product
on $H^*(BG)$ with a dual of the Pontryagin product on $H_*(G)$, by
choosing an umkehr map for the right-hand map in
\eqref{eq:stringtop101}.  Note in particular that the
algebra structure makes $\bbH^*(LBG)$ a free module of rank one over
itself on the unit of the string product, which is a class in $H^d(LBG)$ that maps
non-trivially to $H^d(G)$.

Our goal in this paper is to connect the string topology of $LBG$
to the study of finite groups of Lie type via a module structure.
To understand where this module structure comes from, 
observe that we may concatenate a 
path with a loop starting at the end point of the path 
to obtain a new path with the same start and end points.
Hence we also have a diagram
\begin{equation}\label{eq:stringmodules101}
\xymatrix@R=0pt{
	LX \times X^{h\sigma}  
	&
	\ar[l]
	LX \times_X  X^{h\sigma}
	\ar[r] 
	&
	X^{h\sigma}
	\\	
	(\alpha,\beta) 
	&
	\ar@{|->}[l]
	(\alpha,\beta)
	\ar@{|->}[r]
	&
	\alpha \star \beta
}
\end{equation}
paralleling \eqref{eq:stringtop101}.
Here a point $(\alpha,\beta)$ in the space
$LX \times_X  X^{h\sigma}$
in the middle
may be pictured as follows:
\[
    \begin{tikzpicture}
    	[
    		baseline=0cm,
    		endpoint/.style={
    			circle,
    			draw=black,
    			fill=black,
    			inner sep=0pt,
    			minimum size=2.6pt
    		},
    		label/.style={
    			font=\scriptsize
    		},
    		->-/.style={
    			decoration={
    				markings,
    				mark=at position .52 with {\arrow{>}}
    			},
    			postaction={decorate}
    		}
    	]
    	\node [endpoint] (start) at (1.5,0) {};
    	\node [endpoint] (end) at (0,0) {};
    	\node [label] at (start) [above] {$\sigma(x)$};
    	\node [label] at (end) [above] {$x$};
    	
    	\draw [->-,thick] (start) .. controls (1.0,-0.26) and (0.6,0.2) .. (end)
	   		node [label,above,midway] {$\beta$}
		;
    	\draw [->-,thick] 
			(end) ..controls (-0.15,-0.1) and (-1.3,-1.0)  .. (-1.3,0)
			node [label,left] {$\alpha$}
    		..  controls (-1.3,1.0) and (-0.2,0.066) .. (end);
    \end{tikzpicture}
\]
In this paper, we show that, for $X = BG$ 
a semisimple $\ell$--compact group, we can choose an umkehr map for the
right-hand map in \eqref{eq:stringmodules101} in such a way that
$H^\ast(BG^{h\sigma};\,\F_\ell)$ becomes a module over 
$\bbH^\ast(LBG;\F_\ell)$ with remarkable properties. This again, via
\eqref{lie} and \eqref{lietype}, endows the cohomology groups of
finite groups of Lie type with the desired module structure.

\begin{Th}
\label{thm:mainresult} 
Let $\ell$ be a prime and let $BG$ be a semisimple $\ell$--compact group of
dimension $d$ with a self-map $\sigma\co BG \to BG$.
	The cohomology groups $H^\ast(BG^{h\sigma};\F_\ell)$ admit a
        $\bbH^\ast(LBG;\F_\ell)$--module structure
\[\stringprod\co \bbH^\ast(LBG;\F_\ell) \otimes H^\ast(BG^{h\sigma};\F_\ell)
  \longto H^\ast(BG^{h\sigma};\F_\ell)\]
via the above constructions, extending
the natural $H^*(BG;\F_\ell)$--module structure on these cohomology groups (Definition~\ref{def:stringprodandmod} and Corollary~\ref{cor:bulletalgandmodstr}).
Moreover, the string product on $\bbH^\ast(LBG)$ and
this module structure lift to the level of 
the Serre spectral sequences of the
evaluation maps
$LBG \to BG$ and $BG^{h\sigma} \to BG$
(Theorem~\ref{thm:ssforlbgbghsigma}).
\end{Th}

In good cases both Serre spectral sequences of Theorem~\ref{thm:mainresult} collapse at the $E_2$--pages, providing a
structured isomorphism between the $E_\infty$--pages, and hence a
structured isomorphism between $H^*(LBG)$ and $H^*(BG^{h\sigma})$.
In general the spectral sequences do not collapse, however, and indeed while the product
on $\bbH^*(LBG)$ is commutative (see Theorem~\ref{thm:cmcomparison2}),
the product on the $E_2$--page of the corresponding spectral sequence
in general is not,
forcing nontrivial differentials to appear.
See Remark~\ref{rk:noncommutativee2}.
Nevertheless, we show that the question whether
$H^*(BG^{h\sigma};\F_\ell)$ is free of rank one over
$\bbH^*(LBG;\F_\ell)$ boils down to a single class
being nonzero.
Consider the homotopy fibre sequence 
\begin{equation}\label{eq:fibseq}
	G \xrightarrow{\ i\ } BG^{h\sigma} \longto BG
\end{equation}
associated to the evaluation fibration 
$BG^{h\sigma} \to BG$, $\alpha \mapsto \alpha(1)$, where by definition
$G = \loops BG$.

\begin{defn}
\label{def:fundamentalclass} 
Given a connected $\ell$--compact group $BG$ and a self-map $\sigma\colon BG \to BG$,
we say that \emph{$BG^{h\sigma}$ has a $[G]$--fundamental class}
if the map $i_\ast \colon H_d(G)\to H_d(BG^{h\sigma})$
is nontrivial.
\end{defn}
As  $H_d(G)$ is one-dimensional, the existence of a $[G]$--fundamental class
is equivalent to the map $i^*\co
H^d(BG^{h\sigma}) \to H^d(G)$ being surjective, which in 
turn is equivalent to a
nontrivial class in $E^{0,d}_2$ in the cohomological Serre
spectral sequence of the fibration $G \to BG^{h\sigma} \to BG$ being a
permanent cycle. When $\sigma$ is the identity, i.e., when $BG^{h\sigma} = LBG$, the existence of a $[G]$--fundamental class is a consequence of the non-obvious fact that 
$\bbH^*(LBG)$ is a unital algebra and hence free of rank one over itself (see Remark~\ref{rk:unitality}).

Our next theorem says that in general the existence of 
a $[G]$--fundamental class is equivalent to
being free-of-rank-one, and implies that 
$H^\ast(LBG)$ and $H^\ast(BG^{h\sigma})$
are isomorphic as rings up to specified filtrations.

\begin{Th}
\label{thm:strtoptezukacrit} 
Let $BG$ be a semisimple $\ell$--compact group of
dimension $d$ with a self-map $\sigma\co BG \to BG$.
Then the following three conditions are equivalent.

\begin{enumerate}
\item \label{item:fundclass} 
The homomorphism 
$i_*\co H_d(G;\F_\ell) \to H_d(BG^{h\sigma};\F_\ell)$ for the map $i$
of \eqref{eq:fibseq}
is non-trivial, i.e., $BG^{h\sigma}$ has a $[G]$--fundamental class.
\item \label{item:freerk1}
$H^\ast(BG^{h\sigma};\F_\ell)$ is free of rank $1$
as an  $\bbH^\ast (LBG;\F_\ell)$--module.
\item \label{item:assgraded} 
There exists an $x  \in H^d(BG^{h\sigma};\F_\ell)$ for which
$\stringprod$--product with $x$ induces an $H^*(BG;\F_\ell)$--algebra isomorphism
\begin{equation}
	\gr H^\ast(LBG;\F_\ell) \xto{\ \isom\ } \gr H^\ast(BG^{h\sigma};\F_\ell)
\end{equation}
on the associated graded algebras for
the Serre spectral sequences of Theorem~\ref{thm:mainresult}.
\end{enumerate}
\end{Th}

The equivalence of \eqref{item:fundclass} and \eqref{item:freerk1} is proved in
Section~\ref{subsec:for1crit}, and that of
\eqref{item:freerk1} and \eqref{item:assgraded} in Section~\ref{subsec:isosonassocgradeds}.
The existence of a $[G]$--fundamental 
class is easily seen to imply that $H^*(BG;\F_\ell) \to H^*(BG^{h\sigma};\F_\ell)$
is injective and that $\sigma$ induces the identity map on
$H^*(BG;\F_\ell)$ (see Corollary~\ref{cor:tezukaimpliesid}).
We do not know an example where
these necessary conditions are not also sufficient for a
$[G]$--fundamental class to exist.

\subsubsection*{First results on existence of fundamental classes}
To make our main theorems about the string module structure easier to apply, 
we now embark on the study of when fundamental classes exist.
We first record that they exist when
$H^*(BG;\F_\ell)$ is a polynomial ring with $\sigma$ acting as the
identity.
\begin{prop}
 \label{prop:polycollapse}
Suppose $BG$ is a connected $\ell$--compact group
for which $H^\ast(BG;\F_\ell)$ is a polynomial ring,
and let  $\sigma \colon BG \to BG$ be a self-map 
of $BG$ inducing the identity map on $H^\ast(BG;\F_\ell)$.
Then the  Serre spectral sequence of fibre sequence 
$G \xrightarrow{i} BG^{h\sigma} \to BG$ 
of \eqref{eq:fibseq}
  collapses at the $E_2$--page. In particular, the map 
$H^*(BG^{h\sigma};\F_\ell) \xrightarrow{i^*} H^*(G;\F_\ell)$ 
is surjective and
 $BG^{h\sigma}$ has a $[G]$--fundamental class.
\end{prop}
For completeness, we give a proof of Proposition~\ref{prop:polycollapse} 
using the Eilenberg--Moore spectral sequence in Section~\ref{subsec:polycase}.

Our next result implies that for any automorphism $\sigma$ of $BG$, some
power $\sigma^k$ of it has a fundamental class. More precisely:

\begin{Th} \label{thm:tezukasubgrp}
Let $BG$ be a semisimple $\ell$--compact group.
Then the set 
$$D = \{ [\sigma] \in \Out(BG) \mid BG^{h\sigma} \text{ has a $[G]$--fundamental class} \}$$
is a finite-index closed normal subgroup of $\Out(BG)$ contained in 
the kernel
$\{ [\sigma] \in  \Out(BG) \mid \sigma^* = \id \in  \Aut(H^\ast
BG)\}$. 
In particular,
\begin{equation}
\label{eq:psiqsinD}
\{q \in 1 + 2\ell\Z_\ell \mid [\psi^q] \in D\} = 1 + 2\ell^k \Z_\ell
\end{equation}
for some $k \geq 1$.
\end{Th}
The proof of Theorem~\ref{thm:tezukasubgrp}
is given in Section~\ref{subsec:Dclosed}.
Here $\Out(BG)$ is equipped with the natural $\ell$--adic topology induced
by its action on $H^*(BG;\Z/\ell^s)$,
$s\geq 1$, which agrees with the natural topology on $\Out(\bbD_G)$ by Proposition~\ref{prop:topcomparison}.
We do not know an example of a connected $\ell$--compact group for which the subgroup in
Theorem~\ref{thm:tezukasubgrp} is proper (even without the semi-simplicity assumption).
A concrete bound on the number $k$ in \eqref{eq:psiqsinD} is implied by
Theorem~\ref{thm:ss-comparison2}.
Note that Theorem~\ref{thm:tezukasubgrp} implies, for example, that
$E_8(q)$ at $\ell=2$ has a fundamental class 
as long as $q-1$ is highly $2$--divisible, despite
the exact cohomology ring not being known.

\begin{rem}
\label{rk:droppingsemisimplicity2}
We expect Theorems~\ref{thm:mainresult}, \ref{thm:strtoptezukacrit}, 
and \ref{thm:tezukasubgrp} to generalize from semisimple $\ell$--compact groups
to arbitrary connected $\ell$--compact groups.
See Remark~\ref{rk:semisimplicityassumptionnontechnical}.
\end{rem}

\subsubsection*{Steinberg endomorphisms and untwisting}
The necessary condition $\sigma^\ast = \id$  for $H^\ast(BG^{h\sigma};\F_\ell)$ to be free of rank $1$
over $\bbH^\ast(LBG;\F_\ell)$
is often not satisfied
when $\sigma$ is a Frobenius operation $\psi^q$
with $q \not\equiv 1$  modulo $\ell$. 
However, we show that
in this case 
$H^\ast(BG^{h\sigma};\F_\ell)$ is often
free of rank $1$ over $\bbH^\ast(LBH;\F_\ell)$
where $BH$ is a certain sub--$\ell$--compact group of $BG$
on whose cohomology $\sigma$ does act trivially.
We note that even when $BG$ comes from a compact Lie group, 
the $\ell$--compact group $BH$ often does not, 
again motivating our setting of $\ell$--compact groups.

Before stating our results, we introduce some notation.
Recall that the classification of automorphisms of
$\ell$--compact groups (see \cite[Sec.~8.4]{AG09} and \cite[Sec.~13]{AGMV08}) shows that 
any automorphism $\sigma$ of a simple $\ell$--compact group 
can be written as $\sigma = \tau \psi^q$ for some $\tau \in
\Out(\bbD_G)$ of finite order and $q \in \Z_\ell^\times$, and we will
henceforth in the introduction consider 
automorphisms of this form.
We write
\begin{equation}\label{eq:lietypegrp}
\BtGq = BG^{h(\tau\psi^q)}
\end{equation}
for short, and abbreviate further to $BG(q)$ when $\tau =1$.
By~\eqref{lietype},
this 
agrees with standard Lie theoretic notation, as in e.g.\ 
\cite[Table~22.1]{MT11}, when
the latter makes sense.
Thus $\BtGq = B{}^{(\tau\psi^{-1})}\!G(-q)$ (and in fact one
calculates that
$\psi^{-1}$ equals $1$ in $\Out(\bbD)$ if $G$ has type $A_1$, $B_n$,
$C_n$, $D_{2n}$, $E_7$, $E_8$,
$F_4$, or $G_2$ and is the non-trivial graph automorphism
if $G$ has type $A_n$, $D_{2n+1}$ or $E_6$; see
Proposition~\ref{prop:twistingclassification}).

The following theorem allows us to reduce the
study of $\BtGq$ to the case where $\tau =1$
and $q$ is congruent to $1$ modulo $\ell$,
at least when $\tau$ has order prime to $\ell$. It is a small generalization of a result of Broto--M{\o}ller
\cite[Thms.~B and E(1)]{BM07}.

\newcommand{\untwistingThmStatement}{%
Let $BG$ be a connected $\ell$--compact group, let $q \in
  \Z_\ell^\times$, and let $\tau \in \Out(BG) \isom \Out(\bbD_G)$ 
  be of finite order prime to $\ell$. 
  Let $e$ be the multiplicative order of $q$
  mod $\ell$, and write $q =
  \zeta_e q'$  in $\Z_\ell^\times$
  for $\zeta_e$ a primitive $e$--th root of unity and
$q' \equiv 1$ mod $\ell$.  Set $\tau' = \tau \psi^{\zeta_e}$
 and $\muet = \langle \tau'\rangle$.
Then the finite $\ell'$--group $\muet$ has a
canonical homotopical action on $BG$ 
such that the homotopy fixed point space $BG^{h\muet}$ is
a connected $\ell$--compact group (semisimple or simply connected if $BG$ is),
and there exists a homotopy equivalence
\begin{equation*}%
\BtGq \xto{\ \homot\ } (BG^{h\muet})(q').
\end{equation*}%
}

\begin{thrm}[Untwisting theorem]\label{thm:untwisting-intro} 
\untwistingThmStatement
\end{thrm}
We will deduce this result in Appendix~\ref{subapp:untwisting} from  the more general Theorem~\ref{thm:class-lcglie}, which also gives a reduction when $\tau$ is not of coprime order, and addresses the question of how the \emph{homotopy type} of $\BtGq$ depends on e.g.\ $q$, generalizing \cite{BM07,BrotoMoellerOliver}.
The ``untwisting'' procedure above is analogous to the
$\Phi_e$--theory of finite groups of Lie
type of Brou\'e--Malle--Michel \cite{BM92,malle98}, but has the
advantage that the untwisted group $BG^{h\muet}$ is again an actual $\ell$--compact
group, rather than a mythical ``Spets''.  Theorem~\ref{thm:root-datum}
describes the root datum of  $BG^{h\muet}$, which in
general does not lift to a root datum over $\Z$ even if the root datum of
$BG$ does.

Using Theorem~\ref{thm:untwisting-intro} to consider the
$\psi^{q'}$--action on $BG^{h\muet}$, we  obtain a fibration
sequence
\begin{equation}\label{eq:fibseqtwist}
G^{h\muet} \xrightarrow{\ i\ }  \BtGq \longto BG^{h\muet}.
\end{equation}
With the notation of Theorem~\ref{thm:untwisting-intro},
we can now make the following definition.
\begin{defn}[{$[G^{h\muet}]$}--fundamental class]
\label{def:ghmuetfunclass}
We say that $\BtGq$ has a \emph{$[G^{h\muet}]$--fundamental class} if
the map  $H_d( G^{h\muet};\F_\ell) \xrightarrow{i_*} H_d( \BtGq;\F_\ell)$ is
non-trivial for $d$ the dimension of the $\ell$--compact
group $BG^{h\muet}$.
\end{defn}

We note that when $BG$ is semisimple, 
the cohomology $H^\ast(\BtGq;\F_\ell)$ acquires 
via the homotopy equivalence of Theorem~\ref{thm:untwisting-intro}
the structure of an $\bbH^\ast(LBG^{h\muet};\F_\ell)$--module,
and that in view of Theorem~\ref{thm:strtoptezukacrit}
this module structure is free of rank $1$ precisely 
when $\BtGq$ has a $[G^{h\muet}]$--fundamental class.

By applying Theorem~\ref{thm:tezukasubgrp} to the twisted version
$BG^{h\muet}$, we immediately get that twisted fundamental classes
always exist as long as $q'-1$ is sufficiently divisible by $\ell$.
\begin{cor} \label{cor:tezukasubgrp} 
Let $BG$ be a semisimple $\ell$--compact group and let $\tau \in\Out(BG)\isom \Out(\bbD_G)$ be
an element of finite order prime to $\ell$.
Then $\BtGq$ has a 
$[G^{h\muet}]$--fundamental class as long as
$q'-1$ is divisible by a sufficiently high power of 
$\ell$ where $q'$ is as in Theorem~\ref{thm:untwisting-intro}.
\end{cor}

Again, an explicit bound for what counts as ``sufficiently high power of $\ell$''
in Corollary~\ref{cor:tezukasubgrp}
is implied by Theorem~\ref{thm:ss-comparison2}.

With the twisting in place, we are now able to state what is in some
sense our most general result. It uses the classification of
$\ell$--compact groups and their automorphisms along with looking at
individual cases to conclude that for simply
connected groups, away from a few
potential exceptions, twisted fundamental classes always exist.
Moreover, the result shows that these potential exceptions can be eliminated
by verifying the existence of fundamental classes in only a few specific 
cases---a task that may eventually be feasible, for example,
with the aid of a computer.

\begin{Th}
\label{thm:examples}
Suppose $BG$ is a simply connected $\ell$--compact group, $q\in \Z_\ell^\times$,
and $\tau \in\Out(BG)$ is an element of finite order prime to $\ell$.
Then
\begin{enumerate}[(i)]
\item\label{it:fcexistwithexclusions}
	 $\BtGq$ has a $[G^{h\muet}]$--fundamental class except possibly in the 
	following cases: $\ell = 5$ and $G$ contains an $E_8$--summand;
	$\ell=3$ and $G$ contains an $F_4$ or an $E_i$--summand for $i=6$, $7$, or $8$;
	and $\ell=2$ and $G$ contains an $E_i$--summand for $i=6$, $7$, or $8$.	
\item\label{it:liftexclusions}
	The restrictions placed on $BG$ in part (\ref{it:fcexistwithexclusions}) 
	are unnecessary	when
    \begin{enumerate}
    \item 	
    	$\ell=5$ if $BE_8(11)\fivecom$ turns out to have a 
    	$[(E_8)\fivecom]$--fundamental class;
    \item
    	$\ell=3$ if
    	$BF_4(7)\threecom$ and $BE_i(7)\threecom$ turn out to have
    	$[(F_4)\threecom]$-- and $[(E_i)\threecom]$--fundamental classes, respectively,
    	for $i=6$, $7$, and $8$; and
    \item
    	$\ell=2$ if
    	$BE_6(3)\twocom$, $BE_6(5)\twocom$, $BE_7(5)\twocom$, and $BE_8(5)\twocom$ turn out to have
        $[(E_6)\twocom]$, $[(E_6)\twocom]$, $[(E_7)\twocom]$, and 
        $[(E_8)\twocom]$--fundamental classes, respectively. 
    \end{enumerate}
\end{enumerate}
\end{Th}
We prove Theorem~\ref{thm:examples} in Section~\ref{subsec:putittogether}.
\subsubsection*{Isomorphisms preserving ring structure and Steenrod operations}
In view of Theorem~\ref{thm:strtoptezukacrit}, the existence of a 
$[G^{h\muet}]$--fundamental class as in Theorem~\ref{thm:examples}
implies that for a suitably chosen class $x \in H^\ast(\BtGq;\F_\ell)$, 
the map 
\begin{equation}\label{map:stringprodwithx}
- \stringprod x 
\co 
H^\ast(LBG^{h\muet};\F_\ell)
\longto
H^\ast(\BtGq;\F_\ell)
\end{equation}
induced by string module multiplication  by $x$ is a ring isomorphism
up to specified filtrations. The final result of this introduction 
shows that in the situation of 
Theorem~\ref{thm:examples}(\ref{it:fcexistwithexclusions}),
at least under some additional assumptions for $\ell=2$,
it is possible to strengthen this result and 
choose the class $x$ so that \eqref{map:stringprodwithx}
is an honest ring isomorphism which furthermore commutes with most Steenrod operations.
Write $\calA_\ell$ for the mod $\ell$ Steenrod algebra, and let $\calA' \subset \calA_\ell$
denote the subalgebra of $\calA_\ell$ generated by the Steenrod reduced $\ell$--th power 
operations when $\ell$ is odd and all of $\calA_2$ when $\ell=2$.
Then we have 

\begin{Th} \label{thm:polyexamples} 
Suppose $BG$ is a simply connected $\ell$--compact group, $q\in \Z_\ell^\times$,
and $\tau \in\Out(BG)$ is an element of finite order prime to $\ell$.
Assume that we are away from the cases excluded in 
Theorem~\ref{thm:examples}(\ref{it:fcexistwithexclusions}),
and in the case $\ell=2$ assume furthermore that 
$4 | (q-1)$ and that $G$ has no $\Spin(n)$--summands for $n \geq 10$.
Then there exists an element $x \in H^d(\BtGq;\F_\ell)$ such that
the map \eqref{map:stringprodwithx}
induced by string module multiplication by $x$ 
is an isomorphism of $H^*(BG;\F_\ell)$--algebras and
$\calA'$--modules.
\end{Th}
The proof of Theorem~\ref{thm:polyexamples} 
is given in Section~\ref{subsec:polysimplyconn2cptgrps}.
The assumption that $4 | (q-1)$ when $\ell=2$ is necessary, essentially
because $\Z/4$ and $\Z/2$ do not have the same
mod $2$ cohomology ring 
(see Example~\ref{ex:slnfq} for elaboration), but the 
restriction away from the cases excluded in 
Theorem~\ref{thm:examples}(\ref{it:fcexistwithexclusions})
and cases with $\Spin(n)$--summands may not be.
If $H^*(BG;\F_\ell)$ is concentrated in even degrees, then the class
$x$ in Theorem~\ref{thm:polyexamples} is unique and given by the product
of the exterior generators of $H^\ast(\BtGq;\F_\ell)$, up to a unit (see
Theorem~\ref{thm:polytezukaelaboration}). 
See also Remark~\ref{rem:uniqueclass} for further discussion.

Theorems~\ref{thm:examples} and \ref{thm:polyexamples}  and can be seen as generalizing a number of
previous results in the literature where both $H^*(BG(q))$ and $H^*(LBG)$
were calculated to various degrees of precision and amount of
structure, and subsequently observed to coincide. See e.g.\ 
\cite{kleinerman82, milgram96, KMT00,  Kameko-pb,  grbic06, KK10, KMN06,KTY12}
for previous work on polynomial cases, and
\cite{Kameko-spin} for the spin groups, where an isomorphism of graded
abelian groups
was established---it was noticing this last
paper in an arXiv listing which originally alerted
us to Tezuka's question.

For emphasis, we now state explicitly several questions raised by the preceding discussion.

\begin{question} \label{qu:thequestion} \mbox{ }
\begin{enumerate}
\item \label{q1} 
    Does $BG^{h\sigma}$ have a $[G]$--fundamental
    class if and only if
    $\sigma$ acts as the identity on $H^*(BG)$?
\item \label{q2} 
	Can the exclusions in Theorem~\ref{thm:examples}(\ref{it:fcexistwithexclusions}) 
	be avoided, i.e., do the fundamental classes listed in 
	Theorem~\ref{thm:examples}(\ref{it:liftexclusions}) exist?
	Furthermore, can the assumption that $BG$ is simply connected be dropped
	in Theorem~\ref{thm:examples}?
\item \label{q3} 
	Is there for every connected $\ell$--compact group
	a general algebraic or geometric construction of a dual 
	$x \in H^*( \BtGq;\F_\ell)$ for the $[G^{h\muet}]$--fundamental class
	such that multiplication by $x$ induces a ring isomorphism as in 
	Theorem~\ref{thm:polyexamples}? Does the property that $x$
	induces a ring isomorphism characterize it uniquely up to an
	automorphism of $\BtGq$?
\end{enumerate}
\end{question}
Note that a positive answer to \eqref{q2}  would imply that, for $q
\equiv 1\ \mathrm{mod}\ \ell$,  the map $\psi^q$
 induces the identity map on $H^*(B\G(\bar \F_q))$,  that $H^*(B\G(\bar
  \F_q)) \to H^*(B\G(\F_q))$ is injective, 
  and that every elementary abelian $\ell$--subgroup of $\G(\bar \F_q)$ is
  conjugate to a subgroup
in $\G(\F_q)$ whenever $\G$ is a connected split reductive group
scheme over the integers; 
whether these properties hold does
not seem to be known in full generality. Note also that in the 
calculational literature it is common to work under a simply connectedness 
or simplicity assumption, and Tezuka, in his original question \cite{tezuka98},
seems implicitly to make those assumptions, though their roles seem
unclear.

\subsubsection*{Outline of the paper}
Section~\ref{sec:notation} establishes notation and conventions.
The string product and the string module
structure are constructed in Section~\ref{sec:products}.
In Section~\ref{sec:spectralsequences} we set up the Serre
spectral sequences and show how they interact with the product
structures, and use this to prove Theorem~\ref{thm:mainresult}. In
Section~\ref{sec:tezuka} we prove
Theorem~\ref{thm:strtoptezukacrit}, and start our investigation on when fundamental classes exist. Section~\ref{sec:funclass2}  proves that they exist generically by establishing 
Theorem~\ref{thm:tezukasubgrp}. Then in Section~\ref{sec:funclass1} we use untwisting and case-by-case arguments to prove Theorem~\ref{thm:examples}. Finally, in Section~\ref{sec:algiso}, we examine when the isomorphism 
arising from the string module structure
can be upgraded to a ring isomorphism with extra structure, establishing Theorem~\ref{thm:polyexamples}. 
Appendix~\ref{app:pcg} lays out the relationship between $\Z_\ell$--root data, $\ell$--compact groups and the
$\ell$--local theory of finite groups of Lie type needed in the paper
proper, and in particular contains a proof
Theorem~\ref{thm:untwisting-intro}. The appendix provides refinements of the existing literature throughout, and may also be of independent interest.

\subsubsection*{Acknowledgements}  We would like to thank Andrea Bianchi, Masaki Kameko, Katsuhiko Kuribayashi, and Luc Menichi for exchange of calculations on string topology of classifying spaces, and Kasper Andersen, Roman Bezrukavnikov, Radha Kessar, Richard Lyons, Gunter Malle, Bob Oliver, Rapha\"el Rouquier, and Ron Solomon for helpful conversations and correspondence about finite groups of Lie type and their subgroups. Finally we would like to thank Nathalie Wahl for pointing the first-named author to the second-named author during the initial stages of this project.

\smallskip

\section{Notation and conventions}\label{sec:notation}
As explained in the introduction, we will work in the setting of 
$\ell$--compact groups for a fixed prime $\ell$. A summary of this theory is contained Section~\ref{subsec:pcg-recol}.
Throughout the paper, $BG$ denotes an $\ell$--compact group,
and \emph{unless specified otherwise, $BG$ is furthermore assumed to be semisimple, i.e., 
that $\pi_1(BG) =0$ and $\pi_2(BG)$ is finite}; see Remark~\ref{rk:semisimplicityassumptionnontechnical}.
We set $G = \Omega BG$ and write $d$ for the dimension of $BG$ as an
$\ell$--compact group $BG$, i.e., for the largest $n$ such that $H^n(G;\F_\ell) \neq 0$. 
We also assume that $G$ has been equipped 
with a fixed $\F_\ell$--orientation, i.e., an
identification
\begin{equation} \label{as:gorientation}  
H^d(G;\F_\ell) \xto{\ \isom\ }  \F_\ell.
\end{equation}Unless indicated otherwise, all
homology and cohomology is with $\F_\ell$--coefficients.
As in the introduction, we write $\bbH^\ast$
for cohomology groups shifted by $d$, so that $\bbH^k(X) = H^{k+d}(X)$.
The shifted and unshifted cohomology groups are
related by graded maps
\begin{equation}
\label{eq:degreeshiftmaps} 
	s^d \colon \bbH^\ast \longto H^\ast
	\qquad\text{and}\qquad
	s^{-d} \colon H^\ast \longto \bbH^\ast
\end{equation}
of degree $d$ and $-d$, respectively, sending each element to
itself, but now considered as having a different degree. Remembering
the degrees of these maps is important for ensuring that products
obey the Koszul sign rule.
When working with topological spaces, we always work within the 
category $\calT$ of compactly generated weak Hausdorff spaces.

In this paper we will make repeated use of umkehr maps, which have a long history and have been developed in different settings. 
In Sections~\ref{subsec:secondconstruction}---\ref{sec:spectralsequences}, we rely on \cite{umkehr-maps} for results about these maps, and we hope that the reader will benefit from a single reference 
adapted to the needs of the present paper.
We use the shorthand Theorem~U.X.Y to refer to Theorem X.Y of \cite{umkehr-maps}. Section~\ref{subsubsec:umkehrmapssummary} contains a summary of the 
main points of \cite{umkehr-maps} needed in the present work. 

\section{Construction of the string products: Proof of Theorem~\ref{thm:mainresult}, part 1}
\label{sec:products}

In this section we construct the 
string product on $\bbH^\ast(LBG)$
and the string module structure on $\bbH^\ast (BG^{h\sigma})$.
It turns out to be convenient to define these structures structures in terms
of a more general pairing
\begin{equation}
\label{eq:pairing}
	\stringprod
	\co
	\bbH^\ast P(g,h) \tensor \bbH^\ast P(f,g) 
	\longto 
	\bbH^\ast P(f,h)
\end{equation}
where $P(u,v)$ is a space which for suitable choices of $u$ and $v$
recovers the spaces
$LBG$, $BG^{h\sigma}$ and $G$, as we will observe in Proposition~\ref{prop:specialcasesofpathspaceconstr}.
Section~\ref{subsec:resultsonproducts}
will introduce this pairing along with many of its properties,
and the remaining subsections are devoted to the actual construction
of the string pairing and the verification of the asserted 
properties.

\subsection{The string pairing and its properties}
\label{subsec:resultsonproducts}

\subsubsection{The string pairing 
$\stringprod
	\co
	\bbH^\ast P(g,h) \tensor \bbH^\ast P(f,g) 
	\to 
	\bbH^\ast P(f,h)
$}
\label{subsubsec:thepairingbullet}

\begin{defn}[The space $P(f,g)$]
\label{def:pfg}
For a fixed space $B$ and maps 
$f,g\co B \to BG$, define  $P(f,g)$, \emph{the space of paths in
  $BG$ from $f$ to $g$},
  and the map $\pi_{f,g}\colon P(f,g) \to B$
via the  pullback diagram
$$ 
\xymatrix{
    P(f,g)  
    \ar[r]
    \ar[d]_{\pi_{f,g}}  
    \pb 
    & 
    BG^I \ar[d]^{(\ev_0,\ev_1)} 
    \\
    B \ar[r]^-{(f,g)}
    & 
    BG \times BG
}
$$
where the maps $\ev_0$ and $\ev_1$ are evaluation at $0$ and $1$, respectively.
\end{defn}
Explicitly,
\[
	P(f,g) 
	= 
	\big\{
		(b,\gamma)\in B\times BG^I
		\,|\, 
		\gamma(0) = f(b), \gamma(1) = g(b)
	\big\},
\]
with $\pi_{f,g}\co P(f,g) \to B$ 
given by projection onto the first coordinate. We may picture a point in $P(f,g)$ as follows:
\[
	\left(
		b, 
		\begin{tikzpicture}
			[
				baseline=0cm,
				endpoint/.style={
					circle,
					draw=black,
					fill=black,
					inner sep=0pt,
					minimum size=2.6pt
				},
				label/.style={
					font=\scriptsize
				},
				->-/.style={
					decoration={
						markings,
						mark=at position .5 with {\arrow{>}}
					},
					postaction={decorate}
				}
			]
			\node [endpoint] (start) at (0,0) {};
			\node [endpoint] (end) at (1.5,0) {};
			\node [label] at (start) [above] {$f(b)$};
			\node [label] at (end) [above] {$g(b)$};
			\draw [->-,thick] (start) .. controls (0.5,0.2) and (1.0,-0.26) .. (end) node [label,above,midway] {$\gamma$};
		\end{tikzpicture}
	\right).
\]

Over the course of Section~\ref{sec:products},
we will give three different, but equivalent, constructions
of the string pairing $\stringprod$ of equation~\eqref{eq:pairing},
each of them useful for proving different properties of the string pairing.
While the constructions differ in details, they are all based
on the same fundamental idea, which we now explain.
Given $f, g, h\co B \to BG$, 
the spaces $P(f,g)$, $P(g,h)$ and $P(f,h)$
all fit into a commutative diagram
\begin{equation}
\label{diag:pushpull} 
\vcenter{\xymatrix@C-1em{
	P(g,h) \times P(f,g) 
	\ar[d]_{\pi_{g,h}\times \pi_{f,g}} 
	&& \pbop 
	P(f,g,h)
	\ar[d]^(0.4){\pi_{f,g,h}}
	\ar[ll]_-{\mathrm{split}}
	\ar[rr]^-{\mathrm{concat}}
	&&
	P(f,h)
	\ar[dll]^(0.4){\pi_{f,h}}
	\\
	B\times B
	&&
	B
	\ar[ll]_{\Delta}
}}
\end{equation}
Here the square on the left is a pullback square,
so that 	
\[
	P(f,g,h) 
	= 
	\big\{
    	(b,\gamma_2,\gamma_1) \in B \times BG^I \times BG^I 
    	\,|\,
    	\gamma_1(0) = f(b), 
    	\gamma_1(1) = \gamma_2(0)=g(b), 
    	\gamma_2(1) = h(b)
	\big\},
\]
the map 
$\pi_{f,g,h}$ is the projection onto $B$, and the map 
`$\mathrm{split}$' is given by
\[
	(b,\gamma_2,\gamma_1) \longmapsto ((b,\gamma_2),(b,\gamma_1)).
\]
We may picture a point in $P(f,g,h)$ as follows:
\[
	\left(
		b, 
		\begin{tikzpicture}
			[
				baseline=0cm,
				endpoint/.style={
					circle,
					draw=black,
					fill=black,
					inner sep=0pt,
					minimum size=2.6pt
				},
				label/.style={
					font=\scriptsize
				},
				->-/.style={
					decoration={
						markings,
						mark=at position .5 with {\arrow{>}}
					},
					postaction={decorate}
				}
			]
			\node [endpoint] (start) at (0,0) {};
			\node [endpoint] (mid) at (1.5,0) {};
			\node [endpoint] (end) at (3,0) {};
			\node [label] at (start) [above] {$f(b)$};
			\node [label] at (mid) [above] {$g(b)$};
			\node [label] at (end) [above] {$h(b)$};
			\draw [->-,thick] (start) .. controls (0.5,-0.2) and (1.0,-0.12) .. (mid) node [label,above,midway] {$\gamma_1$};
			\draw [->-,thick] (mid) .. controls (2.0,0.12) and (2.5,0.15) .. (end) node [label,above,midway] {$\gamma_2$};			
		\end{tikzpicture}
	\right).
\]
The map `$\mathrm{concat}$' on the right
in diagram \eqref{diag:pushpull}
is given by concatenation of paths:
\[
	\mathrm{concat}
	\co
	(b,\gamma_2,\gamma_1) \longmapsto (b,\gamma_2 \star \gamma_1).
\]
The string pairing $\stringprod$ is now obtained by a push--pull construction 
in the top row of 
diagram \eqref{diag:pushpull}. 
Modulo the degree shift maps $s^d\colon \bbH^\ast \to H^\ast$ 
and $s^{-d} \colon H^\ast \to \bbH^\ast$
of equation~\eqref{eq:degreeshiftmaps}
and appropriate signs following from the Koszul sign rule,
this means that the string pairing $\stringprod$ 
is given by the composite
\begin{equation}
\label{eq:pairingdescr}
	H^\ast P(g,h) \tensor H^\ast P(f,g) 
	\xto{\,\times\,}
	H^\ast (P(g,h) \times P(f,g))
	\xto{\, \mathrm{split}^\ast\, }
	H^\ast P(f,g,h)
	\xto{\, \mathrm{concat}_!\, }
	H^{\ast-d} P(f,h)
\end{equation}
of the cross product, the induced map $\mathrm{split}^\ast$
associated to $\mathrm{split}$, and an umkehr map $\mathrm{concat}_!$
associated to $\mathrm{concat}$. 

The nontrivial ingredient in the above construction is, of course,
the umkehr map $\mathrm{concat}_!$, and 
the essential differences between
the constructions of the string pairing $\stringprod$ lie in how this map 
is defined.
In the first construction, to construct $\mathrm{concat}_!$,
we make use of the fact that the homotopy fibre of $\mathrm{concat}$
is homotopy equivalent to $G$, and hence satisfies a finiteness condition.
On the other hand, in the second and third constructions,
we rely on the fact that the fibres of $\pi_{f,g,h}$ and $\pi_{f,h}$
are homotopy equivalent to $G\times G$ and $G$, respectively,
and therefore satisfy suitable finiteness conditions.

In order to state the main properties of the string pairing $\stringprod$,
we first note that  $\bbH^\ast P(f,g)$  carries a natural $H^\ast(B)$--module
structure.
\begin{defn}[$H^\ast(B)$--module structure on $\bbH^\ast P(f,g)$]
\label{def:cupmodstr}
Given $f,g\colon B \to BG$, equip $\bbH^\ast P(f,g)$ 
with a $(H^\ast(B),\cupprod)$--module structure 
by first giving $H^\ast P(f,g)$
the $H^\ast(B)$--module structure 
induced by $\pi_{f,g}$, so that
\[
	ax = \pi_{f,g}^\ast(a)\cupprod x \in H^\ast P(f,g)
\]
for $a \in H^\ast(B)$ and $x\in H^\ast P(f,g)$,
and by then 
using the degree shift map $s^{-d}$ 
to define the $H^\ast(B)$--module structure
on $\bbH^\ast P(f,g)$ by the formula
\begin{equation}
\label{eq:cupmodstr}
	as^{-d}(x) = (-1)^{d\deg(a)}s^{-d}(ax)
	\in \bbH^\ast P(f,g)
\end{equation}
for $a \in H^\ast(B)$ and $s^{-d}(x) \in \bbH^\ast P(f,g)$.
Notice that the sign in \eqref{eq:cupmodstr}
follows the Koszul sign rule.
\end{defn}

The following theorem summarizes the principal properties of our pairing.
\begin{thrm}[Properties of the string pairing $\stringprod$]
\label{thm:pairings4}
Suppose $BG$ is a semisimple $\ell$--compact group.
Given a space $B$, the string pairings
\begin{equation}
\label{eq:pairingbullet}
	\stringprod
	\co
	\bbH^\ast P(g,h) \tensor \bbH^\ast P(f,g) 
	\longto 
	\bbH^\ast P(f,h).
\end{equation}
for maps $f,g,h\colon B \to BG$ 
are $H^\ast(B)$--bilinear and
define the composition law
in a category enriched in graded $H^\ast(B)$--modules
whose objects are maps $f\colon B \to BG$ and whose morphisms
from $f\colon B\to BG$ to $g\colon B \to BG$ 
are given by $\bbH^\ast P(f,g)$. 
\end{thrm}
The proof of Theorem~\ref{thm:pairings4} is given
at the beginning of Section~\ref{subsec:furtherpropertiesofthepairing}.
Given a map $f\colon B\to BG$, we write $\bbOne$ or $\bbOne_f$
for the identity element in $\bbH^\ast P(f,f)$
in the enriched category of Theorem~\ref{thm:pairings4}.
The string pairing $\stringprod$ and the identity elements $\bbOne_f$ %
will depend on the chosen 
orientation of $G$, fixed in \eqref{as:gorientation}, 
but only in a rather mild way: changing the orientation 
replaces these data by some nonzero scalar
multiples. 

\begin{rem}
\label{rk:unitdescription}
An inspection of the proof of Theorem~\ref{thm:pairings4},
and in particular the construction of the string pairing $\stringprod$
given in Section~\ref{subsec:secondconstruction}, reveals that 
modulo degree shifts, the unit $\bbOne_f$ for a map $f\colon B \to BG$
is given by the element $s_!(1) \in H^d P(f,f)$
where $1 \in H^0 (B)$ is the unit with respect to cup product
and 
\[
	s_! \colon H^\ast(B) \longto H^{\ast+d} P(f,f)
\]
is an umkehr map induced by the section 
\[
	s \colon B \longto P(f,f), \qquad b \mapsto (b,c_{f(b)})
\]
of the projection $\pi_{f,f} \colon P(f,f) \to B$ where 
$c_{f(b)}$ denotes the constant path at the point $f(b)$.
\end{rem}

Theorem~\ref{thm:pairings4} has the following immediate corollary, 
which we will specialize 
to the desired string product and string module
structures in Definition~\ref{def:stringprodandmod} below.

\begin{cor}
\label{cor:bulletalgandmodstr}
\mbox{}
\begin{enumerate}[(i)]
\item\label{it:bulletalgstr}
	For every $f\colon B \to BG$, 
	the string pairing $\stringprod$ makes $\bbH^\ast P(f,f)$ into a graded ring with unit 
	$\bbOne_f$ (described in Remark~\ref{rk:unitdescription}).
	Moreover, the map 
    \begin{equation}
	   \label{eq:iotaf}
		\iota = \iota_f
		\colon
		H^\ast(B) \longto \bbH^\ast P(f,f), 
		\qquad
		a \longmapsto a\bbOne_f.
    \end{equation}
    is a ring homomorphism making $(\bbH^\ast P(f,f),\stringprod)$
	into a $(H^\ast(B),\cup)$--algebra.
\item\label{it:bulletmodstr}
	For every $f,g\colon B \to BG$, 
	the string pairing $\stringprod$ makes $\bbH^\ast P(f,g)$
	into a  left $(\bbH^\ast P(g,g),\stringprod)$--module.
    This module structure extends, via the map $\iota_g$, the $H^\ast(B)$--module
	structure on $\bbH^\ast P(f,g)$.
	\qed
\end{enumerate}
\end{cor}
Later, in Proposition~\ref{prop:augmentation}, we will show
that the $H^\ast(B)$--algebra structure on $\bbH^\ast P(f,f)$
admits an augmentation 
\begin{equation}
\label{eq:rho}
	\rho =\rho_f \colon \bbH^\ast P(f,f)\longto H^\ast(B).
\end{equation}
In particular, the map $\iota_f$ of equation \eqref{eq:iotaf}
is always a monomorphism. 
The ring structure 
on $\bbH^\ast P(f,f)$
is in general noncommutative (see Remark~\ref{rk:noncomm}), 
but turns out to be commutative
in the case of $\bbH^*(LBG)$ (see Theorem~\ref{thm:cmcomparison2}).

We conclude the discussion of the general properties of the string pairing
by explaining how the pairing
$\stringprod$ of \eqref{eq:pairingbullet}
behaves when the space $B$ varies.
First note that the formation of the space $P(f,g)$ is well behaved 
under base change in $B$: given a
map $\phi\co A \to B$, we bave a pullback
\begin{equation}
\label{sq:inducedmap}
\vcenter{\xymatrix{
    P(f\phi,g\phi)
    \ar[r]^{\bar{\phi}}
    \ar[d]_{\pi_{f\phi,g\phi}} \pb
    &
    P(f,g)
    \ar[d]^{\pi_{f,g}}
    \\
    A
    \ar[r]^{\phi}
    &
    B
}}
\end{equation}
where $\bar{\phi}$ is the map $\bar{\phi}(a,\gamma) =(\phi(a),\gamma)$.
Write $F_\phi$ for the induced map
\begin{equation}
\label{eq:fphidef}
	F_\phi = \bar{\phi}^\ast\co \bbH^\ast P(f,g) \longto \bbH^\ast P(f\phi,g\phi) 
\end{equation}
on cohomology groups. We have
\begin{thrm}
\label{thm:functoriality}
Suppose $BG$ is a semisimple $\ell$--compact group.
\begin{enumerate}[(i)]
\item\label{it:functoriality2}
The category
of Theorem~\ref{thm:pairings4},
	viewed as enriched in graded $\F_\ell$--modules,
	depends functorially on the space $B$: 
	given a map $\phi \colon A \to B$,
	the maps $F_\phi$ of equation~\eqref{eq:fphidef}
	for varying $f,g \colon B\to BG$
	define a functor of categories enriched in graded $\F_\ell$--modules
	which on objects is given by the assignment $f\mapsto f\phi$.
\item\label{it:functorialityandmodstr}
	The functor of part (\ref{it:functoriality2}) is compatible with the $H^\ast(B)$-- and 
	$H^\ast(A)$--module structures
	on the hom-objects of its source and target in the sense that
	\[
		F_\phi(bx) = \phi^\ast(b) F_\phi(x)
	\]
	for all $b\in H^\ast(B)$ and $x\in \bbH^\ast P(f,g)$.
\end{enumerate}
\end{thrm}
The proof of 
Theorem~\ref{thm:functoriality} is
given at the beginning of Section~\ref{subsec:furtherpropertiesofthepairing}
together with the proof of Theorem~\ref{thm:pairings4}.
Theorem~\ref{thm:functoriality}(\ref{it:functoriality2})
in particular implies
that the string pairing $\stringprod$ of equation \eqref{eq:pairingbullet}
is homotopy invariant in $f$, $g$, and $h$. 
See Proposition~\ref{prop:homotopyinvariance}.

\subsubsection{The string product on $\bbH^*(LBG)$ and 
the string module structure on $\bbH^*(BG^{h\sigma})$}
We will now specialize Corollary~\ref{cor:bulletalgandmodstr}
to the string product on $\bbH^*(LBG)$ and 
the string module structure on $\bbH^*(BG^{h\sigma})$.
It is straightforward to 
verify the following result.
\begin{prop}
\label{prop:specialcasesofpathspaceconstr}
The  projection map $P(f,g) \to BG^I$,  $(b,\gamma) \mapsto \gamma$,
induces homeomorphisms
\begin{multicols}{2}
\begin{enumerate}[(i)]
\item\label{it:lbg}
  $\begin{tikzcd}
    P(\id_{BG},\id_{BG}) \arrow{r}{\homeom}
    \arrow[d, "{\pi_{\id_{BG},\id_{BG}}}"'] &LBG^{\mathclap{\phantom{I}}} \arrow{d}{\ev_1} \\
    BG \arrow{r}{=}& BG
  \end{tikzcd}$
\item\label{it:bghsigma}
  $\begin{tikzcd}
    P(\sigma,\id_{BG}) \arrow{r}{\homeom}
    \arrow[d, "{\pi_{\sigma,\id_{BG}}}"'] &BG^{h\sigma} \arrow{d}{\ev_1} \\
    BG \arrow{r}{=}& BG
  \end{tikzcd}$ \\
      \item \label{it:pathspaceconstrbgi}
  $\begin{tikzcd}
    P(\pi_1,\pi_2) \arrow{r}{\homeom}
    \arrow[d, "{\pi_{\pi_1,\pi_2}}"'] & BG^I \arrow{d}{(\ev_0,\ev_1)} \\
    BG \times BG  \arrow{r}{=}& BG \times BG
  \end{tikzcd}$\\

  \item \label{it:ptloopspace}
  $\begin{tikzcd}
    P(c,c) \arrow{r}{\homeom}
    \arrow[d, "{\pi_{c,c}}"'] & \loops BG \arrow{d} \\
    \pt \arrow{r}{=}& \pt
  \end{tikzcd}$

\end{enumerate}
\end{multicols}
\noindent
where $\pi_1,\pi_2\colon BG\times BG \to BG$ are the projections,
$c\colon \pt \to BG$ is the basepoint inclusion, and 
$\sigma \colon BG \to BG$ is any self-map of $BG$.
\qed      
\end{prop}
We will use the 
homeomorphisms of Proposition~\ref{prop:specialcasesofpathspaceconstr}
freely throughout the paper to identify the spaces in question.

\begin{defn}
\label{def:stringprodandmod}
The \emph{string product}
on $\bbH^\ast(LBG)$ is obtained by taking 
$f=\id_{BG}$ in 
Corollary~\ref{cor:bulletalgandmodstr}(\ref{it:bulletalgstr}), and using the identification of
Proposition~\ref{prop:specialcasesofpathspaceconstr}(\ref{it:lbg}).
The \emph{string module structure} on
$\bbH^\ast(BG^{h\sigma})$
over 
$\bbH^\ast(LBG)$
is obtained, via
the identifications in
Proposition~\ref{prop:specialcasesofpathspaceconstr}(\ref{it:lbg}) and (\ref{it:bghsigma}),
by taking
$f=\sigma$ and $g=\id_{BG}$ 
in Corollary~\ref{cor:bulletalgandmodstr}(\ref{it:bulletmodstr}).
Finally, we define the 
string module structure over $\bbH^\ast(LBG)$
on the unshifted cohomology $H^\ast(BG^{h\sigma})$
by setting
\begin{equation}
\label{eq:unshiftedmodstr}
	s^{-d}(x\stringprod y) 
	= 
	(-1)^{d \deg(x)} x \stringprod s^{-d}(y)
	\in 
	\bbH^*(BG^{h\sigma})
\end{equation}
for all 
$x\in \bbH^\ast(LBG)$ and 
$y\in H^\ast(BG^{h\sigma})$
so that, taking into account the Koszul sign rule,
the degree shift map $s^{-d}$ 
of equation \eqref{eq:degreeshiftmaps}
is $\bbH^\ast(LBG)$--linear.
\end{defn}

\begin{thrm}
\label{thm:cmcomparison2} 
Suppose $BG$ is a semisimple $\ell$--compact group.
The string product $\stringprod$ on $\bbH^\ast(LBG)$ 
makes $\bbH^\ast(LBG)$ into a commutative 
unital
$H^\ast(BG)$--algebra
and agrees with the product
$\odot$ of Chataur and Menichi \cite[p.~848]{KM19}
on $\bbH^\ast(LBG)$.
\end{thrm}

The proof of Theorem~\ref{thm:cmcomparison2}
is given at the beginning of Section~\ref{subsec:furtherpropertiesofthepairing}.
Part of Theorem~\ref{thm:cmcomparison2} is immediate from
Corollary~\ref{cor:bulletalgandmodstr}(\ref{it:bulletalgstr}),
which asserts that the product $\stringprod$ makes $\bbH^\ast P(f,f)$ 
into a unital $H^\ast(BG)$--algebra
for any $f$. Proving the commutativity of $\stringprod$ on $\bbH^\ast(LBG)$
requires considerations specific to $LBG$, however, as the following remark
demonstrates.

\begin{rem}[On (non-)commutativity of $\stringprod$]
	\label{rk:noncomm}
	For general $f$, the product $\stringprod$ on $\bbH^\ast P(f,f)$ 
    fails to be commutative.
    This follows from 
    Proposition~\ref{prop:pontryaginproduct}
    below together with the non-commutativity of the Pontryagin product
    on $H_\ast(G)$ in general. Indeed for $\ell$ odd, $H_*(G)$ is
    always non-commutative when $H_*(G;\Z_\ell)$ has
    $\ell$--torsion, as follows from \cite[Thm.~1 and text
    after]{Browder68}, using as input that $H^*(\Omega G;\Z_\ell)$ is
    always torsion-free for $\ell$--compact groups \cite[text after Thm~1.5]{AG09}.
    Borel also showed in \cite[Thm.~16.4]{borel54} that
    $H_*(\Spin(10);\F_2)$ is non-commutative providing a $\ell =2$
    example.
Indeed, as far as we know it could be true for all primes $\ell$ that $H_*(G)$ is
commutative if and only if  $H^*(BG;\F_\ell)$ is a polynomial ring,
but we do not believe this is known in full generality. (See also \cite[Thm.~1.1]{Kane76}.)
\end{rem}

As this route is the quickest for us,
we will deduce the commutativity of $\stringprod$ on $\bbH^\ast(LBG)$
from the comparison with Chataur and Menichi's product $\odot$, which 
is already known to be commutative \cite[Cor.~B.3]{KM19}.
A more direct proof would also be possible, however.  The key features
of $LBG$ implying the desired commutativity are the self-homeomorphism of $LBG$
given by rotation of loops by 180 degrees and the fact that this homeomorphism
is homotopic to the identity map; cf.\ \cite[Pf.~of Prop.~3.4]{tamanoi2009cap}.

\begin{rem}[Unitality of $\stringprod$]
\label{rk:unitality}
The fact that the string product on $\bbH^\ast(LBG)$ 
admits a unit appears to have been previously established  
in special cases only. When $\ell=2$
and $G$ is a compact Lie group, the unitality of $\stringprod$ can be deduced from
\cite{HL15}. Moreover, Kuribayashi and Menichi have given a computational
proof of the unitality of the string product on $\bbH^\ast(LBG)$ when $H^\ast (BG)$
is a polynomial algebra \cite[Cor.~4.2 and Thm.~5.5]{KM19}, 
providing concrete formulas for the unit. 
When $H^\ast(BG)$ is a polynomial ring concentrated in even degrees, 
the unit turns out to be the degree shift of a product of the 
exterior classes in $H^\ast(LBG)$ under cup product. See \cite[Pf.~of Cor.~4.2]{KM19}.

Specializing Remark~\ref{rk:unitdescription}, we see that the unit 
$\bbOne \in \bbH^0(LBG)$ can for general $BG$ be described as the degree shift of the 
image of the unit $1 \in H^0(BG)$ under an umkehr map 
\[
	s_! \colon H^\ast(BG) \longto H^{\ast+d}(LBG)
\]
associated with the section $s\colon BG \to LBG$ given by constant loops.
The key ingredient in establishing the unitality of $\stringprod$ in general
is a theory of umkehr maps where both $s_!$ and the umkehr map $\mathrm{concat}_!$
needed in the construction of $\stringprod$ exist and fit together in the 
required way.

\end{rem}

\begin{rem}[{$[G]$--fundamental class for $LBG$}]
Let $i \colon G \incl LBG$ be the inclusion of the fibre of evaluation map 
$LBG \to BG$, $\alpha \mapsto \alpha(1)$.
By Theorem~\ref{thm:functoriality}, the induced map
$i^\ast \colon \bbH^\ast(LBG) \to \bbH^\ast(G)$
sends $\bbOne$ to $\bbOne$,
so the map $i^\ast \colon H^d(LBG) \to H^d(G)$ and 
hence the map $i_\ast \colon H_d(G) \to H_d(LBG)$ are nontrivial.
Thus $LBG = BG^{h \id}$ has a $[G]$--fundamental class in the sense of 
Definition~\ref{def:fundamentalclass}. Indeed, writing $[G] \in H_d(G)$
for the fundamental class, the unit $\bbOne \in \bbH^\ast(LBG)$ is dual to 
the class $i_\ast[G] \in H_\ast(LBG)$ in the sense that 
$\langle s^d(\bbOne), i_\ast[G] \rangle \neq 0$.
\end{rem}

As we noted earlier in
Proposition~\ref{prop:specialcasesofpathspaceconstr}\eqref{it:ptloopspace},  
the space $P(c,c)$ for $c: \pt \to BG$ the basepoint inclusion
recovers the based loop space $G=\loops BG$.
The resulting product on $\bbH^\ast(G)$ can be described
in more familiar terms:
\begin{prop}
\label{prop:pontryaginproduct}
There exists an isomorphism
$\bbH^\ast  (G) \isom  H_{-\ast}(G)$
under which the product $\stringprod$ on $\bbH^\ast(G)$
corresponds to the Pontryagin product
\[
	\mathrm{concat}_\ast 
	\co 
	H_\ast(G) \tensor H_\ast(G)
	\longto
	H_\ast(G)	
\]
on $H_\ast(G)$
where $\mathrm{concat} \co G\times G \to G$
is the concatenation map.
\end{prop}
The proof of 
Proposition~\ref{prop:pontryaginproduct} is given 
in Section~\ref{subsec:thirdperspective}.
The isomorphism 
$\bbH^\ast  (G) \isom  H_{-\ast}(G)$
in the proposition should be thought of as a manifestation 
of Poincaré duality. See Remark~\ref{rk:pd}.

To give the reader a computational example to keep in mind, 
we conclude by describing
the ring $(\bbH^\ast(LBG),\stringprod)$
explicitly in a common special case.
\begin{thrm}
\label{thm:hlbgcompintro}
Suppose $\ell$ is odd and $BG$ is a semisimple $\ell$--compact group such that 
$H^\ast(BG)$ is a polynomial algebra $H^\ast(BG) = \F_\ell[x_1,\ldots,x_n]$. 
Then there exist ring isomorphisms
$$ \bbH^*(LBG) \cong  H^*(BG) \otimes \bbH^*(G) \cong H^*(BG) \otimes H_{-*}(G)
\isom \F_\ell[x_1,\ldots,x_n] \tensor \Lambda(y_1,\ldots,y_n)$$
where $\deg(y_i) = -(\deg(x_i)-1)$.
Here the second isomorphims is induced by the Poincar{\'e} duality isomorphism of Proposition~\ref{prop:pontryaginproduct}.
\end{thrm}
We will prove Theorem~\ref{thm:hlbgcompintro} in
Section~\ref{subsec:polycase} by combining 
Proposition~\ref{prop:polycollapse},
Proposition~\ref{prop:pontryaginproduct},
and the compatibility of the string product 
with Serre spectral sequences
established in Section~\ref{sec:spectralsequences}.
We refer to \cite[Sections 4 and 5]{KM19} for
further computations of the string product on $\bbH^\ast(LBG)$ using different methods, 
including the $\ell=2$ case of Theorem~\ref{thm:hlbgcompintro}.
These computations also provide detailed information on the
relationship between the string product and 
cup product structures of $H^\ast(LBG)$.

\subsubsection{Outline of the rest of Section~\ref{sec:products}}
In the remainder of Section~\ref{sec:products}, we will 
construct the string pairing $\stringprod$ and 
prove the results stated in Section~\ref{subsec:resultsonproducts}
(with the exception of Theorem~\ref{thm:hlbgcompintro},
which will be proven in Section~\ref{subsec:polycase}).
Indeed, we will construct the string pairing $\stringprod$ in three
different ways, with each construction revealing different
properties of the pairing, and show that the constructions 
produce the same result.
Our first construction of the pairing, carried out in 
Section~\ref{subsec:firstconstruction},
is similar to Chataur and Menichi's 
construction of a string product on $\bbH^\ast(LBG)$,
and has the advantages of being simple,
being straightforward to compare with Chataur and Menichi's product 
(Proposition~\ref{prop:cmcomp}),
and allowing easy proofs of the 
the $H^\ast(B)$--bilinearity of the pairing $\stringprod$
(Proposition~\ref{prop:cupmodstrbilin})
and 
the compatibility of the pairing 
with cartesian products of $\ell$--compact groups
(Proposition~\ref{prop:pairingforproduct}).
This construction does not lend itself
to proving the existence of the units $\bbOne_f$ 
or to the construction of the pairings
on the level of spectral sequences
alluded to in Theorem~\ref{thm:mainresult}, however, so in 
Section~\ref{subsec:secondconstruction}
we present another construction of the pairing
based on the 
technology developed in \cite{umkehr-maps}.
The two constructions are shown to produce
the same result in 
Section~\ref{subsec:comparison}.
In Section~\ref{subsec:furtherpropertiesofthepairing},
we will combine the viewpoints afforded by the
two constructions to complete the proofs of 
Theorems~\ref{thm:pairings4}, \ref{thm:functoriality},
and \ref{thm:cmcomparison2},
in addition to which we will prove the existence of the
augmentation $\rho$ of equation \eqref{eq:rho}.
Finally, in Section~\ref{subsec:thirdperspective},
we provide the third construction on the pairing by
reinterpreting the second construction in 
terms of fibrewise duals,
and use this perspective to prove 
Proposition~\ref{prop:pontryaginproduct}.
This third point of view will also 
prove useful for us in Section~\ref{sec:spectralsequences}
when we prove that the pairing lifts to the 
level of Serre spectral sequences.

\begin{rem}[The semisimplicity assumption on $BG$]
\label{rk:semisimplicityassumptionnontechnical}
As mentioned in the introduction, 
we expect Theorems~\ref{thm:mainresult}, 
\ref{thm:strtoptezukacrit} and \ref{thm:tezukasubgrp}
to generalize from semisimple $\ell$--compact groups to 
arbitrary connected $\ell$--compact groups.
Indeed, the first construction
of the pairing  $\stringprod$  works unchanged 
in this greater generality, and
the third construction could be adapted to
the more general setting by replacing the natural 
isomorphism \eqref{eq:recognitionthmiso2} with one that is independent of
the second construction. 
Where the semisimplicity assumption enters into our arguments in a crucial 
way is the second construction, without which we do not know 
how to relate the first and third constructions.
See Remark~\ref{rk:semisimplicityassumption}.
\end{rem}

\subsection{The first construction of the string pairing}
\label{subsec:firstconstruction}

In this subsection, we will give a simple construction of 
the string pairing \eqref{eq:pairing} using
``integration along fibre'' maps defined 
in terms of Serre spectral sequences.

\begin{defn}[{Integration along fibre maps; cf.~\cite[\S A.3]{KM19} 
and Definition~U.\ref*{U-def:intalongfibrecohomology}}]
\label{def:integrationalongfibre}
Let $p\colon E\to B$ be a fibration,
and write $\calH^\ast(F)$ for the local coefficient
system of graded $\F_\ell$--vector spaces 
over $B$ given by the cohomology
groups of the fibres of $p$.
Call $p$ \emph{orientable} if 
there exists an $n$ such that 
$\calH^k(F) = 0$ for $k > n$
and $\calH^n(F)$ is isomorphic to the 
trivial coefficient system given by $\F_\ell$.
An \emph{orientation} for an orientable fibration $p$ is
an isomorphism $o \colon \calH^n(F) \xto{\isom} \F_\ell$,
of local coefficient systems,
and $p$ equipped with an orientation is called an 
\emph{oriented fibration}.
Given an orientable fibration $p$ with an orientation $o$,
we define the \emph{integration along fibre map}
\[
	p_! = (p,o)_! \colon H^\ast(E) \longto H^{\ast-n}(B)
\]
to be the map given by the composite
\begin{equation}
\label{eq:iafdef}
\xymatrix{
	H^{k+n}(E) 
	\longto 
	E_\infty^{k,n} 
	\longincl 
	E_2^{k,n} 
	=
	H^{k}(B;\, \calH^n(F))
	\ar[r]_-\isom^-{o_\ast} 
	&
	H^k(B;\,\F_\ell)
	=
	H^k(B)
} 
\end{equation}
where $E_2$ and $E_\infty$ refer to pages in the relative Serre 
spectral sequence of $p$
and $o_\ast$ is the map induced by $o$.
\end{defn}

\begin{rem}
For various signs to work out correctly, in
\eqref{eq:iafdef} one must interpret $\calH^n(F)$
as a graded object concentrated in degree $n$,
and $o$ as a map lowering degrees by $n$.
\end{rem}

\begin{rem}
\label{rk:orientationwithconnectedbase}
For an orientable fibration $p\colon E \to B$ 
with $B$ path connected, the data of an orientation 
amounts to the choice of an isomorphism 
$o\colon H^n(F)\isom \F_\ell$
for a fibre $F$ of $p$.
Moreover, in this case
\eqref{eq:iafdef} amounts to the composite
\begin{equation}
\xymatrix{
	H^{k+n}(E) 
	\longto 
	E_\infty^{k,n} 
	\longincl 
	E_2^{k,n} 
	\isom H^{k}(B) \tensor H^n (F) 
	\ar[r]_-\isom^-{\id \tensor o} 
	&
	H^k(B)
} 
\end{equation}
where 
\[
	(\id \tensor o)(\beta \tensor\phi) = (-1)^{n\deg \beta} \beta o(\phi)
\]
in keeping with the Koszul sign rule. 
\end{rem}

\begin{rem}
\label{rk:orfibpb}
Suppose
\[\xymatrix{
	D 
	\ar[r]
	\ar[d]_q
	\pb
	&
	E
	\ar[d]^p
	\\
	A
	\ar[r]
	&
	B
}\]
is a pullback square with $p$ an orientable fibration.
Then $q$ is an orientable fibration, and an orientation for $p$
induces one for $q$.
\end{rem}

A naive attempt to use 
the integration along fibre maps
to define \eqref{eq:pairing}
using the strategy outlined in 
Section~\ref{subsubsec:thepairingbullet}
runs into the problem that the map $\mathrm{concat}$ of 
diagram \eqref{diag:pushpull}
is usually not a fibration.
To overcome this problem, we will replace 
the map $\mathrm{concat}$ 
with an equivalent fibration.

\begin{defn}
Given maps $f,g,h\colon B \to BG$,
write $P'(f,g,h)$ for the space
\[
	P'(f,g,h) 
	= 
	\{
		(b,\alpha) \in B \times BG^{\Delta^2} 
		\mid 
		\alpha(e_0) = f(b),
		\alpha(e_1) = g(b),
		\alpha(e_2) = h(b)
	\}
\]
where $e_0$, $e_1$ and $e_2$ are the vertices of 
the standard $2$--simplex $\Delta^2$.
Moreover, let 
\[
	\mathrm{split}' \colon P'(f,g,h) \longto P(g,h) \times P(f,g) 
\qquad\text{and}\qquad
	\mathrm{concat}' \colon P'(f,g,h) \longto P(f,h)
\]
be the evident maps given induced by the various face inclusions
of $I\homeom \Delta^1$ into $\Delta^2$, and let
\[
	\pi'_{f,g,h} \colon P'(f,g,h) \longto B
\]
be the map given by projection onto the first coordinate.
\end{defn}

We observe that
the usual deformation retraction of $\Delta^2$ onto the horn 
$\Lambda^2_1 \subset \Delta^2$ induces a homotopy equivalence
$P(f,g,h) \to P'(f,g,h)$ making the following diagram commutative.
\begin{equation}
\label{diag:pandpprime}
\vcenter{\xymatrix@!0@C=8em{
	&
	P'(f,g,h)
	\ar[dl]_{\mathrm{split}'}
	\ar[dr]^{_{}\mathrm{concat}'}
	\\
	P(g,h) \times P(f,g) 
	&&
	P(f,h)
	\\
	&
	P(f,g,h)
	\ar[uu]_{\homot}
	\ar[ul]^{\mathrm{split}}
	\ar[ur]_{\mathrm{concat}}
}}
\end{equation}

Our next aim is to equip the fibration $\mathrm{concat}'$
with an orientation. To this end,
we note that there is a canonical pullback square
\begin{equation}
\label{sq:concatprimepb} 
\vcenter{\xymatrix@C+2em{
	P'(f,g,h) 
	\ar[r]^{\mathrm{concat}'}
	\ar[d]
	\pb
	&
	P(f,h)
	\ar[d]
	\\
	P'(\pi_1,\pi_2,\pi_3) 
	\ar[r]^{\mathrm{concat}'}
	&
	P(\pi_1,\pi_3)
}}
\end{equation}
where $\pi_1,\pi_2,\pi_3\colon BG^3 \to BG$ are the coordinate
projection maps. As the bottom map can be identified with 
the map
\begin{equation}
\label{eq:univconcatprimedesc}
	BG^{\Delta^2} \longto BG \times BG^I,
	\qquad
	\alpha
	\longmapsto
	(\alpha(e_1), \alpha  | [e_0,e_2]),
\end{equation}
an orientable fibration with fibre homotopy equivalent to $G$,
we see that the map 
$\mathrm{concat}' \colon P'(f,g,h) \to P(f,h)$
is an orientable fibration.
Let $F$ be the fibre of the map \eqref{eq:univconcatprimedesc}
over the point of $BG \times BG^I$ given by the 
basepoint of $BG$ and the constant path at the basepoint.
Restriction to the line segment 
$[e_1,(e_0 + e_2)/2] \subset \Delta^2$
provides an explicit homotopy equivalence between 
$F$ and $G$, and we give \eqref{eq:univconcatprimedesc}
the orientation defined by the composite
\[
	H^d(F) \xto{\ \isom\ } H^d(G) \xto{\ \isom\ } \F_\ell
\]
of the resulting isomorphism $H^d(F) \isom H^d(G)$
and the isomorphism $H^d(G)\isom \F_\ell$
of \eqref{as:gorientation}.
The identification of the bottom map
$\mathrm{concat}'$ in \eqref{sq:concatprimepb}
with the map \eqref{eq:univconcatprimedesc}
now yields an orientation for the former,
and we equip the map 
$\mathrm{concat}' \colon P'(f,g,h) \to P(f,h)$
with the induced orientation from 
\eqref{sq:concatprimepb}.

We are now ready to give the first definition of the string pairing $\stringprod$.
\begin{defn}[The string pairing $\stringprod$]
\label{def:bulletdef1}
Given $f,g,h \colon B \to BG$, we let
\[
	\tildestringprod 
	\colon 
	H^\ast P(g,h) \tensor H^\ast P(f,g)
	\longto
	H^{\ast-d} P(f,h)
\]
be the composite
\def\firstentry{H^\ast P(g,h) \tensor H^\ast P(f,g)\;}
\begin{equation}
\label{eq:tildebulletdef}
\vcenter{\xymatrix@1@!0@C=5em{
	*!R{\firstentry}
	\ar[r]^\times
	&
	*!L{\;H^\ast (P(g,h) \times P(f,g))}
	\\
	*!R{\phantom{\firstentry}}
	\ar[r]^-{(\mathrm{split}')^\ast}
	&
	*!L{\;H^\ast P'(f,g,h)}
	\\
	*!R{\phantom{\firstentry}}
	\ar[r]^{\mathrm{concat}'_!}
	&
	*!L{\;H^{\ast-d} P(f,h)}
}}
\end{equation}
and define the pairing
\[
	\stringprod 
	\colon
	\bbH^\ast P(g,h) \tensor \bbH^\ast P(f,g)
	\longto
	\bbH^\ast P(f,h)
\]
by setting
\begin{equation}
\label{eq:bulletandtildebullet}
	x \stringprod y = (-1)^{d \deg(x)} s^{-d}(s^d(x) \tildestringprod s^d(y)) 
\end{equation}
for all $x\in \bbH^\ast P(g,h)$, $y\in \bbH^\ast P(f,g)$.
\end{defn}

\subsection{Consequences of the first construction of the string pairing}

Our aim in this subsection is to prove various properties 
of the string pairing $\stringprod$ that follow easily from 
the construction given in Section~\ref{subsec:firstconstruction}.
We will start by proving that on $\bbH^\ast(LBG)$,
our string product $\stringprod$ agrees with 
product $\odot$ constructed by Chataur and Menichi 
(Proposition~\ref{prop:cmcomp})
after which we will prove that the pairing $\stringprod$
is $H^\ast(B)$--bilinear (Proposition~\ref{prop:cupmodstrbilin}).
We will continue by proving that the pairing $\stringprod$
is compatible cartesian products of $\ell$--compact groups 
(Proposition~\ref{prop:pairingforproduct}).
Finally, we will establish a formula which is frequently
useful in computations of $\stringprod$
and which will be needed later in the paper
in the proof of Theorem~\ref{thm:realization}
(Proposition~\ref{prop:middle} and Remark~\ref{rk:middle}).

Let us first recall the definition of the Chataur--Menichi product
\[
\odot\colon \bbH^\ast(LBG) \tensor \bbH^\ast(LBG) \longto \bbH^\ast(LBG).
\]
Let $\Sigma$ be the pair of pants surface with 
one incoming and two outgoing boundary circles,
and let $\mathrm{in} \colon S^1 \to \Sigma$
and $\mathrm{out} \colon S^1 \sqcup S^1 \to \Sigma$
be the inclusions of the incoming and outgoing boundaries
to $\Sigma$.
Then the product $\odot$ is defined by
\[
	x\odot y 
	= 
	(-1)^{d \deg(x)} s^{-d}\mathrm{Dlcop}(s^d(x) \tensor s^d(y))
\]
where $\mathrm{Dlcop}$ is the composite
\def\firstentry{H^\ast (LBG) \tensor H^\ast (LBG)\;}
\begin{equation}
\label{eq:dlcopdef}
\vcenter{\xymatrix@1@!0@C=6em{
	*!R{\firstentry}
	\ar[r]^\times
	&
	*!L{\;H^\ast (LBG \times LBG)}
	\\
	*!R{\phantom{\firstentry}}
	\ar[r]^-{\map(\mathrm{out},BG)^\ast}
	&
	*!L{\;H^\ast (\map(\Sigma, BG))}
	\\
	*!R{\phantom{\firstentry}}
	\ar[r]^{\map(\mathrm{in},BG)_!}
	&
	*!L{\;H^{\ast-d} (LBG)}
}}
\end{equation}
See \cite[pp.~847--848]{KM19}. 

\begin{prop}
\label{prop:cmcomp}
On $\bbH^\ast(LBG)$, the string pairing $\stringprod$ agrees with the 
product $\odot$ of Chataur and Menichi.
\end{prop}

\begin{proof}
Recall that we have identified $LBG$ with 
$P(\id_{BG},\id_{BG})$,
and observe that %
$P'(\id_{BG},\id_{BG},\id_{BG})$
can be identified with the mapping space
$\map(\Delta^{2}/\{e_0,e_1,e_2\},BG)$.
View the surface
$\Sigma$ as built from $\Delta^2 / \{e_0,e_1,e_2\}$
by attaching a cylinder to each of the $3$ boundary 
circles, and observe that collapsing the attached cylinders
yields a deformation retraction 
$r\colon \Sigma \xto{\homot} \Delta^2 / \{e_0,e_1,e_2\}$
such that the following diagram commutes:
\[\xymatrix@!0@C=10em@R=3em{
	&
	\map(\Sigma,BG)
	\ar[dl]_{\map(\mathrm{out},BG)\quad}
	\ar[dr]^{_{}\map(\mathrm{in},BG)}
	\\
	LBG \times LBG
	&&
	LBG
	\\
	&
	\map(\Delta^{2}/\{e_0,e_1,e_2\},BG)
	\ar[uu]_{\homot}^{\map(r,BG)}
	\ar[ul]^(0.6){\mathrm{split}'}
	\ar[ur]_(0.6){\mathrm{concat}'}
}\]
Comparing \eqref{eq:tildebulletdef} and 
\eqref{eq:dlcopdef}, we now see that 
$\stringprod$ and $\odot$ agree as long as
the orientations for $\map(\mathrm{in},BG)$
and $\mathrm{concat}'$ match.
An inspection of Kuribayashi and Menichi's 
orientation for  $\map(\mathrm{in},BG)$ 
\cite[p.~848]{KM19},
also derived from an isomorphism $H^d(G) \isom \F_\ell$,
reveals that it does correspond to our orientation
for $\mathrm{concat}'$ under $\map(r,BG)$.
\end{proof}

The following lemma follows from the 
compatibility of the Serre spectral sequence with products.

\begin{lemma}
\label{lm:intalongfibreandproducts}
Suppose $p_1 \colon E_1 \to B_1$ and  $p_2\colon E_2 \to B_2$
are oriented fibrations with fibres $F_i$ and orientations 
$o_i\colon \calH^{n_i}(F_i)\xto{\ \isom\ } \F_\ell$, $i=1,2$.
Equip $p_1\times p_2 \colon E_1\times E_2 \to B_1\times B_2$
with the orientation 
\[
	o_{12} \colon \calH^{n_1+n_2}(F_1\times F_2) \xto{\ \isom\ } \F_\ell
\]
corresponding to the map 
\[
	o_1\tensor o_2 
	\colon
	\calH^{n_1}(F_1) \tensor \calH^{n_2}(F_2) 
	\xto{\ \isom\ }
	\F_\ell,
	\qquad
	x_1 \tensor x_2
	\longmapsto 
	(-1)^{n_1 n_2} o(x_1)o(x_2)
\]
under the isomorphism
\[\xymatrix{
	\calH^{n_1}(F_1) \tensor \calH^{n_2}(F_2) 
	\ar[r]^-{\times}_-{\isom}
	&
	\calH^{n_1+n_2}(F_1\times F_2) 
}
\]
given by cross product.
Then the following diagram commutes for all $k_1$ and $k_2$:
\[\xymatrix{
	H^{k_1+n_1} (E_1) \tensor H^{k_2+n_2} (E_2)
	\ar[r]^-\times
	\ar[d]_{(p_1)_!\tensor (p_2)_!}
	&
	H^{k_1+k_2+n_1+n_2} (E_1\times E_2)
	\ar[d]^{(p_1\times p_2)_!}
	\\
	H^{k_1}(B_1) \tensor H^{k_2}(B_2)
	\ar[r]^-\times
	&
	H^{k_1+k_2}(B_1\times B_2)
}\]
Here the map $(p_1)_!\tensor (p_2)_!$ is given by
\[
	((p_1)_!\tensor (p_2)_!)(x_1 \tensor x_2) 
	= 
	(-1)^{n_2 \deg(x_1)}
	(p_1)_! (x_1) \tensor (p_2)_!(x_2)
\]
in accordance with the Koszul sign rule. \qed
\end{lemma}

\begin{prop}
\label{prop:cupmodstrbilin}
Let $B$ be a space, and let $f,g,h\co B \to BG$ be maps.
The string pairing
\[
	\stringprod
	\co
	\bbH^\ast P(g,h) \tensor \bbH^\ast P(f,g) 
	\longto 
	\bbH^\ast P(f,h)
\]
is $H^\ast(B)$--bilinear with respect to the 
module structures of Definition~\ref{def:cupmodstr}:
for all $a\in H^{\ast}(B)$, $b\in H^{\ast} (B)$ and 
$x\in  \bbH^{\ast} P(g,h)$, $y\in \bbH^{\ast} P(f,g)$
we have
\[
	(ax) \stringprod (by) = (-1)^{\deg(b)\deg(x)}(ab)(x\stringprod y).
\]
\end{prop}

\begin{proof}%
The claim follows readily from the  formula
\begin{equation}
\label{eq:concatprojformula}
	\mathrm{concat}'_! (
		(\mathrm{concat}')^\ast(a) \cupprod x
	)
	=
	(-1)^{d\deg(a)} a \cupprod 	\mathrm{concat}'_! (x)
\end{equation}
valid for $a\in H^\ast P(f,h)$, $x\in H^\ast P'(f,g,h)$.
The formula in turn follows from the naturality of the Serre
spectral sequence with respect to the square
\[
        \xymatrix@C+3em{
        	P'(f,g,h)
        	\ar[d]_{\mathrm{concat}'}
        	\ar[r]^-{(\mathrm{concat}',\id)}
        	&
        	P(f,h)\times P'(f,g,h)
        	\ar[d]^{\id\times \mathrm{concat}'}
        	\\
        	P(f,h)
        	\ar[r]^-{\Delta}
        	&
        	P(f,h)\times P(f,h)
        }
\]
together with Lemma~\ref{lm:intalongfibreandproducts}.
\end{proof}

Our next goal is to relate the string pairing \eqref{eq:pairing}
for a product of semisimple $\ell$--compact
groups to the pairings for the factors. 
The result (Proposition~\ref{prop:pairingforproduct})
shows that the pairings are compatible in the expected way,
although articulating this compatibility precisely 
will require a bit of effort.

 Suppose $BG$ splits
as a product
\[
	BG = BG_1 \times BG_2
\]
where $BG_1$ and $BG_2$ are semisimple
$\ell$--compact groups of dimension $d_1$ and $d_2$, respectively.
Given maps $f_1,g_1 \colon B_1 \to BG_1$ and $f_2,g_2 \colon B_2 \to BG_2$,
we have a homeomorphism
\[
	c 
	\colon
	P(f_1\times f_2, g_1 \times g_2)
	\xto{\ \homeom\ }
	P(f_1,g_1) \times P(f_2,g_2),
	\quad
	((b_1,b_2),(\gamma_1,\gamma_2))
	\longmapsto 
	((b_1,\gamma_1),(b_2,\gamma_2)) 
\]
making the triangle below commutative.
\[\xymatrix@C-1.5em{%
	P(f_1\times f_2, g_1 \times g_2)
	\ar[rr]_-\homeom^-{c}%
	\ar[dr]_(0.45){\pi_{f_1\times f_2, g_1\times g_2}\;}
	&&
	P(f_1,g_1) \times P(f_2,g_2)
	\ar[dl]^(0.45){\quad\pi_{f_1,g_1} \times \pi_{f_2,g_2}}
	\\
	&
	B_1\times B_2	
}\]
Generalizing our earlier notation $\bbH^\ast$
and the degree shift maps $s^d \colon \bbH^\ast \to H^\ast$,
given an integer $k$ and a space $X$, we write 
$\shiftedH{k}^\ast(X) = H^{\ast+k}(X)$, and define 
\[
	s^k \colon \shiftedH{k}^\ast(X) \longto H^\ast(X)
\]	
to be the map sending each class to itself, but now considered as
an element of $H^\ast(X)$. Given spaces $X_1$ and  $X_2$ and integers $k_1$ and $k_2$,
we define a cross product 
\[
	\times 
	\colon 
	\shiftedH{k_1}^\ast(X_1) \tensor \shiftedH{k_2}^\ast(X_2)
	\longto
	\shiftedH{k_1+k_2}^\ast(X_1\times X_2)	
\]
by the requirement that the diagram
\[\xymatrix{
	\shiftedH{k_1}^\ast(X_1) \tensor \shiftedH{k_2}^\ast(X_2)
	\ar@{-->}[r]^(0.52){\times}
	\ar[d]_{s^{k_1} \tensor s^{k_2}}
	&
	\shiftedH{k_1+k_2}^\ast(X_1\times X_2)
	\ar[d]^{s^{k_1+k_2}}
	\\
	H^\ast(X_1) \tensor H^\ast(X_2)
	\ar[r]^(0.52){\times}
	&
	H^\ast(X_1\times X_2)
}\]
commutes. Here the map $s^{k_1} \tensor s^{k_2}$ is given by
\[
	(s^{k_1} \tensor s^{k_2})(x_1 \tensor x_2) 
	= 
	(-1)^{k_2 \deg(x_1)} s^{k_1}(x_1) \tensor s^{k_1}(x_2)
\]
in accordance with the Koszul sign rule.

\begin{prop}
\label{prop:pairingforproduct}
Suppose $BG$ factors as $BG = BG_1 \times BG_2$ where $BG_1$ and $BG_2$
are semisimple $\ell$--compact groups
of dimension $d_1$ and $d_2$, respectively. 
Suppose $BG_1$ and $BG_2$ are equipped with orientations 
$o_1 \colon H^{d_1} (G_1) \xto{\ \isom\ } \F_\ell$
and
$o_2 \colon H^{d_2} (G_2) \xto{\ \isom\ } \F_\ell$
respectively, and $BG$ is equipped with the orientation
$H^{d_1 + d_2}(G) \xto{\ \isom\ } \F_\ell$
corresponding to the map
\[
	(-1)^{d_1 d_2} o_1\tensor o_2 
	\colon 
	H^{d_1}(G_1) \tensor H^{d_2}(G_2) \longto \F_\ell,
	\qquad
	x_1 \tensor x_2 \longmapsto o_1(x_1) o_2(x_2)
\]
under the isomorphism $H^{d_1}(G_1) \tensor H^{d_2}(G_2) \isom H^{d_1+d_2}(G)$
given by cross product.
Here $G_1 = \loops BG_1$ and $G_2 = \loops BG_2$,
so that $G = \loops BG = \loops (BG_1\times BG_2) = G_1 \times G_2$.
Then, given maps 
$f_1, g_1, h_1 \colon B_1 \to BG_1$ and $f_2, g_2, h_2 \colon B_2 \to BG_2$,
the following diagram commutes:
\begingroup
\small
\[
\xymatrix@C+1em{
	\shiftedH{d_1}^\ast P(g_1,h_1)
	\tensor
	\shiftedH{d_1}^\ast P(f_1,g_1)
	\tensor
	\shiftedH{d_2}^\ast P(g_2,h_2)
	\tensor
	\shiftedH{d_2}^\ast P(f_2,g_2)
	\ar[r]^-{\stringprod_1\tensor \stringprod_2}
	\ar[d]_{1\tensor \chi \tensor 1}
	&
	\shiftedH{d_1}^\ast P(f_1,h_1)
	\tensor
	\shiftedH{d_2}^\ast P(f_2,h_2)
	\ar[dd]^{\times}
	\\
	\shiftedH{d_1}^\ast P(g_1,h_1)
	\tensor
	\shiftedH{d_2}^\ast P(g_2,h_2)
	\tensor
	\shiftedH{d_1}^\ast P(f_1,g_1)
	\tensor
	\shiftedH{d_2}^\ast P(f_2,g_2)	
	\ar[d]_{\times \tensor \times}
	\\
	\shiftedH{d_1+d_2}^\ast (P(g_1,h_1) \times P(g_2,h_2))
	\tensor
	\shiftedH{d_1+d_2}^\ast (P(f_1,g_1) \times P(f_2,g_2))	
	\ar[d]^\isom_{c^\ast\tensor c^\ast}
	&
	\shiftedH{d_1+d_2} (P(f_1,h_1) \times P(f_2,h_2))
	\ar[d]^{c^\ast}_\isom
	\\
	\shiftedH{d_1+d_2}^\ast P(g_1\times g_2, h_1\times h_2)
	\tensor
	\shiftedH{d_1+d_2}^\ast P(f_1\times f_2, g_1\times g_2)
	\ar[r]^-\stringprod
	&
	\shiftedH{d_1+d_2} P(f_1\times f_2, h_1\times h_2)
}\]
\endgroup
Here $\stringprod_1$ and $\stringprod_2$ refer to 
the string pairing \eqref{eq:pairing} for $BG_1$ and $BG_2$,
respectively, and $\chi$ is the usual symmetry constraint
$x\tensor y \mapsto (-1)^{\deg(x)\deg(y)} y\tensor x$.
\end{prop}

\begin{proof}[Proof of Proposition~\ref{prop:pairingforproduct}]
The claim follows by unpacking definitions and applying 
Lemma~\ref{lm:intalongfibreandproducts}.
\end{proof}

We conclude the section with the following result, which 
is often useful in computations of
the string pairing $\stringprod$. 
Compare with~\cite[Lemma~2.3]{KM19} and \cite[Thm.~8.2]{ChasSullivan}.
In the present paper, we will need the result in the proof of 
Theorem~\ref{thm:realization}.
\begin{prop}
\label{prop:middle}
Fix $f,g,h\colon B \to BG$, and 
suppose $A \in H^\ast P(f,h)$ and 
$a = \sum_i a^{}_i \times a'_i \in H^\ast(P(g,h)\times P(f,g))$
satisfy 
$\mathrm{concat}^\ast(A) = \mathrm{split}^\ast (a)
\in 
H^\ast P(f,g,h)$.
Then 
\begin{equation}
\label{eq:middleformula} 
	A \cupprod(u_1 \tildestringprod u_2) 
	= 
	\sum_i (-1)^{d \deg(a_i)  + (d+ \deg(u_1))\deg(a'_i)} 
		(a_i\cupprod u_1) \tildestringprod (a'_i \cupprod u_2)
\end{equation}
for all $u_1\in H^\ast(LBG)$ and $u_2 \in H^\ast(BG^{h\sigma})$. 
\end{prop}
\begin{proof}
In view of diagram~\eqref{diag:pandpprime},
the equation 
\[
	\mathrm{concat}^\ast(A) = \mathrm{split}^\ast (a)
\]
implies the equation
\[
	(\mathrm{concat}')^\ast(A) = (\mathrm{split}')^\ast (a).
\]
The claim now follows easily from formula~\eqref{eq:concatprojformula}.
\end{proof}

\begin{rem}
\label{rk:middle}
The sign in formula \eqref{eq:middleformula}
becomes less mysterious 
when Proposition~\ref{prop:middle}
is phrased in terms of the pairing $\stringprod$ instead of $\tildestringprod$.
Given a space $X$, equip 
the shifted cohomology groups $\bbH^\ast(X)$
with a  $(H^\ast(X),\cupprod)$--module structure 
by requiring 
$s^d \colon \bbH^\ast(X) \to H^\ast(X)$
to be $H^\ast(X)$--linear, so that
\[
	s^d(a x) = (-1)^{d\deg(a)} a \cupprod s^d(x)
\]
for all $a\in H^\ast(X)$, $x\in \bbH^\ast(X)$.
Then \eqref{eq:middleformula} is equivalent 
to the formula
\begin{equation}
\label{eq:bettermiddleformula}
	A (v_1 \stringprod v_2)
	= 
	\sum_i (-1)^{\deg(v_1) \deg(a'_i)}
		(a_i v_1) \stringprod (a'_i v_2) 
\end{equation}
for all $v_1\in \bbH^\ast(LBG)$ and $v_2 \in \bbH^\ast(BG^{h\sigma})$
where the sign is what one would expect from the Koszul rule. 
\end{rem}

\subsection{The second construction of the string pairing}
\label{subsec:secondconstruction}

In this subsection, we will present the second 
construction of the string pairing $\stringprod$.
To distinguish the pairings produced by the 
first and second constructions,
we will temporarily adopt the notation $\stringprod'$
for the pairing produced by the latter.
The main results of the subsection will be

\begin{thrm}
\label{thm:pairingssecondconstruction}
Suppose $BG$ is a semisimple $\ell$--compact group.
\begin{enumerate}[(i)]
\item\label{it:pairingssecondconstructioncats}
    Given a space $B$, the pairings
    \begin{equation}
    \label{eq:pairingbulletprime}
    	\stringprod'
    	\co
    	\bbH^\ast P(g,h) \tensor \bbH^\ast P(f,g) 
    	\longto 
    	\bbH^\ast P(f,h)
    \end{equation}
    for maps $f,g,h\colon B \to BG$ define the composition law
    in a category enriched in graded $\F_\ell$--modules
    whose objects are maps $f\colon B \to BG$ and whose morphisms
    from $f\colon B\to BG$ to $g\colon B \to BG$ 
    are given by $\bbH^\ast P(f,g)$. 
\item\label{it:pairingssecondconstructionfuns}
	The enriched category of part~(\ref{it:pairingssecondconstructioncats})
	depends functorially on the space $B$:
	given a map $\phi \colon A \to B$,
	the maps $F_\phi$ of equation~\eqref{eq:fphidef}
	for varying $f,g \colon B\to BG$
	define a functor of categories enriched in graded $\F_\ell$--modules
	which on objects is given by the assignment $f\mapsto f\phi$.
\end{enumerate}
\end{thrm}
Compare with Theorems~\ref{thm:pairings4} and 
\ref{thm:functoriality}(\ref{it:functoriality2}).
As was the case with $\stringprod$, the pairing $\stringprod'$
will depend  on a piece of orientation data, but again
only in a mild way, so that changing the orientation
has the effect of replacing $\stringprod'$ by a nonzero scalar multiple.
Later, in Section~\ref{subsec:comparison},
we will see that the orientation data can be chosen
so that the pairings $\stringprod'$ and $\stringprod$ agree,
after which point we will drop the notation $\stringprod'$ 
in favor of $\stringprod$.

\subsubsection{Outline of the construction}
\label{subsubsec:strategy2}

We will phrase the construction of the pairing $\stringprod'$ and the proof of 
Theorem~\ref{thm:pairingssecondconstruction}
using the language of enriched category theory, which we will now briefly 
recall.

\begin{defn}[Enriched categories]
\label{def:enrichedcat}
A \emph{category} $\calC$ 
\emph{enriched in a monoidal category} $\calV$
consists of the following data: 
a collection of objects $\Ob \calC$,
a hom-object $\calC(A,B) \in \calV$ for 
every pair of objects $A,B\in \Ob\calC$,
a composition law $\calC(B,C)\tensor \calC(A,B) \to \calC(A,C)$
for every triple of objects $A,B,C\in\Ob\calC$,
and an identity element $I \to \calC(A,A)$ for every object $A\in\Ob\calC$,
where $I$ denotes the monoidal unit in $\calV$.
These data are supposed to satisfy
the evident analogues of the axioms of an ordinary category
\cite[\S1.2]{KellyEnriched}.
\end{defn}

\begin{defn}[Enriched functors]
\label{def:enrichedfunctor} 
An \emph{enriched functor}
$F\co \calC \to \calD$ between categories enriched in $\calV$
consists of a map $F\co \Ob \calC \to \Ob \calD$ and a 
map $F=F_{A,B}\co\calC(A,B) \to \calC(FA,FB)$ for every pair of objects 
$A,B\in \Ob\calC$, these data being subject to 
the evident analogues of the axioms
for an ordinary functor 
\cite[\S1.2]{KellyEnriched}.
\end{defn}

\begin{terminology}[(Symmetric) monoidal functors]
\label{term:monfun}
By a (symmetric) monoidal functor $F\co\calC \to \calD$
between (symmetric) monoidal categories, 
we mean a \emph{strong} (symmetric) monoidal functor 
in the sense of Mac~Lane~\cite[\S XI.2]{MacLane}, 
meaning that the monoidality and identity constraints
\[
	F_\tensor \co F(X) \tensor F(Y) \longto F(X\tensor Y)
	\qquad\text{and}\qquad
	F_I\co I_\calD \longto F(I_\calC)
\]
are assumed to be isomorphisms. Here $I_\calC$ and $I_\calD$ 
denote the unit objects of $\calC$ and $\calD$, respectively.
In a \emph{lax} (symmetric) monoidal functor the requirement that the 
maps are isomorphisms is dropped, and in an \emph{oplax}
(symmetric) monoidal functor the direction of the maps 
is in addition reversed.
\end{terminology}

The following construction gives a basic way of 
obtaining new enriched categories and functors 
from existing ones.

\begin{constr}
\label{constr:newenrichedcats}
Let 
$M\co \calV \to \calW$ be a lax monoidal functor.
Then from a $\calV$--enriched
category $\calC$ we obtain a $\calW$--enriched category $M_\ast\calC$
with $\Ob M_\ast\calC = \Ob \calC$ and hom-objects 
\[
	(M_\ast\calC)(A,B) = M\calC(A,B)
\]
by taking as the composition law the composite
\[
	M\calC(B,C)\tensor M\calC(A,B) 
	\xto{\ M_\tensor\ } 
	M(\calC(B,C)\tensor \calC(A,B))
	\xto{\ M(\mu_{A,B,C})\ }
	M\calC(A,C)
\]
and as the identity element for an object $A$ the composite
\[
	I_\calW \xto{\ M_I\ } M(I_\calV) \xto{\ M(\iota_A)\ } M\calC(A,A),
\]
where $\mu_{A,B,C}$ and $\iota_A$ refer to 
the composition law and the identity element in $\calC$,
respectively, and $M_\tensor$ and $M_I$
are the monoidality and identity constraints of $M$.
Moreover, a $\calV$--enriched functor $F\co \calC \to \calD$
induces a $\calW$--enriched functor 
$M_\ast(F) \co M_\ast\calC \to M_\ast\calD$
by letting $M_\ast(F) = F$ on objects, and by defining
\[
	M_\ast(F)_{A,B} = M(F_{A,B})
	\co 
	(M_\ast\calC)(A,B) \longto (M_\ast\calD)(FA,FB)
\]
on morphisms.
\end{constr}

In view of 
Construction~\ref{constr:newenrichedcats},
one might try to construct the pairing $\stringprod'$ and prove 
Theorem~\ref{thm:pairingssecondconstruction}(\ref{it:pairingssecondconstructioncats})
by first constructing an enriched category $\calP_B$
where the objects are maps 
$f\co B \to BG$,
where the hom-object from $f$ to $g$ is given by the fibration 
$P(f,g) \to B$, and where the composition law is given by 
diagram~\eqref{diag:pushpull}, and then applying
Construction~\ref{constr:newenrichedcats} to $\calP_B$
with a suitable monoidal functor $M$ to obtain the category of 
Theorem~\ref{thm:pairingssecondconstruction}(\ref{it:pairingssecondconstructioncats}).
Moreover, 
Theorem~\ref{thm:pairingssecondconstruction}(\ref{it:pairingssecondconstructionfuns})
would follow if the pullback squares \eqref{sq:inducedmap}
assembled into an enriched 
functor $K_\phi\co \calP_B \to \calP_A$ which, upon
application of $M$, yielded $F_\phi$.
This strategy is indeed the one we will follow. 
The functor $M$ will be the composite
\def\firstentry{(h\calF^\fop)^\op\;}
\begin{equation}
\label{eq:functorMnew2}
\vcenter{\xymatrix@1@!0@C=5em{
	*!R{\firstentry}
	\ar[r]^{\fH_\bullet^\op}
	&
	*!L{\;\ho(\Mod^{H\F_\ell})^\op}
	\\
	*!R{\phantom{\firstentry}}
	\ar[r]^{H^\ast}
	&
	*!L{\;\grMod^{\F_\ell}}
}}
\end{equation}
Here 
$\ho(\Mod^{H\F_\ell})$ denotes the 
homotopy category of $H\F_\ell$--modules,
$\grMod^{\F_\ell}$ the category of graded
$\F_\ell$--modules,
and
$H^\ast$ cohomology with $\F_\ell$ coefficients.
The category $h\calF^\fop$ will be constructed in 
Section~\ref{subsubsec:hcalffop}, 
the categories $\calP_B$ enriched in 
$(h\calF^\fop)^\op$ 
in Section~\ref{subsubsec:calpb2},
and the functor $\fH_\bullet$ in 
Section~\ref{subsubsec:fhbullet}.
Once all the categories and functors in \eqref{eq:functorMnew2}
have been constructed, the proof of 
Theorem~\ref{thm:pairingssecondconstruction}
will be completed by identifying 
the image of $P(f,g)$ under the composite functor~\eqref{eq:functorMnew2}
as $\bbH^\ast P(f,g)$. 
This will be done in Section~\ref{subsubsec:result2}.
In our work, we will be relying on the foundations developed in \cite{umkehr-maps}. 
We will therefore begin by summarizing in Section~\ref{subsubsec:umkehrmapssummary} 
the main points of \cite{umkehr-maps} the reader of the present work 
needs to know to follow the construction.

\subsubsection{A brief summary of \cite{umkehr-maps}}
\label{subsubsec:umkehrmapssummary}

The paper \cite{umkehr-maps} develops 
a theory of umkehr maps suited to 
the needs of the present paper. 
The theory is phrased in terms of
``twisted homology objects'' $H_\bullet(B;X)$
and maps between them, where
$B$ is a space and $X$ is an object parametrized by $B$
such as a parametrized spectrum or a parametrized $H\F_\ell$--module
over $B$.
In more detail, suppose $\calC$ is a 
presentable symmetric monoidal $\infty$--category
with symmetric monoidal product $\tensor$
such as the $\infty$--category 
$\Spectra$ of spectra with smash product $\smashprod$, 
the $\infty$--category
$\Spectra^\ell$ of $\ell$--complete 
(that is, $H\F_\ell$--local) spectra with smash product $\smashprod^\ell$,
or the $\infty$--category $\Mod^{R}$ of modules
over a commutative ring spectrum $R$ with smash product $\smashprod^R$.
Given a space $B$ and a parametrized $\calC$--object $X$ over $B$,
there is then an associated object $H_\bullet(B;X) \in \ho(\calC)$,
the ``twisted homology of $B$ with coefficients in $X$'' 
(see Definition~U.\ref*{U-def:hbullethc} or below).
In special cases, these objects recover the usual generalized
homology groups of $B$ and ordinary twisted homology groups of $B$:
when  $\calC$ is $\Spectra$ or a similar $\infty$--category,
Example~U.\ref*{U-ex:hbulletfortrivialcoeffs2}
shows that the homotopy groups of the objects $H_\bullet(B;X)$
for trivial parametrized $\calC$--objects $X$ over $B$
agree with the usual untwisted generalized homology groups of $B$,
and for $\calC = \Mod^{H\Z}$,
Corollary~U.\ref*{U-cor:homologywithlocalcoefficientscomp}
shows how to recover ordinary homology groups of $B$ 
with local coefficients from the homotopy groups of the objects 
$H_\bullet(B;X)$. 

The construction of the objects $H_\bullet(B;X)$
is conveniently expressed as a functor
\begin{equation}
\label{eq:hbullethpctohocalc}
	H_\bullet \colon \hpC \longto \ho(\calC)
\end{equation}
where $\hpC$ is a category 
obtained by assembling the homotopy categories
$\ho(\calC_{/B})$ of parameterized $\calC$--objects over $B$
for varying base spaces $B$
with their symmetric monoidal products $\tensor_B$ induced by $\tensor$
together into a symmetric monoidal category
fibred and opfibred over the category $\calT$
of compactly generated weak Hausdorff spaces.
The symmetric monoidal product on $\hpC$ is denoted by $\exttensor$.
Formally, $(\hpC,\exttensor)$ is the symmetric monoidal 
category obtained by the 
Grothendieck construction from the pseudofunctor
\[
	\calT^\op \longto \smCat,
	\qquad
	B \longmapsto (\ho(\calC_{/B}), \tensor_B),
	\qquad
	f \longmapsto f^\ast
\]
where $\smCat$ is the $2$--category of symmetric monoidal categories
and $f^\ast$ is the pullback functor induced by $f$.
Concretely, the objects of $\hpC$
are pairs $(B,X)$ where $B \in \calT$
and $X$ is a parametrized $\calC$--object over $B$.
We often write just $X$ for $(B,X)$, 
leaving the base space $B$ implicit.
In terms of the fibred and opfibred category $\hpC \to \calT$, 
the object $H_\bullet(B;X)$
is determined by a universal property: identifying 
$\ho(\calC)$ with $\ho(\calC_{/\pt})$, the object $H_\bullet(B;X)$
is characterized up to unique equivalence as the 
target of an opcartesian morphism $X \to H_\bullet(B;X)$
covering the unique map $B \to \pt$.

On morphisms, the functor $H_\bullet$
of equation~\eqref{eq:hbullethpctohocalc}
sends a map
$\phi \colon (A,X) \to (B,Y)$ of $\hpC$ covering 
a map $f\colon A \to B$ in $\calT$ to the map
\[
	(f,\phi)_\bullet \colon H_\bullet(A;X) \longto H_\bullet(B;Y)
\]
induced by the universal property of opcartesian morphisms.
These are the usual induced maps on twisted homology.
In the theory of \cite{umkehr-maps}, the umkehr maps
arise from a contravariant functor (also denoted by $H_\bullet$)
\begin{equation}
\label{eq:hbullethpcdfoptohocalc}
	H_\bullet \colon \hpC^\dfop \longto \ho(\calC)
\end{equation}
where $\hpC^\dfop$ is a category opfibred over $\calT$
obtained from $\hpC$ by, roughly speaking, retaining 
all opcartesian morphisms of $\hpC$
while replacing all fibres of the projection
$\hpC\to \calT$ 
by their opposite categories.
Formally, the objects of $\hpC^\dfop$ are precisely the 
objects of $\hpC$, and morphisms of $\hpC^\dfop$ covering a continuous
map $f\colon A \to B$ are given by equivalence classes of zigzags
\[
	X \xto{\ \alpha\ } X' \xot{\ \beta\ } Y
\]
where $\alpha$ is an opcartesian morphism of $\hpC$ covering $f$
and $\beta$ is a morphism of $\hpC$ covering the identity map of $B$.
The functors 
\eqref{eq:hbullethpctohocalc}
and 
\eqref{eq:hbullethpcdfoptohocalc}
agree on objects.
We use the notation  $\theta\colon (A,X) \oto (B,Y)$ to indicate
that $\theta$ is a morphism from $(A,X)$ to $(B,Y)$ in $\hpC^\dfop$.
Given such a $\theta$ covering a map $f\colon A \to B$
in $\calT$, applying the functor $H_\bullet$ to it 
provides us an ``umkehr map''
\begin{equation}
\label{eq:fthetaumk}
	(f,\theta)^\leftarrow \colon H_\bullet(B;Y) \longto H_\bullet(A;X).
\end{equation}
In Section~U.\ref*{U-sec:comparisonofumkehrmaps},
it is shown that umkehr maps obtained in this way
recover various classically defined umkehr maps.

While in principle any morphism $\theta$ of $\hpC^\dfop$ 
gives rise to an umkehr map as in \eqref{eq:fthetaumk},
it is the cartesian morphisms of $\hpC^\dfop$
that are of particular interest in this regard.
(The opcartesian morphisms of $\hpC^\dfop$
turn out not to induce interesting umkehr maps as
$H_\bullet$ sends opcartesian morphisms to equivalences.)
Recall that $\hpC^\dfop$ is opfibred rather than fibred
over $\calT$, so morphisms of $\calT$
do not necessarily admit cartesian morphisms of $\hpC^\dfop$ 
covering them.
A sufficient condition on a map $f\colon A \to B$
ensuring a plentiful supply of cartesian morphisms
in $\hpC^\dfop$ covering $f$ is that $f$ is 
small-fibred with respect to $\calC$ in the sense of 
Definition~U.\ref*{U-def:smallfibred}:
when $f$ is small-fibred, Theorem~U.\ref*{U-thm:hypercartexistence}
implies that for every object $Y$ of $\hpC^\dfop$
over $B$, there exists a hypercartesian morphism $X \oto Y$
in $\hpC^\dfop$ covering $f$.
(Hypercartesian morphisms are cartesian morphisms with 
particularly strong properties; 
see Definition~U.\ref*{U-def:hypercart}
and Proposition~U.\ref*{U-prop:superandhypercartmorprops}.)
For the purposes of the present paper, 
the crucial example to know is that a map whose homotopy 
fibres are homotopy equivalent to $G$
is small-fibred with respect to $\Spectra^\ell$.
See Theorem~U.\ref*{U-thm:ellcptgrpscwdualizable}.

\subsubsection{The category \texorpdfstring{$h\calF^\fop$}{hF\textasciicircum fop}}
\label{subsubsec:hcalffop}

Our goal now is to construct the category 
$h\calF^\fop$ appearing in \eqref{eq:functorMnew2}.

\begin{defn}
\label{def:calf}
We let $\calF$ be the category defined as follows.
The objects of $\calF$ are fibrations in $\calT$ 
which are small-fibred in $\Spectra^\ell$ in the sense of 
Definition~U.\ref*{U-def:smallfibred},
and a morphisms in $\calF$ from $\pi \co E\to B$
to $\pi'\co E' \to B'$ is a pair $(f,\bar{f})$
of maps making the square
\begin{equation}
\label{eq:calfmor}
\vcenter{\xymatrix{
	E
	\ar[r]^{\bar{f}}
	\ar[d]_\pi
	&
	E'
	\ar[d]^{\pi\smash{'}}
	\\
	B
	\ar[r]^f
	&
	B'
}}
\end{equation}
commutative.
We equip $\calF$ with the symmetric monoidal structure given 
by the direct product
\begin{equation}
\label{eq:calftensorprod} 
	(E\xto{\ \pi\ } B) \times (E' \xto{\ \pi'\ } B') 
	= 
	(E\times E' \xto{\ \pi\times \pi'\ } B\times B');
\end{equation}
the monoidal unit given by the identity map $\pt \xto{\ \id\ }\pt$.
(The product of small-fibred fibrations is again
small-fibred by~U.\ref*{U-prop:cwdualityproduct}.)
\end{defn}

\begin{prop}
\label{prop:pfgobjinf}
Suppose $BG$ is a semisimple $\ell$--compact group. 
Then for all $f,g\colon B \to BG$, the fibration
$\pi_{f,g}\colon P(f,g) \to B$  is an object in $\calF$.
\end{prop}

\begin{proof}
The fibres of  $\pi_{f,g}\colon P(f,g) \to B$ 
are homotopy equivalent to $G$, so the claim follows from
Theorem~U.\ref*{U-thm:ellcptgrpscwdualizable}.
\end{proof}

\begin{rem}
\label{rk:semisimplicityassumption}
The reason why we are working under the assumption that $BG$ is a semisimple
$\ell$--compact group instead of an arbitrary connected $\ell$--compact group
is that Theorem~U.\ref*{U-thm:ellcptgrpscwdualizable} and hence 
Proposition~\ref{prop:pfgobjinf}
fail to hold in that greater generality. 
See Remark~U.\ref*{U-rk:s1lcomnotcwdualizable}.
Indeed, apart from Proposition~\ref{prop:pfgobjinf},
almost everything in the proofs of Theorems~\ref{thm:mainresult}, 
\ref{thm:strtoptezukacrit} and \ref{thm:tezukasubgrp}
generalizes without change to arbitrary connected $\ell$--compact 
groups. The only exception is the proof of Proposition~\ref{prop:Dopen},
which in the more general case requires the elaboration indicated
in Remark~\ref{rk:Dopengeneralization}. 
\end{rem}

The functor $\calF \to \calT$ given by 
sending a fibration $\pi\co E \to B$ to the space $B$ and 
a pair of maps $(f,\bar{f})$ to the map $f$
makes $\calF$ a category fibred over $\calT$
where a morphism $(f,\bar{f})$ is cartesian 
if and only if the corresponding 
square \eqref{eq:calfmor} is a pullback square in $\calT$.
See Definition~U.\ref*{U-def:fibredcats}.

\begin{defn}
We define $\calF^\fop \to \calT$ to be the fibrewise opposite of 
the fibred category $\calF\to \calT$.
See Section~U.\ref*{U-subsubsec:fop}.
\end{defn}

\begin{rem}
\label{rk:calffopdesc}
Explicitly, $\calF^\fop$ is the category
whose objects are the objects of $\calF$, and where
morphisms in $\calF^\fop$ from $\pi \colon E\to B$ 
to $\pi' \colon E'\to B'$ covering a map $f\colon B\to B'$
in $\calT$ are given by equivalence classes of 
zigzags 
\begin{equation}
\label{eq:calffopzigzag}
	\pi \xot{\ (\id_B,\alpha)\ } \tau \xto{\ (f,\bar{f})\ } \pi'
\end{equation}
of morphisms in $\calF$ where $(f,\bar{f})$ is cartesian.
We may depict~\eqref{eq:calffopzigzag} 
more fully as a commutative diagram
\[\vcenter{\xymatrix@!C=0.7em{
	E
	\ar[dr]_{\pi}
	&&
	E''
	\pb
	\ar[dl]^{\tau}
	\ar[ll]_-\alpha
	\ar[rr]^-{\bar{f}}
	&&
	E'
	\ar[d]^{\pi'}
	\\
	&
	B
	\ar[rrr]^f
	&&&
	B'
}}\]
where the trapezoid on the right is a pullback square.
Here two zigzags
\[
	\pi \xot{\ (\id_B,\alpha)\ } \tau \xto{\ (f,\bar{f})\ } \pi'
	\qquad\text{and}\qquad
	\pi \xot{\ (\id_B,\alpha')\ } \tau' \xto{\ (f,\bar{f}')\ } \pi'	
\]
are equivalent if there exists a morphism
$(\id_B,\theta)\co \tau \to \tau'$ such that
$\bar{f} = \bar{f}'\theta$ and $\alpha = \alpha'\theta$;
such a $\theta$ is necessarily 
unique and a homeomorphism. The composite of 
two equivalence classes
\[
	[\pi \xot{\ (\id,\alpha)\ } \tau \xto{\ (f,\bar{f})\ } \pi']
	\qquad\text{and}\qquad
	[\pi' \xot{\ (\id,\alpha')\ } \tau' \xto{\ (f',\bar{f}')\ } \pi'']
\]
is represented by the zigzag given by the composites
along the two sides of the diagram
\[\xymatrix{
	&&
	\sigma
	\ar[dl]_{(\id,\tilde{\alpha}')}
	\ar[dr]^{(f,\tilde{f})}
	\\
	&
	\tau
	\ar[dl]_{(\id,\alpha)}
	\ar[dr]_{(f,\bar{f})}
	&&
	\tau'
	\ar[dl]^{(\id,\alpha')}
	\ar[dr]^{(f',\bar{f}')}
	\\
	\pi
	&&
	\pi'
	&&
	\pi''
}\]
where $\sigma$, $\tilde{\alpha}'$ and $\tilde{f}$
are determined by the requirement that $(f,\tilde{f})$
is cartesian and $\tilde{\alpha}'$ is the unique morphism making
the diamond in the middle commutative.
The symmetric monoidal product on $\calF^\fop$ is 
given by~\eqref{eq:calftensorprod} on objects and by
\[
	[\pi \xot{\, (\id,\alpha)\, } \tau \xto{\, (f,\bar{f})\, } \pi']
	\times
	[\pi' \xot{\, (\id,\alpha')\, } \tau \xto{\, (f',\bar{f}')\, } \pi'']
	=
	[
		\pi\times \pi' 
		\xot{\, (\id\times\id,\alpha\times\alpha')\, } 
		\tau\times\tau'
		\xto{\, (f\times f',\bar{f}\times \bar{f}')\, } 
		\pi'\times \pi''
	]
\]
on morphisms.
\end{rem}

\begin{defn}
\label{def:calffophtopy}
Given an object $\pi \colon E\to B$ in $\calF^\fop$, 
the composite of $\pi$ with  the projection $E\times I \to E$
gives an object $\mathrm{cocyl}(\pi) =  \pi\circ\pr$
of $\calF^\fop$
we call the $\emph{cocylinder}$ of $\pi$ in $\calF^\fop$.
(That $\mathrm{cocyl}(\pi)$ is small-fibred in $\Spectra^\ell$ and 
therefore again an object of $\calF^\fop$
follows from Proposition~U.\ref*{U-prop:cwdualityweyinvariance}.)
The inclusions of $E$ as the two ends of the cylinder $E\times I$ 
define maps $p_0,p_1 \colon \pi \to \mathrm{cocyl}(\pi)$ in 
$\calF^\fop$, and we call two maps 
$\phi_0,\phi_1\colon \pi \to \pi'$ 
in $\calF^\fop$
\emph{homotopic}
if there exists a map 
$\psi \colon \pi \to\mathrm{cocyl}(\pi')$
in $\calF^\fop$ such that $\phi_0 = p_0\psi$ and $\phi_1 = p_1 \psi$. 
Such a map is called a \emph{homotopy} between $\phi_0$ and $\phi_1$.
Finally, 
we define $h\calF^\fop$ to be the homotopy category of $\calF^\fop$
resulting from this notion of homotopy between maps in $\calF^\fop$.
\end{defn}

Explicitly, a homotopy between maps
$\phi_0,\phi_1 \colon \pi \to \pi'$ in $\calF^\fop$ 
is a morphism represented by a diagram
\[\xymatrix@!C=0.7em{
	E
	\ar[dr]_{\pi}
	&&
	E'' \times I
	\pb
	\ar[dl]^{\tau\circ \pr}
	\ar[ll]_-\alpha
	\ar[rr]^-{\bar{f}\times I}
	&&
	E'\times I
	\ar[d]^{\pi'\circ \pr}
	\\
	&
	B
	\ar[rrr]^f
	&&&
	B'
}\]
such that 
$\phi_i = [\pi \xot{\ (\id,\alpha_i)\ } \tau \xto{\ (f,\bar{f})\ } \pi']$,
$i=0,1$, where $\alpha_i \colon E'' \to E$ is the map obtained by
restricting $\alpha$ to $E''\times \{i\} \subset E'' \times I$.
We note that homotopic maps in $\calF^\fop$ cover the same 
map in $\calT$, so that the functor $\calF^\fop \to \calT$
descends to a functor $h\calF^\fop \to \calT$. Moreover,
we note that the symmetric monoidal structure on $\calF^\fop$
descends to one on $h\calF^\fop$.

\subsubsection{The category \texorpdfstring{$\calP_B$}{P\_B} enriched in \texorpdfstring{$(h\calF^\fop)^\op$}{(F\textasciicircum fop)\textasciicircum op}}
\label{subsubsec:calpb2}
Our goal now is to construct the category $\calP_B$
enriched in $(h\calF^\fop)^\op$ discussed at the end of
Section~\ref{subsubsec:strategy2}. 
From Proposition~\ref{prop:pfgobjinf} and 
the description of $\calF^\fop$ given in Remark~\ref{rk:calffopdesc},
we see that diagram~\eqref{diag:pushpull} defines a morphism
\begin{equation}
\label{eq:complaw} 
	\Big(P(f,h)\xto{\pi_{f,h}} B\Big) 
	\longto 
	\Big(P(g,h) \xto{\pi_{g,h}} B\Big) \times \Big(P(f,g) \xto{\pi_{f,g}} B\Big)
\end{equation}
in $\calF^\fop$ and hence in $h\calF^\fop$.
Moreover, for $f\co B \to BG$, we also have a morphism (in $h\calF^\fop$)
\begin{equation}
\label{eq:idmor} 
	\Big(P(f,f) \xto{\pi_{f,f}} B\Big) \longto \Big(\pt \xto{\id} \pt\Big)
\end{equation}
into the monoidal unit
given by the diagram
\[\xymatrix@!C=0.7em{
	P(f,f)
	\ar[dr]_(0.45){\pi_{f,f}}
	&&
	B
	\ar[dl]^(0.45){\id}
	\ar[ll]_-s
	\ar[rr]^-r
	\pb
	&&
	\pt
	\ar[d]^{\id}
	\\
	&
	B
	\ar[rrr]^r
	&&&
	\pt
}\]
Here the map $s$ is given by $s(b) = (b,c_{f(b)})$
where $c_{f(b)}$ denotes the constant path onto $f(b)\in BG$.
Finally, the pullback square \eqref{sq:inducedmap}
gives a map (in $h\calF^\fop$)
\begin{equation}
\label{eq:kphicomp}
	\Big(P(f\phi,g\phi)\xto{\pi_{f\phi,g\phi}} A\Big) 
	\longto 
	\Big(P(f,g) \xto{\pi_{f,g}} B\Big). 
\end{equation}

\begin{defn}
\label{def:calpb2}
The category $\calP_B$ is the category enriched in 
$(h\calF^\fop)^\op$
whose objects are maps $B\to BG$, where
the hom-object of maps from 
$f\co B \to BG$ to $g\co B\to BG$
is $\pi_{f,g}\colon P(f,g) \to B$, and where the composition law
and identity elements are given by 
the maps \eqref{eq:complaw} and \eqref{eq:idmor},
respectively. The functor $K_\phi \co \calP_B \to \calP_A$
is the enriched functor given by the maps \eqref{eq:kphicomp}.
\end{defn}

\begin{rem}
In Definition~\ref{def:calpb2}, we are forced to work 
in the homotopy category $h\calF^\fop$ instead of 
$\calF^\fop$ by the fact that composition of paths is
unital and associative only up to homotopy.
\end{rem}

\subsubsection{The functor \texorpdfstring{$\fH_\bullet$}{fH.}}
\label{subsubsec:fhbullet}

We will now proceed to construct the functor
$\fH_\bullet\colon h\calF^\fop \to \ho(\Mod^{H\F_\ell})$
featuring in \eqref{eq:functorMnew2}.

\begin{notation}
\label{ntn:unotations}
As in \cite{umkehr-maps}, 
for $B$ a space, we write $S_{B,\ell} \in \ho(\Spectra^\ell_{/B})$ 
for the parametrized $\ell$--complete sphere spectrum over $B$.
Moreover, we write $\smashprod^\ell$ (resp.\ $\smashprod^{H\F_\ell}$) 
for the smash product in $\Spectra^\ell$ (resp.\ $\Mod^{H\F_\ell}$)
and $\extsmashprod^\ell$ (resp.\ $\extsmashprod^{H\F_\ell}$) 
for the induced product
on $\hpSpectra^\ell$ (resp.\ $\hpMod^{H\F_\ell}$).
\end{notation}

\begin{defn}[The functor 
$\tilde{\fH}_\bullet \co \calF^\fop \to \ho(\Mod^{H\F_\ell})$]
\label{def:tildefhbullet}
For each object $\pi\colon E\to B$ in $\calF^\fop$, 
fix a hypercartesian morphism 
$\theta_\pi \colon \omega_\pi \oto S_{B,\ell}$
in $(\hpSpectra^\ell)^\dfop$ covering $\pi$.
(Such a hypercartesian morphism exists 
by Theorem~U.\ref*{U-thm:hypercartexistence}.)
When $\pi = \id_B$, we may and do choose
$\theta_\pi$ to be the identity map of $S_{B,\ell}$.
Write $(-)^{H\F_\ell}$ for the functors
\[
	\hpSpectra^\ell \longto \hpMod^{H\F_\ell}
	\qquad\text{and}\qquad
	(\hpSpectra^\ell)^\dfop \longto (\hpMod^{H\F_\ell})^\dfop
\]
induced by the left adjoint 
\[
	H\F_\ell \smashprod^\ell(-) 
	\colon 
	\Spectra^\ell 
	\longto 
	\Mod^{H\F_\ell}
\]
to the forgetful functor
$\Mod^{H\F_\ell} \to \Spectra^\ell$.
See Definition~U.\ref*{U-def:fwfun} and 
Proposition~U.\ref*{U-prop:fprimefw}.
Define a functor 
\[
	\tilde{\fH}_\bullet 
	\colon
	\calF^\fop
	\longto 
	\ho(\Mod^{H\F_\ell})
\]
on objects by setting 
\[
	\tilde{\fH}_\bullet (E\xto{\ \pi\ } B) 
	= 
	H_\bullet(E;\, \omega_\pi^{H\F_\ell})
\]
where
\[
	H_\bullet \colon \hpMod^{H\F_\ell} \longto \ho(\Mod^{H\F_\ell})
\]
is the functor discussed in Section~U.\ref*{U-subsec:hbullethc}.
On morphism, $\tilde{\fH}_\bullet$
is defined by sending the morphism 
$\psi \colon \pi \to \pi'$ of $\calF^\fop$
represented by the diagram
\begin{equation}
\label{diag:calffopmor} 
\vcenter{\xymatrix@!C=0.7em{
	E
	\ar[dr]_{\pi}
	&&
	E''
	\ar[dl]^{\tau}
	\ar[ll]_-\alpha
	\ar[rr]^-{\bar{f}}
	&&
	E'
	\ar[d]^{\pi'}
	\\
	&
	B
	\ar[rrr]^f
	&&&
	B'
}}
\end{equation}
to the morphism constructed as follows.
Construct a commutative diagram 
\begin{equation}
\label{diag:liftedcalffopmor}
\vcenter{\xymatrix@!C=0.7em{
	\omega_\pi
	\ar[dr]|{\circdec}_{\theta_\pi}
	&&
	\zeta
	\ar[dl]|{\circdec}^{\kappa_\tau}
	\ar[ll]|{\circdec}_-{\kappa_\alpha}
	\ar[rr]^-{\phi_{\bar{f}}}_-{\cart}
	&&
	\omega_{\pi'}
	\ar[d]|{\circdec}^{\theta_{\pi'}}
	\\
	&
	S_{B,\ell}
	\ar[rrr]^{\phi_f}_{\cart}
	&&&
	S_{B',\ell}
}}
\end{equation}
in $\hpSpectra^\ell$ and $(\hpSpectra^\ell)^\dfop$
covering \eqref{diag:calffopmor}
by taking $\phi_f$ to be the canonical
cartesian morphism in $\hpSpectra^\ell$,
choosing a cartesian morphism 
$\phi_{\bar{f}} \colon \zeta \to \omega_{\pi'}$
in  $\hpSpectra^\ell$ covering $\bar{f}$,
taking $\kappa_\tau$ to be the base change of 
$\theta_{\pi'}$ along $\phi_f$ and $\phi_{\bar{f}}$
in the sense of Definition~U.\ref*{U-def:basechange},
and by taking $\kappa_\alpha$ in $(\hpSpectra^\ell)^\dfop$
to be the unique morphism covering $\alpha$ making 
the triangle in \eqref{diag:liftedcalffopmor}
commutative. Now $\tilde{\fH}_\bullet(\psi)$
is defined to be the composite
\[
	H_\bullet(E;\, \omega_\pi^{H\F_\ell})
	\xto{\ (\alpha,\kappa_\alpha^{H\F_\ell})^{\leftarrow}}
	H_\bullet(E'';\, \zeta^{H\F_\ell})
	\xto{\ (\bar{f},\phi_{\bar{f}}^{H\F_\ell})_\bullet\ }
	H_\bullet(E';\, \omega_{\pi'}^{H\F_\ell})
\]
where $(\alpha,\kappa_\alpha^{H\F_\ell})^{\leftarrow}$ 
and $(\bar{f},\phi_{\bar{f}}^{H\F_\ell})_\bullet$ are the
maps defined in Definitions~U.\ref*{U-def:hprime} 
and U.\ref*{U-def:hbullethc}, respectively.
\end{defn}

It is readily verified that $\tilde{\fH}_\bullet(\psi)$
as defined above is independent of the choices 
made during its construction.
That $\tilde{\fH}_\bullet$ respects composition of morphisms
follows from 
Proposition~U.\ref*{U-prop:umkehrprops}.
Finally, 
$\tilde{\fH}_\bullet$
is a symmetric monoidal functor.
Given objects $\pi_1 \colon E_1 \to B_1$ 
and $\pi_2 \colon E_2 \to B_2$
in $\calF^\fop$,
the monoidality constraint
\[
	H_\bullet(E_1;\, \omega_{\pi_1}^{H\F_\ell})
	\smashprod^{H\F_\ell}
	H_\bullet(E_2;\, \omega_{\pi_2}^{H\F_\ell})
	\xto{\ \homot\ }
	H_\bullet(E_1\times E_2;\, \omega_{\pi_1\times \pi_2}^{H\F_\ell})
\]
for  $\tilde{\fH}_\bullet$
is the composite of the cross product
\[
	\times
	\colon
	H_\bullet(E_1;\, \omega_{\pi_1}^{H\F_\ell})
	\smashprod^{H\F_\ell}
	H_\bullet(E_2;\, \omega_{\pi_2}^{H\F_\ell})
	\xto{\ \homot\ }
	H_\bullet(
		E_1\times E_2;\, 
		\omega_{\pi_1}^{H\F_\ell}
		\extsmashprod^{H\F_\ell}
		\omega_{\pi_2}^{H\F_\ell}
	)
\]
of  Proposition~U.\ref*{U-prop:hbulletsm}
and the equivalence
\[
	H_\bullet(
		E_1\times E_2;\, 
		\omega_{\pi_1}^{H\F_\ell}
		\extsmashprod^{H\F_\ell}
		\omega_{\pi_2}^{H\F_\ell}
	)
	\xto{\ \homot\ }
	H_\bullet(
		E_1\times E_2;\, 
		\omega_{\pi_1\times \pi_2}^{H\F_\ell}
	)
\]
induced by the monoidality constraint
$\omega_{\pi_1}^{H\F_\ell}
\extsmashprod^{H\F_\ell}
\omega_{\pi_2}^{H\F_\ell}
\homot
(\omega_{\pi_1}
\extsmashprod^\ell
\omega_{\pi_2})^{H\F_\ell}$
of $(-)^{H\F_\ell}$
and the 
equivalence 
$\omega_{\pi_1} \extsmashprod^\ell \omega_{\pi_2}
\homot
\omega_{\pi_1\times \pi_2}$
obtained as follows:
By 
Proposition~U.\ref*{U-prop:superandhypercartmorprops}(\ref*{U-it:hcexttensorhc}),
the map 
\[
	\theta_{\pi_1} \extsmashprod^\ell \theta_{\pi_2}
    \colon 
    \omega_{\pi_1} \extsmashprod^\ell  \omega_{\pi_2}
    \oto
    S_{B_1,\ell} \extsmashprod^\ell S_{B_2,\ell}
\]
is hypercartesian.
Composing it with 
the isomorphism given by 
the  canonical equivalence
$S_{B_1,\ell} \extsmashprod^\ell S_{B_2,\ell}
\homot
S_{B_1 \times B_2,\ell}$
in $\ho(\Spectra^\ell_{/{B_1\times B_2}})$
therefore gives a cartesian morphism 
$\omega_{\pi_1} \extsmashprod^\ell  \omega_{\pi_2}
\oto
S_{B_1\times B_2,\ell}$
in $(\hpSpectra^\ell)^\dfop$
covering $\pi_1\times \pi_2$.
Since 
$\theta_{\pi_1\times\pi_2}
\colon 
\omega_{\pi_1\times\pi_2}
\oto
S_{B_1\times B_2,\ell}$ 
is another cartesian morphism
in $(\hpSpectra^\ell)^\dfop$
covering $\pi_1\times\pi_2$,
the uniqueness of cartesian morphisms
(Proposition~U.\ref*{U-prop:cartmorprops}(\ref*{U-it:cartsourceiso}))
yields the desired equivalence 
$\omega_{\pi_1} \extsmashprod^\ell \omega_{\pi_2}
\homot
\omega_{\pi_1\times \pi_2}$.

\begin{lemma}
\label{lm:fhhtopyinvariance}
Suppose $\phi_0,\phi_1 \colon \pi \to \pi'$ in $\calF^\fop$
are homotopic. Then 
$\tilde{\fH}_\bullet(\phi_0) = \tilde{\fH}_\bullet(\phi_1)$.
\end{lemma}
\begin{proof}
It suffices to show that 
$\tilde{\fH}_\bullet(p_0) = \tilde{\fH}_\bullet(p_1)$
for the maps $p_0,p_1 \colon \mathrm{cocyl}(\pi) \to \pi$
of Definition~\ref{def:calffophtopy}
for every object $\pi \colon E \to B$ of $\calF^\fop$.
Let $i_0,i_1\colon E \to E\times I$ be the 
inclusions of $E$ as the two ends of the cylinder $E\times I$.
By definition, for $\lambda = 0,1$, the map 
$\tilde{\fH}_\bullet(p_\lambda)$
is equal to 
\begin{equation}
\label{eq:ilmabdaumk} 
	H_\bullet(E\times I;\,\omega_{\pi\circ\pr}^{H\F_\ell})
	\xto{\ (i_\lambda,\kappa_\lambda^{H\F_\ell})^{\leftarrow}\ }
	H_\bullet(E;\, \omega_\pi^{H\F_\ell})
\end{equation}
where $\kappa_\lambda \colon \omega_\pi \oto \omega_{\pi\circ\pr}$
is the unique map covering $i_\lambda$ such that 
$\theta_{\pi\circ\pr}\circ \kappa_\lambda= \theta_\pi$.
By the universal property of the cartesian morphism 
$\theta_\pi$, we may factor $\theta_{\pi\circ \pr}$
as a composite 
$\theta_{\pi\circ \pr} = \theta_\pi \circ \kappa_{\pr}$
where $\kappa_{\pr} \colon \omega_{\pi\circ\pr} \oto \omega_{\pi}$
is a cartesian morphism covering the map $\pr\colon E\times I \to E$.
Now $\theta_\pi \circ \kappa_\pr \circ \kappa_\lambda = \theta_\pi$
for $\lambda=0,1$. Since $\theta_\pi$ is cartesian, it follows that 
$\kappa_\pr \circ \kappa_\lambda = \id$ for $\lambda = 0,1$.
Consequently, the composite of \eqref{eq:ilmabdaumk} with 
\begin{equation}
\label{eq:prumk}
	H_\bullet(E;\,\omega_\pi^{H\F_\ell})
	\xto{\ 	(\pr,\kappa_\pr^{H\F_\ell})^{\leftarrow} \ }
	H_\bullet(E\times I;\,\omega_{\pi\circ\pr}^{H\F_\ell})	
\end{equation}
is equal to the identity map of 
$H_\bullet(E;\,\omega_\pi^{H\F_\ell})$
for both $\lambda = 0$ and $\lambda = 1$.
But as the map $\pr\colon E\times I \to E$ 
is a homotopy equivalence,
Proposition~U.\ref*{U-prop:superandhypercartmorprops}(\ref*{U-it:hypercartiffopcart})
implies that the cartesian morphism $\kappa_\pr$ is 
also opcartesian in $(\hpSpectra^\ell)^\dfop$,
so \eqref{eq:prumk} is an equivalence by
Remark~U.\ref*{U-rk:opcartpresfprimefw}
and
Proposition~U.\ref*{U-prop:opcarttoequivdfop}.
The claim follows.
\end{proof}

By Lemma~\ref{lm:fhhtopyinvariance},
the symmetric monoidal functor 
$\tilde{\fH}_\bullet \colon \calF^\fop \to \ho(\Mod^{H\F_\ell})$
factors through the homotopy category $h\calF^\fop$,
yielding the desired symmetric monoidal functor 
\[
	\fH_\bullet \colon h\calF^\fop \longto \ho(\Mod^{H\F_\ell}).
\]

\subsubsection{Identifying the result}
\label{subsubsec:result2}

We have now constructed the categories
$(H^\ast \fH_\bullet^\op)_\ast\calP_B$
enriched in graded $\F_\ell$--modules
we set out to construct at the end of
Section~\ref{subsubsec:strategy2},
along with the enriched functors
\[
	(H^\ast \fH_\bullet^\op)_\ast (K_\phi)
	\co
	(H^\ast \fH_\bullet^\op)_\ast \calP_B
	\longto 
	(H^\ast \fH_\bullet^\op)_\ast \calP_A
\]
between them. 
To complete the construction of the pairing $\stringprod'$
and the proof of 
Theorem~\ref{thm:pairingssecondconstruction},
it now suffices to prove 
the following result.

\begin{thrm}
\label{thm:recognitionthm3}
Suppose $BG$ is a semisimple $\ell$--compact group.
For all $f,g\co B \to BG$, there is an isomorphism
\begin{equation}
\label{eq:recognitionthmiso} 
	H^\ast H_\bullet(P(f,g);\,\omega_{\pi(f,g)}^{H\F_\ell})
	\isom 
	\bbH^\ast(P(f,g))
\end{equation}
natural with respect to the homomorphisms induced by diagram
\eqref{sq:inducedmap}.
\end{thrm}

Here we have written $\pi(f,g)$ for $\pi_{f,g}$
to avoid an excess of subscripts. The rest of the 
subsubsection is dedicated to the proof of 
Theorem~\ref{thm:recognitionthm3}.

\begin{notation}
As in \cite{umkehr-maps}, 
given an $H\F_\ell$--module $M \in \ho(\Mod^{H\F_\ell})$ and a space $B$,
we write $\underline{M} \in \ho(\Mod^{H\F_\ell}_{/B})$ for the trivial parametrized
$H\F_\ell$--module over $B$ defined by $M$.
\end{notation}

With the above notation, we have natural isomorphisms
\begin{equation}
\label{eq:isosbbhpfg2}
\begin{aligned}
	\bbH^\ast P(f,g)
	=
	H^{\ast+d} P(f,g)
	&\isom
	H^{\ast+d} H_\bullet(P(f,g);\, \underline{H\F_\ell})
	\\
	&\isom
	H^\ast(\suspension^{-d} H_\bullet(P(f,g);\, \underline{H\F_\ell}))
	\isom
	H^\ast H_\bullet(P(f,g);\, \suspension^{-d}_{P(f,g)}\underline{H\F_\ell}).
\end{aligned}
\end{equation}
Here the isomorphism on the first line 
follows from Example~U.\ref*{U-ex:hbulletfortrivialcoeffs2},
the first isomorphism on the second line is given by the 
suspension isomorphism, and the final isomorphism
is induced by an instance of U.\eqref*{U-eq:fshriektcomm}.
To prove Theorem~\ref{thm:recognitionthm3},
it is therefore enough to construct an equivalence
\begin{equation}
\label{eq:omegapifghfltriv}
	\omega_{\pi(f,g)}^{H\F_\ell} 
	\homot
	\suspension^{-d}_{P(f,g)} \underline{H\F_\ell}
\end{equation}
in $\ho(\Mod^{H\F_\ell}_{/P(f,g)})$
natural with respect to the maps induced
by \eqref{sq:inducedmap}.

Let us be more explicit about what these maps induced
by \eqref{sq:inducedmap} are. First, tracing through the isomorphisms
\eqref{eq:isosbbhpfg2}, 
the relevant map induced  
by \eqref{sq:inducedmap}
on the right hand side 
of \eqref{eq:omegapifghfltriv}
is the cartesian morphism 
$\suspension^{-d}_{P(f\phi,g\phi)} \underline{H\F_\ell} 
\to 
\suspension^{-d}_{P(f,g)} \underline{H\F_\ell}$
obtained by applying the functor 
$\suspension^{-d}_\fw$ to the cartesian morphism
$\underline{H\F_\ell} \to  \underline{H\F_\ell}$ covering the map
$\bar{\phi} \colon P(f\phi,g\phi) \to P(f,g)$.
On the other hand, an inspection of 
the definition of $\fH_\bullet$ reveals that  
the map induced by \eqref{sq:inducedmap}
on the left hand side 
of \eqref{eq:omegapifghfltriv}
is induced by the unique cartesian morphism 
$\omega_{\pi(f\phi,g\phi)}\to \omega_{\pi(f,g)}$
making the square on the left below 
a commutative square (in 
$\hpSpectra^\ell$ and $(\hpSpectra^\ell)^\dfop$, 
in the sense of Definition~U.\ref*{U-def:mixedcomm})
covering the square on the right.
\[
	\xymatrix{
    	\omega_{\pi(f\phi,g\phi)}
    	\ar[r]^\cart
    	\ar[d]|\circdec_{\theta_{\pi(f\phi,g\phi)}}
    	&
    	\omega_{\pi(f,g)}
    	\ar[d]|\circdec^{\theta_{\pi(f,g)}}
    	\\
    	S_{A,\ell}
    	\ar[r]^\cart
    	&
    	S_{B,\ell}
	}
	\qquad\qquad
	\xymatrix{
		P(f\phi,g\phi)
		\ar[d]_{\pi_{f\phi,g\phi}}
		\ar[r]^{\bar{\phi}}
		&
		P(f,g)
		\ar[d]^{\pi_{f,g}}
		\\
		A
		\ar[r]^{\phi}
		&
		B
	}
\]
See 
Propositions~U.\ref*{U-prop:dfopcartmixedprop}
and
U.\ref*{U-prop:cartandhypercartrel}.
Here the bottom map on the left is the canonical cartesian morphism
$S_{A,\ell} \to S_{B,\ell}$ covering $\phi$.

Notice that for all $f,g\colon B \to BG$, we have a canonical
pullback square
\[\xymatrix{
	P(f,g)
	\ar[d]_{\pi_{f,g}}
	\ar[r]
	\pb
	&
	P(\pi_1,\pi_2)
	\ar[d]^{\pi_{\pi_1,\pi_2}}
	\\
	B
	\ar[r]^-{(f,g)}
	&
	BG\times BG	
}\]
where $\pi_1,\pi_2\co BG\times BG \to BG$ 
are the coordinate projections (see
Proposition~\ref{prop:specialcasesofpathspaceconstr}\eqref{it:pathspaceconstrbgi}).
In view of the above discussion, 
to construct the desired 
natural equivalence~\eqref{eq:omegapifghfltriv},
it is therefore enough to prove the following lemma.
The equivalence \eqref{eq:omegapifghfltriv}
for general $f$ and $g$ then follows by
pulling back \eqref{eq:omegapiunivtriv}
along the cartesian morphisms
$\omega_{\pi(f,g)} \to \omega_{\pi(\pi_1,\pi_2)}$
and 
$\suspension^{-d}_{P(f,g)} \underline{H\F_\ell}
\to
\suspension^{-d}_{P(\pi_1,\pi_2)}\underline{H\F_\ell}$.

\begin{lemma}
\label{lm:omegapipipitriv}
There exists an equivalence
\begin{equation}
\label{eq:omegapiunivtriv}
 	\omega_{\pi(\pi_1,\pi_2)}^{H\F_\ell} 
	\homot
	\suspension^{-d}_{P(\pi_1,\pi_2)}\underline{H\F_\ell}
\end{equation}
in $\ho(\Mod^{H\F_\ell}_{/P(\pi_1,\pi_2)})$.
\end{lemma}
In the proof of Lemma~\ref{lm:omegapipipitriv},
we will need 
\begin{lemma}
\label{lm:omegapipipifibres}
The fibres of 
$\omega_{\pi(\pi_1,\pi_2)}\in \ho(\Spectra^\ell_{/P(\pi_1,\pi_2)})$ are $(-d)$--dimensional $\ell$--complete spheres. 
\end{lemma}
\begin{proof}
As the functor 
$b^{[\ast]} \colon \Ex_{P(\pi_1,\pi_2)}(\Spectra^\ell) \to  \Ex(\Spectra^\ell)$
of Appendix~U.\ref*{U-app:exbc}
induced by the inclusion $b \colon \pt \to P(\pi_1,\pi_2)$
of the basepoint into $P(\pi_1,\pi_2)$
preserves dual pairs, it follows from 
Theorem~U.\ref*{U-thm:hypercartdata} that 
$b^{[\ast]} (\omega_{\ev_1})$ is the 
Costenoble--Waner dual of $G$ in $\Spectra^\ell$.
The claim now follows from 
Theorem~U.\ref*{U-thm:ellcptgrpscwdualizable}.
\end{proof}

\begin{proof}[Proof of Lemma~\ref{lm:omegapipipitriv}]
For brevity, let us write $Q$ for 
$\omega_{\pi(\pi_1,\pi_2)}^{H\F_\ell}$.
Since $P(\pi_1,\pi_2) \homeom BG^I$ is simply connected,
by Theorem~U.\ref*{U-thm:serresss}
there is a spectral sequence
\[
	E^{s,t}_2 = H^s(P(\pi_1,\pi_2)) 
	\tensor 
	H^t(Q_b)
	\Longrightarrow 
	H^{s+t} H_\bullet({P(\pi_1,\pi_2)};\, Q)
\]
where $b\in P(\pi_1,\pi_2)$ is a basepoint.
By Lemma~\ref{lm:omegapipipifibres},
the spectral sequence is concentrated on the $t=-d$ line,
so the spectral sequence collapses on the $E_2$ page
and converges strongly to the indicated target. Let 
\[
	u \in H^{-d}H_\bullet({P(\pi_1,\pi_2)};\, Q)
\]
be the class corresponding to the class
$1\tensor x \in H^0(P(\pi_1,\pi_2)) \tensor H^{-d}(Q_b)$
where $x$ is a generator of $H^{-d}(Q_b) \isom \F_\ell$.
By functoriality of the spectral sequence, 
the class $u$ now has the property 
that for every $b\in P(\pi_1,\pi_2)$, the restriction of
$u$ to $H^{-d} H_\bullet(\{b\}; Q_b) \isom H^{-d}(Q_b) \isom \F_\ell$
is a generator. 
Letting $r$ to be the unique map from $P(\pi_1,\pi_2)$ to $\pt$,
the class $u$ is equivalent to the data of a map
\[
	r_! Q \longto \suspension^{-d} H\F_\ell 
\]
in $\ho(\Mod^{H\F_\ell})$, which via the 
$(r_!,r^\ast)$ adjunction
is equivalent to the data of a map
\[
	\tilde{u}
	\co
	Q 
	\longto
	r^\ast (\suspension^{-d}H\F_\ell)
	\homot
	\suspension^{-d}_{P(\pi_1,\pi_2)}\underline{H\F_\ell}
\]
in $\ho(\Mod^{H\F_\ell}_{/P(\pi_1,\pi_2)})$. 
The property that $u$ restricts to a generator for
each fibre is easily seen to translate
precisely to the property that 
$\tilde{u}$ is an equivalence on all fibres. 
Since equivalences in 
$\ho(\Mod^{H\F_\ell}_{/P(\pi_1,\pi_2)})$
are detected on fibres, the claim follows.
\end{proof}

The construction of the pairing $\stringprod'$
and the proof of 
Theorem~\ref{thm:pairingssecondconstruction}
are now complete.

\begin{rem}
\label{rk:orientation2}
The equivalence \eqref{eq:omegapiunivtriv} 
in Lemma~\ref{lm:omegapipipitriv} is not unique:
the set of all such equivalences forms an $\F_\ell^\times$--torsor,
as follows by reversing the argument in the proof of
Lemma~\ref{lm:omegapipipitriv}. 
The choice of such an equivalence should be thought of as
an $H\F_\ell$--orientation for the ``sphere bundle''
$\omega_{\pi(\pi_1,\pi_2)}$ over $P(\pi_1,\pi_2)$.
This indeterminacy in the construction of 
the pairing $\stringprod'$ parallels the one found in the construction
of the pairing $\stringprod$, which depended on 
the choice of an orientation for $G$ (encoded as 
the choice of an isomorphism $H^d(G)\isom  \F_\ell$ in \eqref{as:gorientation}).
Notice that changing the choice of equivalence \eqref{eq:omegapiunivtriv}
results in pairings $\stringprod'$ which differ from each other 
by some nonzero scalar multiple.
\end{rem}

\subsection{Comparison of the two constructions of the string pairing}
\label{subsec:comparison}

This subsection is dedicated to the proof of the following theorem.

\begin{thrm}
\label{thm:pairingcomparison}
For a suitable choice of the equivalence \eqref{eq:omegapiunivtriv},
the pairing $\stringprod'$ constructed in 
Section~\ref{subsec:secondconstruction}
agrees with the pairing $\stringprod$ constructed in 
Section~\ref{subsec:firstconstruction}.
\end{thrm}
\noindent
The proof is based on Theorem~U.\ref*{U-thm:umkehrmapcomparison}
which expresses the integration along fibre maps $(p,o)_!$ featuring 
in the construction of the  pairing $\stringprod$ 
in terms of the umkehr maps $(p,\theta)^\leftarrow$ 
featuring in the construction of the pairing $\stringprod'$.

In view of Theorem~\ref{thm:pairingcomparison},
from the next subsection onwards, we agree to choose 
equivalence \eqref{eq:omegapiunivtriv}
so that the pairings $\stringprod$ and $\stringprod'$
agree, and cease to distinguish the pairings notationally.
For the remainder of this subsection, however, we will 
 use the more precise notation $\stringprod'_t$
for the pairing $\stringprod'$ resulting by the construction of 
Section~\ref{subsec:secondconstruction}
from an equivalence
\begin{equation} 
\label{eq:trivt}
	t 
	\colon
	\omega_{\pi(\pi_1,\pi_2)}^{H\F_\ell} 
	\xto{\ \homot\ }
	\suspension^{-d}_{P(\pi_1,\pi_2)}\underline{H\F_\ell}
\end{equation}
where $\pi_1,\pi_2 \colon BG^2 \to BG$ are the coordinate projections.
We call such an equivalence $t$ a \emph{universal trivialization}.
Rephrasing Theorem~\ref{thm:pairingcomparison}, 
our task is to show that $\stringprod'_t = \stringprod$ for a
suitably chosen $t$.

\begin{notation}
For a space $B$ and integer $n$, write $u^n$ for the composite isomorphism
\[	
	u^n 
	\colon 
	H^\ast(B) 
	\xto{\ \isom\ }
	H^\ast H_\bullet(B;\,\underline{H\F_\ell})
	\xto{\ \isom\ }
	H^{\ast+n} (\suspension^n H_\bullet(B;\,\underline{H\F_\ell}))
	\xto{\ \isom\ }
	H^{\ast+n} H_\bullet(B;\,\suspension_B^n\underline{H\F_\ell})
\]
where the first isomorphism
follows from Example~U.\ref*{U-ex:hbulletfortrivialcoeffs2},
the second is given by the suspension isomorphism, and 
the final isomorphism is an instance of U.\eqref*{U-eq:fshriektcomm}.
Compare with \eqref{eq:isosbbhpfg2}. 
\end{notation}

\begin{notation}
Given a universal trivialization $t$ and maps
$f,g\colon B \to BG$, write $t_{f,g}$ for the  
equivalence
\[
	t_{f,g} 
	\colon 
	\omega_{\pi(f,g)}^{H\F_\ell}
	\xto{\ \homot\ }
	\suspension_B^{-d} \underline{H\F_\ell}
\] 
induced by $t$.
\end{notation}

\begin{notation}
We write $\blackdiamond$
for the composition in the categories
$(H^\ast \fH_\bullet^\op)_\ast\calP_B$
enriched in graded $\F_\ell$--modules.
\end{notation}

With the above notations, 
by the construction of $\stringprod'_t$,
we have the following lemma.
\begin{lemma}
\label{lm:bullettandblackdiamond}
Given a universal trivialization $t$, 
the following diagram commutes for all $f,g,h \colon B \to BG$ and $a,b\in\Z$:
\[	
	\pushQED{\qed}
	\raisebox{\dimexpr\depth-\fboxsep-6.25pt}{%
    	\xymatrix{
        	\bbH^a P(g,h) \tensor \bbH^b P(f,g)
        	\ar[r]^-{\stringprod'_t}
        	\ar[d]_{u^{-d} s^d\tensor u^{-d} s^d}^\isom
        	&
        	\bbH^{a+b} P(f,h)
        	\ar[d]^{u^{-d} s^d}_\isom
        	\\
        	H^a H_\bullet(P(g,h);\,\suspension_{P(g,h)}^{-d} \underline{H\F_\ell})
        	\tensor
        	H^b H_\bullet(P(f,g);\,\suspension_{P(f,g)}^{-d} \underline{H\F_\ell})
        	\ar[d]_{t_{g,h}^\ast \tensor t_{f,g}^\ast}^\isom
        	&
        	H^{a+b} H_\bullet(P(f,h);\,\suspension_{P(f,h)}^{-d} \underline{H\F_\ell})
        	\ar[d]^{t_{f,h}^\ast}_\isom
        	\\
        	H^a H_\bullet(P(g,h);\, \omega_{\pi(g,h)}^{H\F_\ell})
        	\tensor
        	H^b H_\bullet(P(f,g);\, \omega_{\pi(f,g)}^{H\F_\ell})
        	\ar[r]^-{\blackdiamond}
        	&
        	H^{a+b} H_\bullet(P(f,h);\, \omega_{\pi(f,h)}^{H\F_\ell}) 
		}
	} 
	\qedhere
    \popQED
\] 
\end{lemma}

We will make use of the following description of the product  $\blackdiamond$.

\begin{lemma}
\label{lm:blackdiamondproddesc}
Let $f,g,h \colon B \to BG$. Suppose
\begin{equation}
\label{diag:blackdiamondproddesc}
\vcenter{\xymatrix@!C=2.2em{
	\omega_{\pi(g,h)} \extsmashprod^\ell \omega_{\pi(f,g)}
	\ar[d]|{\circdec}_{
		\theta_{\pi(g,h)} \extsmashprod^\ell \theta_{\pi(f,g)}
	}
	&&
	\zeta
	\ar[ll]_-{\phi}
	\ar[rr]|{\circdec}^{\kappa}
	\ar[dr]|{\circdec}_\theta
	&&
	\omega_{\pi(f,h)}
	\ar[dl]|{\circdec}^{\theta_{\pi(f,h)}}
	\\
	S_{B,\ell} \extsmashprod^\ell S_{B,\ell}
	&&&
	S_{B,\ell}
	\ar[lll]_{\phi_{\Delta}}
}}
\end{equation}
is a commutative diagram in $\hpSpectra^\ell$ and $(\hpSpectra^\ell)^\dfop$ 
covering 
\begin{equation}
\label{diag:pprimepushpull}
\vcenter{\xymatrix@!C=2.2em{
	P(g,h) \times P(f,g)
	\ar[d]_{\pi_{g,h}\times \pi_{f,g}}
	&&
	P'(f,g,h)
	\ar[ll]_-{\mathrm{split}\mathrlap{'}}
	\ar[rr]^-{\mathrm{concat}'}
	\ar[dr]_(0.4){\pi'_{f,g\mathrlap{,h}}}
	&&
	P(f,h)
	\ar[dl]^{\pi_{f,h}}
	\\
	B\times B
	&&&
	B
	\ar[lll]_\Delta
}} 
\end{equation}
such that the trapezoid in the upper diagram is a base change square
in the sense of Definition~U.\ref*{U-def:basechange}.
Here $\phi_\Delta$ is the cartesian morphism covering $\Delta$
given by the composite of the canonical cartesian morphism 
$S_{B,\ell} \to S_{B\times B,\ell}$ covering $\Delta$ and the canonical equivalence
$S_{B\times B,\ell}\homot S_{B,\ell} \extsmashprod^\ell  S_{B,\ell}$.
Then the pairing
\begin{equation}
\label{eq:blackdiamondpairing}
	\blackdiamond
	\colon
	H^\ast H_\bullet(P(g,h);\omega_{\pi(g,h)}^{H\F_\ell})
	 \tensor 
	H^\ast H_\bullet(P(f,g);\omega_{\pi(f,g)}^{H\F_\ell})
	\longto
	H^\ast H_\bullet(P(f,h);\omega_{\pi(f,h)}^{H\F_\ell}) 
\end{equation}
agrees with the composite
\begin{equation}
\label{eq:blackdiamondfactorization}
\begin{aligned}
	H^\ast H_\bullet(P(g,h);\,\omega_{\pi(g,h)}^{H\F_\ell})
	\tensor 
	H^\ast H_\bullet(&P(f,g);\,\omega_{\pi(f,g)}^{H\F_\ell})
	\\
	\xto{\qquad\qquad\mathclap{\times}\qquad\qquad}
	&
	H^\ast H_\bullet(
		P(g,h) \times P(f,g);\, 
		(\omega_{\pi(g,h)} \extsmashprod^\ell \omega_{\pi(f,g)})^{H\F_\ell}
	) 
	\\
	\xto{
		\qquad\qquad
		\mathclap{H^{\ast}((\mathrm{split}', \phi^{H\F_\ell})_\bullet)}
		\qquad\qquad
	}
	&
	H^\ast H_\bullet(
		P'(f,g,h);\,
		\zeta^{H\F_\ell}
	)
	\\
	\xto{
		\qquad\qquad
		\mathclap{
			H^{\ast}((\mathrm{concat}', \kappa^{H\F_\ell})^{\leftarrow})
		}
		\qquad\qquad
	}
	&
	H^\ast H_\bullet(
		P(f,h);\,
		\omega_{\pi(f,h)}^{H\F_\ell}
	).
\end{aligned}
\end{equation}
Here the first map is given by the composite
\begin{equation}
\label{eq:timesexpansion}
\begin{aligned}
	H^\ast H_\bullet(P(g,h);\,\omega_{\pi(g,h)}^{H\F_\ell})
	\tensor 
	H^\ast H_\bullet(&P(f,g);\,\omega_{\pi(f,g)}^{H\F_\ell})
	\\
	\xto{\qquad\qquad\mathclap{\times}\qquad\qquad}
	&
	H^\ast\bigl(
		H_\bullet(P(g,h);\,\omega_{\pi(g,h)}^{H\F_\ell})
		\smashprod^{H\F_\ell}
		H_\bullet(P(f,g);\,\omega_{\pi(f,g)}^{H\F_\ell})
	\bigr)
	\\
	\xto{\qquad\qquad\mathclap{\isom}\qquad\qquad}
	&
	H^\ast H_\bullet\bigl(
		P(g,h)\times P(f,g);\,
		\omega_{\pi(g,h)}^{H\F_\ell}
		\extsmashprod^{H\F_\ell}
		\omega_{\pi(f,g)}^{H\F_\ell}
	\bigr)
	\\
	\xto{\qquad\qquad\mathclap{\isom}\qquad\qquad}
	&
	H^\ast H_\bullet\bigl(
		P(g,h) \times P(f,g);\, 
		(\omega_{\pi(g,h)} \extsmashprod^\ell \omega_{\pi(f,g)})^{H\F_\ell}
	\bigr) 
\end{aligned}
\end{equation}
of the cross product on cohomology and the maps induced 
by the monoidality constraints for the functors
$H_\bullet$ and $(-)^{H\F_\ell}$.
\end{lemma}

\begin{rem}
A diagram as in \eqref{diag:blackdiamondproddesc}
satisfying the requirements of Lemma~\ref{lm:blackdiamondproddesc}
exists for all $f,g,h\colon B\to BG$.
To construct an  example, 
one can choose $\phi$ to be a cartesian morphism with the indicated
target covering 
$\mathrm{split}'$, 
obtain $\theta$ from 
$\theta_{\pi(g,h)} \extsmashprod^\ell \theta_{\pi(f,g)}$
by base change, and finally obtain $\kappa$
from $\theta$ by using the universal property of the 
cartesian morphism $\theta_{\pi(f,h)}$.
\end{rem}

\begin{rem}
\label{rk:phicartkappahypercart}
In any diagram \eqref{diag:blackdiamondproddesc} satisfying the 
requirements of Lemma~\ref{lm:blackdiamondproddesc},
the map $\phi$ is cartesian and the map $\kappa$ is
hypercartesian. That $\phi$ is cartesian is part of the assumption
that the trapezoid in \eqref{diag:blackdiamondproddesc} 
is a base change square. To see that $\kappa$ is hypercartesian,
notice that the product
$\theta_{\pi(g,h)} \extsmashprod^\ell \theta_{\pi(f,g)}$
is hypercartesian by 
Proposition~U.\ref*{U-prop:superandhypercartmorprops}(\ref*{U-it:hcexttensorhc}),
so by the definition of hypercartesian morphisms 
the morphism $\theta$ is supercartesian and hence cartesian. 
Since $\theta_{\pi(f,h)}$ is cartesian, by properties
of cartesian morphisms it follows that the morphism $\kappa$ is 
cartesian; see Proposition~U.\ref*{U-prop:cartmorprops}(\ref*{U-it:cartfactor}).
The map $\mathrm{concat}'$ is a fibration with 
fibres homotopy equivalent to $G$, so $\mathrm{concat}'$
is small-fibred with respect to $\Spectra^\ell$
in the sense of Definition~U.\ref*{U-def:smallfibred},
and Theorem~U.\ref*{U-thm:hypercartexistence}
implies that $\mathrm{concat}'$ is covered with a hypercartesian
morphism $\kappa'$ whose target is the target of $\kappa$.
By the uniqueness of cartesian morphism 
(Proposition~U.\ref*{U-prop:cartmorprops}(\ref*{U-it:cartsourceiso})),
$\kappa$ factors as a composite of $\kappa'$ and an isomorphism, so 
Proposition~U.\ref*{U-prop:superandhypercartmorprops}(\ref*{U-it:isosarehypercart})
and (\ref*{U-it:hypercartcomp}) imply that $\kappa$ is hypercartesian, as claimed.
\end{rem}

\begin{proof}[Proof of Lemma~\ref{lm:blackdiamondproddesc}]
Given the data in the lemma, choose a cartesian morphism
$\psi\colon \hat{\zeta} \to \zeta$ in $\hpSpectra^\ell$
covering the vertical homotopy equivalence 
in diagram \eqref{diag:pandpprime}, and set $\hat{\phi} = \phi \circ \psi$.
Since $\psi$ covers a homotopy equivalence, it is also
opcartesian in $\hpSpectra^\ell$ by Proposition~U.\ref*{U-prop:wecartiffopcart}.
Therefore we may also interpret $\psi$ as a morphism in $(\hpSpectra^\ell)^\dfop$,
and form the composites $\hat{\theta} = \theta \circ \psi$ and 
$\hat{\kappa} = \kappa \circ \psi$.
Now 
\begin{equation}
\label{diag:hats}
\vcenter{\xymatrix@!C=2.2em{
	\omega_{\pi(g,h)} \extsmashprod^\ell \omega_{\pi(f,g)}
	\ar[d]|{\circdec}_{
		\theta_{\pi(g,h)} \extsmashprod^\ell \theta_{\pi(f,g)}
	}
	&&
	\hat{\zeta}
	\ar[ll]_-{\hat{\phi}}
	\ar[rr]|{\circdec}^{\hat{\kappa}}
	\ar[dr]|{\circdec}_{\hat{\theta}}
	&&
	\omega_{\pi(f,h)}
	\ar[dl]|{\circdec}^{\theta_{\pi(f,h)}}
	\\
	S_{B,\ell} \extsmashprod^\ell S_{B,\ell}
	&&&
	S_{B,\ell}
	\ar[lll]_{\phi_{\Delta}}
}}
\end{equation}
is a commutative diagram in $\hpSpectra^\ell$ and $(\hpSpectra^\ell)^\dfop$
covering the analogue of \eqref{diag:pprimepushpull}
where $P'(f,g,h)$ has been replaced with $P(f,g,h)$,
and the trapezoid in \eqref{diag:hats} is a base change square.
Tracing through the construction of $\blackdiamond$,
we see that \eqref{eq:blackdiamondpairing}
is given by the composite obtained from 
\eqref{eq:blackdiamondfactorization}
by replacing
$\mathrm{split}'$, $\mathrm{concat}'$, $\zeta$, $\phi$, and $\kappa$ by
$\mathrm{split}$, $\mathrm{concat}$, $\hat{\zeta}$, $\hat{\phi}$, and $\hat{\kappa}$,
respectively. Since $\psi$ is opcartesian in 
$\hpSpectra^\ell$,
by Proposition~U.\ref*{U-prop:opcartpreservation}
the map $\psi^{H\F_\ell}$
is opcartesian in $\hpMod^{H\F_\ell}$.
Thus $\psi^{H\F_\ell}$ induces an equivalence
upon application of $H_\bullet$, and the claim follows.
\end{proof}

\begin{lemma}
\label{lm:compdiags}
Let $t$ be a universal trivialization, 
let $f,g,h\colon B \to BG$, and suppose we have been given 
a diagram \eqref{diag:blackdiamondproddesc}
satisfying the conditions of Lemma~\ref{lm:blackdiamondproddesc}.
Then the following diagrams commute
for all $a,b\in\Z$:
\begin{equation}
\label{diag:comp1}
\vcenter{\xymatrix@!0@C=15em@R=8.5ex{
	H^{a+d} P(g,h) \tensor H^{b+d} P(f,g)
	\ar[rd]^\times
	\ar[dd]_{u^{-d}\tensor u^{-d}}^\isom
	\\
	&
	H^{a+b+2d}(P(g,h) \times P(f,g))
	\ar[dd]^{u^{-2d}}_\isom
	\\
	H^a H_\bullet(P(g,h);\, \suspension_{P(g,h)}^{-d} \underline{H\F_\ell})
	\tensor
	H^b H_\bullet(P(f,g);\, \suspension_{P(f,g)}^{-d} \underline{H\F_\ell})
	\ar[rd]^\times
	\ar[dd]_{t_{g,h}^\ast \tensor t_{f,g}^\ast}^\isom
	\\
	&
	H^{a+b} H_\bullet (
		P(g,h)\times P(f,g);\, 
		\suspension_{P(g,h)\times P(f,g)}^{-2d}\underline{H\F_\ell}
	)
	\ar[dd]^{t_1^\ast}_\isom
	\\
	H^a H_\bullet (P(g,h);\, \omega_{\pi(g,h)}^{H\F_\ell})
	\tensor
	H^b H_\bullet (P(f,g);\, \omega_{\pi(f,g)}^{H\F_\ell})
	\ar[rd]^\times
	\\
	&
	H^{a+b} H_\bullet(
		P(g,h)\times P(f,g);\, 
		(\omega_{\pi(g,h)}\extsmashprod^\ell \omega_{\pi(f,g)})^{H\F_\ell}
	)
}}
\end{equation}
\begin{equation}
\label{diag:comp2}
\vcenter{\xymatrix@!0@C=15em@R=8.5ex{
	H^{a+b+2d}(P(g,h) \times P(f,g))
	\ar[dd]_{u^{-2d}}^\isom
	\ar[rd]|*+<0.2em>{_{(\mathrm{split}')^\ast}}
	\\
	&
	H^{a+b+2d} P'(f,g,h)
	\ar[dd]^{u^{-2d}}_\isom
	\\
	H^{a+b} H_\bullet (
		P(g,h)\times P(f,g);\, 
		\suspension_{P(g,h)\times P(f,g)}^{-2d}\underline{H\F_\ell}
	)
	\ar[dd]_{t_1^\ast}^\isom
	\ar[rd]|*+<0.2em>{_{H^\ast((\mathrm{split}',\phi_{\mathrm{can}})_\bullet)}}
	\\
	&
	H^{a+b} H_\bullet (P'(f,g,h);\, \suspension_{P'(f,g,h)}^{-2d}\underline{H\F_\ell})
	\ar[dd]^{t_2^\ast}_\isom
	\\
	H^{a+b} H_\bullet(
		P(g,h)\times P(f,g);\, 
		(\omega_{\pi(g,h)}\extsmashprod^\ell \omega_{\pi(f,g)})^{H\F_\ell}
	)
	\ar[rd]|*+<0.2em>{_{H^\ast((\mathrm{split}',\phi)_\bullet)}}
	\\
	&
	H^{a+b} H_\bullet (P'(f,g,h);\, \zeta^{H\F_\ell})
}}
\end{equation}
\begin{equation}
\label{diag:comp3}
\vcenter{\xymatrix@!0@C=15em@R=8.5ex{
	H^{a+b+2d} P'(f,g,h)
	\ar[dd]_{u^{-2d}}^\isom
	\ar[rd]|*+<0.2em>{_{(\mathrm{concat}',\kappa_{t;f,g,h})_\sharp}}
	\\
	&
	H^{a+b+d} P(f,h)
	\ar[dd]^{u^{-d}}_\isom
	\\
	H^{a+b} H_\bullet (P'(f,g,h);\, \suspension_{P'(f,g,h)}^{-2d}\underline{H\F_\ell})
	\ar[dd]_{t_2^\ast}^\isom
	\ar[rd]|*+<0.2em>{_{H^\ast((\mathrm{concat}',\kappa_{t;f,g,h})^{\leftarrow})}}
	\\
	&
	H^{a+b} H_\bullet (P(f,h);\, \suspension_{P(f,h)}^{-d} \underline{H\F_\ell})
	\ar[dd]^{t_{f,h}^\ast}_\isom
	\\
	H^{a+b} H_\bullet (P'(f,g,h);\, \zeta^{H\F_\ell})
	\ar[rd]|*+<0.2em>{_{H^\ast((\mathrm{concat}',\kappa^{H\F_\ell})^{\leftarrow})}}
	\\
	&
	H^{a+b} H_\bullet (P(f,h);\, \omega_{\pi(f,h)}^{H\F_\ell})
}}
\end{equation}
Here the bottom map $\times$ in \eqref{diag:comp1} is 
the composite \eqref{eq:timesexpansion};
the middle map $\times$ in \eqref{diag:comp1} is the composite
\begin{equation*}
\begin{aligned}
	H^a H_\bullet(P(g,h);\, \suspension_{P(g,h)}^{-d} \underline{H\F_\ell})
	\tensor
	H^b &H_\bullet(P(f,g);\, \suspension_{P(f,g)}^{-d} \underline{H\F_\ell})
	\\
	\xto{\qquad\qquad\mathclap{\times}\qquad\qquad}
	&
	H^{a+b}\bigl(
		H_\bullet(P(g,h);\, \suspension_{P(g,h)}^{-d} \underline{H\F_\ell})
		\smashprod^{H\F_\ell}
		H_\bullet(P(f,g);\, \suspension_{P(f,g)}^{-d} \underline{H\F_\ell})
	\bigr)
	\\
	\xto{\qquad\qquad\mathclap{\isom}\qquad\qquad}
	&
	H^{a+b} H_\bullet\bigl(
		P(g,h)\times P(f,g);\,
		\suspension_{P(g,h)}^{-d} \underline{H\F_\ell}
		\extsmashprod^{H\F_\ell}
		\suspension_{P(f,g)}^{-d} \underline{H\F_\ell}
	\bigr)
	\\
	\xto{\qquad\qquad\mathclap{\isom}\qquad\qquad}
	&
	H^{a+b} H_\bullet\bigl(
		P(g,h) \times P(f,g);\, 
		\suspension_{P(g,h)_\times P(f,g)}^{-2d} \underline{H\F_\ell}
	\bigr) 
\end{aligned}
\end{equation*}
where the first isomorphism is induced by the monoidality 
constraint of $H_\bullet$ and the second one  by the evident equivalence
\[
	\suspension_{P(g,h)}^{-d} \underline{H\F_\ell}
	\extsmashprod^{H\F_\ell}
	\suspension_{P(f,g)}^{-d} \underline{H\F_\ell}
	\homot 
	\suspension_{P(g,h)_\times P(f,g)}^{-2d} \underline{H\F_\ell};
\]
$t_1$ is the composite
\begin{multline*}
	t_1 
	\colon 
	(\omega_{\pi(g,h)}\extsmashprod^\ell \omega_{\pi(f,g)})^{H\F_\ell}
	\xto{\quad\homot\quad}
	\omega_{\pi(g,h)}^{H\F_\ell} \extsmashprod^{H\F_\ell} \omega_{\pi(f,g)}^{H\F_\ell}
	\\
	\xto[\homot]{\ (t_{g,h}) \extsmashprod^{H\F_\ell} (t_{f,h})\ }
	\suspension_{P(g,h)}^{-d} \underline{H\F_\ell} 
	\extsmashprod^{H\F_\ell} 
	\suspension_{P(f,g)}^{-d} \underline{H\F_\ell}	
	\xto{\quad\homot\quad}
	\suspension_{P(g,h)\times P(f,g)}^{-2d} \underline{H\F_\ell}
 \end{multline*}
where the first equivalence is the inverse of the monoidality constraint for 
$(-)^{H\F_\ell}$ and the last equivalence is the evident one;
\[
    \phi_{\mathrm{can}} 
    \colon  
    \suspension_{P'(f,g,h)}^{-2d}\underline{H\F_\ell}
    \longto
    \suspension_{P(g,h)\times P(f,g)}^{-2d}\underline{H\F_\ell}
\]
is the canonical cartesian morphism covering $\mathrm{split}'$;
\[
	t_2
	\colon 
	\zeta^{H\F_\ell} 
	\xto{\ \homot\ }
	\suspension_{P'(f,g,h)}^{-2d} \underline{H\F_\ell}
\]
is the unique morphism in $\ho(\Mod^{H\F_\ell}_{/P'(f,g,h}))$
making the square
\[\xymatrix{
	\zeta^{H\F_\ell} 
	\ar[d]_{t_2}^{\homot}
	\ar[r]^-{\phi_\mathrm{can}}_-\cart
	&
	(\omega_{\pi(g,h)}\extsmashprod^\ell \omega_{\pi(f,g)})^{H\F_\ell}
	\ar[d]^{t_1}_{\homot}
	\\
	\suspension_{P'(f,g,h)}^{-2d} \underline{H\F_\ell}
	\ar[r]^-{\phi^{H\F_\ell}}_-\cart
	&
	\suspension_{P(g,h)\times P(f,g)}^{-2d} \underline{H\F_\ell}
}\]
commutative;
\[
	\kappa_{t;f,g,h} 
	\colon
	\suspension_{P'(f,g,h)}^{-2d} \underline{H\F_\ell}
	\longoto
	\suspension_{P(f,h)}^{-d} \underline{H\F_\ell}
\]
is the morphism in $(\hpMod^{H\F_\ell})^\dfop$
covering $\mathrm{concat}'$
corresponding to %
$\kappa^{H\F_\ell} \colon \zeta^{H\F_\ell} \oto \omega_{\pi(f,h)}^{H\F_\ell}$
under the equivalences
\[
	t_2 \colon \zeta^{H\F_\ell} 
	\xto{\ \homot\ } 
	\suspension_{P'(f,g,h)}^{-2d} \underline{H\F_\ell}
	\qquad\text{and}\qquad
	t_{f,h} \colon \omega_{\pi(f,h)}^{H\F_\ell} 
	\xto{\ \homot\ } 
	\suspension_{P(f,h)}^{-d} \underline{H\F_\ell};
\]
and $(\mathrm{concat}',\kappa_{t;f,g,h})_\sharp$ is 
the umkehr map of Definition~U.\ref*{U-def:psharpumkehrmaps}.
\end{lemma}

\begin{rem}
By Remark~\ref{rk:phicartkappahypercart} and 
Proposition~U.\ref*{U-prop:hypercartmorpreservation},
the morphism $\kappa^{H\F_\ell}$ and hence the morphism $\kappa_{t;f,g,h}$
in Lemma~\ref{lm:compdiags} are hypercartesian.
\end{rem}

\begin{rem}
\label{rk:ududsign}
In accordance with the Koszul sign rule, 
the morphism $u^{-d} \tensor u^{-d}$ in diagram~\eqref{diag:comp1}
is given by
\[
	(u^{-d} \tensor u^{-d})(x\tensor y) = (-1)^{d(a+d)} u^{-d}(x) \tensor u^{-d}(y)
\]
for $x\tensor y \in H^{a+d} P(g,h) \tensor H^{b+d} P(f,g)$.
\end{rem}

\begin{proof}[Proof of Lemma~\ref{lm:compdiags}]
The bottom parallelograms in \eqref{diag:comp1}, \eqref{diag:comp2}, \eqref{diag:comp3}
commute by the choice of $t_1$, $t_2$, and $\kappa_{t;f,g,h}$, respectively.
That the top parallelogram in \eqref{diag:comp1} commutes is a tedious but
essentially straightforward exercise in unrolling the relevant definitions.
Checking the commutativity of the top parallelogram in \eqref{diag:comp2}
is straightforward, and the top parallelogram in \eqref{diag:comp3}
commutes by the definition of the umkehr map
$(\mathrm{concat}',\kappa_{t;f,g,h})_\sharp$.
\end{proof}

Stacking diagrams~\eqref{diag:comp1}, \eqref{diag:comp2}, 
and \eqref{diag:comp3}
together horizontally, and taking into account the sign 
explained in Remark~\ref{rk:ududsign},
Lemmas~\ref{lm:bullettandblackdiamond}, \ref{lm:blackdiamondproddesc}, and
\ref{lm:compdiags} imply 
\begin{lemma}
Let $t$ be a universal trivialization, and 
let $f,g,h\colon B \to BG$.
Then for all $x\in \bbH^a P(g,h)$, $y\in \bbH^b P(f,g)$
\[
	x\stringprod'_t y 
	= 
	(-1)^{d(a+d)}
	s^{-d} \bigl( 
		(\mathrm{concat}',\kappa_{t;f,g,h})_\sharp
		(\mathrm{split}')^\ast 
		(s^d(x) \times s^d(y))
	\bigr)
\]
where $\kappa_{t;f,g,h}$ is as in Lemma~\ref{lm:compdiags}.
\qed
\end{lemma}

\begin{proof}[Proof of Theorem~\ref{thm:pairingcomparison}]
Suppose $f,g,h\colon B \to BG$.
By Theorem U.\ref*{U-thm:umkehrmapcomparison}, 
for any universal trivialization $t$ we have
\[
	(\mathrm{concat}',\kappa_{t;f,g,h})_\sharp = (\mathrm{concat}',o^{\kappa(t;f,g,h)})_!
\]
where
\[
	o^{\kappa(t;f,g,h)} 
	\colon 
	\calH^d(F)
	\longto
	\F_\ell
\]
is the map of local coefficient systems over $P(f,h)$ defined in
Definition~U.\ref*{U-def:othetat}.
Here $F$ denotes the fibre
of the fibration $\mathrm{concat}' \colon P'(f,g,h) \to P(f,h)$
and we have written $\kappa(t;f,g,h)$ for
$\kappa_{t;f,g,h}$.
Notice that since $F\homot G$ is connected,
Remark~U.\ref*{U-rk:othetaaltdesc} implies that
$o^{\kappa(t;f,g,h)}$
is an orientation for $\mathrm{concat}'$
in the sense of Definition~\ref{def:integrationalongfibre},
as required in our definition of integration along fibre maps
there.
Thus we obtain the formula
\[
	x\stringprod'_t y 
	= 
	(-1)^{d(\deg(x)+d)}
	s^{-d} \bigl( 
		(\mathrm{concat}',o^{\kappa(t;f,g,h)})_!
		(\mathrm{split}')^\ast 
		(s^d(x) \times s^d(y))
	\bigr)
\]
for the product
\[
	\stringprod'_t 
	\colon 
	\bbH^\ast P(g,h) \tensor \bbH^\ast P(f,g)
	\longto
	\bbH^\ast P(f,h).
\]
Comparing this formula with the definition of $\stringprod$ at 
Definition~\ref{def:bulletdef1}, we see that 
\[
	\stringprod = \stringprod'_t 
	\colon 
	\bbH^\ast P(g,h) \tensor \bbH^\ast P(f,g)
	\longto
	\bbH^\ast P(f,h)
\]
as long as $(-1)^d o^{\kappa(t;f,g,h)}$ agrees with the 
orientation for $\mathrm{concat}'$ used in the construction 
of $\stringprod$. Let us write $o^{f,g,h}$ for the latter orientation.
Given a map $\phi \colon A \to B$ of spaces, notice that
\[
	o^{\kappa(t;f\phi,g\phi,h\phi)} = \bar{\phi}^\ast (o^{\kappa(t;f,g,h)})
	\qquad\text{and}\qquad
	o^{f\phi,g\phi,h\phi} = \bar{\phi}^\ast (o^{f,g,h})
\]
where $\bar{\phi} \colon P(f\phi,h\phi) \to P(f,h)$ is the map induced by $\phi$.
Consequently, to ensure that $\stringprod = \stringprod'_t$ for all $f$,$g$, and $h$,
it suffices to choose $t$ so that 
\begin{equation}
\label{eq:targetorel}
 	o^{\kappa({t;p_1,p_2,p_3})} = (-1)^d o^{p_1,p_2,p_3}
\end{equation}
for $p_1,p_2,p_3 \colon BG^3 \to BG$ the coordinate projections.

Recall that a universal trivialization is an equivalence
\[
	\omega_{\pi(\pi_1,\pi_2)}^{H\F_\ell} 
	\xto{\ \homot\ }
	\suspension^{-d}_{P(\pi_1,\pi_2)}\underline{H\F_\ell}
\]
where $\pi_1,\pi_2 \colon BG^2 \to BG$ are the
coordinate projections, and that an orientation for 
\[
	\mathrm{concat}'\colon P'(p_1,p_2,p_3) \longto P(p_1,p_3)
\]
amounts to the data of an isomorphism $H^d(F) \isom \F_\ell$
for a fibre $F$ of $\mathrm{concat}'$.
See Remark~\ref{rk:orientationwithconnectedbase}.
Thus the orientations  of $\mathrm{concat}'$ form an $\F_\ell^\times$--torsor.
Moreover,
associated to each unit $u \in \F_\ell^\times$ we have an automorphism
of $H\F_\ell$, and hence an automorphism of 
the parametrized $H\F_\ell$--module 
$\suspension_{P(\pi_1,\pi_2)}^{-d}\underline{H\F_\ell}$
over $P(\pi_1,\pi_2)$, and postcomposition with 
these automorphism yields an action of $\F_\ell^\times$ 
on the set of universal trivializations. 
Tracing through the construction of $o^{\kappa({t;f,g,h})}$, 
we see that 
\[
	o^{\kappa({ut;p_1,p_2,p_3})} = u o^{\kappa({t;p_1,p_2,p_3})}
\]
for all $u \in \F_\ell^\times$ and universal trivializations $t$.
Consequently, given a  universal trivialization $t_0$, we may 
find a unit $u\in \F_\ell^\times$ such that $t = ut_0$
satisfies \eqref{eq:targetorel}.
\end{proof}

\subsection{Further properties of the string pairing}
\label{subsec:furtherpropertiesofthepairing}

As mentioned after Theorem~\ref{thm:pairingcomparison},
we will from now on assume that orientations 
have been chosen so that the pairings $\stringprod$
and $\stringprod'$ agree, and we will no longer 
distinguish between the two notationally.
Armed with the first two constructions of the pairing $\stringprod$
and Theorem~\ref{thm:pairingcomparison},
in this subsection we will establish further properties
of the pairing $\stringprod$. We begin with

\begin{proof}[Proof of Theorems~\ref{thm:pairings4} and~\ref{thm:functoriality}]
In view of Theorem~\ref{thm:pairingcomparison},
Theorem~\ref{thm:pairings4}
follows by combining 
Theorem~\ref{thm:pairingssecondconstruction}(\ref{it:pairingssecondconstructioncats})
with 
Proposition~\ref{prop:cupmodstrbilin}
while 
Theorem~\ref{thm:functoriality}(\ref{it:functoriality2})
follows from 
Theorem~\ref{thm:pairingssecondconstruction}(\ref{it:pairingssecondconstructionfuns}).
Finally, 
Theorem~\ref{thm:functoriality}(\ref{it:functorialityandmodstr})
follows readily from Definition~\ref{def:cupmodstr}
and the construction of the map $F_\phi$.
\end{proof}

\begin{proof}[Proof of Theorem~\ref{thm:cmcomparison2}]
The claim is immediate from 
Corollary~\ref{cor:bulletalgandmodstr}(\ref{it:bulletalgstr}) of
Theorem~\ref{thm:pairings4}; Proposition~\ref{prop:cmcomp}; 
and the fact that Chataur and Menichi's product $\odot$ on $\bbH^\ast(LBG)$
is commutative \cite[Cor.~B.3]{KM19}.
\end{proof}

Having proven 
Theorem~\ref{thm:functoriality}(\ref{it:functoriality2}),
we note that it implies the following homotopy invariance property for the 
pairing $\stringprod$.
\begin{prop}
\label{prop:homotopyinvariance}
Let $f_i,g_i,h_i\co B \to BG$, $i=0,1$ be maps, and let 
$H_f\co f_0\homot f_1$,
$H_g\co g_0\homot g_1$, and
$H_h\co h_0\homot h_1$
be homotopies.
Let $j_i \colon B \to B\times I$ be the inclusions
$b \mapsto (b,i)$, $i=0,1$.
Then the following diagram commutes:
\[
    \pushQED{\qed}
    \raisebox{\dimexpr\depth-\fboxsep-3.0pt}{%
		\xymatrix{
        	\bbH^\ast P(g_0,h_0) \tensor \bbH^\ast P(f_0,g_0) 	
        	\ar[r]^(0.62)\stringprod
        	\ar[d]^\isom_{F_{j_0}\tensor F_{j_0}}
        	&
        	\bbH^\ast P(f_0,h_0)
        	\ar[d]_\isom^{F_{j_0}}
        	\\
        	\bbH^\ast P(H_g,H_h) \tensor \bbH^\ast P(H_f,H_h)
        	\ar[r]^(0.62)\stringprod
        	&
        	\bbH^\ast P(H_f,H_g)
        	\\
        	\bbH^\ast P(g_1,h_1) \tensor \bbH^\ast P(f_1,g_1) 
        	\ar[r]^(0.62)\stringprod
        	\ar[u]^{F_{j_1}\tensor F_{j_1}}_\isom
        	&
        	\ar[u]_{F_{j_1}}^\isom
        	\bbH^\ast P(f_1,h_1) 
		}
    }
    \qedhere
    \popQED
\]
\end{prop}

In the remainder of the subsection, 
we will show that
the $H^\ast(B)$--algebra structure 
on $\bbH^\ast P(f,f)$ 
admits an augmentation 
\[
	\rho = \rho_f \colon \bbH^\ast P(f,f) \longto H^\ast B
\]
as asserted after Corollary~\ref{cor:bulletalgandmodstr}.
Moreover, we will show that these augmentations 
are compatible with the maps
\[
	F_\phi \colon \bbH^\ast P(f,f) \longto \bbH^\ast P(f\phi,f\phi)
\]
of equation~\eqref{eq:fphidef}.
We start by constructing the map
$\rho$ and providing what will turn out 
to be an alternative construction of the map 
\begin{equation}
\label{eq:iota}
	\iota = \iota_f \co H^\ast B \longto \bbH^\ast P(f,f),
	\quad
	a \longmapsto a\bbOne_f
\end{equation}
of Corollary~\ref{cor:bulletalgandmodstr}(\ref{it:bulletalgstr}).

\begin{defn}%
For a map $f\co B \to BG$, consider the commutative triangles
\begin{equation}
\label{triangles:algstr} 
    \vcenter{\xymatrix@!0@R=8ex@C=4em{
    	B 
    	\ar[rr]^-s
    	\ar[dr]_{\id_B}
    	&&
    	P(f,f)
    	\ar[dl]^{\pi_{f,f}}
    	\\
    	&
    	B
    }}
    \qquad\text{and}\qquad
    \vcenter{\xymatrix@!0@R=8ex@C=4em{
    	P(f,f)
    	\ar[rr]^-{\pi_{f,f}}
    	\ar[dr]_{\pi_{f,f}}
    	&&
    	B
    	\ar[dl]^{\id}
    	\\
    	&
    	B
    }}
\end{equation}
where $s$ is the section sending a point $b\in B$ to the pair consisting of $b$ and the constant path 
at $f(b)$. The triangles define morphisms 
\begin{equation}
\label{mors:iotarhopre} 
	(B\xto{\id} B) \longto (P(f,f)\xto{\pi(f,f)} B)
	\qquad\text{and}\qquad
	(P(f,f) \xto{\pi(f,f)} B) \longto (B\xto{\id} B)
\end{equation}
in the category $(h\calF^\fop)^\op$. 
We define the maps
\begin{equation}
\label{eq:iotaprimeandrho}
	\iota' = \iota'_f \co H^\ast B \longto \bbH^\ast P(f,f)
	\qquad\text{and}\qquad
	\rho = \rho_f \co \bbH^\ast P(f,f) \longto H^\ast B 
\end{equation}
to be the morphisms obtained by first applying 
the composite functor $H^\ast \fH_\bullet^\op$
to the morphisms~\eqref{mors:iotarhopre}
and then using the isomorphism
\[
	H^\ast H_\bullet(P(f,f);\,\omega_{\pi(f,f)}^{H\F_\ell})
	\isom 
	\bbH^\ast P(f,f)
\]
of Theorem~\ref{thm:recognitionthm3} and the computation
\begin{equation*}
     H^\ast H_\bullet(B; \omega_{\id_B}^{H\F_\ell})
     = 
     H^\ast H_\bullet(B; S_{B,\ell}^{H\F_\ell})
     \isom
     H^\ast H_\bullet(B; \underline{H\F_\ell})
     \isom
     H^\ast B
\end{equation*}
to recognize the source and the target.
The equality 
follows from our choice $\omega_{\id_B} = S_{B,\ell}$ 
(see Definition~\ref{def:tildefhbullet}),
the first isomorphism is induced by the unitality 
constraint for the functor $(-)^{H\F_\ell}$,
and the second isomorphism follows from 
Example~U.\ref*{U-ex:hbulletfortrivialcoeffs2}.
\end{defn}

\begin{lemma}
\label{lm:iotaprimeisiota}
Given $f\colon B \to BG$, we have
$\iota'_f = \iota_f \colon H^\ast B \to \bbH^\ast P(f,f)$.
\end{lemma}

\begin{prop}
\label{prop:augmentation}
Let $B$ be a space, and let $f\co B\to BG$ be a map. 
Then the map 
\[
	\rho = \rho_f \co (\bbH^\ast P(f,f),\stringprod) \longto (H^\ast B,\cupprod)
\]
is a homomorphism of graded rings
satisfying $\rho_f \iota_f = \id_{H^\ast(B)}$.
\end{prop}

\begin{proof}[Proof of Lemma~\ref{lm:iotaprimeisiota} and Proposition~\ref{prop:augmentation}]
The diagrams
\[
    \vcenter{\xymatrix{
    	B
    	\ar[r]^-\Delta
    	\ar[d]_\id
    	&
    	B\times B 	
    	\ar[d]^{\id\times\id}
    	\\
    	B
    	\ar[r]^-\Delta
    	&
    	B\times B
    }}
    \qquad\text{and}\qquad
    \vcenter{\xymatrix{
    	B 
    	\ar[r]^r
    	\ar[d]_\id
    	&
    	\pt
    	\ar[d]^\id
    	\\
    	B 
    	\ar[r]^r
    	&
    	\pt
    }}
\]
define morphisms in $h\calF^\fop$ making 
the identity map $(B \xto{\id} B)$ into a monoid
object in the category $(h\calF^\fop)^\op$. 
Moreover, the morphisms \eqref{eq:idmor} and 
\eqref{eq:complaw} (with $g=h=f$) make $P(f,f)\to B$
a monoid object in  $(h\calF^\fop)^\op$.
It is straightforward to verify that the 
morphisms~\eqref{mors:iotarhopre}
in $(h\calF^\fop)^\op$
are monoid object homomorphisms whose composite is the identity.
Applying the functor $H^\ast \fH_\bullet^\op$,
recognizing the image of the monoid object
$(B\xto{\id} B)$ under this functor as $H^\ast(B)$ equipped
with the cup product, and using Theorem~\ref{thm:recognitionthm3}
to recognize the image of $(P(f,f)\to B)$ as the graded ring 
$\bbH^\ast P(f,f)$,
we see that $\iota'_f$ and $\rho_f$ are $\F_\ell$--algebra
homomorphisms with $\iota'_f\rho_f = \id$.
Thus Proposition~\ref{prop:augmentation} 
follows from Lemma~\ref{lm:iotaprimeisiota}.
Comparing the definition of $\iota'_f$ 
to that of $\bbOne_f\in \bbH^\ast P(f,f)$, 
we see that $\iota'_f(1) = \bbOne_f$.
Thus to prove Lemma~\ref{lm:iotaprimeisiota}, it suffices
to show that the map $\iota'_f$ is $H^\ast(B)$--linear.
To do that, notice that the map defined by the diagram
\begin{equation}
\label{diag:pffbbbmodstr}
	\vcenter{\xymatrix@C+3.4em{
		P(f,f)
		\ar[r]^-{(\pi_{f,f},\id)}
		\ar[d]_{\pi_{f,f}}
		&
		B \times P(f,f)
		\ar[d]^{\id \times \pi_{f,f}}
		\\
		B
		\ar[r]^{\Delta}
		&
		B\times B
	}}
\end{equation}
makes $(P(f,f) \xto{\pi(f,f)} B)$ into a module object over the monoid $(B\xto{\id} B)$
in the category $(h\calF^\fop)^\op$, and that map 
\[
	(B\xto{\id} B) \longto (P(f,f) \xto{\pi(f,f)} B)
\]
in $(h\calF^\fop)^\op$ defined
by the diagram on the left in \eqref{triangles:algstr}
is a homomorphism of $(B\to B)$--modules.
It follows that $\iota'_f$ is $H^\ast(B)$--linear 
when we equip $\bbH^\ast P(f,f)$ with the 
$H^\ast(B)$--module structure induced by
\eqref{diag:pffbbbmodstr}.
The proof is now completed by 
the verification that this $H^\ast(B)$--module
structure  on $\bbH^\ast P(f,f)$
agrees with the one we placed on $\bbH^\ast P(f,f)$
in Definition~\ref{def:cupmodstr}.
\end{proof}

\begin{prop}
\label{prop:functorcompat}
Given a map of spaces $\phi\co A \to B$ and a map $f\co B \to BG$,
the following diagram commutes:
\begin{equation}
\label{diag:functorcompat}
\vcenter{\xymatrix@C+1em{
	H^\ast B 
	\ar[r]^{\phi^\ast}
	\ar[d]_{\iota_f}
	&
	H^\ast A
	\ar[d]^{\iota_{f\phi}}
	\\
	\bbH^\ast P(f,f)
	\ar[r]^{F_\phi}
	\ar[d]_{\rho_f}
	&
	\bbH^\ast P(f\phi,f\phi)
	\ar[d]^{\rho_{f\phi}}
	\\
	H^\ast B 
	\ar[r]^{\phi^\ast}
	&
	H^\ast A	
}} 
\end{equation}
\end{prop}

\begin{proof}%
The commutativity of the top square is
immediate from the definitions of the maps
involved. It remains to prove the commutativity
of the bottom square.
Consider the diagram in $(h\calF^\fop)^\op$ 
\[\xymatrix{
	(P(f,f) \to B)
	\ar[r]
	\ar[d]
	&
	(P(f\phi,f\phi) \to A)
	\ar[d]
	\\
	(B\xto{\id} B)
	\ar[r]
	&
	(A\xto{\id} A)
}\]
where the vertical arrows are induced by the right-hand diagram
in~\eqref{triangles:algstr} and the analogous diagram for $f\phi$;
where the top horizontal arrow is induced by 
square~\eqref{sq:inducedmap}; and
where the bottom horizontal arrow is induced by the pullback square
\[\xymatrix{
	A
	\ar[r]^\phi
	\ar[d]_\id
	\pb
	&
	B
	\ar[d]^\id
	\\
	A
	\ar[r]^\phi
	&
	B
}\]
It is straightforward to verify 
that the diagram commutes. The claim now follows by 
applying the functor $H^\ast \fH_\bullet^\op$
to the diagram and recognizing the result as 
the bottom square in diagram~\eqref{diag:functorcompat}.
\end{proof}

\subsection{The third construction of the string pairing}
\label{subsec:thirdperspective}

In this subsection, we will give a third perspective on the 
construction of the string pairing
\[
	\stringprod
	\co
	\bbH^\ast P(g,h) \tensor \bbH^\ast P(f,g) 
	\longto 
	\bbH^\ast P(f,h)
\]
by reinterpreting the second construction of $\stringprod$
in terms of fibrewise duals,
and use this perspective to prove
Proposition~\ref{prop:pontryaginproduct}.
Later, in
Section~\ref{sec:spectralsequences},
we will use this point of view
on the pairing $\stringprod$ to
prove that the pairing lifts to the
level of Serre spectral sequences.
We begin by constructing a functor sending 
an object $(E\to B)\in h\calF^\fop$ to its 
``fibrewise dual with respect to $H\F_\ell$.''

\begin{defn}[The fibrewise dual functor $D_\fw$]
\label{def:dfw}
Construct a symmetric monoidal morphism 
\begin{equation}
\label{eq:dfwhcalffoptohpmodhfell}
	D_\fw \colon h\calF^\fop \longto \hpMod^{H\F_\ell}
\end{equation}
of fibrations over $\calT$ (see Definition~U.\ref*{U-def:smfibs})
as follows. 
Write $\hpMod^{H\F_\ell}_{\mathrm{dl}}$ 
for the full subcategory of $\hpMod^{H\F_\ell}$ 
spanned by the objects admitting a fibrewise dual, 
so that a parametrized $H\F_\ell$--module $X$ over a space $B$ 
belongs
to $\hpMod^{H\F_\ell}_{\mathrm{dl}}$ if and only if 
$X$ is dualizable in $\ho(\Mod^{H\F_\ell}_B)$,
and notice that $\hpMod^{H\F_\ell}_{\mathrm{dl}}$
can alternatively be constructed by applying the Grothendieck construction
of Theorem~U.\ref*{U-thm:smgrothendieckconstr}
to the pseudofunctor
\[
	\calT^\op \longto \smCat,
	\quad
	B\longmapsto \ho(\Mod^{H\F_\ell}_{/B})_{\mathrm{dl}},
	\quad
	f\longmapsto f^\ast
\]
where $\ho(\Mod^{H\F_\ell}_{/B})_{\mathrm{dl}}$ is the 
full subcategory of $\ho(\Mod^{H\F_\ell}_{/B})$
spanned by dualizable objects. Applying
the Grothendieck construction to the pseudo natural transformation
given by the functors
\[
	D_B
	\colon 
	\ho(\Mod^{H\F_\ell}_{/B})_{\mathrm{dl}}^\op
	\longto 
	\ho(\Mod^{H\F_\ell}_{/B})
\]
sending each object to its dual
therefore yields a symmetric monoidal morphism 
\[
	\tilde{D}
	\colon
	(\hpMod^{H\F_\ell}_{\mathrm{dl}})^\fop 
	\longto 
	\hpMod^{H\F_\ell}
\]	
of fibrations over $\calT$.
Now recall from Proposition~U.\ref*{U-prop:tcalc2}
the symmetric monoidal functor 
\[
	t = t_{\Mod^{H\F_\ell}} \colon \pT \longto \hpMod^{H\F_\ell}
\]
over $\calT$
given on objects by
\[
	t (E\xto{\pi} B) = \pi_! \underline{H\F_\ell},
\]
and notice that Proposition~U.\ref*{U-prop:hypercartandfwduality}
combined with Theorem~U.\ref*{U-thm:hypercartexistence}
and Proposition~U.\ref*{U-prop:hypercartmorpreservation}
ensures that the restriction of $t$  
to the subcategory $\calF$ of $\pT$ takes values 
in the subcategory $\hpMod^{H\F_\ell}_{\mathrm{dl}}$.
Thus $t$ restricts to a functor
\[
	t' \colon \calF \longto \hpMod^{H\F_\ell}_{\mathrm{dl}}
\]
which with the aid of Proposition~U.\ref*{U-prop:tcalc2} 
is easily verified to be a symmetric monoidal 
morphism of fibrations in the sense of 
Definition~U.\ref*{U-def:smfibs}.
The composite
\begin{equation}
\label{eq:calffoptohpmodhfell}
	\calF^\fop 
	\xto{\ (t')^\fop\ }
	(\hpMod^{H\F_\ell}_{\mathrm{dl}})^\fop 
	\xto{\ \tilde{D}\ }
	\hpMod^{H\F_\ell}
\end{equation}
factors through the projection $\calF^\fop \to h\calF^\fop$, 
and we define $D_\fw$ in \eqref{eq:dfwhcalffoptohpmodhfell}
to be the symmetric monoidal morphism of fibrations
induced by \eqref{eq:calffoptohpmodhfell}.
We will often write $D_B E$ for the image of an object
$(E\to B) \in h\calF^\fop$ under $D_\fw$.
\end{defn}

The following proposition is the key result of this subsection.

\begin{prop}
\label{prop:fhbulletfactorization}
The functor $\fH_\bullet$
of Section~\ref{subsubsec:fhbullet}
is symmetric monoidally naturally equivalent to the composite
\begin{equation}
\label{eq:thirdconstructioncomposite}
	h\calF^\fop
	\xto{\ D_\fw\ }
	\hpMod^{H\F_\ell}
	\xto{\ H_\bullet\ }
	\ho(\Mod^{H\F_\ell}).
\end{equation}
\end{prop}
\begin{proof}
At an object $(E\xto{\,\pi\,} B) \in h\calF^\fop$, the equivalence is 
given by the composite
\begin{equation}
\label{eq:hbulletehbulletdbe}
	H_\bullet(E;\, \omega_\pi^{H\F_\ell}) 
	\homot
	H_\bullet(B;\, \pi_! \omega_\pi^{H\F_\ell})
	\homot
	H_\bullet(B;\, D_B E) 
\end{equation}
of the equivalences afforded by 
Corollary~U.\ref*{U-cor:hbulletandshriek}
and Proposition~U.\ref*{U-prop:hypercartandfwduality}.
That these composites assemble to a natural equivalence
between $\fH_\bullet$ and $H_\bullet D_\fw$
follows from Proposition~U.\ref*{U-prop:dualmaps}.
Finally, a tedious verification using the explicit description 
of the unit of the dual pair 
$(\pi_! \underline{H\F_\ell}, \pi_! \omega_{\pi})$
given in Proposition~U.\ref*{U-prop:hypercartandfwduality},
Proposition~U.\ref*{U-prop:superandhypercartmorprops}(\ref*{U-it:hcexttensorhc}),
and the fact that $\extsmashprod^{H\F_\ell}$ preserves both cartesian and
opcartesian morphism in $\hpMod^{H\F_\ell}$
(Definition~U.\ref*{U-def:smfibs} 
and Proposition~U.\ref*{U-prop:exttensoropcarts})
shows that the natural equivalence
so obtained is symmetric monoidal.
\end{proof}

By Proposition~\ref{prop:fhbulletfactorization},
we may replace the functor $\fH_\bullet$
in the second construction of the pairing $\stringprod$
by the composite $H_\bullet D_\fw$.
Thus we have

\begin{cor}
\label{cor:bulletconstr3}
For $f,g\co B \to BG$, there is an isomorphism
\begin{equation}
\label{eq:recognitionthmiso2} 
	\xi_{f,g}
	\colon
	H^\ast H_\bullet(B;\,D_B P(f,g))
	\isom 
	\bbH^\ast(P(f,g))
\end{equation}
natural with respect to the homomorphisms induced by diagram
\eqref{sq:inducedmap}
so that for all $f,g,h\colon B \to BG$, the 
pairing 
\[	
	\stringprod\colon \bbH^\ast P(g,h)\tensor \bbH^\ast P(f,g) \longto \bbH^\ast P(f,h)
\]
corresponds 
under $\xi_{g,h}$, $\xi_{f,g}$ and $\xi_{f,h}$
to the pairing
\begin{equation}
\label{eq:dfwpairing}
	H^\ast H_\bullet(B;D_B P(g,h))
	\tensor
	H^\ast H_\bullet(B;D_B P(f,g))
	\longto
	H^\ast H_\bullet(B;D_B P(f,h)) 
\end{equation}
given by composition in the enriched category 
$(H^\ast H_\bullet^\op D_\fw^\op)_\ast \calP_B$.
\qed
\end{cor}

\begin{rem}
\label{rk:xifgdesc}
Explicitly, tracing through the construction
we see that the isomorphism $\xi_{f,g}$ of 
\eqref{eq:recognitionthmiso2}
is given by the inverse of the composite
\begin{equation}
\label{eq:xifgdesc} 
	\bbH^\ast P(f,g) 
	\isom 
	H^\ast H_\bullet (P(f,g);\, \suspension^{-d}_{P(f,g)} \underline{H\F_\ell})
	\isom 
	H^\ast H_\bullet (P(f,g);\, \omega^{H\F_\ell}_{\pi(f,g)})
	\isom
	H^\ast H_\bullet (B; D_B P(f,g))
\end{equation}
where the first isomorphism is given by \eqref{eq:isosbbhpfg2},
the second by \eqref{eq:omegapifghfltriv}, and the third by
\eqref{eq:hbulletehbulletdbe}.
\end{rem}

\begin{proof}[Proof of Proposition~\ref{prop:pontryaginproduct}]
The product on $\bbH^\ast \loops BG$ agrees under the isomorphism
\[
	\bbH^\ast \loops BG  \isom H^\ast D\loops BG
\]
provided  by Corollary~\ref{cor:bulletconstr3}
with the composite
\def\firstentry{H^\ast(D \loops BG) \tensor H^\ast(D \loops BG)\;}
\[\xymatrix@!0@C=5em{
	*!R{\firstentry}
	\ar[r]^{ \times }
	&
	*!L{\;H^\ast(D \loops BG \smashprod^{H\F_\ell} D \loops BG)}
	\\
	*!R{\phantom{\firstentry}}
	\ar[r]^{ \isom }
	&
	*!L{\;H^\ast D (\loops BG \times \loops BG)}
	\\
	*!R{\phantom{\firstentry}}
	\ar[r]^{ (D\mathrm{concat})^\ast }
	&
	*!L{\;H^\ast(D \loops BG)}.
}\]
Here $DX$ 
for a space $X$
denotes the dual of $H\F_\ell\smashprod \suspension^\infty_+ X$ in 
$\ho(\Mod^{H\F_\ell})$,
and the middle isomorphism is induced by the equivalence
$D(\loops BG \times \loops BG) \homot D\loops BG \smashprod^{H\F_\ell} D\loops BG$.
The claim now follows using 
the natural isomorphism $H^\ast DX \isom H_{-\ast} X$
valid for spaces $X$ such that $H\F_\ell\smashprod \suspension^\infty_+ X$
is dualizable in $\ho(\Mod^{H\F_\ell})$.
\end{proof}

\begin{rem}[The isomorphism of Proposition~\ref{prop:pontryaginproduct}
and Poincaré duality]
\label{rk:pd}
At least morally, we may regard the isomorphism
$H_{-\ast}(\loops BG) \isom \bbH^\ast(\loops BG)$
of Proposition~\ref{prop:pontryaginproduct}
as arising from Poincaré duality.
Recall that for a closed oriented $d$--manifold $M$,
the Poincaré duality isomorphism $H_{-\ast}(M) \isom \bbH^\ast (M)$
factors as a composite of the isomorphism
$H_{-\ast}(M) \isom H^\ast(DM)$ arising from Spanier--Whitehead duality,
the isomorphism $H^\ast(DM) \isom H^\ast(M^{-\tau_M})$
induced by the Atiyah duality equivalence $DM \homot M^{-\tau_M}$,
and the Thom isomorphism $H^\ast(M^{-\tau_M}) \isom \bbH^\ast(M)$.
Unwinding the construction, we see that the isomorphism
$H_{-\ast}(\loops BG) \isom \bbH^\ast(\loops BG)$
of Proposition~\ref{prop:pontryaginproduct}
is obtained by combining similar ingredients:
\begin{enumerate}
\item 
	An isomorphism $H_{-\ast}(\loops BG) \isom H^\ast (D\loops BG)$
	given by Spanier--Whitehead duality.
\item 
	An isomorphism $H^\ast (D\loops BG) \isom H^\ast(r_! \omega_{r}^{H\F_\ell})$
	arising from an equivalence $D\loops BG \homot r_! \omega_{r}^{H\F_\ell}$.
	Here $r$ is the unique map $\loops BG \to \pt$.
	By Theorems~U.\ref*{U-thm:hypercartdata} and U.\ref*{U-thm:ellcptgrpscwdualizable},
	there is an equivalence
	$\omega_r \homot \suspension^{-d}_{\loops BG} S_{\loops BG,\ell}$.
	Thinking of $\loops BG$ as a kind of $d$--dimensional parallelizable manifold
	at the prime $\ell$, it is natural to think of $\omega_r$
	as the sphere bundle associated to the opposite of the tangent bundle of 
	$\loops BG$ and $r_! \omega_r$ as the associated Thom spectrum.	
\item 
	An isomorphism $H^\ast(r_! \omega_{r}^{H\F_\ell}) \isom \bbH^\ast(\loops BG)$
	arising from a trivialization 
	$\omega_r^{H\F_\ell} \homot \suspension^{-d}_{\loops BG} \underline{H\F_\ell}$,
	the commutation equivalence 
	$%
		r_! \suspension_{\loops BG}^{-d} \underline{H\F_\ell}
		\homot
		\suspension^{-d} r_! \underline{H\F_\ell},
	$ %
	the suspension isomorphism
	\[
		H^\ast (\suspension^{-d} r_! \underline{H\F_\ell})
		\isom
		\bbH^\ast (r_! \underline{H\F_\ell}),
	\]
	and the observation that 
	$\bbH^\ast r_! \underline{H\F_\ell} \isom \bbH^\ast (\loops BG)$.
	Thinking of $r_! \omega_r$ as a Thom spectrum, 
	the isomorphism $H^\ast(r_! \omega_{r}^{H\F_\ell}) \isom \bbH^\ast(\loops BG)$
	amounts to a Thom isomorphism.	
\end{enumerate}
Indeed, when $BG = BK\lcom$ for a semisimple compact Lie group $K$, 
making use of Remark~U.\ref*{U-rk:cwdualofgalt},
our isomorphism 
$H_{-\ast}(\loops BG) \isom \bbH^\ast(\loops BG)$
can be shown to coincide with the usual Poincaré duality isomorphism 
$H_{-\ast}(K) \isom \bbH^\ast(K)$
under the maps induced by the evident $\F_\ell$--homology equivalence
$K \to \loops BG$.
\end{rem}

\section{Pairings on Serre spectral sequences: Proof of
  Theorem~\ref{thm:mainresult}, part 2}
\label{sec:spectralsequences}

For a fibration $\pi\co X\to B$, let us write $\bbE(X)$ for the 
strongly convergent spectral sequence
\[
	\bbE^{s,t}_2(X) \Longrightarrow \bbH^{s+t}(X)
\]
obtained from the Serre spectral sequence $E(X)$
of $\pi$ by setting $\bbE^{s,t}_r(X) = E^{s,t+d}_r(X)$
and multiplying all differentials by $(-1)^d$. 
We call $\bbE(X)$ the \emph{shifted Serre spectral sequence of $\pi$}.
In this section, we will prove the following theorem 
showing that the string product on $\bbH^\ast(LBG)$
and the string module structure on $H^\ast(BG^{h\sigma})$ over $\bbH^\ast(LBG)$
lift to the level of Serre spectral sequences. 
This result makes precise the assertions
concerning Serre spectral sequences made in Theorem~\ref{thm:mainresult}.

\begin{thrm}
Suppose $BG$ is a semisimple $\ell$--compact group of dimension $d$.
\label{thm:ssforlbgbghsigma}
\begin{enumerate}[(i)]
\item \label{it:mainresult-ringss} 
    The shifted Serre spectral sequence of the evaluation
    fibration $LBG\to BG$, $\omega \mapsto \omega(1)$, 
    is a strongly convergent
    spectral sequence of algebras
    \begin{equation}
    \label{ss:alg}
    		\bbE^{s,t}_2(LBG)
    		\isom
    		H^s(BG) \tensor \bbH^t (G)
   		    \Longrightarrow  \bbH^{s+t}(LBG).
    \end{equation}
    Here $H^\ast(BG)$ is equipped with the cup product
	and $\bbH^\ast (G)$ and $\bbH^\ast(LBG)$ with the product $\stringprod$.
\item \label{it:mainresult-modss}
    The Serre spectral sequence
    \begin{equation}
    \label{ss:mod}
    	E^{s,t}_2(BG^{h\sigma})
    		\isom
    	H^s(BG) \tensor H^t (G) \Longrightarrow H^{s+t}(BG^{h\sigma})
    \end{equation}
    of the fibration $BG^{h\sigma} \to BG$, $\alpha \mapsto \alpha(1)$,
    with fibre homotopy equivalent to $G$,
    is a module spectral sequence over 
    the spectral sequence \eqref{ss:alg}
	and converges to $H^\ast(BG^{h\sigma})$ as a module over
    $\bbH^\ast(LBG)$. On the $E_2$--page, the module structure 
    is free of rank $1$ on 
	a generator of 
	$E_2^{0,d} \isom H^d(G) \isom \F_\ell$.
\end{enumerate}
\end{thrm}

\begin{rem}
\label{rk:noncommutativee2}
We remind the reader that the ring $(\bbH^\ast(G),\stringprod)$ 
in Theorem~\ref{thm:ssforlbgbghsigma}(\ref{it:mainresult-ringss})
is isomorphic to $H_{-\ast}(G)$ equipped with the Pontryagin product.
See Proposition~\ref{prop:pontryaginproduct}.
Since, as discussed in Remark~\ref{rk:noncomm}, the Pontryagin product
is in general non-commutative, so is the product on the $E_2$--page of 
spectral sequence~\eqref{ss:alg}.
Because the target $(\bbH^\ast(LBG),\stringprod)$ 
of the spectral sequence is by Theorem~\ref{thm:cmcomparison2}
commutative, the spectral sequence must in such cases have nontrivial differentials.
\end{rem}

The proof of Theorem~\ref{thm:ssforlbgbghsigma}
will be given at the end of the section.
In the proof, we will make use of the following analogue of 
Theorem~\ref{thm:pairings4}:

\begin{thrm}
\label{thm:sspairings2}
Suppose $BG$ is a semisimple $\ell$--compact group and let $B$ be a space.
Then there exists a category enriched in spectral 
sequences whose objects are maps $f\colon B \to BG$;
whose hom-object from $f\colon B \to BG$ to $g\colon B\to BG$
is given by the shifted Serre spectral sequence
\[
	\bbE_2^{s,t} (P(f,g)) \Longrightarrow \bbH^{s+t} P(f,g)
\]
of the fibration $\pi_{f,g} \colon P(f,g) \to B$;
and whose composition law
\begin{equation}
\label{eq:ssbullet}
	\stringprod 
	\co 
	\bbE^{s_1,t_1}_r (P(g,h)) \tensor \bbE^{s_2,t_2}_r (P(f,g))
	\longto
	\bbE^{s_1+s_2,t_1+t_2}_r(P(f,h))	
\end{equation}
converges to the pairing
\[
	\stringprod
	\co
	\bbH^\ast P(g,h) \tensor \bbH^\ast P(f,g) 
	\longto 
	\bbH^\ast P(f,h).
\]
on targets.
\end{thrm}

See Proposition~\ref{prop:e2pageprod}
for a description of the pairing \eqref{eq:ssbullet}
on the level of $E_2$--pages.
The category of spectral sequences in which the
enrichment takes place is given a precise definition 
in Definition~\ref{def:sscatdef} below.

In addition to Theorem~\ref{thm:sspairings2},
we have the following analogue of 
Theorem~\ref{thm:functoriality}:

\begin{thrm}
\label{thm:ssfunctors}
Suppose $BG$ is a semisimple $\ell$--compact group.
Then the category
of Theorem~\ref{thm:sspairings2}
enriched in spectral sequences
depends functorially on the space $B$:
given a map $\phi \colon A \to B$, the maps
\begin{equation}
\label{map:ssfuncomponent}
	\bbE(P(f,g)) \longto \bbE(P(f\phi,g\phi))
\end{equation}
induced by the maps $F_\phi$ of equation~\eqref{eq:fphidef}
for varying $f,g \colon B\to BG$
define a functor of categories 
enriched in spectral sequences
which on objects is given by the assignment $f\mapsto f\phi$.
\end{thrm}

Our strategy for proving Theorems~\ref{thm:sspairings2} 
and~\ref{thm:ssfunctors} is based on the second 
and third constructions of the pairing 
$\stringprod$ of Theorem~\ref{thm:pairings4},
carried out in 
Sections~\ref{subsec:secondconstruction} and \ref{subsec:thirdperspective},
respectively,
and we continue to rely 
on the paper \cite{umkehr-maps}
briefly summarized in 
Section~\ref{subsubsec:umkehrmapssummary}.
Starting with the categories
$\calP_B$ enriched in $(h\calF^\fop)^\op$
and the enriched functors $K_\phi$ between them,
all constructed in Section~\ref{subsec:secondconstruction},
we will apply 
Construction~\ref{constr:newenrichedcats}
with a suitable lax monoidal functor
\begin{equation}
\label{eq:sscomposite}
	M
	\co
	(h\calF^\fop)^\op
	\longto
	\SpectralSequences
\end{equation}
to obtain categories $M_\ast \calP_B$
enriched in the category $\SpectralSequences$ 
of spectral sequences
along with $\SpectralSequences$--enriched functors 
$M_\ast K_\phi$.
The proof is then completed by showing that the hom-objects
$M P(f,g)$ in 
$M_\ast \calP_B$ 
are isomorphic to the shifted Serre spectral sequences
$\bbE(P(f,g))$,
and observing that under these isomorphisms, the enriched functor 
$M_\ast K_\phi$ is given by the map 
\eqref{map:ssfuncomponent}. 
The functor $M$ will be a composite
\begin{equation}
\label{eq:ssfunctorM}
	M
	\colon
	(h\calF^\fop)^\op
	\xto{\ \hat{D}_\fw^\op\ }
	(\hpMod^{H\F_\ell}_\mathrm{bb})^\op
	\xto{\ \tilde{E}\ }
	\SpectralSequences
\end{equation}
analogous to the composite~\eqref{eq:thirdconstructioncomposite}
encountered in the third construction of the pairing $\stringprod$.
The category $h\calF^\fop$
was already constructed in 
Section~\ref{subsubsec:hcalffop}.
Our next goal is to construct the remaining
categories and functors
appearing in \eqref{eq:ssfunctorM}.

\begin{defn}(The category $\hpMod^{H\F_\ell}_\mathrm{bb}$)
\label{def:bbhpmodhfell}
Call a parametrized $H\F_\ell$--module $X$ over a space $B$
\emph{bounded from below} if for every $b\in B$ there exists a $t_0\in\Z$
such that $H^t(X_b) = 0$ for all $t < t_0$.
We define  $\hpMod^{H\F_\ell}_\mathrm{bb}$
to be the full subcategory of $\hpMod^{H\F_\ell}$
spanned by the objects which are bounded from below.
\end{defn}

We note that the $\extsmashprod^{H\F_\ell}$--product of 
bounded-from-below parametrized $H\F_\ell$--modules
is again such, so that $\hpMod^{H\F_\ell}_\mathrm{bb}$
is a symmetric monoidal subcategory of $\hpMod^{H\F_\ell}$.

\begin{lemma}
The functor
\[
	D_\fw \colon h\calF^\fop \longto \hpMod^{H\F_\ell}
\]
of Definition~\ref{def:dfw}
takes values in the subcategory $\hpMod^{H\F_\ell}_\mathrm{bb}$.
\end{lemma}

\begin{proof}
Suppose $\pi\colon E\to B$ is an object in $h\calF^\fop$.
By construction of $D_\fw$, we then have a dual pair 
$(\pi_! \underline{H\F_\ell}, D_B E)$
in $\ho(\Mod^{H\F_\ell}_{/B})$.
For any $b \in B$, this dual pair 
restricts to a dual pair
$((\pi_! \underline{H\F_\ell})_b, (D_B E)_b)$
in $\ho(\Mod^{H\F_\ell})$
on fibres over $b$. 
In particular, the fibre
$(D_B E)_b$ is dualizable in $\ho(\Mod^{H\F_\ell})$,
and hence has finite-dimensional mod $\ell$ cohomology.
\end{proof}

\begin{defn}[The functor $\hat{D}_\fw$]
We let $\hat{D}_\fw$ in \eqref{eq:ssfunctorM} be the functor
obtained by restricting the codomain of $D_\fw$ 
to $\hpMod^{H\F_\ell}_\mathrm{bb}$.
\end{defn}

\begin{defn}[The category $\SpectralSequences$]
\label{def:sscatdef}
A \emph{spectral sequence} $E$ consists of the following data:
a sequence  $E^{\ast,\ast}_1,E^{\ast,\ast}_2,\ldots$
of bigraded $\F_\ell$--vector spaces;
a differential $d_r$ of bidegree $(r,1-r)$
on each $E^{\ast,\ast}_r$; and an isomorphism
\[
	\phi_r\co H(E^{\ast,\ast}_r) \xto{\ \isom\ } E^{\ast,\ast}_{r+1}
\]
for each $r$. A \emph{morphism $f\co E \to D$
of spectral sequences} consists of a sequence
\[
	f_r \co  E^{\ast,\ast}_r \longto D^{\ast,\ast}_r,
	\quad 
	r= 1,2,\ldots
\]
of morphisms commuting with the differentials and having the
property that $f_{r+1}$ corresponds to $H(f_r)$ under the 
isomorphisms $\phi_r$. The \emph{tensor product}
of spectral sequences $E$ and $D$ is the spectral sequence
$E \tensor D$ with 
\[
	(E \tensor D)^{s,t}_r
	=
	\bigoplus_{\substack{s_1+s_2 = s\\t_1+t_2=t}}
	E^{s_1,t_1}_r \tensor D^{s_2,t_2}_r,
\]
differential 
\[
	d_r (x\tensor y) = d_r(x) \tensor y + (-1)^{\deg(x)} x \tensor d_r(y)
\]
where $\deg(x) = s+t$ for $x\in E^{s,t}_r$, and isomorphisms $\phi_r$
given by the Künneth theorem. We let $\SpectralSequences$
be the resulting symmetric monoidal category of spectral sequences.
The symmetry constraint in $\SpectralSequences$  is given by 
\[
	E^{s_1,t_1}_r \tensor D^{s_2,t_2}_r 
	\longto	
	D^{s_2,t_2}_r \tensor E^{s_1,t_1}_r,
	\qquad
	x\tensor y \longmapsto (-1)^{(s_1+t_1)(s_2+t_2)} y\tensor x,
\]
while the monoidal unit is given by the
spectral sequence which on each page is a single copy of $\F_\ell$
concentrated in degree $(0,0)$.
\end{defn}

\begin{defn}[The functor $\tilde{E}$]
Given a parametrized $H\F_\ell$--module $X$ over $B$, 
that is, an $\infty$--functor 
$X \colon \Pi_\infty(B)^\op \to \Mod^{H\F_\ell}$,
write $\calL^\ast(X)$ for the local coefficient system
\[
	\xymatrix@C+1.3em{
		\calL^\ast(X)
    	\colon 
    	\Pi_1(B)
    	\ar[r]^-{X^\op}
		&
		\ho(\Mod^{H\F_\ell})^\op
		\ar[r]^-{H^\ast}
		&
		\grMod^{\F_\ell}
	}
\] 
where we have continued to write $X$ for the 
functor $\Pi_1(B)^\op \to \ho(\Mod^{H\F_\ell})$
induced by $X$.
Compare with Definition~U.\ref*{U-def:lmxs}.
We let 
\[
	\tilde{E}
	\colon
	(\hpMod^{H\F_\ell}_\mathrm{bb})^\op
	\longto
	\SpectralSequences
\]
be the functor provided by the Serre spectral sequences
\[
	E_2^{s,t}(X)
	= 
	H^s(B; \calL^t(X))
	\Longrightarrow 
	H^{s+t}H_\bullet(B;X)
\]
of Theorem U.\ref*{U-thm:serresss}(\ref*{U-it:serresscohomology}).
\end{defn}

Our restriction to bounded-from-below
parametrized $H\F_\ell$--modules
in the definition of $\tilde{E}$ ensures
that the values of $\tilde{E}$ 
converge strongly to the indicated target.
An argument analogous to 
\cite[\S XIII.8]{Whitehead}
gives

\begin{thrm}
\label{thm:sspairing}
Given parametrized $H\F_\ell$--modules $X$ over $B$ and $Y$ over $C$,
there is an associative pairing 
\[
	E^{s,t}_r (X) \tensor E^{s't'}_r (Y) 
	\longto
	E^{s+s',t+t'}_r (X\extsmashprod^{H\F_\ell} Y)
\]
of the Serre spectral sequences of 
Theorem U.\ref*{U-thm:serresss}(\ref*{U-it:serresscohomology})
which on the $E_2$--page is given by the cross product
\[
	H^s(B;\, \calL^t(X))
	\tensor
	H^{s'}(C;\, \calL^{t'}(Y))
	\xto{\ \times\ }
	H^{s+s'}(B\times C ;\, \calL^{t+t'}(X\extsmashprod^{H\F_\ell} Y))
\]
where the pairing on local coefficient systems is 
given by the products
\[	
	H^t(X_b) \tensor H^{t'}(Y_c)
	\xto{\ \times\ }
	H^{t+t'}(X_b \smashprod^{H\F_\ell} Y_c)
\]
for $b\in B$ and $c\in C$.
If $X$ and $Y$ (and hence $X\extsmashprod^{H\F_\ell} Y$) are bounded 
from below in the sense of
Definition~\ref{def:bbhpmodhfell},
ensuring strong convergence,
the pairing on the $E_\infty$--page is the one induced by
the cross product
\begin{equation*}
	\begin{split}
	\pushQED{\qed} 
	H^\ast(H_\bullet(B; X)) \tensor H^\ast(H_\bullet(C; Y)) 
	&
	\xto{\qquad \mathclap{\times} \qquad}
	H^\ast(H_\bullet(B; X) \smashprod^{H\F_\ell} H_\bullet(C; Y))
	\\
	&
	\xto[\isom]{\qquad\mathclap{H^\ast(\times)^{-1}}\qquad}
	H^\ast(H_\bullet(B\times C; X\extsmashprod^{H\F_\ell} Y)). 
	\qedhere
    \popQED
 	\end{split}
\end{equation*}
\end{thrm}

The pairing of Theorem~\ref{thm:sspairing}
in particular makes the functor $\tilde{E}$
into a lax monoidal functor.
This completes the construction of the functor $M$ 
of equation~\eqref{eq:ssfunctorM}.
The proof of Theorems~\ref{thm:sspairings2} 
and \ref{thm:ssfunctors} 
is now completed by
\begin{prop} 
For all $f,g\co B \to BG$, there is an 
isomorphism 
\[
	\tilde{E} (D_B P(f,g)) \isom \bbE(P(f,g))
\]
of spectral sequences natural 
with respect to the homomorphisms induced by diagram
\eqref{sq:inducedmap}
and compatible with the isomorphism
\[
	\xi_{f,g} 
	\colon 
	H^\ast H_\bullet(B; D_B P(f,g)) 
	\isom
	\bbH^\ast P(f,g)
\]
of Corollary~\ref{cor:bulletconstr3}
on targets.
\end{prop}

\begin{proof}
Consider the equivalences of parametrized $H\F_\ell$--modules over $B$
\begin{equation}
\label{eq:dbpfgtosuspbpifghfell}
	D_B P(f,g) 
	\homot
	(\pi_{f,g})_! \omega_{\pi(f,g)}^{H\F_\ell}
	\homot
	(\pi_{f,g})_! \suspension^{-d}_{P(f,g)} \underline{H\F_\ell}
	\homot
	\suspension^{-d}_B (\pi_{f,g})_!  \underline{H\F_\ell}
\end{equation}
where the first equivalence is given by
Proposition~U.\ref*{U-prop:hypercartandfwduality},
the second by~\eqref{eq:omegapifghfltriv},
and the third follows from Remark~U.\ref*{U-rk:opcartpreservation2}.
The desired isomorphism of spectral sequences is given by the 
sequence of isomorphisms
\begin{multline*}
	E_r^{s,t} (D_B P(f,g))
 	\isom
	E_r^{s,t} (\suspension^{-d}_B (\pi_{f,g})_! \underline{H\F_\ell})
	\isom
	E_r^{s,t+d} ((\pi_{f,g})_! \underline{H\F_\ell})
	\isom
	E_r^{s,t+d} (P(f,g))
	= 
	\bbE_r^{s,t} (P(f,g)) 
\end{multline*}
where the first isomorphism is induced by the equivalences of 
equation~\eqref{eq:dbpfgtosuspbpifghfell},
the second is given by Proposition~U.\ref*{U-prop:sssusp},
and the third follows from Remark~U.\ref*{U-rk:ordinaryserress}.
(Here the first three spectral sequences are  
Serre spectral sequences for parametrized $H\F_\ell$--modules
provided by Theorem~U.\ref*{U-thm:serresss}(\ref*{U-it:serresscohomology}),
while the last two are unshifted and shifted versions of 
the usual mod $\ell$ Serre spectral sequence of the fibration
$\pi_{f,g}\colon P(f,g) \to B$.)
In the corresponding sequence of isomorphisms
\begin{multline*}
	\qquad
	H^\ast H_\bullet (B;\, D_B P(f,g))
	\isom
	H^\ast H_\bullet (B;\, \suspension^{-d}_B (\pi_{f,g})_! \underline{H\F_\ell})
	\\
	\isom
	H^{\ast+d} H_\bullet (B;\, (\pi_{f,g})_! \underline{H\F_\ell})
	\isom 
	H^{\ast+d} P(f,g)
	=
	\bbH^\ast P(f,g)
	\qquad
\end{multline*}
between the targets of the spectral sequences,
the first isomorphism is induced by the equivalences 
\eqref{eq:dbpfgtosuspbpifghfell}, 
the second by the isomorphism $\bar{\sigma}^u$ of 
U.\eqref*{U-eq:barsigmaucohomology},
and the third follows by combining
Corollary~U.\ref*{U-cor:hbulletandshriek}
and Example~U.\ref*{U-ex:hbulletfortrivialcoeffs2}.
With the aid of Remark~\ref{rk:xifgdesc},
this composite is easily recognized as
the map $\xi_{f,g}$.
\end{proof}

Tracing through the constructions, we obtain the following
description of the pairing $\stringprod$
between spectral sequences on the $E_2$ page.

\begin{prop}
\label{prop:e2pageprod}
Let $B$ be a path connected space, 
let $f,g,h\co B \to BG$ be maps, and let $b\in B$
be a basepoint.
On the $E_2$ page, the pairing
\[
	\stringprod 
	\co 
	\bbE^{s_1,t_1}_r (P(g,h)) \tensor \bbE^{s_2,t_2}_r (P(f,g))
	\longto
	\bbE^{s_1+s_2,t_1+t_2}_r(P(f,h))	
\]
of spectral sequences is given by the map
\[
	H^\ast(B) \tensor \bbH^\ast(P(g,h)_b) 
	\tensor
	H^\ast(B) \tensor \bbH^\ast(P(f,g)_b) 
	\longto
	H^\ast(B) \tensor \bbH^\ast(P(f,h)_b)
\]
induced by the cup product on $B$ and the pairing
\[
	\stringprod
	\co
	\bbH^\ast(P(g,h)_b) 
	\tensor
	\bbH^\ast(P(f,g)_b) 
	\longto
	\bbH^\ast(P(f,h)_b)
\]
arising from the identifications
$P(g,h)_b = P(gi,hi)$,
$P(f,g)_b = P(fi,gi)$,  and
$P(f,h)_b = P(fi,hi)$
for $i$  the inclusion of $b$ into $B$.  \qed
\end{prop}

\begin{proof}[Proof of Theorem~\ref{thm:ssforlbgbghsigma}]
The spectral sequences of parts~(\ref{it:mainresult-ringss}) 
and (\ref{it:mainresult-modss})
were constructed in Theorem~\ref{thm:sspairings2};
the module structure
on the Serre spectral without a degree shift
in part~(\ref{it:mainresult-modss})
is obtained from the module structure 
on the degree-shifted Serre spectral sequence simply 
by regrading as in Definition~\ref{def:stringprodandmod}.
In part~(\ref{it:mainresult-modss}), the assertion about
the $E_2$--page follows from Proposition~\ref{prop:e2pageprod}
and Proposition~\ref{prop:homotopyinvariance} (which allows one to 
compare the $\bbH^\ast(G)$--module structure 
on the cohomology of the fibre of 
$BG^{h\sigma} \to BG$ to that on $\bbH^\ast(G)$ itself).
\end{proof}

\section{%
Proof of Theorem~\ref{thm:strtoptezukacrit} and results in the polynomial case}
\label{sec:tezuka}

In this section, we will prove Theorem~\ref{thm:strtoptezukacrit}, in addition to which 
we will cast an eye on the case where $H^\ast(BG)$ is polynomial and prove
Proposition~\ref{prop:polycollapse} and Theorem~\ref{thm:hlbgcompintro}.
A common theme in the section is the use of Serre spectral sequences:
at key steps in the arguments, we will make essential use of the 
results on the Serre spectral sequence established in the previous section.
The proof of Theorem~\ref{thm:strtoptezukacrit} will occupy
Sections~\ref{subsec:for1crit} and~\ref{subsec:isosonassocgradeds}
while the discussion of the polynomial case will be given in
Section~\ref{subsec:polycase}.

\subsection{The criterion for being free of rank \texorpdfstring{$1$}{1}: Proof of Theorem~\ref{thm:strtoptezukacrit} 
(\ref{item:fundclass}) \texorpdfstring{$\Leftrightarrow$}{<=>} (\ref{item:freerk1})}
\label{subsec:for1crit}

In this subsection, we will prove the following elaboration of
the equivalence of conditions (\ref{item:fundclass}) and (\ref{item:freerk1})
of Theorem~\ref{thm:strtoptezukacrit}.

\begin{thrm}
\label{thm:conj-red3}
Let $BG$ be a semisimple $\ell$--compact group of dimension $d$,
let $B$ be a path connected space, 
let $f,g\co B \to BG$ be maps,
let $F\homot \loops BG$ be a fibre of the fibration 
$\pi_{f,g}\colon P(f,g) \to B$,
and  let $i\co F\incl P(f,g)$ be the inclusion.
Then the following are equivalent conditions on an element
$x \in \bbH^0 P(f,g)$:
\newcounter{savedvalue}
\begin{enumerate}
\item \label{it:freerank1}
	$\bbH^\ast P(f,g)$ is free of rank $1$ with basis $\{x\}$
	as a graded left module over $\bbH^\ast P(g,g)$. 
\item \label{it:freerank1right}
	$\bbH^\ast P(f,g)$ is free of rank $1$ with basis $\{x\}$
	as a graded right module over $\bbH^\ast P(f,f)$. 
\item \label{it:cyclic}
	$\bbH^\ast P(f,g)$ is generated by $x$ 
	as a graded left module over $\bbH^\ast P(g,g)$.
\item \label{it:cyclicright}
	$\bbH^\ast P(f,g)$ is generated by $x$ 
	as a  graded right module over $\bbH^\ast P(f,f)$.
\item \label{it:nonzerores}
	$i^\ast(x) \neq 0 \in \bbH^0 F$.
\item  \label{it:ssiso2}
	The map
	\[
		\bbE(P(g,g)) \longto \bbE(P(f,g)),
		\quad 
		z \longmapsto z \stringprod (1\tensor i^\ast(x))
	\]
	is an isomorphism from the Serre spectral
	sequence of the fibration  $\pi_{g,g}\colon P(g,g)\to B$ 
	to that of the fibration $\pi_{f,g} \colon P(f,g)\to B$.
\item \label{it:ssiso2right}
	The map
	\[
		\bbE(P(f,f)) \longto \bbE(P(f,g)),
		\quad 
		z \longmapsto  (1\tensor i^\ast(x))\stringprod z
	\]
	is an isomorphism from the Serre spectral
	sequence of the fibration  $\pi_{f,f} \colon P(f,f)\to B$ 
	to that of the fibration $\pi_{f,g} \colon P(f,g)\to B$.
	\setcounter{savedvalue}{\value{enumi}}
\end{enumerate}
Moreover, the following conditions are equivalent:
\begin{enumerate}
\setcounter{enumi}{\value{savedvalue}}
\item \label{it:xexists}
	There exists an element $x \in \bbH^0 P(f,g)$
	satisfying conditions (\ref{it:freerank1})--(\ref{it:ssiso2right}).
\item \label{it:fclass}
	The map $i_\ast\co H_d F \to H_d P(f,g)$ is nontrivial.
\item \label{it:ssiso}
	The Serre spectral sequences of $\pi_{f,g} \colon P(f,g)\to B$ 
	and $\pi_{g,g}\colon P(g,g)\to B$
	are isomorphic.
\item \label{it:ssisoright}
	The Serre spectral sequences of $\pi_{f,g}\colon P(f,g)\to B$ 
	and $\pi_{f,f}\colon P(f,f)\to B$
	are isomorphic.
\item \label{it:permanentcycle2}
	The generator of
	$E_2^{0,d}(P(f,g))= H^0 B \tensor H^d F \isom \F_\ell$
	is a permanent cycle in the Serre spectral sequence of 
	$\pi_{f,g}\colon P(f,g)\to B$.
\end{enumerate}
\end{thrm}
Here the product $\stringprod$ on spectral sequences
in conditions (\ref{it:ssiso2})
and (\ref{it:ssiso2right}) refers to the pairing 
\eqref{eq:ssbullet} of Theorem~\ref{thm:sspairings2}.
\begin{proof}[Proof of Theorem~\ref{thm:conj-red3}]
The implications 
(\ref{it:freerank1})\,$\Rightarrow$\,(\ref{it:cyclic})
and 
(\ref{it:freerank1right})\,$\Rightarrow$\,(\ref{it:cyclicright})
are obvious. To show
(\ref{it:cyclic})\,$\Rightarrow$\,(\ref{it:nonzerores}),
let $j\co \loops BG \incl P(g,g)$ be the inclusion of a fibre.
By assumption, there exists some element $y\in \bbH^{-d} P(g,g)$ 
such that $y\stringprod x = 1\in \bbH^{-d}P(f,g)$.
We now have
\[
	j^\ast(y) \stringprod i^\ast(x) 
	= 
	i^\ast(y\stringprod x) 
	= 
	i^\ast(1) 
	= 
	1 \in \bbH^{-d}(F),
\]
showing that $i^\ast(x) \neq 0\in \bbH^0 F$. 
Here the first equality follows from 
Theorem~\ref{thm:functoriality}.
The implication 
(\ref{it:cyclicright})\,$\Rightarrow$\,(\ref{it:nonzerores})
follows similarly.

Let us now show that 
(\ref{it:nonzerores})\,$\Rightarrow$\,(\ref{it:ssiso2}).
By naturality of the Serre spectral sequence, the element 
$1\tensor i^\ast(x) \in \bbE_2^{0,0}(P(f,g)) = H^0(B)\tensor \bbH^0(F)$
is the image of $x$ under the composite
\[
	\bbH^0 P(f,g) 
	\longto
	\bbE_\infty^{0,0}(P(f,g))
	\longto
	\bbE_2^{0,0}(P(f,g))
\]
of the quotient and inclusion maps.
Thus $1\tensor i^\ast(x)$ is a
permanent cycle, and multiplication by it 
using the pairing $\stringprod$ of Theorem~\ref{thm:sspairings2}
does define a
morphism of spectral sequences
\[
	m_{1\tensor i^\ast(x)} \co \bbE(P(g,g)) \longto \bbE(P(f,g)),
	\quad 
	z\mapsto z\stringprod (1\tensor i^\ast(x))
\]
By Proposition~\ref{prop:e2pageprod}, on the $E_2$ page,
this map is given by the map
\[
	\id \tensor m_{i^\ast(x)}
	\co
	H^\ast(B) \tensor \bbH^\ast(\loops BG)
	\longto
	H^\ast(B) \tensor \bbH^\ast(F)
\]
where $m_{i^\ast(x)}$ is the multiplication map 
$z \mapsto z \stringprod i^\ast(x)$.
By Proposition~\ref{prop:homotopyinvariance}, 
there is an isomorphism $\bbH^\ast(F) \isom \bbH^\ast(\loops BG)$
under which the map $m_{i^\ast(x)}$ corresponds to the map
\[
	m_{y} 
	\co
	\bbH^\ast(\loops BG)\longto \bbH^\ast(\loops BG),
	\quad 
	z \mapsto z \stringprod y
\]
for some $y \in \bbH^0(\loops BG)$. Since $i^\ast(x)$ is nonzero,
so is $y$. Thus $y$ is a nonzero multiple of the unit 
element for the product on $\bbH^\ast(\loops BG)$.
It follows that the map $m_{y}$ and hence the map
$m_{i^\ast(x)}$ are isomorphisms. Thus the 
map $m_{1\tensor i^\ast(x)}$ is an isomorphism on the $E_2$ pages,
and hence on all further pages as well, giving an isomorphism
of spectral sequences. 
Again, the implication 
(\ref{it:nonzerores})\,$\Rightarrow$\,(\ref{it:ssiso2right})
follows similarly.

We now prove the implication
(\ref{it:ssiso2})\,$\Rightarrow$\,(\ref{it:freerank1}).
Since $x \in \bbH^0 P(f,g)$ is a lift of the element
$1\tensor i^\ast(x) \in \bbE_\infty^{0,0}(P(f,g))$,
the multiplication map
\[
	m_{x}
	\co 
	\bbH^\ast P(g,g) \longto \bbH^\ast P(f,g),
	\quad
	z \mapsto z\stringprod x
\]
induces on the associated graded modules
corresponding to the Serre spectral sequences
of $P(g,g)$ and $P(f,g)$ the isomorphism
\[
	m_{1\tensor i^\ast(x)}
	\co 
	\bbE^{\ast,\ast}_\infty (P(g,g))
	\xto{\ \isom\ } 
	\bbE^{\ast,\ast}_\infty (P(f,g)).
\]
Therefore the map $m_{x}$ itself
must be an isomorphism. Thus $\bbH^\ast P(f,g)$ is free of 
rank $1$ over $\bbH^\ast P(g,g)$ with basis $\{x\}$.
The implication 
(\ref{it:ssiso2right})\,$\Rightarrow$\,(\ref{it:freerank1right})
follows similarly.

In view of condition (\ref{it:nonzerores}), condition (\ref{it:xexists})
is equivalent to the map $i^\ast \co H^d P(f,g) \to H^d F$
being nontrivial, which in turn is equivalent to condition
(\ref{it:fclass}). Thus 
(\ref{it:xexists})\,$\Leftrightarrow$\,(\ref{it:fclass}).
The implications 
(\ref{it:xexists})\,$\Rightarrow$\,(\ref{it:ssiso})
and
(\ref{it:xexists})\,$\Rightarrow$\,(\ref{it:ssisoright})
follow from conditions (\ref{it:ssiso2}) and (\ref{it:ssiso2right}).
To show 
(\ref{it:ssiso})\,$\Rightarrow$\,(\ref{it:permanentcycle2}),
it suffices to show that the generator of 
$E_2^{0,d}(P(g,g)) = H^0(B) \tensor H^d (\loops BG) \isom \F_\ell$
is a permanent cycle.
Let $j\co \loops BG \incl P(g,g)$ be the inclusion of a fibre.
As before, the element
$1\tensor j^\ast (\bbOne) \in \bbE_2^{0,0}(P(g,g))
= H^0(B) \tensor \bbH^0 (\loops BG)$
is a permanent cycle.
Since 
$j^\ast(\bbOne) = \bbOne \neq 0 \in \bbH^0 (\loops BG)$
by Theorem~\ref{thm:functoriality},
it is nonzero, and hence generates 
$\bbE_2^{0,0}(P(g,g)) \isom \F_\ell$.
Thus the claim follows.
Again, the implication
(\ref{it:ssisoright})\,$\Rightarrow$\,(\ref{it:permanentcycle2})
follows similarly.
Finally, to show that 
(\ref{it:permanentcycle2})\,$\Rightarrow$\,(\ref{it:xexists}),
it suffices to observe that by naturality of the Serre spectral sequence, 
a lift of the nontrivial permanent cycle from $E_\infty^{0,d}(P(f,g))$ to 
an element of $H^d P(f,g)$ satisfies condition~(\ref{it:nonzerores}).
\end{proof}

The first part of Theorem~\ref{thm:strtoptezukacrit}
is immediate from Theorem~\ref{thm:conj-red3}:

\begin{proof}[Proof that 
(\ref{item:fundclass}) $\Leftrightarrow$ (\ref{item:freerk1})
in Theorem~\ref{thm:strtoptezukacrit}]
The claim is a special case of the equivalence
(\ref{it:xexists}) $\Leftrightarrow$ (\ref{it:fclass})
in Theorem~\ref{thm:conj-red3}.
\end{proof}

Theorem~\ref{thm:strtoptezukacrit}
also gives the following criterion for 
recognizing elements 
which generate $H^\ast(BG^{h\sigma})$
as a free of rank 1 module over $\bbH^\ast(LBG)$.
\begin{cor}
Let $\sigma \colon BG \to BG$ be a map and
let $i\colon G \to BG^{h\sigma}$
be the inclusion of the fibre of the evaluation
fibration $BG^{h\sigma}\to BG$, $\alpha \mapsto \alpha(1)$.
Then an element $x\in H^\ast(BG^{h\sigma})$
generates $H^\ast(BG^{h\sigma})$ as a free of rank $1$
module over $\bbH^\ast(LBG)$
if and only if $\deg(x) = d$ and 
$i^\ast(x) \neq 0 \in H^d(G)$.
\end{cor}

\begin{proof}
For degree reasons, any generator of 
$H^\ast(BG^{h\sigma})$ 
as a free of rank $1$
module over $\bbH^\ast(LBG)$
has to be of degree $d$, so the claim follows from the 
equivalence
(\ref{it:freerank1}) $\Leftrightarrow$ (\ref{it:nonzerores})
in Theorem~\ref{thm:conj-red3}.
\end{proof}

Finally, we record the following corollary of 
Theorem~\ref{thm:conj-red3}.

\begin{cor}
\label{cor:tezukaimpliesid}

Suppose $\sigma \colon BG \to BG$ is such that
$H^\ast(BG^{h\sigma})$ is free of rank $1$ over $\bbH^\ast LBG$.
Then the map $\ev_1^\ast\colon H^\ast BG \to H^\ast BG^{h\sigma}$
induced by the evaluation fibration $\ev_1\colon BG^{h\sigma}\to BG$,
$\alpha \mapsto \alpha(1)$
is injective, and $\sigma$ induces the identity map on $H^\ast(BG)$.
\end{cor}

\begin{proof}
Write $\ev_t\colon BG^{h\sigma} \to BG$
for the evaluation map $\alpha \mapsto \alpha(t)$.
The existence of 
a section for the evaluation fibration
$LGB\to BG$ implies that the Serre spectral 
sequence for $LGB\to BG$ has no differentials hitting the
bottom row. By  Theorem~\ref{thm:conj-red3}, 
the same is therefore true for the Serre spectral sequence
of the evaluation fibration $\ev_1\co BG^{h\sigma} \to BG$.
It follows that the map
\[
	\ev_1^\ast \co H^\ast BG \longto H^\ast BG^{h\sigma}
\]
is injective. 
Since the maps $\ev_0,\ev_1 \co BG^{h\sigma} \to BG$
are homotopic, we have $\ev_0^\ast = \ev_1^\ast$
on $H^\ast$
and hence
\[
	\ev_1^\ast \sigma^\ast 
	= 
	(\sigma\circ \ev_1)^\ast  
	= 
	\ev_0^\ast
	=
	\ev_1^\ast.
\]
Thus $\sigma^\ast$ is the identity, as claimed.
\end{proof}

\subsection{Isomorphisms on associated graded algebras: Proof of
  Theorem~\ref{thm:strtoptezukacrit} (\ref{item:freerk1}) 
  \texorpdfstring{$\Leftrightarrow$}{<=>} (\ref{item:assgraded})}
\label{subsec:isosonassocgradeds}
In this subsection, we will prove Theorem~\ref{thm:sscupstr} below
and use it to deduce the equivalence of conditions 
(\ref{item:freerk1}) and (\ref{item:assgraded})
in Theorem~\ref{thm:strtoptezukacrit}.

\begin{notation}[Degree shift maps $s^d$ and $s^{-d}$ for spectral sequences]
Paralleling \eqref{eq:degreeshiftmaps},
we define degree shift maps
\[
	s^d \colon \bbE(X) \longto E(X)
	\qquad\text{and}\qquad
	s^{-d} \colon E(X) \longto \bbE(X)
\]
between the shifted and unshifted Serre spectral sequences of a fibration
$X\to B$ as follows: the map $s^d$ sends each element 
$x \in \bbE_r^{s,t}(X)$ to itself considered as an element of $E_r^{s,t+d}(X)$,
while $s^{-d}$ is the inverse of $s^d$.
\end{notation}

\begin{thrm}
\label{thm:sscupstr}
Let $BG$ be a semisimple $\ell$--compact group of dimension $d$, 
let $B$ be a path connected space, 
let $f,g\co B \to BG$ be maps,
let $F\homot \loops BG$ be a fibre of the fibration 
$\pi_{f,g}\co P(f,g) \to B$,
and  let $i\co F\incl P(f,g)$ be the inclusion.
Suppose $x\in \bbH^0 P(f,g)$ is a generator of 
$\bbH^\ast P(f,g)$ as a free rank $1$ module over
$\bbH^\ast P(g,g)$ such that  
the composite
\begin{equation}
\label{map:miastx}
\xymatrix@C+1em{
	H^\ast(\loops BG)
	\ar[r]^-{s^{-d}}
	&
	\bbH^\ast(\loops BG)
	\ar[r]^-{\stringprod i^\ast(x)}
	&
	\bbH^\ast(F)
	\ar[r]^-{s^d}
	&
	H^\ast(F)
}
\end{equation}
sends $1\in H^0 \loops BG$ to $1 \in H^0 F$.
Then the composite 
\begin{equation}
\label{map:m1x}
\xymatrix{
	E(P(g,g))
	\ar[r]^{s^{-d}}
	&
	\bbE(P(g,g))
	\ar[rr]^-{\stringprod (1\tensor i^\ast(x))}_\isom
	&&
	\bbE(P(f,g))
	\ar[r]^{s^d}
	&
	E(P(f,g))
}
\end{equation}
of the isomorphism of Theorem~\ref{thm:conj-red3}(\ref{it:ssiso2})
with degree shift maps 
is an isomorphism of spectral sequences of algebras,
where the spectral sequences are equipped with the 
usual algebra structures induced by cup product.
In particular, the composite
\begin{equation}
\label{eq:bulletxconj}
\xymatrix{
	H^\ast P(g,g) 
	\ar[r]^{s^{-d}}
	&	
	\bbH^\ast P(g,g) 
	\ar[r]^{\stringprod x}
	&
	\bbH^\ast P(f,g)
	\ar[r]^{s^d}
	&	
	H^\ast P(f,g)
}
\end{equation}
induces an $H^\ast(B)$--algebra isomorphism
\begin{equation}
\label{eq:assocgradediso}
	\gr H^\ast P(g,g) \xto{\ \isom\ } \gr H^\ast P(f,g) 
\end{equation}
on the associated graded algebras of 
$H^\ast P(g,g)$ and $H^\ast P(f,g)$ 
corresponding to the Serre spectral sequences.
\end{thrm}

\begin{rem}
\label{rk:replacementtrick}
Notice that, given \emph{any} generator 
$x$ of $\bbH^\ast P(f,g)$ as a free rank $1$ module over
$\bbH^\ast P(g,g)$,
the condition that \eqref{map:miastx}
sends $1\in H^0 \loops BG$ to $1 \in H^0 F$
can be arranged to hold simply by replacing $x$ by
a suitable nonzero scalar multiple.
\end{rem}

\begin{proof}[Proof of Theorem~\ref{thm:sscupstr}]
To show that the map \eqref{map:m1x} respects the algebra
structures, it is enough to show that it does so on the 
$E_2$ pages. By Proposition~\ref{prop:e2pageprod},
on $E_2$ pages the map \eqref{map:m1x} is given by
the composite
\[\xymatrix@C+0.1em{
	H^\ast(B) \tensor H^\ast(\loops BG)
	\ar[r]^-{1\tensor s^{-d}}
	&
	H^\ast(B) \tensor \bbH^\ast(\loops BG)
	\ar[rr]^-{1\tensor (\stringprod i^\ast(x))}
	&&
	H^\ast(B) \tensor \bbH^\ast(F)
	\ar[r]^-{1\tensor s^d}
	&
	H^\ast(B) \tensor H^\ast(F),
}\]
so it is enough to show that the composite \eqref{map:miastx}
is a ring homomorphism.
Picking a path connecting $f(b_0)$ and $g(b_0)$
where $b_0$ is the point in $B$ over which $F$ 
is a fibre of $\pi_{f,g}\colon P(f,g)\to B$ and $\loops BG$ is a fibre of 
$\pi_{g,g}\colon P(g,g)\to B$,
from Proposition~\ref{prop:homotopyinvariance} we obtain 
an isomorphism
\[
	\Xi \co \bbH^\ast(F) \xto{\ \isom\ } \bbH^\ast(\loops BG)
\]
of $\bbH^\ast(\loops BG)$--modules. As the isomorphism $\Xi$
is induced by a zigzag of homotopy equivalences of spaces,
the composite
\begin{equation}
\label{eq:Xiconjugate}
\xymatrix{
	H^\ast(F)
	\ar[r]^{s^{-d}}
	&
	\bbH^\ast(F) 
	\ar[r]_-\isom^-{\Xi}
	&
	\bbH^\ast(\loops BG)
	\ar[r]^{s^d}
	&
	H^\ast(\loops BG)
} 
\end{equation}
is an algebra isomorphism with respect to 
cup products. Thus it is enough to show that the 
composite of \eqref{map:miastx} and \eqref{eq:Xiconjugate}
is a ring endomorphism of $H^\ast(\loops BG)$,
which we will do by showing that this composite
is in fact the identity map of $H^\ast(\loops BG)$.
To this end, it is enough to show that 
the composite 
\begin{equation}
\label{map:xicomp}
    \xymatrix@C+0.5em{
    	\bbH^\ast (\loops BG)
    	\ar[r]^-{\stringprod i^\ast(x)}
    	&
    	\bbH^\ast (F)
    	\ar[r]^-{\Xi}
    	&
    	\bbH^\ast (\loops BG)
    },
    \quad
    \beta \longmapsto \beta \stringprod \Xi(i^\ast(x))
\end{equation}
is the identity map of $\bbH^\ast (\loops BG)$.
We have 
\[
	s^{-d}(1) \stringprod \Xi(i^\ast(x)) 
	=
	\Xi(s^{-d}(1) \stringprod i^\ast(x))
	=
	\Xi(s^{-d}(1))
	=
	s^{-d}(1)
\]
where the first equality follows by the
$\bbH^\ast(\loops BG)$--linearity of $\Xi$,
the second equality holds by the assumption on $x$,
and the final equality holds since the composite
\eqref{eq:Xiconjugate} is 
an algebra homomorphism. Since by $\F_\ell$--linearity
the group $\bbH^0(\loops BG)\isom \F_\ell$
contains at most one element $y$ with the property that 
$s^{-d}(1)\stringprod y = s^{-d}(1) \in \bbH^{-d}(\loops BG)$, 
and both $\bbOne$ and $\Xi(i^\ast(x))$ have this property, 
we must have $\Xi(i^\ast(x)) = \bbOne$.
Thus the map \eqref{map:xicomp}
is the identity map as claimed.

The last part of the theorem follows by noting that
on $E_\infty$--pages, the isomorphism \eqref{map:m1x}
agrees with the map induced by \eqref{eq:bulletxconj}.
To see that \eqref{eq:assocgradediso}
is a map of $H^\ast(B)$--algebras,
it suffices to show that it is $H^\ast(B)$--linear,
which in turn follows by observing that 
the filtrations on $H^\ast P(g,g)$
and $H^\ast P(f,g)$ from the Serre specral sequences
are ones of $H^\ast(B)$--modules and the maps of 
equation \eqref{eq:bulletxconj} are $H^\ast(B)$--linear.
\end{proof}

\begin{proof}[Proof that 
(\ref{item:freerk1}) $\Leftrightarrow$ (\ref{item:assgraded})
in Theorem~\ref{thm:strtoptezukacrit}]
The implication 
(\ref{item:assgraded}) $\Rightarrow$ (\ref{item:freerk1})
follows by a repeated application of the short five lemma.
Moreover, in view of Remark~\ref{rk:replacementtrick},
the implication (\ref{item:freerk1}) $\Rightarrow$ (\ref{item:assgraded})
follows from Theorem~\ref{thm:sscupstr}.
\end{proof}

\subsection{Results in the polynomial case}
\label{subsec:polycase}

In this subsection, we will use the Eilenberg--Moore spectral 
sequence to provide a proof of the well-known Proposition~\ref{prop:polycollapse}.
Moreover, we will make use of the properties of the Serre spectral sequence
established in Section~\ref{sec:spectralsequences} 
to provide a proof of Theorem~\ref{thm:hlbgcompintro}.
All told, by the end of the subsection, we will have established the 
following picture in the case where $H^\ast(BG)$ is a polynomial ring
$H^\ast(BG) = \F_\ell[x_1,\ldots,x_n]$ and $\sigma \colon BG \to BG$
is a map inducing the identity map on cohomology.
\begin{enumerate}
\item $H^\ast(BG^{h\sigma})$ is free of rank $1$ over the ring $(\bbH^\ast(LBG),\stringprod)$. 
	This follows from
	the equivalence of conditions (\ref{item:fundclass}) and (\ref{item:freerk1})
	of Theorem~\ref{thm:strtoptezukacrit}
	together with Proposition~\ref{prop:polycollapse}.
\item When $\ell$ is odd, the ring 
	$(\bbH^\ast(LBG),\stringprod)$ has a simple description: we have
	\[
		\bbH^*(LBG) 
		\isom  
		\F_\ell[x_1,\ldots,x_n] \tensor \Lambda(y_1,\ldots,y_n)
	\]
	where $\deg(y_i) = -(\deg(x_i)-1)$. This follows from Theorem~\ref{thm:hlbgcompintro}.
\item \label{it:stringiso}
	For a suitably chosen element $x \in H^d(BG^{h\sigma})$, the map
	$H^\ast(LBG) \to H^\ast(BG^{h\sigma})$ induced by $\stringprod$--product with $x$
	is a $H^\ast(BG)$--module isomorphism which is close to being a 
	ring isomorphism in the sense that it induces
	an isomorphism
	\[
		\gr H^\ast(LBG) \xto{\ \isom\ } \gr H^\ast(BG^{h\sigma})
	\]
	on certain associated graded algebras. This follows from 
	the equivalence of conditions (\ref{item:fundclass}) and (\ref{item:assgraded})
	of Theorem~\ref{thm:strtoptezukacrit}
	together with Proposition~\ref{prop:polycollapse}
	and the $H^\ast(BG)$--bilinearity of the string pairing $\stringprod$.	
\end{enumerate}
The result in part (\ref{it:stringiso}) can typically be improved upon. 
We will return to the polynomial case later in Section~\ref{sec:algiso}
where our focus is in obtaining more highly structured isomorphisms
$H^\ast(LBG) \isom H^\ast(BG^{h\sigma})$ given by $\stringprod$--product
with a suitably chosen element $x \in H^\ast(BG^{h\sigma})$.
See Theorems~\ref{thm:realization}, \ref{thm:polytezukaelaboration} and~\ref{thm:polysimplyconn2cptgrps}.

\begin{proof}[Proof of Proposition~\ref{prop:polycollapse}]
Recall that for a pullback diagram of spaces
\[\xymatrix{
	X \times_B Y
	\ar[r]
	\ar[d]
	\pb
	&
	Y
	\ar[d]^{g}
	\\
	X
	\ar[r]^{f}
	&
	B
}\]
with at least one of $f$ and $g$ a fibration and 
the space $B$ simply connected, the Eilenberg--Moore
spectral sequence is a strongly convergent second-quadrant spectral 
sequence of algebras
\[
	E_2^{s,t} = \Tor^{H^\ast(B)}_{-s,t}(H^\ast(X), H^\ast(Y)) 
	\Longrightarrow
	H^{s+t}(X\times_B Y)
\]
where $t$ is the internal degree and $-s$ the 
homological degree. The entries on the $E_\infty$-page of the spectral 
sequence are filtration quotients
\[
	E_\infty^{s,t} = F^{s} H^{s+t}(X\times_B Y) / F^{s+1} H^{s+t}(X\times_B Y)
\]
for a certain descending filtration
\[
	H^\ast(X\times_B Y) 
	\supset \cdots \supset 
	F^{-2} \supset F^{-1} \supset F^0 \supset F^1 = 0.
\]
of $H^\ast(X\times_B Y)$.
In particular, we may interpret the line
$E_\infty^{0,\ast}$ as a subring of $H^\ast (X\times_B Y)$.
By naturality of the spectral sequence, it is easy to see
that this subring is precisely the image of the 
map
\[
	H^\ast(X \times Y) \longto H^\ast(X\times_B Y)	
\]
induced by the inclusion of $X\times_B Y$ into $X\times Y$.

Suppose now $H^\ast(BG) = \F_\ell[x_1,\ldots,x_n]$,
and let $\sigma \colon BG \to BG$ be a map inducing 
the identity on cohomology.
Consider the Eilenberg--Moore spectral sequence for 
the pullback square
\[\xymatrix{
	BG^{h\sigma} 
	\ar[r]
	\ar[d]_{\ev_1}
	\pb
	&
	BG^I
	\ar[d]^{(\ev_0,\ev_1)}
	\\
	BG
	\ar[r]^-{(\sigma,1)}
	&
	BG\times BG
}\]
By the assumption on $\sigma$,
the $E_2$-page of the spectral sequence amounts to 
\[
	\Tor^{\F_\ell[x_1,\ldots,x_n,x'_1,\ldots,x'_n]}(
		\F_\ell[x_1,\ldots,x_n],
		\F_\ell[x_1,\ldots,x_n]
	)
\]
where 
$\F_\ell[x_1,\ldots,x_n,x'_1,\ldots,x'_n]$ 
acts on the two copies of 
$\F_\ell[x_1,\ldots,x_n]$
via the map $x_i \mapsto x_i$, $x'_i \mapsto x_i$.
Interpreting
\[
	\F_\ell[x_1,\ldots,x_n,x'_1,\ldots,x'_n] 
	=
	\F_\ell[x_1,\ldots,x_n,y_1,\ldots,y_n]
\]
with $y_i = x'_i-x^{}_i$, it is now 
easy to compute that the $E_2$-page is 
\[
	\F_\ell[x_1,\ldots,x_n] \tensor \Lambda_{\F_\ell}(z_1,\ldots,z_n)
\]
where the $x_i$'s occur on the line $s=0$ and the $z_i$'s
on the line $s=-1$. There can be no 
differentials, so the spectral sequence collapses on the $E_2$-page.
From the freeness of the $E_2$-page as an $H^\ast(BG)$-module we now 
deduce that the $H^\ast(BG^{h\sigma})$ is free as an 
$H^\ast(BG)$ module (with the module structure 
induced by $\ev_1\colon BG^{h\sigma} \to BG$).
It follows that the Eilenberg--Moore spectral sequence
of the pullback diagram
\[\xymatrix{
	G
	\ar[r]^-{i}
	\ar[d]
	\pb
	&
	BG^{h\sigma}
	\ar[d]^{\ev_1}
	\\
	\pt
	\ar[r]
	&
	BG	
}\]
is concentrated on the line $s=0$, which in turn implies that the 
map $i^\ast \colon H^\ast(BG^{h\sigma})\to H^\ast(G)$ is an epimorphism.  
In particular, $i^\ast$ and hence $i_\ast$ are nonzero in degree $d$, 
so $BG^{h\sigma}$ has a $[G]$--fundamental class by
Definition~\ref{def:fundamentalclass}.

The Serre spectral sequence of the fibre sequence
$G \xto{i} BG^{h\sigma} \to BG$ collapses at the $E_2$ page, as surjectivity of $i^\ast$ 
guarantees that there can be no differentials 
originating on the $E_r^{0,*}$ line,  and hence there cannot be any nonzero 
differentials at all by multiplicativity.
\end{proof}

\begin{proof}[Proof of Theorem~\ref{thm:hlbgcompintro}]
  By Proposition~\ref{prop:polycollapse},
  the shifted Serre spectral sequence
\begin{equation}
\label{eq:lbgss}
	\bbE_2^{s,t}(LBG) \isom H^s(BG) \tensor \bbH^t(G) 
	\Longrightarrow 
	\bbH^{s+t} (LBG)
\end{equation}
of Theorem~\ref{thm:ssforlbgbghsigma}(\ref{it:mainresult-ringss})
collapses on the $E_2$-page.
By Proposition~\ref{prop:pontryaginproduct}, the ring $\bbH^\ast(G)$
is isomorphic to the homology $H_{-\ast}(G)$ equipped with the 
Pontryagin product. Notice that the assumption that $\ell$ is odd
implies that $\deg(x_i)$ is even for all $i$, say
$\deg(x_i) = 2k_i$.
By Borel's theorem 
(see e.g.\ \cite[Cor.~VII.2.8]{MT91}),
the cohomology ring $H^\ast(G)$ is an exterior algebra on 
generators $z_1,\ldots,z_n$ with $\deg(z_i) = 2k_i-1$,
so by the Leray--Samelson theorem 
(see \cite[Thm.~2-1B]{kane88}), the homology
$H_\ast(G)$ equipped with Pontryagin product 
is isomorphic to an exterior algebra on 
generators $y'_1,\ldots,y'_n$, $\deg(y'_i) = 2k_i-1$.
In particular, the $E_2$-page of 
the spectral sequence~\eqref{eq:lbgss}
is a free graded commutative algebra.
By Theorem~\ref{thm:cmcomparison2},
the target $\bbH^\ast(LBG)$ 
of the spectral sequence
is graded commutative,
so no extension problems arise in passing from
the spectral sequence to the target.
The claim follows.
\end{proof}

\section{Fundamental classes exist generically: Proof of Theorem~\ref{thm:tezukasubgrp}}\label{sec:funclass2}

In this section, we will prove Theorem~\ref{thm:tezukasubgrp}.
The proof consists of three parts. First, 
in Section~\ref{subsec:sigmakinD}, we will 
show that for any self-map $\sigma\co BG \to BG$,
some power of $[\sigma]$ lies in the set 
\begin{equation}
\label{eq:Ddef}
	D 
	= 
	\{
		[\sigma] \in \Out(BG)	
		\,|\, 
		BG^{h\sigma} \text{ has a $[G]$--fundamental class}
	\}
\end{equation}
of Theorem~\ref{thm:tezukasubgrp}. Then, in Section~\ref{subsec:Dnormalsg}, we show
that $D$ is a normal subgroup of $\Out(BG)$. 
Finally, in Section~\ref{subsec:Dclosed}, we show that 
the subgroup $D$ is closed and of finite index in $\Out(BG)$,
and finish the proof of Theorem~\ref{thm:tezukasubgrp}.

\subsection{A power of \texorpdfstring{$[\sigma]$}{[sigma]} lies in \texorpdfstring{$D$}{D}}
\label{subsec:sigmakinD}

In this subsection, our aim is to prove that for any self-map $\sigma \colon BG \to BG$,
some power of $[\sigma] \in \Out(BG)$ lies in the set $D$ of 
equation~\eqref{eq:Ddef}. We will do so by proving
the following sharpening of a result of Kameko \cite[Thm.~1.5]{Kameko-pb}.

\begin{thrm}
\label{thm:ss-comparison2}
Let $BG$ be a connected $\ell$--compact group (not necessarily semisimple).
Let $\sigma \colon BG \to BG$ be a self-map of $BG$, and 
let $k > 0$ be an integer divisible by $\ell^{d^2+d}r^2$
where $d$ is the dimension of $BG$ and
$r$ is the order of the automorphism of $H^\ast(G)$
induced by $\sigma$. Then the Serre spectral sequences
with coefficients in $\F_\ell$
of the evaluation fibrations 
$\ev_1 \colon BG^{h\sigma^k} \to BG$ and $\ev_1 \colon LBG \to BG$ 
are isomorphic.
In particular, in this case, if $BG$ is semisimple,
$BG^{h\sigma^k}$ has a $[G]$--fundamental
class in the sense of Definition~\ref{def:fundamentalclass}.
\end{thrm}

We will prove Theorem~\ref{thm:ss-comparison2} by 
showing that the Serre spectral sequences of 
$\ev_1 \colon BG^{h\sigma^k} \to BG$ and $\ev_1 \colon LBG \to BG$ 
embed with the same image
as sub-spectral sequences of the Serre spectral sequence of a third fibration
indicated in diagram~\eqref{eqn:fibrations-sigmak1-sigmak2} below.
The proof is an abstraction of the earlier work of  Kameko \cite{Kameko-pb}, 
which again builds on work of 
Friedlander and Mislin \cite{FM84}. As we are working in a more general
setup,  and the draft
\cite{Kameko-pb} remains unpublished, we will give a self-contained
treatment, which also leads to slightly improved bounds.

The remainder of Section~\ref{subsec:sigmakinD}
is divided into two subsections.
First, in Section~\ref{subsubsec:mapxhstoxhsk},
we will define and consider the homological behavior of a certain map
$X^{h\sigma} \to X^{h\sigma^k}$ where $\sigma$ is a self-map of 
a space $X$.
Then, in Section~\ref{subsubsec:comparisonofspectralsequences},
we will develop tools for comparing spectral sequences,
and provide the proof of Theorem~\ref{thm:ss-comparison2}.

\subsubsection{The map $X^{h\sigma} \to X^{h\sigma^k}$ on homology}
\label{subsubsec:mapxhstoxhsk}
 
\begin{defn}[The maps $i \colon X^{h\sigma} \to X^{h\sigma^k}$ and 
$\bar{i} \colon \loops X \to \loops X$]
\label{def:iandbari}
Let $X$ be a connected space equipped with a basepoint.
Given a basepoint-preserving self-map $\sigma \colon X \to X$
and an integer $k > 0$, we define a map 
\[
	i = i^\sigma_k \colon X^{h\sigma} \longto  X^{h\sigma^k}
\]
by the formula
\[
	i^\sigma_k(\alpha) 
	= 
	\alpha \star (\sigma \circ \alpha) \star \cdots \star (\sigma^{k-1} \circ \alpha) 
\]
where $\star$ refers to concatenation of paths. Moreover, we write
\[
	\bar{i} = \bar{i}^\sigma_k \colon \loops X \to \loops X
\]
for the restriction of $i$ to a map between the fibres of the evaluation
fibrations $\ev_1 \colon X^{h\sigma} \to X$ and $\ev_1 \colon X^{h\sigma^k} \to X$,
so that altogether we have a commutative diagram
\[\xymatrix{
	\loops X 
	\ar[r]^{\bar{i}}
	\ar[d]
	&
	\loops X
	\ar[d]
	\\
	X^{h\sigma}
	\ar[r]^{i}
	\ar[d]_{\ev_1}
	&
	X^{h\sigma^k}
	\ar[d]^{\ev_1}
	\\
	X
	\ar@{=}[r]
	&
	X
}\]
where the columns are fibration sequences.
\end{defn}

\begin{rem}
\label{rk:fromunpointedtopointed}
In Definition~\ref{def:iandbari} and many other places in this section, it is 
most convenient to work with basepoint-preserving self-maps $\sigma$
instead of arbitrary ones, like in Theorem~\ref{thm:ss-comparison2}.
The restriction to basepoint-preserving self-maps is essentially
harmless, as it is easy to convert a non-basepoint-preserving self-map
into a basepoint-preserving one e.g.\ by the following ``whisker trick.''
Suppose $X$ is a connected space and $\sigma \colon X \to X$ is a self-map
of $X$. Let $x_0\in X$ be a basepoint, and consider the wedge sum 
$X \vee I = (X,x_0) \vee (I,1)$ where  $I$ is the interval $[0,1]$.
Equip $X \vee I$ with the basepoint given by $0\in I$.
Defining $\sigma' \colon X \vee I \to X \vee I$
to agree with $\sigma$ on $X$ and to be given by the formula 
$t \mapsto 2t$ on $[0,1/2] \subset I$ and 
by a path in $X$ connecting $x_0$ to $\sigma(x_0)$ on $[1/2,1] \subset I$
now yields a basepoint-preserving self-map $\sigma'$ of $X\vee I$
fitting into a commutative diagram 
\[\xymatrix{
	X
	\ar[r]^\sigma
	\ar[d]_\homot
	&
	X\ar[d]^\homot
	\\
	X\vee I
	\ar[r]^{\sigma'}
	&
	X\vee I
}\]
where both vertical maps are the (unpointed) homotopy equivalence given 
by the inclusion of $X$ into $X\vee I$.
\end{rem}
 
The following is a generalization of a well-known result of Friedlander and Mislin
\cite[Lem.~1.2]{FM84}.
\begin{prop} \label{prop:fixed-point-map2} 
Suppose that in the situation of Definition~\ref{def:iandbari} the space $X$ is simply connected and that the induced map $(\loops\sigma)_\ast \colon H_\ast(\Omega X;\F_\ell)\to H_\ast(\Omega X;\F_\ell)$ is an automorphism of finite order $r$. Then

\begin{enumerate}[(i)]

\item \label{item:prop-prim2}
If $\ell r | k$, then 
$\bar{i}_\ast \colon H_\ast(\loops X;\F_\ell) \to H_\ast(\loops X;\F_\ell)$ is zero
on primitive elements. In particular, it is the zero map 
in positive degrees if $H_\ast(\loops X;\F_\ell)$  is primitively generated.

\item \label{item:prop-general2} 
For $m > 0$,
the map $\bar{i}_\ast \colon H_m(\loops X;\F_\ell) \to H_m(\loops X;\F_\ell)$
vanishes if $\ell^m r | k$.
In particular, if $H_*(\loops X;\F_\ell)$ vanishes in degrees higher than
$d$ and $\ell^d r | k$, then 
$\bar{i}_\ast \colon H_\ast(\loops X;\F_\ell) \to H_\ast(\loops X;\F_\ell)$
is the zero map in positive degrees.
\end{enumerate}
\end{prop}

For the proof of Proposition~\ref{prop:fixed-point-map2}, we need a small algebraic lemma.
For a bialgebra $A$ and $n\geq 1$, let $P_n \colon A \to A^{\tensor n} \to A$
be the map obtained by composing the iterated coproduct with the iterated product.
\begin{lemma}
\label{lemma:nullmap2} 
Let $A$ be a cocommutative bialgebra. Then
\begin{enumerate}[(i)]
\item \label{it:coalgebramap}
	$P_n$ is a map of coalgebras for all $n\geq 1$.
\item \label{it:composition} 
	$P_{mn} = P_m \circ P_n \colon A\to A$ for all $m,n \geq 1$.
\item \label{it:primitives} 
	On primitive elements, the map $P_n \colon A \to A$ is given by multiplication by $n$.
\item \label{it:vanishing}
	Suppose $A$ is a connected cocommutative graded bialgebra in positive characteristic 
	$\ell$. Then $P_n$ vanishes in degree $m \geq 1$ as long as $\ell^m | n$.
\end{enumerate}
\end{lemma}
\begin{proof}
Parts (\ref{it:coalgebramap}), (\ref{it:composition}) and (\ref{it:primitives}) 
follow readily from the definition of the maps $P_n$, 
using that comultiplication is a map of coalgebras by cocommutativity.
In view of part (\ref{it:composition}), to show part (\ref{it:vanishing}),
it suffices to prove that for any $m\geq 0$, the map $P_{\ell^m}$ 
vanishes in positive degrees $\leq m$, which we will do by induction on $m$.
When $m = 0$, there is nothing to prove. 
When $m \geq 1$, we have
$P_{\ell^m} = P_{\ell} \circ P_{\ell^{m-1}}$ 
by part (\ref{it:composition}). Thus $P_{\ell^m}$ vanishes
in positive degrees $\leq m-1$ by the inductive assumption on $P_{\ell^{m-1}}$.
To see that $P_{\ell^m}$ also vanishes on degree $m$ elements,
notice that by part (\ref{it:coalgebramap}) and the inductive assumption,
the map $P_{\ell^{m-1}}$ sends degree $m$ elements to primitive elements.
As the map $P_{\ell}$ vanishes on primitive elements by part (\ref{it:primitives}), 
the claim follows. 
\end{proof}

\begin{proof}[Proof of Proposition~\ref{prop:fixed-point-map2}]
From the definition of $\bar{i}$, we have the formula
\[
	\bar{i}_\ast(x) = \sum x_{(1)} \sigma_\ast(x_{(2)}) \cdots \sigma_\ast^{k-1}(x_{(k)})
	\in
	H_\ast(\loops X)
\]
where $\sum x_{(1)}\tensor \cdots \tensor x_{(k)}$ is the iterated coproduct of $x$.
When $x$ is primitive, we therefore have
$
	\bar{i}_\ast(x) = \sum_{i = 0}^{k-1} \sigma_\ast^i(x)
$
which, when $\ell r | k$, gives us 
$
	\bar{i}_\ast(x) = (k/r) \sum_{i = 0}^{r-1} \sigma_\ast^i(x) = 0.
$
Thus part (\ref{item:prop-prim2}) follows. 
To show part (\ref{item:prop-general2}), 
notice that we have
$\bar{i} = \bar{i}^\sigma_k \homot \bar{i}^{\sigma^r}_{k/r} \circ \bar{i}^{\sigma}_r$
when $r | k$. As 
$(\bar{i}^{\sigma^r}_{k/r})_\ast = P_{k/r} \colon H_\ast(\loops X) \to H_\ast(\loops X)$,
part (\ref{item:prop-general2}) follows from Lemma~\ref{lemma:nullmap2}(\ref{it:vanishing}).\end{proof}

If a fibration admits a section, i.e., if it receives a map from the
fibration $* \to B \xto{\id} B$, then the cohomology of the base obviously
injects into the cohomology of the total space. The following
proposition records a more general version of this fact, just assuming that the
fibration receives a map from fibrations with the same base, with
trivial map on fibers.

\begin{prop}\label{prop:base-mono}
Suppose that we have a commutative diagram of nonempty connected spaces where the 
columns are fibration sequences 
\begin{equation}
\label{eq:fibrationsequenceladder}
\vcenter{\xymatrix{
    F_1 
    \ar[r] 
    \ar[d] 
    & 
    F_{2}
    \ar[r]
    \ar[d] 
    & 
    \cdots
    \ar[r]  
    &
    F_n 
    \ar[d]
    \\
    E_1 
    \ar[r]
    \ar[d]_{\pi_1}
    & 
    E_{2}
    \ar[r] 
    \ar[d]_{\pi_2}
    & 
    \cdots 
    \ar[r] 
    &
    E_n 
    \ar[d]^{\pi_n}
    \\
    B 
    \ar@{=}[r] 
    &
    B 
    \ar@{=}[r]
    & 
    \cdots 
    \ar@{=}[r] 
    & 
    B
}}
\end{equation}
Let $M$ be an abelian group, and consider the 
induced map 
\[
	\pi_n^\ast \colon H^*(B;M) \longto H^*(E_n;M).
\]
\begin{enumerate}[(i)]
\item \label{it:monoinlowdegrees}
	If $H^i(F_{k+1};M) \to H^i(F_k;M)$ is trivial when $1 \leq i\leq k < n$,
	the map $\pi_n^\ast$ is a monomorphism in degrees $\leq n$.
\item \label{it:monoinalldegrees}
	If in addition $H^i(F_n;M) = 0$  for $i\geq n$, the map $\pi_n^\ast$ is 
	a monomorphism in all degrees.
\item \label{it:generalcase}
	More generally, if $C \to B$ is a nonempty
	space over $B$ and $C' \subset C$ is a subspace,
	the same conclusions hold for the map
	\[
		(\pi_n \times_B C)^\ast 
		\colon 
		H^\ast(C,C';M) 
		\to 
		H^\ast(E_n \times_B C, E_n\times_BC';M)
	\]
	in place of $\pi_n^\ast$.
\end{enumerate}
\end{prop}

\begin{proof}

(\ref{it:monoinlowdegrees}):
We prove by induction on $n$ the stronger statement that 
in the Serre spectral sequence 
of the fibration $F_n \to E_n \to B$ 
for cohomology with coefficients in $M$,
all differentials $d_r: E_r^{*-r,r-1} \to E_r^{*,0}$ 
are zero when $2 \leq r \leq n$.
Notice that $E_2^{*,0} = H^*(B;M)$ as $F_k$ is connected.
For $n=1$ the claim holds vacuously.
When $n \geq 2$, we already know by induction that 
$d_r: E_r^{*-r,r-1} \to E_r^{*,0}$ is zero both for the fibration $F_{n-1} \to E_{n-1}
\to B$ and the fibration $F_n \to E_n \to B$ when $2 \leq r < n$. But then  $d_n: 
E_n^{*-n,n-1} \to E_n^{*,0}$ is zero for the fibration $F_n \to E_n \to B$ as
it factors as a composite of a map induced by the zero map
$H^{n-1}(F_n;M) \to  H^{n-1}(F_{n-1};M) $ with the corresponding differential for 
$F_{n-1} \to E_{n-1} \to B$.

(\ref{it:monoinalldegrees}):
If $H^i(F_n) = 0$ for $i\geq n$, then $d_r: E_r^{*-r,r-1} \to  E_r^{*,0}$
for $r>n$ as well because the domain is zero, so the map $\pi_n^\ast$
is a monomorphism in all degrees.

(\ref{it:generalcase}):
Proceeding component by component, we may assume that $C$ is connected.
The claim now follows by the same argument as in 
(\ref{it:monoinlowdegrees}) and (\ref{it:monoinalldegrees})
by considering the diagram of fibration sequences obtained by 
pulling back the fibrations $\pi_1,\ldots,\pi_n$ along the map $C\to B$
and 
using the relative version of the Serre spectral sequence.
\end{proof}

We note that Proposition~\ref{prop:fixed-point-map2} and
Proposition~\ref{prop:base-mono} have the
following corollary, which is an analog of \cite[Thm.~1.4]{FM84}.

\begin{cor} \label{cor:FM-cor-general} Suppose $X$ is simply connected 
pointed space with $H_*(\loops X;\F_\ell)$ 
zero in high degrees, and assume that $\sigma: X \to X$
is a basepoint-preserving 
self-map which induces an automorphism of finite order on
$H_*(\loops X;\F_\ell)$.
Then the evaluation fibrations $\ev_1 \colon X^{h\sigma^k} \to X$
for $k \geq 1$ induce an isomorphism
$$\colim_k H_*(X^{h\sigma^k};\F_\ell) \xto{\ \isom\ } H_*(X;\F_\ell)$$
where the colimit is taken over the poset of positive integers
ordered by divisibility, with the map 
$H_*(X^{h\sigma^{a}};\F_\ell) \to H_*(X^{h\sigma^b};\F_\ell)$
for $a | b$ induced by the map $i^{\sigma^{a}}_{b/a}$
of Definition~\ref{def:iandbari}.
\end{cor}

\begin{proof}
Let $\{E_{s,t}^r(k)\}$ be the  Serre spectral sequence 
of $\ev_1\colon X^{h\sigma^k} \to X$ for homology with coefficients in $\F_\ell$, 
and let
\[
	0 
	= 
	F^{-1}_n(k) 
	\subset 
	F^0_n(k) 
	\subset 
	\cdots 
	\subset 
	F^n_n(k) 
	= 
	H_n(X^{h\sigma^k};\F_\ell)
\]
be the associated filtration of $H_n(X^{h\sigma^k};\F_\ell)$ so that
$F^s_n(k) / F^{s-1}_n(k)$ is isomorphic to the $E_{s,n-s}^\infty(k)$
for all $0 \leq s \leq n$.
By Proposition~\ref{prop:fixed-point-map2}\eqref{item:prop-general2},
it is possible to find a $u > 0$ such that the map 
$E^2_{s,n-s}(k) \to E^2_{s,n-s}(ku)$
and hence the map
$E^\infty_{s,n-s}(k) \to E^\infty_{s,n-s}(ku)$
induced by the map
$i^{\sigma^k}_u \colon X^{h\sigma^k} \to X^{h\sigma^{ku}}$
are zero for $s < n$. It follows that for this $u$, we have 
$(i^{\sigma^k}_u)_\ast F^s_n(k) \subset F^{s-1}_n(ku)$ when $s < n$.
Iterating this observation, we can find a $v > 0$ such that 
$(i^{\sigma^k}_{v})_\ast F^{n-1}_n(k) \subset F^{-1}_n(kv) = 0$.
Thus any class in the kernel of 
$(\ev_1)_\ast \colon H_n(X^{h\sigma^k};\F_\ell) \to H_\ast(X;\F_\ell)$
already lies in the kernel of  
$(i^{\sigma^k}_{v})_\ast \colon H_n(X^{h\sigma^k};\F_\ell) \to H_n(X^{h\sigma^{kv}};\F_\ell)$. 
It follows that the map 
$\colim_k H_*(X^{h\sigma^k};\F_\ell) \to H_*(X;\F_\ell)$ 
is a monomorphism.
On the other hand, it follows readily from Propositions~\ref{prop:base-mono}
and \ref{prop:fixed-point-map2}(\ref{item:prop-general2}) that
there exists a $k$ for which 
the map $(\ev_1)_\ast \colon H_*(X^{h\sigma^k};\F_\ell) \to H_*(X;\F_\ell)$ 
is an epimorphism.  Thus the map 
$\colim_k H_*(X^{h\sigma^k};\F_\ell) \to H_*(X;\F_\ell)$ 
is also an epimorphism, and the claim follows.     
\end{proof}

Just for clarity, we emphasize that the above corollary in particular applies
to $\ell$--compact groups.

\begin{cor}
\label{cor:FM-cor}
Let $BG$ be a connected $\ell$--compact group (not necessarily semisimple),
and let $\sigma: BG \to BG$ be a basepoint-preserving self-map of $BG$
inducing an isomorphism on $H_*(G;\F_\ell)$. 
Then
\begin{equation}
\label{eq:colimbghsk}
\colim_k H_*(BG^{h\sigma^k};\F_\ell) \xrightarrow{\ \cong\ } H_*(BG;\F_\ell) 
\end{equation}
where the colimit is taken over the poset of positive integers
ordered by divisibility.
\qed
\end{cor}

\begin{rem}
According to Theorem~\ref{thm:ss-comparison2},
the modules $H_\ast(BG^{h\sigma^k};\F_\ell)$ 
in \eqref{eq:colimbghsk} are, from certain point onwards, all 
abstractly isomorphic to $H_\ast(LBG;\F_\ell)$.
What happens is that in the colimit system, 
the elements in the kernel of
$H_\ast(LBG;\F_\ell) \to H_\ast(BG;\F_\ell)$
are all eventually killed. Note that this is just like what
happens for $BG = BS^1\lcom$ where, after $\ell$--completion,
the space-level colimit system has a cofinal subsystem
$B\Z/\ell \to B\Z/\ell^2 \to \cdots$ with colimit
$B\Z/\ell^\infty$.
\end{rem}

\subsubsection{Comparison of spectral sequences}
\label{subsubsec:comparisonofspectralsequences}

Our aim in this sububsection is to finish the 
proof of Theorem~\ref{thm:ss-comparison2}.
The proof will be based on the following proposition.

\begin{prop}\label{prop:SS-zigzag}
Suppose that $\{'E_r\} \xrightarrow{f} \{E_r\} \xleftarrow{g} \{''E_r\}$ are maps of
spectral sequences.  
\begin{enumerate}[(i)]
\item \label{it:SS-zigzag-inj}
If $\im(f_2) \subseteq \im(g_2)$ and $g_r$ is injective
for all $r$, then $\im(f_r) \subseteq \im(g_r)$ for all $r$, and in particular
we have a well-defined map $h_r= g^{-1}_r \circ f_r\co {}'E_r \to {}''E_r$ for all $r$.
\item \label{it:SS-zigzag-iso}
If in addition $\im(f_2) = \im(g_2)$ and $f_2$ is injective, then $\im(f_r) =
\im(g_r)$ and $f_r$ is injective for all $r$, and in particular
$\{h_r\}\co \{'E_r\} \to \{''E_r\}$ defines an isomorphism of spectral
sequences. 
\end{enumerate}
\end{prop}
\begin{proof}
(\ref{it:SS-zigzag-inj}):
We prove (\ref{it:SS-zigzag-inj}) by induction on $r$, the case $r=2$ being clear. 
Suppose that $\bar x \in E_r$ is of the
form $\bar{x} = f_r(\bar y)$ for some $\bar y \in {}'E_r$. We can choose
representative $y \in {}'E_{r-1}$ of $\bar y$ such that $d_{r-1}y = 0$. 
By induction, we can find $z \in {}''E_{r-1}$ such that $f_{r-1}(y) =
g_{r-1}(z)$. But then 
$$g_{r-1}(d_{r-1}(z)) = d_{r-1}(g_{r-1}(z)) =
d_{r-1}(f_{r-1}(y)) =  f_{r-1}(d_{r-1}(y)) = 0$$
which by injectivity of $g_{r-1}$ means that $d_{r-1}(z) = 0$. Thus 
$z$ survives to $\bar z \in {}''E_{r}$ where 
$f_{r}(\bar y) = g_{r}(\bar z)$ as wanted.
  
(\ref{it:SS-zigzag-iso}): 
We again proceed by induction, the case $r=2$ being clear. 
Given an element $\bar z \in {}''E_r$, choose a lift 
$z \in {}''E_{r-1}$ with $d_{r-1}(z) =0$. By the inductive assumption,
we can find an element $y\in {}'E_{r-1}$ such that
$f_{r-1}(y) = g_{r-1}(z)$. Now
$$f_{r-1}(d_{r-1}(y)) = d_{r-1}(f_{r-1}(y)) = d_{r-1}(g_{r-1}(z)) =
g_{r-1}(d_{r-1}(z)) =0, $$
and hence $d_{r-1}(y) =0$, as $f_{r-1}$ is injective by
the inductive assumption. Hence $y$ defines a class  $\bar y \in E_r$ with $f_r(\bar y) =
g_r(\bar z)$. We conclude that $\im(g_r) \subseteq \im(f_r)$. Combining this with part (\ref{it:SS-zigzag-inj}), we see that $\im(f_r) = \im(g_r)$.

To prove that $f_r$ is injective, assume that $\bar y \in {}'E_r$ is in the kernel
of $f_r$, and choose $y \in
{}'E_{r-1}$ with $d_{r-1}(y) =0$ representing $\bar y$.
As $h_r(\bar y) =0$, we can choose $z' \in
{}''E_{r-1}$ with $h_{r-1}(y) = d_{r-1}(z')$. As 
$\im(f_{r-1}) = \im(g_{r-1})$ by the inductive assumption, we can find
$y' \in {}'E_{r-1}$ such that $f_{r-1}(y') = g_{r-1}(z')$. Then
$$h_{r-1}(y-d_{r-1}(y')) = h_{r-1}(y) - d_{r-1}(h_{r-1}(y')) =
h_{r-1}(y) - d_{r-1}(z') = 0$$
and hence $y = d_{r-1}(y')$ as $h_{r-1}$ is injective by induction. So $\bar
y = 0$ as wanted, proving that $f_r$ is injective.
\end{proof}

To apply Proposition~\ref{prop:SS-zigzag},
we need a criterion for checking that a
map between spectral sequences is injective
on all pages.
To this end, recall that one formalism for spectral sequences is via
Cartan--Eilenberg systems \cite[Section XV.7]{CE56}. 
Such a system consists of modules $H^\ast(i,j)$
for $-\infty \leq i \leq j \leq \infty$,
together with maps $\eta\colon H^\ast(i',j') \to H^\ast(i,j)$
for $i \leq i'$ and  $j \leq j'$ and 
$\delta\colon H^\ast(i,j) \to H^\ast(j,k)$ for 
$-\infty \leq i \leq j \leq k \leq \infty$,
all satisfying certain axioms we will not restate here.
For us, the first important point about Cartan--Eilenberg systems
is that the Serre spectral sequence
arises from such a system.

\begin{example}[Cartan--Eilenberg system for the Serre spectral sequence]
\label{ex:serressfromcartaneilenberg}
Suppose $\pi \colon E \to B$ is a fibration, 
and let $\Gamma B \to B$ be a CW approximation of $B$.
For $-\infty \leq s \leq \infty$, let $\sk^B_s E$ be the pullback of $E$ to 
$(\Gamma B)^{(s)}$, where $(\Gamma B)^{(s)}$ is the $s$--skeleton of $\Gamma B$
for $0 \leq s < \infty$; $(\Gamma B)^{(s)} = \emptyset$ for $s < 0$; and 
$(\Gamma B)^{(s)} = \Gamma B$ for $s = \infty$. 
Setting 
\[
	H^\ast(i,j) = H^\ast(\sk^B_{j-1} E, \sk^B_{i-1} E)
\]
for $-\infty \leq i \leq j \leq \infty$
and viewing these groups together with the maps induced by inclusions of pairs and 
the connecting homomorphisms of the long exact sequences of the triples
$(\sk_{k-1} E, \sk_{j-1} E, \sk_{i-1} E)$ yields a Cartan--Eilenberg system
whose associated spectral sequence is the cohomological Serre spectral sequence of~$\pi$.
This follows by combining \cite[Section~XV.7, Example~2]{CE56} 
with the usual construction of the Serre spectral sequence
from the filtration $\{\sk^B_i E\}_i$ of the pullback of $E$ to $\Gamma B$.
\end{example}

The second important point about Cartan--Eilenberg systems for us
is that they provide the following criterion for verifying that
a map between spectral sequences is injective.

\begin{prop} \label{prop:CE-systems} 
Assume that $\alpha\co \{H^*(i,j)\} \to \{'H^*(i,j)\}$ 
is a map of Cartan--Eilenberg systems which is 
injective
for all $-\infty < i \leq j < \infty$.
Then the associated map of spectral sequences $E_r \to
{}'E_r$ is injective for all $r \geq 2$.
\end{prop}
\begin{proof}
Suppose $\bar{x} \in E_r^i$ satisfies $\alpha(\bar{x})=0$ in ${}'E_r^i$.
Pick a representative $x\in Z_r^i \subset H^\ast(i,i+1)$
for $\bar{x}$. Then by the choice of $x$, we have $\alpha(x)\in {}'B^i_r$,
so that $\alpha(x)$ is in the image of 
$\delta \colon {}'H^\ast(i-r+1,i) \to {}'H^\ast(i,i+1)$,
or what is the same,
in the kernel of 
$\eta \colon {}'H^\ast(i,i+1) \to {}'H^\ast(i-r+1,i+1)$.
Consequently $\alpha(\eta(x)) = \eta(\alpha(x)) = 0$,
so $\eta(x) = 0$ in $H^\ast(i-r+1,i+1)$
since $\alpha \colon H^\ast(i-r+1,i+1) \to {}'H^\ast(i-r+1,i+1)$
is injective. But this means that $x$ is in the image of 
$\delta \colon  H^\ast(i-r+1,i) \to H^\ast(i,i+1)$, showing that $\bar{x} = 0$.
\end{proof}

We note parenthetically that the analogous criterion 
for checking that a map between spectral sequences
is surjective 
in terms of Cartan--Eilenberg systems
also holds. We will not need this criterion, however.

We continue with two results which will help us verify that 
the hypotheses in Proposition~\ref{prop:SS-zigzag} are satisfied.

\begin{prop}
\label{prop:e2monoswithsameimage}
Assume that $X$ is a simply connected pointed space
for which $H^\ast(\loops X; \F_\ell)$ 
is finite-dimensional and vanishes in degrees 
higher than $d$, 
and that $\sigma$ is a basepoint-preserving
self-map of $X$ such that 
$(\loops \sigma)^\ast \in \Aut(H^\ast(\loops X;\F_\ell))$
has finite order $r$.
Suppose $k_1$ and $k_2$ are positive integers with $k_2 | k_1$, 
and consider the commutative diagram
\begin{equation} \label{eqn:fibrations-sigmak1-sigmak2}
\vcenter{\xymatrix{
	\loops X \ar[d]
	& 
	\loops X \times \loops X   
	\ar[l]_(0.66){\bar f} 
	\ar[r]^(0.66){\pr_1} 
	\ar[d] 
	& 
	\loops X \ar[d]
	\\
	LX
	\ar[d]_{\ev_1}
	& 
	X^{h\sigma^{k_1}} \times_{X} X^{h\sigma^{k_2}}  
	\ar[l]_(0.66)f
	\ar[r]^(0.66){\pr_1} 
	\ar[d] 
	&
	X^{h\sigma^{k_1}}
	\ar[d]^{\ev_1}
	\\
	X 
	& 
	X
	\ar@{=}[r]
	\ar@{=}[l] 
	&
	X
}}
\end{equation}
where the columns are fibration sequences, $f$ is the composite
\[
	X^{h\sigma^{k_1}} \times_{X} X^{h\sigma^{k_2}}  
	\xto{\ \id \times_X i\ }
	X^{h\sigma^{k_1}} \times_{X} X^{h\sigma^{k_1}}
	\xto{\ c\ }
	LX
\]
where $i = i^{\sigma^{k_2}}_{k_1/k_2}$ is the map of 
Definition~\ref{def:iandbari} and $c$
maps $(\alpha,\beta)$ to $\alpha \star \beta^{-1}$,
and $\bar{f}$ is the restriction of $f$ to a map between fibers.
Then, if $\ell^d r | (k_1/k_2)$,
the induced maps from the $E_2$--pages of the mod $\ell$ cohomological 
Serre spectral sequences of the leftmost and rightmost columns 
of \eqref{eqn:fibrations-sigmak1-sigmak2}
into that of the middle column 
of \eqref{eqn:fibrations-sigmak1-sigmak2}
are monomorphisms with the same image.
\end{prop}
\begin{proof}
In view of Proposition~\ref{prop:fixed-point-map2}(\ref{item:prop-general2}),
in both cases the map on the $E_2$--pages agrees with the map
\[
	H^\ast(X;\F_\ell) \tensor H^\ast(\loops X;\F_\ell)
	\longto
	H^\ast(X;\F_\ell) \tensor H^\ast(\loops X;\F_\ell) \tensor H^\ast(\loops X;\F_\ell)
\]
induced by the monomorphism
\[
	H^\ast(\loops X;\F_\ell) 
	\longto
	H^\ast(\loops X;\F_\ell)\tensor H^\ast(\loops X;\F_\ell),
	\qquad
	x \longmapsto x \tensor 1. \qedhere
\]
\end{proof}

\begin{prop}
\label{prop:ermono}
Suppose in the situation of Proposition~\ref{prop:e2monoswithsameimage}
we have $\ell^{d^2}r | k_2$. 
Then the induced map between the mod $\ell$ cohomological
Serre spectral sequences of the rightmost and middle columns in 
\eqref{eqn:fibrations-sigmak1-sigmak2}
is a monomorphism on $E_r$--pages for all $2 \leq r < \infty$.
\end{prop}
\begin{proof}
In view of 
Example~\ref{ex:serressfromcartaneilenberg}
and
Proposition~\ref{prop:CE-systems},
it suffices to show that the map
\begin{equation}
\label{map:skmap}
	\sk_{j-1}^X (X^{h\sigma^{k_1}} \times_{X} X^{h\sigma^{k_2}})
	\longto
	\sk_{j-1}^X (X^{h\sigma^{k_1}})
\end{equation}
induced by the projection 
$\pr_1 \colon X^{h\sigma^{k_1}} \times_{X} X^{h\sigma^{k_2}} \to X^{h\sigma^{k_1}}$
induces a monomorphism
\[
	H^\ast(	\sk_{j-1}^X (X^{h\sigma^{k_1}}), \sk_{i-1}^X (X^{h\sigma^{k_1}}))
	\longto
	H^\ast(
		\sk_{j-1}^X (X^{h\sigma^{k_1}} \times_{X} X^{h\sigma^{k_2}}),
		\sk_{i-1}^X (X^{h\sigma^{k_1}} \times_{X} X^{h\sigma^{k_2}})
	)
\]
for all $-\infty < i \leq j < \infty$.
For $0 \leq a \leq d$, let $e_a = \ell^{ad-d^2} k_2$, and
consider the diagram
\begin{equation}
\label{diag:fibrationsequences}
\vcenter{\xymatrix{
	\loops X
	\ar[r]^{\bar{i}}
	\ar[d]
	&
	\loops X
	\ar[r]^{\bar{i}}
	\ar[d]
	&
	\cdots
	\ar[r]^{\bar{i}}
	&
	\loops X
	\ar[d]
	\ar@{=}[r]
	&
	\loops X
	\ar[d]
	\\
	X^{h\sigma^{e_0}}
	\ar[r]^{i}
	\ar[d]_{\ev_1}
	&
	X^{h\sigma^{e_1}}
	\ar[r]^{i}
	\ar[d]_{\ev_1}
	&
	\cdots
	\ar[r]^{i}
	&
	X^{h\sigma^{e_d}}
	\ar[d]^{\ev_1}
	\ar@{=}[r]
	&
	X^{h\sigma^{k_2}}
	\ar[d]^{\ev_1}
	\\
	X
	\ar@{=}[r]
	&
	X
	\ar@{=}[r]
	&
	\cdots
	\ar@{=}[r]
	&
	X
	\ar@{=}[r]
	&
	X
}} 
\end{equation}
where the columns are fibration sequences
and the maps labeled by $i$ and $\bar{i}$
are various instances of the similarly named maps
defined in Definition~\ref{def:iandbari}.
By Proposition~\ref{prop:fixed-point-map2}(\ref{item:prop-general2}),
the maps labeled $\bar{i}$ in \eqref{diag:fibrationsequences}
vanish on $H^\ast(\loops X; \F_\ell)$ in positive degrees,
so the claim follows from Proposition~\ref{prop:base-mono}(\ref{it:generalcase})
by observing that \eqref{map:skmap} agrees with the pullback of 
the map $\ev_1 \colon X^{h\sigma^{k_2}} \to X$ along the composite map
\[
	\sk_{j-1}^X (X^{h\sigma^{k_1}}) \longto  X^{h\sigma^{k_1}} \xto{\ \ev_1\ } X. \qedhere
\]
\end{proof}

\begin{proof}[Proof of Theorem~\ref{thm:ss-comparison2}]
Using the whisker trick of Remark~\ref{rk:fromunpointedtopointed},
we may without loss of generality assume that $\sigma$ preserves
the basepoint. 
That the Serre spectral sequences of 
$\ev_1 \colon BG^{h\sigma^k} \to BG$ and $\ev_1 \colon LBG \to BG$ 
are isomorphic
now follows from
Proposition~\ref{prop:SS-zigzag}(\ref{it:SS-zigzag-iso})
combined with Propositions~\ref{prop:e2monoswithsameimage} and \ref{prop:ermono}
by taking $X = BG$, $k_1 = k$, and $k_2 = k/(\ell^d r)$ in 
Proposition~\ref{prop:e2monoswithsameimage}.
Finally, the claim about $[G]$--fundamental classes follows from
the equivalence of conditions~(\ref{it:ssiso}) and (\ref{it:fclass})
of Theorem~\ref{thm:conj-red3}.
\end{proof}

\begin{rem}
The argument presented here only shows that,
with $X = BG$ and the numbers $k_1 = k$ and $k_2$ chosen suitably,
the maps $f$  and $\pr_1$ of \eqref{eqn:fibrations-sigmak1-sigmak2}
induce monomorphisms with the same image
on the $E_\infty$--pages of the relevant Serre spectral sequences.
It however seems very likely that the 
same is true already on the level of cohomology,
that is, that the maps induced by $f$ and $\pr_1$ embed
$H^*(BG^{h\sigma^k})$ and $H^*(LBG)$ as the same submodule of 
$H^*(BG^{h\sigma^k} \times_{BG} BG^{h\sigma^{k_2}})$, a stronger statement, 
after perhaps substituting $\sigma$ by a power. 
This would establish that $H^*(BG^{h\sigma^k})$ and $H^*(LBG)$ 
are isomorphic as 
rings over the Steenrod algebra, and also provide a distinguished 
``dual fundamental class'' in $H^*(BG^{h\sigma^k})$, namely the element
corresponding to the unit in $\bbH^*(LBG)$.
\end{rem}

\subsection{The set \texorpdfstring{$D$}{D} is a normal subgroup of 
\texorpdfstring{$\Out(BG)$}{Out(BG)}}
\label{subsec:Dnormalsg}

The aim of this subsection is to prove the following proposition.
\begin{prop}
\label{prop:dsubgrp}
When $BG$ is semisimple, 
the set $D$ of equation~\eqref{eq:Ddef} is a normal subgroup of $\Out(BG)$.
\end{prop}

We will start with

\begin{defn}
Call fibrations $\pi \colon E \to B$ and $\pi' \colon E' \to B'$
\emph{equivalent} if there exists a zigzag
\[\xymatrix@!C=0.15em{
	E
	\ar[d]_\pi
	\ar[rr]^{\homot}
	&&
	\bullet
	\ar[d]
	&&
	\ar[ll]_\homot
	\bullet
	\ar[d]
	&
	\cdots
	&
	\bullet
	\ar[d]
	\ar[rr]^\homot
	&&
	E'
	\ar[d]^{\pi'}
	\\
	B
	\ar[rr]^\homot
	&&
	\bullet
	&&
	\ar[ll]_\homot
	\bullet
	&
	\cdots
	&
	\bullet
	\ar[rr]^\homot
	&&
	B'
}\]
where all the squares commute, all horizontal morphisms are homotopy equivalences,
and all vertical morphisms are fibrations. We write $\pi\sim \pi'$
to indicate that $\pi$ and $\pi'$ are equivalent.
\end{defn}
For the fibrations $\pi(f,g) \colon P(f,g) \to B$ of Definition~\ref{def:pfg},
we then have the following.
\begin{lemma}
\label{lm:equivfreerank1}
\label{A: Here we do need $BG$ to be semisimple, since that is needed for 
the existence of the string module structure.}
Suppose $\pi(f,g)\colon P(f,g) \to B$ and $\pi(f',g')\colon P(f',g') \to B'$ 
are equivalent fibrations with $B$ and $B'$ path connected. 
Then $\bbH^\ast P(f,g)$ is free of rank $1$ as an 
$\bbH^\ast P(g,g)$--module 
if and only if 
$\bbH^\ast P(f',g')$ is free of rank $1$ as an
$\bbH^\ast P(g',g')$--module.
\end{lemma}
\begin{proof}
Clearly the Serre spectral sequences of equivalent fibrations are isomorphic.
Thus the claim follows from the equivalence of conditions (\ref{it:xexists})
and (\ref{it:permanentcycle2}) in Theorem~\ref{thm:conj-red3}.
\end{proof}

\begin{lemma}
\label{lm:equivalentfibrationcriteria}
In each of the following, the indicated fibrations are equivalent.
\begin{enumerate}[(i)]
\item \label{it:swap}
	$\pi(f,g) \colon P(f,g) \to B$ 
	and
	$\pi(g,f)\colon P(g,f) \to B$ 
	for maps $f,g\colon B\to BG$.
\item \label{it:precomp}
	$\pi(f,g) \colon P(f,g) \to B$ 
	and 
	$\pi(f\phi,g\phi) \colon P(f\phi,g\phi) \to A$
	for maps $f,g\colon B\to BG$ and a homotopy equivalence 
	$\phi \colon A \to B$
\item \label{it:postcomp}
	$\pi(f,g)\colon P(f,g) \to B$  
	and
	$\pi(\sigma f, \sigma g) \colon P(\sigma f, \sigma g) \to B$ 
	for maps $f,g\colon B\to BG$ and a homotopy equivalence 
	$\sigma \colon BG \to BG$
\item \label{it:homot}
	$\pi(f_0,g_0)\colon P(f_0,g_0) \to B$ 
	and 
	$\pi(f_1,g_1)\colon P(f_1,g_1)\to B$
	for maps $f_i,g_i \colon B \to BG$, $i=0,1$,
	with $f_0\homot f_1$ and $g_0\homot g_1$.
\end{enumerate}
\end{lemma}
\begin{proof}
Part~(\ref{it:swap}): 
The claim follows by observing that the map 
$P(f,g) \to P(g,f)$, $(b,\gamma)\mapsto (b,\gamma^{-1})$
is a homeomorphism over $B$.

Part~(\ref{it:precomp}):
The claim 
follows from the fact that the map $\bar{\phi}$ in \eqref{sq:inducedmap}
is a homotopy equivalence when $\phi$ is. To prove this fact,
one can work directly with the definition of a fibration in terms
of a homotopy lifting property, or observe that 
the Strøm model structure on topological spaces is right proper,
so that the claim follows from the definition of a right proper model category.
See \cite[Def.~13.1.1 and Cor.~13.1.3(3)]{Hirschhorn}.

Part~(\ref{it:postcomp}):
Observe that we have a commutative cube
\[\xymatrix@!0@C=5em@R=8ex{
	P(f,g)
	\ar[dd]_(0.26){\pi(f,g)}
	\ar[rr]
	\ar[dr]
	&&
	BG^I
	\ar[dd]_(0.26){(\ev_0,\ev_1)}|!{[dl];[dr]}\hole
	\ar[dr]^{\sigma_\sharp}_\homot
	\\
	&
	P(\sigma f,\sigma g)
	\ar[dd]^(0.26){\pi(\sigma f, \sigma g)}
	\ar[rr]
	&&
	BG^I
	\ar[dd]^(0.26){(\ev_0,\ev_1)}
	\\
	B
	\ar@{=}[dr]
	\ar[rr]^(0.26){(f,g)}|!{[ur];[dr]}\hole
	&&
	BG\times BG
	\ar[dr]_(0.4){\sigma\times\sigma}^{\homot}
	\\
	&
	B
	\ar[rr]^(0.26){(\sigma f,\sigma g)}
	&&
	BG\times BG
}\]
where the back and front faces are the defining pullback squares for $\pi(f,g)$
and $\pi(\sigma f, \sigma g)$, respectively, 
and the map from $P(f,g)$ to $P(\sigma f, \sigma g)$
is the unique one making the cube commutative. By \cite[Prop.~13.3.14]{Hirschhorn},
applied with the Strøm model structure on topological spaces, 
this map is a homotopy equivalence, implying the claim.

Part~(\ref{it:homot}):
Choose homotopies $F\colon f_0\homot f_1$ and $G\colon g_0 \homot g_1$.
Let $j_i \colon B \to B\times I$  be the inclusion $b \mapsto (b,i)$, $i=0,1$,
and let $\bar{j}_i \colon P(f_i,g_i) \to P(F,G)$
be the map $(b,\gamma)\mapsto ((b,i),\gamma)$,
$i=0,1$.
The claim now follows from the zigzag
\[\xymatrix{
	P(f_0,g_0)
	\ar[d]_{\pi(f_0,g_0)}
	\ar[r]^{\bar{j}_0}
	&
	P(F,G)
	\ar[d]_{\pi(F,G)}
	&
	P(f_1,g_1)
	\ar[d]^{\pi(f_1,g_1)}
	\ar[l]_{\bar{j}_1}
	\\
	B
	\ar[r]^-{j_0}_-{\homot}
	&
	B\times I
	&
	B
	\ar[l]_-{j_1}^-{\homot}
}\]
where $\bar{j}_0$ and $\bar{j}_1$ are homotopy equivalences
by the proof of part~(\ref{it:precomp}).
\end{proof}

\begin{proof}[Proof of Proposition~\ref{prop:dsubgrp}]
By Theorem~\ref{thm:strtoptezukacrit},
the homotopy class of a self homotopy equivalence
$\sigma \colon BG \to BG$
belongs to $D$ if and only if $H^\ast(BG^{h\sigma})$
is free of rank $1$ over $\bbH^\ast(LBG)$.
The cohomology of $BG^{h\id_{BG}} = LBG$ is certainly
free of rank $1$ over itself, so $[\id_{BG}] \in D$.
For a self homotopy equivalence $\sigma$ of $BG$,
let us write $\ev(\sigma)$ for the evaluation map
$\ev_1 \colon BG^{h\sigma} \to BG$.
By Proposition~\ref{prop:specialcasesofpathspaceconstr}\eqref{it:bghsigma},
we may identify $\ev(\sigma)$ with the fibration
$\pi(\sigma,\id_{BG})\colon P(\sigma,\id_{BG}) \to BG$.
By Lemma~\ref{lm:equivalentfibrationcriteria}, we have
\[
	\ev(\sigma^{-1}) 
	= 
	\pi(\sigma^{-1},\id_{BG})
	\sim
	\pi(\sigma^{-1}\sigma,\sigma)
	\sim
	\pi(\id_{BG},\sigma)
	\sim
	\pi(\sigma,\id_{BG})
	=
	\ev(\sigma)
\]
and 
\[
	\ev(\alpha\sigma\alpha^{-1})
	=
	\pi(\alpha\sigma\alpha^{-1},\id_{BG})
	\sim
	\pi(\alpha\sigma\alpha^{-1}\alpha,\alpha)
	\sim
	\pi(\alpha\sigma,\alpha)
	\sim
	\pi(\sigma,\id_{BG})
	=
	\ev(\sigma)	
\]
for self homotopy equivalences $\sigma$ 
and inverse self homotopy equivalences $\alpha$ and $\alpha^{-1}$ of $BG$,
so Lemma~\ref{lm:equivfreerank1}
implies that $D$ is closed under conjugation and taking inverses.

It remains to show that $D$ is closed under composition.
Suppose $[\sigma],[\sigma'] \in D$.
Our aim is to show that $[\sigma\sigma'] \in D$.
In view of Theorem~\ref{thm:strtoptezukacrit} and the
equivalence between (\ref{it:xexists})
and (\ref{it:permanentcycle2}) in Theorem~\ref{thm:conj-red3},
it suffices to show that the generator of
\[
	\bbE^{0,0}_2 (P(\sigma\sigma',\id_{BG}))\isom \F_\ell
\]
in the shifted Serre spectral sequence of 
$\pi(\sigma\sigma',\id_{BG})\colon P(\sigma\sigma',\id_{BG}) \to BG$
is a permanent cycle. Under our pairings between spectral 
sequences, this generator factors as the product of the 
generators of $\bbE^{0,0}_2 (P(\sigma',\id_{BG}))$
and $\bbE^{0,0}_2 (P(\sigma\sigma',\sigma'))$,
so it is enough to show that these two factors 
are permanent cycles.
That the first factor is a permanent cycle 
follows from the assumption that $[\sigma'] \in D$ 
by Theorem~\ref{thm:strtoptezukacrit} and
the equivalence between
(\ref{it:xexists})
and (\ref{it:permanentcycle2}) 
in Theorem~\ref{thm:conj-red3}.
That the second factor is also 
a permanent cycle follows similarly from the 
assumption that $[\sigma] \in D$
by noting that Lemma~\ref{lm:equivalentfibrationcriteria}(\ref{it:precomp})
implies that
\[
	\pi(\sigma\sigma',\sigma') \sim \pi(\sigma,\id_{BG}) = \ev(\sigma'),
\]
and observing that equivalent fibrations have isomorphic Serre spectral sequences.
\end{proof}

\subsection{The subgroup \texorpdfstring{$D$}{D} is closed and of finite index 
in \texorpdfstring{$\Out(BG)$}{Out(BG)}}
\label{subsec:Dclosed}

In this subsection, we will show that the subgroup $D \leq \Out(BG)$
is a closed finite-index subgroup
and finish the proof of Theorem~\ref{thm:tezukasubgrp}.
We will need the following theorem, which is a slight strengthening of a result
that traces back to \cite{BM07}.
\begin{thrm}
\label{thm:closedsubgrphelp} Let $BG$ be a connected $\ell$--compact
group (not necessarily semisimple),
and suppose that $[\sigma],[\sigma'] \in \Out(BG)$ 
generate the same
closed subgroup of $\Out(BG)$. 
Then $BG^{h\sigma}$ and $BG^{h\sigma'}$ are homotopy equivalent over
$BG$, i.e., we can choose a homotopy equivalence so that the following
diagram commutes
\[\xymatrix{
	BG^{h\sigma}
	\ar[dr]_{\ev_1}
	\ar[rr]^\homot
	&&
	BG^{h\sigma'}
	\ar[dl]^{\ev_1}
	\\
	&
	BG
}\]
In particular $[\sigma] \in D$ 
if and only if $[\sigma'] \in D$.
\end{thrm}

\begin{proof}[Proof of Theorem~\ref{thm:closedsubgrphelp}]
Note that $H^*(BG)$
is Noetherian by \cite[Thm.~2.4]{DW94} and the topology on
$\Out(BG)$ is Hausdorff by Lemma~\ref{lm:outbghausdorff}.
Hence the assumptions of \cite[Thm.~2.4]{BrotoMoellerOliver} (generalizing
\cite[Prop.~6.5]{BM07}) are satified, and that result produces a homotopy equivalence
$BG^{h\sigma} \xto{\homot} BG^{h\sigma'}$. An inspection of the proof
of that result
furthermore shows that this homotopy equivalence
can be chosen so that the indicated diagram commutes.
\end{proof}

The theorem implies the following.
\begin{lemma}
\label{lm:closedsubgrp2}
When $BG$ is semisimple,
$x\in D$ implies $\overline{\langle x \rangle} \leq D$ where
$\overline{\langle x \rangle}$ is the closure of $\langle x\rangle$ in $\Out(BG)$.
\end{lemma}
\begin{proof}
  By Proposition~\ref{prop:topcomparison}, $\Out(BG)$ is isomorphic to
  $\Out(\bbD_G)$ also as a topological group, and by Proposition~\ref{prop:outDstructure2},
  $\Out(\bbD_G)$ contains an open subgroup of finite index isomorphic to a
  pro-$\ell$--group.
  Hence some finite power of $x$, say $x^m$,
belongs to an open pro-$\ell$--subgroup of $\Out(BG)$.
The closed procyclic subgroup 
$\overline{\langle x^m\rangle} \leq \Out(BG)$
generated by $x^m$ is therefore a pro-$\ell$-group,
and hence isomorphic either to 
$\Z_\ell$ or $\Z/\ell^n$ for some finite $n$.
In either case, every closed subgroup of 
$\overline{\langle x^m\rangle}$ 
is topologically generated by some power of $x^m$.
Thus, if $y\in \overline{\langle x^m\rangle}$,
then $\overline{\langle y\rangle} = \overline{\langle x^{mk} \rangle}$
for some $k$, and therefore by Theorem~\ref{thm:closedsubgrphelp}
the element
$y$ belongs to $D$ since by Proposition~\ref{prop:dsubgrp} $x^{mk}$ does. 
We conclude that $\overline{\langle x^m \rangle} \subset D$.
Consequently by Proposition~\ref{prop:dsubgrp}
\[
	\overline{\langle x \rangle}
	= 
	\bigcup_{k=0}^{m-1} x^k \overline{\langle x^m \rangle}
	\subset
	D
\]
as claimed.
\end{proof}

\begin{prop}
\label{prop:Dopen}
Suppose $BG$ is a semisimple $\ell$--compact group. Then the set $D$ is an open subset of $\Out(BG)$.
\end{prop}
\begin{proof}
It suffices to show that the subgroup $D \leq \Out(BG)$ contains some open subgroup of 
$\Out(BG)$. 
By the assumption that $BG$ is semisimple we have $\pi_1(\bbD_G)\tensor \Q = 0$,
so by Propositions~\ref{prop:topcomparison} and \ref{prop:outDstructure2},
$\Out(BG)$ contains an open subgroup $U$ isomorphic
to $(\Z_\ell)^t$ for some $t$. Choose a set of topological generators 
$x_1, \ldots, x_t$ for $U$. 
By Theorem~\ref{thm:ss-comparison2}, for each $i$, a power
$x_i^{n_i}$ of $x_i$ lies in $D$, and by Lemma~\ref{lm:closedsubgrp2}, so
does $\overline{\langle x_i^{n_i}\rangle}$. It follows that $D$ contains
the subgroup
\[
	U' 
	= 
	\overline{\langle x_1^{n_1}\rangle} \cdots \overline{\langle x_t^{n_t}\rangle}
	\leq
	U
\]
generated by the subgroups $\overline{\langle x_i^{n_i}\rangle}$. 
On the other hand, $U'$ is closed and finite-index in $U$, and hence open in $U$
and therefore in $\Out(BG)$. The claim follows.
\end{proof}

\begin{proof}[Proof of Theorem~\ref{thm:tezukasubgrp}]
Since open subgroups of the compact group $\Out(BG)$ are
finite-index and closed, the first half of the theorem follows from 
Propositions~\ref{prop:dsubgrp} and~\ref{prop:Dopen}, 
noting that we already made the observation in 
Corollary~\ref{cor:tezukaimpliesid} that 
every $\sigma \in D$ acts trivially on $H^*(BG)$.

The final part also follows from what we have proven so far together with 
standard facts about closed finite-index subgroups of $\Z_\ell^\times$. 
Let $H = 1 + 2 \ell\,\Z_\ell < \Z_\ell^\times$.
By Corollary~\ref{cor:ladicunitstooutbgcont}, the homomorphism 
\[
	\phi \colon H \longto \Out(BG),\qquad q \longmapsto \psi^q.
\]
is continuous, so
it follows from the first part of the theorem
that the preimage $\phi^{-1} (D)$ is a closed finite-index subgroup of $H$.
Since any such subgroup of $H$ is of the form $1 + 2\ell^k \,\Z_\ell$
for some $k\geq 1$, the claim follows.
\end{proof}

\begin{rem}[Generalizing Proposition~\ref{prop:Dopen} to connected $\ell$--compact groups]
\label{rk:Dopengeneralization}
The proof of Proposition~\ref{prop:Dopen} is the 
only place in the proofs of Theorems~\ref{thm:mainresult}, 
\ref{thm:strtoptezukacrit} and \ref{thm:tezukasubgrp}
where we make use of the assumption that $BG$ is a semisimple 
(as opposed to just connected) $\ell$--compact group
that goes beyond simply ensuring that the assumptions of
Proposition~\ref{prop:pfgobjinf} are satisfied. To help clear the way 
for the potential generalization of Theorem~\ref{thm:tezukasubgrp}
from semisimple to arbitrary connected $\ell$--compact groups,
we now indicate the extra argument required to prove
Proposition~\ref{prop:Dopen} in that greater generality
assuming that the other results of Section~\ref{sec:funclass2}
have already been generalized to connected $\ell$--compact groups.

In the proof of Proposition~\ref{prop:Dopen}, the extra use of
semisimplicity went to ensuring that $\Out(BG)$ 
has an open subgroup $U$ isomorphic to $(\Z_\ell)^t$ for some $t$.
In the case of a general connected $\ell$--compact group $BG$,
Propositions~\ref{prop:topcomparison} and \ref{prop:outDstructure2}
only imply the existence of an open subgroup $U$ of $\Out(GB)$
of the form $U = \Gamma_s \times (\Z_\ell)^t$ 
where $\Gamma_s$ is the principal congruence subgroup
\[
	\Gamma_s = \ker(\GL_n(\Z_\ell) \to \GL_n(\Z/\ell^s))
\]
for $n = \dim_{\Q_\ell}(\pi_1(\bbD_G)\tensor \Q)$.
To deal with the extra $\Gamma_s$--factor, we may argue as follows.
For $x\in \Z_\ell$ and $1\leq i,j \leq n$, 
write $E_{ij}(x)$ for the $(n\times n)$--matrix
whose $(i,j)$--th entry is $x$ and which is zero otherwise, and 
write $A_{ij}(x)$ for the matrix $A_{ij}(x) = I_n + E_{ij}(x)$.
We note that the maps $(1+x) \mapsto A_{ii}(x)$ and $x \mapsto A_{ij}(x)$
for $i\neq j$ are continuous embeddings from 
the multiplicative group $1 + 2\ell\Z_\ell \leq \Z_\ell^\times$
and  the additive group $\Z_\ell$ into $\GL_n(\Z_\ell)$,
respectively. Notice also that $A_{ij}(\ell^s) \in \Gamma_s$
for all $1\leq i,j \leq n$. 
With the aid of Theorem~\ref{thm:ss-comparison2},
we may find powers $k_1$ and $k_2$ such that $A_{ii}(\ell^s)^{k_1} \in D$
for all $1\leq i \leq n$ and $A_{ij}(\ell^s)^{k_2} \in D$ for all $i \neq j$.
Moreover, we may choose $k_1$ and $k_2$ so that 
$A_{ii}(\ell^s)^{k_1} = A_{ii}(z_1)$ and $A_{ij}(\ell^s)^{k_2} = A_{ij}(z_2)$
where the $\ell$--adic valuations of the elements $z_1, z_2 \in \Z_\ell$ 
are equal to some $s'$ (where necessarily $s' > s$). By Lemma~\ref{lm:closedsubgrp2},
we have $\overline{\langle  A_{ii}(z_1)\rangle} \leq D$
and $\overline{\langle  A_{ij}(z_2)\rangle} \leq D$
for $1\leq i,j \leq n$, $i\neq j$,
so $D$ contains the matrix $A_{ij}(z)$ for all
$z\in \ell^{s'}\Z_\ell$, $1 \leq i,j \leq n$.
But the elementary row operations represented by these matrices suffice
to reduce an arbitrary matrix in $\Gamma_{s'}$ to the identity matrix,
so we must have $\Gamma_{s'} \leq D$.
Combining this with the argument for the $(\Z_\ell)^t$--factor
given in the proof of Proposition~\ref{prop:Dopen},
we see that $D$ contains an open subgroup $U'$ of $U$, and the claim 
follows.
\end{rem}

\section{Fundamental classes for finite groups of Lie type: Proof of
  Theorem~\ref{thm:examples}}
\label{sec:funclass1}

Our aim in this section is to  prove Theorem~\ref{thm:examples}.
We begin by making some observations about products in
Section~\ref{subsec:productsfund}. 
In Section~\ref{subsec:spin}, we deal with the
$\Spin(n)$ case.
Finally, in Section~\ref{subsec:putittogether},
we put it all together and prove Theorem~\ref{thm:examples}.

\subsection{Products}\label{subsec:productsfund}
Before proving Theorem~\ref{thm:examples},
we will make some general observations about products.

\begin{lemma}\label{lem:productsfund} 
Suppose $BG$ and $BH$ are connected $\ell$--compact
groups, and let $\sigma \co BG \to BG$ and $\tau \co BH \to BH$
be maps such that $BG^{h\sigma}$ and $BH^{h\tau}$ have
$[G]$-- and $[H]$--fundamental classes, respectively.
Then $(BG \times BH)^{h(\sigma\times \tau)}$
has a $[G \times H]$--fundamental class.
\end{lemma}
\begin{proof}
The claim follows from the observation that the fibre 
sequence
\[
	G\times H 
	\longto 
	(BG \times BH)^{h(\sigma\times \tau)}
	\longto 
	BG\times BH	
\]
of equation~\eqref{eq:fibseq} for $G\times H$
is the product of the analogous fibre sequences for $G$ and $H$.
\end{proof}
The next lemma is an analogue of a well-known result for finite groups of Lie type \cite[Exercise~30.2]{MT11}.
\begin{lemma}\label{lem:permute}
Let $BH$ be a simple $\ell$--compact group, and let
$BG = BH^n$.
Suppose $\sigma \in \Out(BG)$ is an element 
(of any order, infinite or finite) such that
the permutation of $\{1,\ldots,n\}$ induced by $\sigma$
\cite[Prop.~8.14]{AG09} is transitive.
Then $\sigma^n = \diag(\rho,\ldots,\rho)$ for  some $\rho \in
\Out(BH)$ and 
\[
	BG^{h\langle\sigma\rangle} \homot BH^{h\langle \rho\rangle}.
\]
\end{lemma}
\begin{proof}
By \cite[Thm.~1.2 and Prop.~8.14]{AG09} $\Out(BG) \isom \Out(\bbD_G)
\isom \Out(\bbD_H)\wr \fS_n$. 
Hence $\sigma^n = \diag(\rho,\ldots, \rho)$ 
for some $\rho \in \Out(\bbD_H)$.
Note that by Proposition~\ref{autBTtoBG}, $\sigma$ and $\rho$
determine well-defined homotopy actions on $BG$ and $BH$. 
Furthermore,
$$(BG)^{h\langle \sigma\rangle} \homot
((BG)^{h\langle\sigma^n\rangle})^{hC_n} \homot
((BH^{h\langle\rho\rangle})^n)^{hC_n} \homot BH^{h\langle \rho\rangle}$$
where the first homotopy equivalence follows by transitivity of homotopy
fixed points, see\ e.g.\ \cite[Lem.~10.5]{DW94}, and the last one by
``Shapiro's lemma,'' see e.g.\ \cite[Lem.~10.8]{DW94}.
\end{proof}

We need to know that the untwisting process cannot introduce the summands excluded in Theorem~\ref{thm:examples}. The proof uses some case-by-case observations, but only on the level of root data.

\begin{prop}

\label{prop:bghtfornonexceptional2cptgrps}
Suppose $BG$ is a simply connected $2$--compact group 
not containing any $E_6$, $E_7$, and $E_8$ summands. Then for 
any $\tau \in \Out(BG)$ of finite odd order, the homotopy fixed point
space $BG^{h\langle\tau\rangle}$ is a simply connected $2$--compact group
not containing these summands. Moreover, if in addition $BG$ does not contain
any $\Spin(n)$--summands for $n\geq 10$, neither does $BG^{h\langle\tau\rangle}$.
\end{prop}

\begin{proof}
That $BG^{h\langle\tau\rangle}$ is again a simply connected $2$--compact
group follows from Proposition~\ref{prop:ghksplit}.
By \cite[Thm.~1.4]{dw:split} (or \cite[Thm.~1.2 and Prop.~8.12]{AG09})
$BG$ splits as a product
$BG \homot \prod_{i=1}^s (BK_i)^{m_i}$ for some
non-isomorphic simple simply-connected $2$--compact groups $BK_i$,
and  by 
\cite[Thm.~1.2 and Prop.~8.14]{AG09}
we then have $\Out(BG) \isom \prod_{i=1}^s\Out({BK_i})\wr \fS_{m_i}$.
Consequently, we can write $BG$ as a product
$BG \homot \prod_{i=1}^t (BH_i)^{n_i}$ for (potentially isomorphic)
simple and simply-connected 
$BH_i$ such that $\tau$
permutes the factors of 
each product  $(BH_i)^{n_i}$ transitively.
Thus  Lemma~\ref{lem:permute} implies that 
\[
	BG^{h\langle \tau \rangle} \homot  \prod_{i=1}^t (BH_i)^{h\langle \rho_i\rangle}
\]
for some odd-order elements $\rho_i \in \Out(BH_i)$.
The possible twistings $\rho_i$ can be read off from \cite[Thm.~13.1]{AGMV08}.
A glance at that list reveals that there is in fact only one possibility 
for a nontrivial $\rho_i$ of odd order, 
namely the triality graph automorphism $\rho$ when
$BH_i \homot B\Spin(8)\twocom$.
The resulting $2$--compact group
$(B\Spin(8)\twocom)^{h\langle \rho \rangle}$ certainly does not have
an $E_i$--summand for rank reasons (and in fact it is easily seen to be
$(BG_2)\twocom$, as one would expect). Thus the claim follows.
\end{proof}

\subsection{Fundamental classes for \texorpdfstring{$B\Spin(n)$}{BSpin(n)}}
\label{subsec:spin}

We now prove the existence of a fundamental class in the $\Spin(n)$ 
case for $\sigma = \psi^q$, 
building on the work of Kameko \cite{Kameko-spin}.
\begin{prop}
\label{prop:spin-case}
$(B\Spin(n)\twocom)^{h\psi^q}$
has a $[\Spin(n)\twocom]$--fundamental class for all 
$q\in \Z_2^\times$ and $n \geq 2$.
\end{prop}
\begin{proof}
Since the mod 2 cohomology of $B\Spin(2)\twocom$ is a polynomial ring,
the claim for $n=2$ follows from 
Propositions~\ref{prop:polycollapse} and \ref{prop:idmappoly}.
Let us assume that $n\geq 3$.
For brevity, let us write $BSO_n$ and $B\Spin_n$ for $BSO(n)\twocom$
and $B\Spin(n)\twocom$, respectively.

Let $D$ be as in Theorem~\ref{thm:tezukasubgrp}.
By Corollary~\ref{cor:ladicunitstooutbgcont}, the homomorphism 
\[
	\phi \colon \Z_2^\times \longto \Out(BG),\qquad q \longmapsto \psi^q.
\]
is continuous,
so it follows from Theorem~\ref{thm:tezukasubgrp}
that the preimage $\phi^{-1}(D)$ is a closed subgroup of $\Z_2^\times$.
Since the closed subgroup of $\Z_2^\times$
generated by the elements $3$ and $5\in \Z_2^\times$
is all of $\Z_2^\times$, 
it follows that
it is enough to prove that $(B\Spin_n)^{h\psi^q}$ 
has a fundamental class when $q=3$ or $q=5$.
So let us assume that $q$ is one of these numbers;
in fact, all that matters for the argument that follows
is that $q$ is an odd prime power. By choosing suitable
models for $BSO_n$ and $B\Spin_n$ and for the self-maps
$\psi^q$ of $BSO_n$ and $B\Spin_n$,
we may assume the following:
\begin{enumerate}
\item The maps $\psi^q \co BSO_n \to BSO_n$ and 
	$\psi^q\co B\Spin_n \to B\Spin_n$
	are basepoint-preserving.
\item The map $p\co B\Spin_n \to BSO_n$
	induced by the projection $\Spin_n \to SO_n$
	preserves basepoints and commutes with the action 
	of the maps $\psi^q$.
\end{enumerate}
For example, by replacing $\psi^q \co BSO_n \to BSO_n$
by a homotopic map if necessary, we may assume that 
it preserves the basepoint, after which models for
$B\Spin_n$ and $\psi^q \co B\Spin_n \to B\Spin_n$
with the desired properties can be obtained
by passing to functorial $2$--connected covers.

Let $E_{n,q}$ and $E_n$ be the spaces obtained as pullbacks
\begin{equation}
\label{eq:epbsqs}
    \vcenter{\xymatrix{
    	E_{n,q}
		\pb
    	\ar[r]
    	\ar[d]
    	&
    	BSO_n^{h\psi^q}
    	\ar[d]^{\ev_1}
    	\\
    	B\Spin_n
    	\ar[r]^{p}
    	&
    	BSO_n	
    }}
    \qquad\text{and}\qquad
    \vcenter{\xymatrix{
    	E_{n}
		\pb
    	\ar[r]
    	\ar[d]
    	&
    	LBSO_n^{\vphantom{h\psi^q}}
    	\ar[d]^{\ev_1}
    	\\
    	B\Spin_n
    	\ar[r]^{p}
    	&
    	BSO_n	
    }}
\end{equation}
Let 
\[
	\pi_{n,q} \co B\Spin_n^{h\psi^q} \longto E_{n,q}
	\qquad\text{and}\qquad
	\pi_{n} \co LB\Spin_n \longto E_n
\]
be the maps over $B\Spin_n$ induced by the evaluation maps
\[
	\ev_1 \co B\Spin_n^{h\psi^q} \longto B\Spin_n
	\qquad\text{and}\qquad
	\ev_1 \co LB\Spin_n  \longto B\Spin_n	
\]
and the maps
\[
	B\Spin_n^{h\psi^q} \longto BSO_n^{h\psi^q}
	\qquad\text{and}\qquad
	LB\Spin_n \longto LBSO_n
\]
induced by $p$. Then, up to homotopy, $\pi_{n,q}$ and $\pi_n$
are two-fold covering spaces, and there are maps 
of fibre sequences
\[
    \vcenter{\xymatrix{
    	\Z/2 
    	\ar[r] 
    	\ar@{=}[d]
    	&
    	\loops B\Spin_n 
    	\ar[r]^-{\loops p}
    	\ar[d]^{i_{n,q}}
    	&
    	\loops BSO_n
    	\ar[d]^{j_{n,q}}
    	\\
    	\Z/2
    	\ar[r]
    	&
    	B\Spin_n^{h\psi^q}
    	\ar[r]^-{\pi_{n,q}}
    	&
    	E_{n,q}
    }}
    \qquad\text{and}\qquad
    \vcenter{\xymatrix{
    	\Z/2 
    	\ar[r] 
    	\ar@{=}[d]
    	&
    	\loops B\Spin_n 
    	\ar[r]^-{\loops p}
    	\ar[d]^{i_n}
    	&
    	\loops BSO_n
    	\ar[d]^{j_n}
    	\\
    	\Z/2
    	\ar[r]
    	&
    	LB\Spin_n^{\phantom{h}}
    	\ar[r]^-{\pi_{n}}
    	&
    	E_{n}
    }}
\]
where the vertical arrows are inclusions of fibres of 
the various spaces over $B\Spin_n$ over the basepoint.
We obtain the following commutative
diagram whose rows are long exact Gysin sequences:
\[\xymatrix{
	&
	\mathllap{\cdots\longto}\,
	H^d(E_{n,q})
	\ar[r]^-{\pi_{n,q}^\ast}
	\ar[d]_{j_{n,q}^\ast}
	&
	H^d (B\Spin_n^{h\psi^q})
	\ar[r]^-{(\pi_{n,q})_!}
	\ar[d]_{i_{n,q}^\ast}%
	&	
	H^d(E_{n,q})
	\ar[r]^-{\cup e_{n,q}}
	\ar[d]_{j_{n,q}^\ast}
	&
	H^{d+1}(E_{n,q})
	\ar[d]_{j_{n,q}^\ast}
	\,\mathrlap{\longto\cdots}
	&
	\\
	&
	\mathllap{\cdots\longto}\,
	H^d(\loops BSO_n)
	\ar[r]^{(\loops p)^\ast}%
	&
	H^d(\loops B\Spin_n)
	\ar[r]^-{(\loops p)_!}
	&
	H^d(\loops BSO_n)
	\ar[r]^-{\cup e}
	&
	H^{d+1}(\loops BSO_n)
	\,\mathrlap{\longto\cdots}
	&
	\\
	&
	\mathllap{\cdots\longto}\,
	H^d(E_{n})
	\ar[r]^-{\pi_{n}^\ast}
	\ar[u]^{j_n^\ast}
	&
	H^d (LB\Spin_n)
	\ar[r]^-{(\pi_{n})_!}
	\ar[u]^{i_n^\ast}
	&
	H^d(E_{n})
	\ar[r]^-{\cup e_{n}}
	\ar[u]^{j_n^\ast}
	&
	H^{d+1}(E_{n})
	\ar[u]^{j_n^\ast}
	\,\mathrlap{\longto\cdots}
	&
}\]
Here $d= \dim SO_n =  \dim \Spin_n$
and $e$, $e_n$ and $e_{n,q}$ are 
the Euler classes for the respective double covers.

Our task is to show that the map $i_{n,q}^\ast$ in the above
diagram is nonzero. To do this,
it suffices to show that the composite map
\[
	(\loops p)_! \circ i_{n.q}^\ast 
	= 
	j_{n,q}^\ast \circ (\pi_{n,q})_!
	\co
	H^d(B\Spin_n^{h\psi^q})
	\longto
	H^d(\loops BSO_n)
\]
is nonzero. Thus it is enough to show that there 
exists an element $x\in H^d(E_{n,q})$ such that $j_{n,q}^\ast(x) \neq 0$
and $x \cup e_{n,q} = 0$.

To show the existence of such an element $x$,
we will compare the upper part of the diagram 
with the lower part. 
As in \cite[Props.~4.1 and 4.3]{Kameko-spin},
the Eilenberg--Moore spectral sequences for the 
pullback squares \eqref{eq:epbsqs} yield
ring isomorphisms
\[
    H^\ast(E_{n,q}) 
    \isom 
    H^\ast(B\Spin_n) \tensor_{H^\ast (BSO_n)} H^\ast(BSO_n^{h\psi^q})
\]
and
\[
    H^\ast(E_{n}) 
    \isom 
    H^\ast(B\Spin_n) \tensor_{H^\ast (BSO_n)} H^\ast(LBSO_n).
\]
Let
\[
	k_{n,q} \co \loops BSO_n \longto BSO_n^{h\psi^q}
	\qquad\text{and}\qquad
	k_{n} \co \loops BSO_n \longto LBSO_n
\]
be inclusions of fibres of the evaluation fibrations
$BSO_n^{h\psi^q}\to BSO_n$ and $LBSO_n \to BSO_n$.
Under the above isomorphisms, the maps
$j_{n,q}^\ast$ and $j_n^\ast$
then correspond to the maps induced by
\[
	k_{n,q}^\ast\co H^\ast(BSO_n^{h\psi^q}) \longto H^\ast(\loops BSO_n)
	\qquad\text{and}\qquad
	k_{n}^\ast\co H^\ast(LBSO_n) \longto H^\ast (\loops BSO_n)
\] 
(along with the augmentation $H^\ast(B\Spin_n) \to \F_2)$,
respectively.
By \cite[Thm.~1.7]{Kameko-spin},
there exists a $H^\ast(BSO_n)$--algebra isomorphism
\[
	H^\ast(BSO_n^{h\psi^q}) \isom H^\ast(LBSO_n).
\]
Under this isomorphism, the maps  $k_{n,q}^\ast$ and $k_n^\ast$ agree;
to see this, observe that both halves of the diagram on p.\ 525
of Kameko's paper \cite{Kameko-spin} are pullbacks, so 
$k_{n,q}$ and $k_n$ factor (up to homotopy) through the same
map from $\loops BSO_n$ to Kameko's space $\tilde{B} A_{n-1}$.
We conclude that there exists a ring isomorphism
\begin{equation}
\label{iso:enqen}
	H^\ast(E_{n,q}) \isom H^\ast(E_n)
\end{equation}
under which the maps $j_{n,q}^\ast$ and $j_n^\ast$ agree.

Using for example the Serre spectral sequence, it is easy to see
that $H^1 E_{n,q} \isom H^1 E_n \isom \F_2$.
The Euler classes $e_{n,q} \in H^1 E_{n,q}$ and $e_n \in H^1 E_n$
both pull back to the class $e\in H^1 \loops BSO_n$, which in 
turn pulls back to the (nontrivial) Euler class of the 
double cover $\Spin_2 \to SO_2$. Thus the classes $e_{n,q}$
and $e_n$ must be the unique nontrivial degree $1$ classes
in their respective cohomology groups, and hence they must correspond
under the isomorphism \eqref{iso:enqen}.

We have reduced the task of constructing the desired class 
$x\in H^d E_{n,q}$ to the task of finding a class 
$y\in H^d E_n$ such that $j_n^\ast(y) \neq 0$ and $y\cup e_n = 0$.
By Theorem~\ref{thm:functoriality}, the map $i_n^\ast$
sends the class $s^d(\bbOne)\in H^d(LB\Spin_n)$ to the nontrivial class
$s^d(\bbOne)\in H^d(\loops B\Spin_n)$.
Moreover, since $H^{d+1}(\loops BSO_n) = 0$, the map $(\loops p)_!$ is an
epimorphism and hence an isomorphism, as its source and
target are both one-dimensional.
Thus $(\loops p)_! i_n^\ast s^d(\bbOne) \neq 0$. Now the class
$y = (\pi_n)_!s^d(\bbOne) \in H^d(E_n)$ is as desired.
\end{proof}

\subsection{Proof of Theorem~\ref{thm:examples}}
\label{subsec:putittogether}

With the preparation of the previous subsections, we can now prove
 Theorem~\ref{thm:examples}. Let us first prove the following.
 
\begin{thrm} 
  \label{thm:examples2} Suppose $BG$ is a connected $\ell$--compact group
  that is a product of an $\ell$--compact group
with polynomial mod
$\ell$ cohomology ring %
and copies of
$B\Spin(n)\lcom$
for various $n$. Let $\tau \in \Out(BG) \isom \Out(\bbD_G)$ 
be of finite order, with
$\ell \nmid |\tau|$ if $\ell$ odd and $\tau =1$ if $\ell =2$.
Then for every $q \in \Z_\ell^\times$,  $\BtGq$ has a
$[G^{h\muet}]$--fundamental class.
\end{thrm}
\begin{proof}
Suppose first that $\ell$ is odd. 
Then the cohomology of $B\Spin(n)\lcom$ 
is a polynomial cohomology ring 
(see~Theorem~\ref{thm:poly-cohom}), so
the theorem is reduced to the case where $BG$ has polynomial
cohomology ring. 
By the untwisting theorem, Theorem~\ref{thm:untwisting-intro}, we can write 
 \[
\BtGq \xto{\ \homot\ } (BG^{h\muet})(q')\]
with $\muet$ of order prime to $\ell$ and $q'$ congruent to $1$ modulo $\ell$.
Furthermore, by Proposition~\ref{prop:ghksplit},
the cohomology of the fixed points $BG^{h\muet}$ is also 
a polynomial ring. And by Proposition~\ref{prop:idmappoly}, $\psi^{q'}$ acts as the identity on $H^*(BG^{h\muet})$.
Hence the assumptions of Proposition~\ref{prop:polycollapse} are satisfied and the theorem follows.

Now consider the case $\ell=2$.  In this case $e$, the multiplicative order of $q$ mod $\ell$, 
is $1$, and by assumption, $\tau=1$, so we are reduced to showing that
$BG(q)=BG^{h\psi^q}$ has a $[G]$--fundamental class. 
In view of Lemma~\ref{lem:productsfund}, we are reduced to proving this in the polynomial case and in the spin case individually. The polynomial case again follows from Proposition~\ref{prop:polycollapse}, as the assumptions are satisfied by Proposition~\ref{prop:idmappoly}, and the spin case is the content of 
Proposition~\ref{prop:spin-case}.
\end{proof}

\begin{proof}[Proof of Theorem~\ref{thm:examples}]
For \eqref{it:fcexistwithexclusions} our task is to show that $\BtGq$ has a
$[G^{h\muet}]$--fundamental class
when we are away from the 8 exceptional cases
listed in the statement of Theorem~\ref{thm:examples}.
If $\ell$ is odd,
our list of exclusions ensures that $BG$
has polynomial cohomology, by Theorem~\ref{thm:poly-cohom}(\ref{poly-torsion-free}) 
and (\ref{polyodd}), and it follows 
from Theorem~\ref{thm:examples2} that $\BtGq$ has a
$[G^{h\muet}]$--fundamental class. If $\ell=2$,
our list of exclusions ensures that 
$BG$ is a product of a $2$--compact group with
polynomial cohomology and $B\Spin(n)\twocom$'s for
various $n\geq 10$, by the classification of $2$--compact groups \cite[Thm.~1.1]{AG09}
and  Theorem~\ref{thm:poly-cohom}(\ref{poly2-pi1torsionfree}), and
by Proposition~\ref{prop:bghtfornonexceptional2cptgrps},
the same holds for $BG^{h\muet}$. Hence $\BtGq$ has a
$[G^{h\muet}]$--fundamental class by
Theorem~\ref{thm:examples2}.

For \eqref{it:liftexclusions} it is, by Theorem~\ref{thm:untwisting-intro}, enough to see that $BG^{h\muet}(q')$ has a $[G^{h\muet}]$--fundamental class for any simply connected $\ell$--compact group $BG$, under the stated assumptions. But, by the same theorem,  $BG^{h\muet}$ is again a simply connected $\ell$--compact group. In other words, the claim reduces to showing that, for any simply connected $\ell$--compact group $BG$,  $BG(q)$ has a $[G]$--fundamental class, as long as $q$ is congruent to $1$ modulo $\ell$, under the stated assumptions.
As $BG$ splits as a product of simple simply connected $\ell$--compact groups, by \cite[Thm.~1.4]{dw:split} (or \cite[Thm.~1.2 and Prop.~8.12]{AG09}), we can, by Lemma~\ref{lem:productsfund}, even assume that $BG$ is simple and simply connected. 

Hence, if $\ell =5$, we have to establish that $(BE_8)\fivecom(q)$ has a $[(E_8)\fivecom]$--fundamental class for any $q$ congruent to $1$ modulo $5$, by \eqref{it:fcexistwithexclusions}. However, as $\nu_5(11-1) =1$, Theorem~\ref{thm:tezukasubgrp} implies that if $BE_8(11)\fivecom$ has a fundamental class, then the same is true for $BE_8(q)\fivecom$ for all $q \in \Z_5$ with $q$ congruent to $1$ modulo 5.  If $\ell =3$ we have to establish the same claim for $(BF_4)\threecom$ and $(BE_i)\threecom$, $i=6,7,8$, and $q$ congruent to $1$ modulo $3$, which follows similarly.  For $\ell =2$ there is a small subtlety as the structure of the $\ell$--adic units, and consequently the statement of Theorem~\ref{thm:tezukasubgrp}, is slightly different at $\ell=2$: By assumption $BE_i(5)\twocom$ have fundamental classes, which by Theorem~\ref{thm:tezukasubgrp} means that the same is true for any $q\in \Z_2$ congruent to $1$ modulo $4$. Now, if $i=7,8$ such $q$ generate $\Out(\bbD)$ as $-1 \in W$ (see Proposition~\ref{prop:twistingclassification}). If $i=6$  such $q$ together with $3$ generate $\Z_2^\times \cong \Out(\bbD)$, which finishes the proof in that case.
\end{proof}

\begin{rem}[The assumption $\ell \nmid |\tau|$] In Section~\ref{subapp:untwisting}, and in particular Proposition~\ref{prop:twistingclassification}, we explain the (fairly mild) restrictions posed by the $\ell \nmid |\tau|$ assumption.
It would be interesting to work out directly
what
happens in the few cases where $\tau$ is of order $\ell$.
\end{rem}

\begin{rem} By Corollary~\ref{cor:tezukaimpliesid}, a fundamental class for 
$E_6(q)$ at $\ell=2$ 
would imply that 
$\psi^q$ acts as the identity $H^*(BE_6;\F_2)$ also without the assumption that $q$ is congruent to $1$ mod $4$, a non-obvious claim as $-1$ is not in the Weyl group. 
One may check that this claim \emph{is} true, as the
cohomology as an $\calA$--algebra is generated by the class in degree
$4$ together with pull-backs of Chern classes, by \cite{KNN19}.
\end{rem}

\section{Isomorphisms of algebras: Proof of Theorem~\ref{thm:polyexamples}}
\label{sec:algiso}

Recall from Theorem~\ref{thm:strtoptezukacrit} 
that when $H^\ast(BG^{h\sigma})$ is free of 
rank 1 over $\bbH^\ast(LBG)$, it is possible to find an element 
$x\in H^\ast(BG^{h\sigma})$ such that the map 
\begin{equation}
\label{eq:introbulletx}
\xymatrix@C-0.3em{
	H^\ast(LBG)
	\ar[r]^-{s^{-d}}
	&
	\bbH^\ast(LBG)
	\ar[rr]^-{\stringprod s^{-d}(x)}
	&&
	\bbH^\ast(BG^{h\sigma})
	\ar[r]^-{s^d}
	&
	H^\ast(BG^{h\sigma})
}
\end{equation}
induced by string multiplication by $x$
is close to being a ring isomorphism 
in the sense that it induces an isomorphism between certain 
associated graded algebras of $H^\ast(LBG)$ and $H^\ast(BG^{h\sigma})$.
In this section, we will go further and prove results giving
sufficient conditions under which the element $x$ can be chosen so that
\eqref{eq:introbulletx} is an actual ring isomorphism and 
furthermore commutes with a large number of Steenrod operations.
Our first result in this direction, Theorem~\ref{thm:polytezukaelaboration}, 
covers in particular all $\ell$--compact groups with polynomial cohomology
when $\ell$ is odd, while the second, Theorem~\ref{thm:polysimplyconn2cptgrps},
pertains to simply connected $2$--compact groups with polynomial cohomology.
Finally, combining these two results, we will 
present a proof of Theorem~\ref{thm:polyexamples}. 
Throughout the section, except where otherwise stated, we continue to assume
that $BG$ is a semisimple $\ell$--compact group of dimension $d$.

The section is divided into 3 subsections. First, in 
Section~\ref{subsec:addingstruc}, we will prove a realization result,
Theorem~\ref{thm:realization}, which allows us to realize a highly 
structured isomorphism from $H^\ast(LBG)$ to $H^\ast(BG^{h\sigma})$---meaning 
one commuting with cup products and many Steenrod operations---as composites 
of the form \eqref{eq:introbulletx} given that such an isomorphism exists 
abstractly. Combining this result with ideas and results of 
Kishimoto and Kono \cite{KK10} yielding such abstract isomorphisms, 
in Section~\ref{subsec:polyinevendegrees}
we will prove Theorem~\ref{thm:polytezukaelaboration}.
Finally, in Section~\ref{subsec:polysimplyconn2cptgrps},
we use Theorem~\ref{thm:realization} and computations of Kishimoto and Kono
\cite{KK10} and Kaji \cite{KajiMod2total} to prove 
Theorem~\ref{thm:polysimplyconn2cptgrps};
and use Theorems~\ref{thm:polytezukaelaboration} and \ref{thm:polysimplyconn2cptgrps}
to prove Theorem~\ref{thm:polyexamples}.

\begin{rem}
It would be very interesting 
to find constructions of 
a class $x$ for which \eqref{eq:introbulletx} is a ring isomorphism
that do not depend on the 
a priori knowledge that $H^\ast(LBG)$ and $H^\ast(BG^{h\sigma})$
are isomorphic as rings, especially in cases where these rings 
are difficult to compute. See Question~\ref{qu:thequestion}(\ref{q3}).
\end{rem}

\subsection{Realizing highly structured isomorphisms in terms of the string module structure}
\label{subsec:addingstruc}

In this subsection, we will prove the following result stating
that in the polynomial case, whenever a highly structured isomorphism 
$H^\ast(LBG) \isom H^\ast(BG^{h\sigma})$
exists abstractly, such an isomorphism can be realized in terms of 
the string module structure.

\begin{thrm}[Realization theorem]
\label{thm:realization}
Suppose $BG$ is a semisimple $\ell$--compact  group
of dimension $d$ for which $H^\ast(BG)$ is a polynomial ring, and let 
$\sigma\colon BG \to BG$ be a map such that there exists
an isomorphism 
\[
	\theta \colon H^\ast(LBG) \xto{\ \isom\ } H^\ast(BG^{h\sigma})
\]
of $H^\ast(BG)$--algebras. Then for a suitable scalar 
$c\in \F_\ell^\times$, the composite 
\begin{equation}
\label{eq:basismultcomp}
	H^\ast(LBG) 
	\xto{\ s^{-d}\ } 
	\bbH^\ast(LBG) 
	\xto{\ \stringprod c\theta(\bbOne)\ } 
	\bbH^\ast(BG^{h\sigma})
	\xto{\ s^{d}\ } 	
	H^\ast(BG^{h\sigma}) 
\end{equation}
is an isomorphism of $H^\ast(BG)$--algebras. Moreover,
if $\theta$ commutes with the action of a Hopf subalgebra
$\tilde{\calA} \subset \calA_\ell$ of the mod $\ell$ Steenrod algebra,
so does~\eqref{eq:basismultcomp}.
\end{thrm}

To prove Theorem~\ref{thm:realization}, we begin with 
a series of auxiliary results.
For the remainder of the subsection, assume that 
$H^\ast(BG) = \F_\ell [x_1,\ldots,x_n]$.
As in \cite[Appendix~E]{KM19}, write
$\Delta \colon H^\ast(LBG) \to H^{\ast-1}(LBG)$
for the operator
characterized by the formula
\begin{equation}
\label{eq:actformula}
	\mathrm{act}^\ast(u) = 1\times u - [S^1]^\ast \times \Delta(u)
\end{equation}
for all $u\in H^\ast(LBG)$ where 
\[
	\mathrm{act} \colon S^1 \times LBG \longto LBG
\]
is the rotation action of $S^1$ on $LBG$
and $[S^1]^\ast \in H^1(S^1)$ is the dual of the 
fundamental class $[S^1] \in H_1(S^1)$.
By \cite[Thm.~3.1]{KM19}, $H^\ast(LBG)$ is then a free
$H^\ast(BG)$--module with basis given by the elements
\[
	y_I =\prod_{i\in I} y_i, \qquad I\subset\{1,\ldots,n\},
\]
where $y_i = \Delta \ev_1^\ast(x_i)\in H^\ast(LBG)$.
\begin{lemma}
\label{lm:sastyi}
Let $s\colon BG \to LBG$
be the section of $\ev_1\colon LBG \to BG$ given by 
constant loops. 
Then for $I\subset\{1,\ldots,n\}$
\[
	s^\ast(y_I) 
	= 
    \begin{cases}
   	1 & \text{if $I=\emptyset$} \\	
	0 & \text{otherwise}
	\end{cases}
\]
\end{lemma}
\begin{proof}
It suffices to show that $s^\ast$ vanishes on the image of $\Delta$,
which in turn follows from 
formula~\eqref{eq:actformula} and
the commutativity of the diagram
\[\xymatrix{
	S^1\times BG 
	\ar[r]^-{\pr}
	\ar[d]_{\id\times s}
	&
	BG
	\ar[d]^{s}
	\\
	S^1\times LBG
	\ar[r]^-{\mathrm{act}}
	&
	LBG
}\]
where  $\pr \colon S^1\times BG \to BG$ is the projection map.
\end{proof}

The following lemma is a reformulation of \cite[Thm.~2.2]{KM19}
in terms of the $(H^\ast(LBG),\cup)$--module
structure on $\bbH^\ast(LBG)$ defined in Remark~\ref{rk:middle}.
\begin{lemma}%
\label{lm:bulletderivation}
Suppose $y \in H^\ast(LBG)$ is of the form $y = \Delta \ev_1^\ast (x)$
for some $x\in H^\ast(BG)$. Then 
\begin{equation}
\label{eq:bulletderivation}
	y(u\stringprod v) = (yu)\stringprod v + (-1)^{\deg(y)\deg(u)} u \stringprod (yv)
\end{equation}
for all $u,v \in \bbH^\ast(LBG)$. 
\end{lemma}
\begin{proof}
Following \cite[Proof of Thm.~4.2 (4-5)]{tamanoi2009cap},
we have the equation
\[
	\mathrm{concat}^\ast (y)
	= 
	\mathrm{split}^\ast (y \times 1 + 1 \times y) 
\]
from which the claim follows by applying formula~\eqref{eq:bettermiddleformula}.
\end{proof}

Finally, since under our assumption that 
$H^\ast(BG) \isom \F_\ell[x_1,\ldots,x_n]$ 
the cohomology  $H^\ast(LBG)$ 
is free as an $H^\ast(BG)$--module,
the Eilenberg--Moore spectral sequence
implies the following computation of the cohomology ring
$H^\ast(LBG\times_{BG} BG^{h\sigma})$
and the composite of $\mathrm{split}^\ast$ and $\times$.
\begin{lemma}
\label{lm:fpcohcomp}
The composite
\[
	H^{\ast} (LBG) \tensor H^{\ast} (BG^{h\sigma})
	\xto{\ \times\ }
	H^{\ast} (LBG \times BG^{h\sigma})
	\xto{\ \mathrm{split}^\ast\ }
	H^{\ast} (LBG \times_{BG} BG^{h\sigma})
\]
induces an isomorphism of $H^\ast(BG)$--algebras and  
$\calA_\ell$--modules
\[
    H^\ast(LBG) \tensor_{H^\ast(BG)} H^\ast(BG^{h\sigma})
    \xto{\ \isom\ }
    H^\ast(LBG \times_{BG} BG^{h\sigma})
\]
where the target is equipped with the $H^\ast(BG)$--algebra 
structure induced by the map $LBG \times_{BG} BG^{h\sigma} \to BG$.
\qed
\end{lemma}

We are now ready to prove Theorem~\ref{thm:realization}.

\begin{proof}[Proof of Theorem~\ref{thm:realization}]
Applying $\theta$ to the $H^\ast(BG)$--module basis 
$(y_I)_{I\subset\{1,\ldots,n\}}$ of $H^\ast(LBG)$,
we obtain an $H^\ast(BG)$--module basis 
$(\theta(y_I))_{I\subset\{1,\ldots,n\}}$ for $H^\ast(BG^{h\sigma})$.
Let $V \subset H^\ast(BG^{h\sigma})$ be the graded $\F_\ell$--vector space
spanned by the elements $\theta(y_I)$, $I\subset \{1,\ldots,n\}$.
Identify 
$H^\ast(LBG\times BG^{h\sigma})$ 
with 
$H^\ast(LBG)\tensor H^\ast(BG^{h\sigma})$
via the Künneth isomorphism 
and 
$H^\ast(LBG\times_{BG} BG^{h\sigma})$ 
with
$H^\ast(LBG)\tensor_{H^\ast(BG)} H^\ast(BG^{h\sigma})$
via the isomorphism of Lemma~\ref{lm:fpcohcomp}.
The composite
\[
	t
	\colon
	H^\ast(LBG) \tensor_{H^\ast(BG)} H^\ast(BG^{h\sigma})
	\xto{\ \isom\ }
	H^\ast(LBG) \tensor V
	\longto
	H^\ast(LBG) \tensor H^\ast(BG^{h\sigma})
\]
of the evident isomorphism with the map induced by the inclusion 
of $V$ into $H^\ast(BG^{h\sigma})$ then provides a section for 
the map $\mathrm{split}^\ast$, and writing $f$ for the map
\[
	f \colon H^\ast(LBG) \tensor H^\ast(BG^{h\sigma})
	\longto
	H^\ast(BG^{h\sigma}),
	\qquad
	u \tensor v 
	\longmapsto
	u \tildestringprod (v\cupprod s^d\theta(\bbOne)),
\]
Proposition~\ref{prop:middle}  implies that 
\begin{equation}
\label{eq:acupformula}
	A \cupprod (1 \tildestringprod s^d \theta(\bbOne))
	=
	(-1)^{d\deg(A)}f t \,\mathrm{concat}^\ast(A)
\end{equation}
for all $A \in H^\ast(BG^{h\sigma})$. Here $\tildestringprod$ 
is as defined in Definition~\ref{def:bulletdef1}.

Let $\varepsilon \colon H^\ast(BG^{h\sigma}) \to H^\ast(BG)$
be the map corresponding to the map 
$s^\ast\colon H^\ast(LBG) \to H^\ast(BG)$
of Lemma~\ref{lm:sastyi} under the isomorphism $\theta$.
Then by  Lemma~\ref{lm:sastyi}, we have
\[
	\varepsilon (\theta(y_I)) 
	=	
	\begin{cases}
   	1 & \text{if $I=\emptyset$} \\	
	0 & \text{otherwise}.
	\end{cases}
\] 
Moreover, taking $u=v=\bbOne$ in 
Lemma~\ref{lm:bulletderivation}, we see that 
$y_i \cupprod s^d(\bbOne) = 0$ for all $i=1,\ldots,n$, so that 
\[
	y_I \cupprod s^d(\bbOne)  
	= 
	\begin{cases}
   	s^d(\bbOne) & \text{if $I=\emptyset$} \\
	0 & \text{otherwise}.
	\end{cases}
\]
Noticing that it is enough to consider elements of the form 
\[
	u\tensor_{H^\ast(BG)}\theta(y_I) 
	\in 
	H^\ast(LBG)\tensor_{H^\ast(BG)} H^\ast(BG^{h\sigma}),
\]
it is now straightforward to check that the square
\[\xymatrix{
	H^\ast(LBG) \tensor_{H^\ast(BG)} H^\ast(BG^{h\sigma})
	\ar[r]^-t
	\ar[d]_{\id\tensor_{H^\ast(BG)} \varepsilon}
	&
	H^\ast(LBG) \tensor H^\ast(BG^{h\sigma})
	\ar[d]^f
	\\
	H^\ast(LBG)
	\ar[r]^{\tildestringprod s^d(\theta(\bbOne))}
	&
	 H^\ast(BG^{h\sigma})
}\]
commutes. Combining the commutativity of this square
with \eqref{eq:acupformula} 
and writing 
$\tilde{\varepsilon} = \id \tensor_{H^\ast(BG)} \varepsilon$
for brevity,
we see that 
\begin{align*}
	A \cupprod (1 \tildestringprod s^d \theta(\bbOne))
	&=
	(-1)^{d\deg(A)} 
	\tilde{\varepsilon} \,\mathrm{concat}^\ast(A)
	\tildestringprod 
	s^d(\theta(\bbOne))
\end{align*}
for all $A \in H^\ast(BG^{h\sigma})$,
or what in view of the equation
\eqref{eq:bulletandtildebullet}
is the same,
\begin{align}
\label{eq:acupformula2}
	A \cupprod s^d (s^{-d}(1) \stringprod \theta(\bbOne))
	&=
	s^d \big( 
		s^{-d} \tilde{\varepsilon} \,\mathrm{concat}^\ast(A)
		\stringprod 
		\theta(\bbOne)
	\big) 
\end{align}
for all $A \in H^\ast(BG^{h\sigma})$.

 By the Eilenberg--Moore spectral sequence, the map 
$H^\ast(BG^{h\sigma}) \to H^\ast(F)$ induced by the inclusion 
of the fibre $F$ of the fibration 
$\ev_1 \colon BG^{h\sigma} \to BG$ into $BG^{h\sigma}$
identifies with the map 
$H^\ast(BG^{h\sigma}) \to H^\ast(BG^{h\sigma}) \tensor_{H^\ast(BG)} \F_\ell$, and similarly for the map $H^\ast(LBG) \to H^\ast(\loops BG)$
induced  by the inclusion $\loops BG \incl LBG$. Since 
the  map   $H^\ast(LBG) \to H^\ast(\loops BG)$
sends  $s^d(\bbOne) \in H^d(LBG)$ to a
nontrivial element, the  map
$H^\ast(BG^{h\sigma}) \to H^\ast(F)$ 
does the same to the element
$s^d \theta(\bbOne) \in H^d(BG^{h\sigma})$. 
By Theorem~\ref{thm:conj-red3}, it follows that the map 
\[
	\bbH^\ast(LBG) 
	\xto{\ \stringprod \theta(\bbOne)\ } 
	\bbH^\ast(BG^{h\sigma})
\]
is an isomorphism of graded vector spaces. Consequently,
we can find a scalar $c\in\F_\ell^\times$ such that
$s^{-d}(1) \stringprod c\theta(\bbOne) 
= 
s^{-d}(1) \in \bbH^{-d}(BG^{h\sigma})$.
From \eqref{eq:acupformula2} we can now deduce that the composite
\begin{equation}
\label{eq:compositeid}
	H^\ast(BG^{h\sigma})
	\xto{\ \tilde{\varepsilon} \,\mathrm{concat}^\ast\ }
	H^\ast(LBG)
	\xto{\ s^d  \, (\stringprod c\theta(\bbOne))\,   s^{-d}\ }
	H^\ast(BG^{h\sigma})
\end{equation}
is the identity map.
As the first map is a map of $H^\ast(BG)$--algebras 
and the second map is a bijection, it follows that
the second map is an isomorphism of $H^\ast(BG)$--algebras.
Finally, if $\theta$ is a map of $\tilde{\calA}$--modules
for some Hopf subalgebra $\tilde{\calA}\subset \calA_\ell$,
then so is the map $\varepsilon$ and hence the map
$\tilde{\varepsilon}$ and consequently the first map in 
\eqref{eq:compositeid}. As before, it follows that the 
second map is a map of $\tilde{\calA}$--modules,
as desired.
\end{proof}

\subsection{Ring isomorphisms when \texorpdfstring{$H^*(BG)$}{H*(BG)} 
is polynomial in even degrees}
\label{subsec:polyinevendegrees}

The aim of this subsection is to prove the following theorem
showing that when $H^\ast(BG)$ is a polynomial algebra
concentrated in even degrees and $\sigma$ induces the identity map on 
$H^\ast(BG;\Z/\ell)$ (on $H^\ast(BG;\Z/4)$ when $\ell=2$),
there exists a unique 
and easily identifiable class $x\in H^\ast(BG^{h\sigma})$
for which the composite \eqref{eq:bulletx} is an isomorphism of rings 
and $\calA'$--modules. We remind the reader that 
$\calA'$ denotes the subalgebra of the mod $\ell$
Steenrod algebra $\calA_\ell$ generated by the
Steenrod reduced $\ell$--th power operations when $\ell$ is odd and 
all of $\calA_2$ when $\ell=2$.

\begin{thrm}
\label{thm:polytezukaelaboration} Let $BG$ be a semisimple $\ell$--compact group 
of dimension $d$, and 
suppose $H^\ast(BG) = \F_\ell[x_1,\ldots,x_n]$ is a polynomial algebra
concentrated in even degrees and
$\sigma\colon BG \to BG$ is a map inducing the identity map on $H^\ast(BG;\Z/\ell)$
(on $H^\ast(BG;\Z/4)$ when $\ell= 2$).
Then 
\begin{enumerate}[$\quad$(1)]
\item \label{it:existsuniquex}
    There exists a unique class $x \in H^d (BG^{h\sigma})$ 
    such that the composite
    \begin{equation}
    \label{eq:bulletx}
    \xymatrix@C-0.3em{
    	H^\ast(LBG)
    	\ar[r]^-{s^{-d}}
    	&
    	\bbH^\ast(LBG)
    	\ar[rr]^-{\stringprod s^{-d}(x)}
    	&&
    	\bbH^\ast(BG^{h\sigma})
    	\ar[r]^-{s^d}
    	&
    	H^\ast(BG^{h\sigma})
    }
    \end{equation}
    is an isomorphism of rings 
    (and hence of $H^\ast(BG)$--algebras, by Proposition~\ref{prop:cupmodstrbilin}).
\item \label{it:steenrodopcompat}
	For this $x$, the composite \eqref{eq:bulletx}
	commutes with the action of $\calA'$. Moreover, when $\ell$ is odd,
	\eqref{eq:bulletx}
	commutes with the action of all of $\calA_\ell$ as long as 
	$\sigma$ induces the identity map on $H^\ast(BG;\Z/\ell^2)$
	for all $i=1,\ldots,n$.
\item \label{it:extalgs}
	As $H^\ast(BG)$--algebras, $H^\ast(LBG)$ and hence $H^\ast(BG^{h\sigma})$
	are isomorphic to exterior algebras:
	\[
		H^\ast(LBG) \isom H^\ast(BG^{h\sigma}) \isom \Lambda_{H^\ast(BG)}(y_1,\ldots y_n)
	\]
	where $\deg(y_i) = \deg(x_i)-1$ for all $i=1,\ldots,n$.
\item \label{it:formulaforx}
	For any classes $y_1,\ldots, y_n \in H^\ast(BG^{h\sigma})$
	such that 
	\[
		H^\ast(BG^{h\sigma}) = \Lambda_{H^\ast(BG)}(y_1,\ldots,y_n),
	\]
	the class $x$ in part (\ref{it:existsuniquex})
	is given by the product
	\[
		x = c y_1 \cdots y_n \in H^d(BG^{h\sigma})
	\]
	where $c\in \F_\ell^\times$ is a scalar 
	determined by the requirement that \eqref{eq:bulletx}
	sends $1\in H^0(LBG)$ to $1\in H^0(BG^{h\sigma})$.
\end{enumerate}
\end{thrm}

We now embark on preparations for proving  Theorem~\ref{thm:polytezukaelaboration}.
For any $BG$, self map $\sigma \colon BG \to BG$, and commutative ring $R$,
we have a long exact sequence
\begin{equation}
\label{eq:mappingtorusles} 
	\cdots
	\longto
	H^{\ast-1} (BG; R)
	\longto 
	H^\ast(BG_{h\sigma}; R)
	\xto{\ \iota^\ast\ }
	H^{\ast} (BG; R) 
	\xto{\ \sigma^\ast - \id\ }
	H^{\ast} (BG; R)
	\longto
	\cdots
\end{equation}
where $BG_{h\sigma}$ is the homomotopy orbit space
(or mapping torus) of $\sigma$ and 
$\iota \colon BG \to BG_{h\sigma}$
is the projection;
see e.g.\ \cite[(9)]{KK10}.
When $\sigma$ induces the identity map on 
$H^\ast(BG)$, we therefore have a short exact sequence
\begin{equation}
\label{eq:mappingtorusses}
	0 
	\longto 
	H^{\ast-1} (BG) 
	\longto 
	H^\ast(BG_{h\sigma}) 
	\xto{\ \iota^\ast\ }
	H^{\ast} (BG) 
	\longto
	0.
\end{equation}
In the proof of 
Theorem~\ref{thm:polytezukaelaboration},
we will make use of the following lemma.

\begin{lemma}
\label{lm:structurediso}
Suppose $H^\ast(BG) = \F_\ell[x_1,\ldots,x_n]$ is a polynomial algebra
(not necessarily concentrated in even degrees) and
$\sigma\colon BG \to BG$ is a map inducing the identity map on $H^\ast(BG)$.
Suppose $\tilde{\calA} \subset \calA_\ell$ is a Hopf subalgebra
such that the map $\iota^\ast$ in \eqref{eq:mappingtorusses}
admits a section $\alpha$ which is a ring homomorphism and
commutes with the $\tilde{\calA}$--action.
When $\ell = 2$, suppose moreover that 
\[
	\Sq^{\deg(x_i)-1}\alpha(x_i) = \alpha(\Sq^{\deg(x_i)-1} x_i)
\]
for all $i=1,\ldots,n$.
Then there exists a ring isomorphism
$H^\ast(LBG) \to H^\ast(BG^{h\sigma})$ 
which commutes with the action of $\tilde{\calA}$. 
\end{lemma}
\begin{proof}
Proposition~3 of \cite{KK10}
exhibits a simple system of generators for $H^\ast(LBG)$
as an $H^\ast(BG)$--algebra. In the case where $\ell$ is odd, the squares
of these generators vanish for degree reasons, while in the case $\ell=2$
the squares of these generators are determined by 
Proposition~2(3),(4) of \cite{KK10} from the elements 
$\Sq^{\deg(x_i)-1}(x_i) \in H^\ast(BG)$. Moreover, for all $\ell$, the action of
$\calA_\ell$ on the said generators is determined by Proposition~2(3),(4) of \cite{KK10}.
Thus Propositions~2 and 3 of \cite{KK10} yield a description 
of the ring structure and $\calA_\ell$--action
on $H^\ast(LBG)$ in terms of those on $H^\ast(BG)$.
In the same way, Propositions~14 and 15 of \cite{KK10}
do the same for the ring structure and $\tilde{\calA}$--module
action on $H^\ast(BG^{h\sigma})$.
The claim now follows by comparing the results of these computations.
\end{proof}

\begin{rem}
Lemma~\ref{lm:structurediso} can be viewed as a strengthened version 
\cite[Cor.~16]{KK10}, and it also corrects a subtle mistake afflicting the 
statement of \cite[Cor.~16]{KK10} in the case $\ell=2$
stemming from the fact that the even-degree subalgebra of
$\calA_2$ is not a Hopf subalgebra of $\calA_2$.
A counterexample to the 
statement of \cite[Cor.~16]{KK10} in the case $\ell=2$
is provided by Example~\ref{ex:slnfq}.
\end{rem}

\begin{proof}[Proof of Theorem~\ref{thm:polytezukaelaboration}]
The map $\sigma$ induces the identity map on $H^\ast(BG)$;
when $\ell$ is odd, this is our assumption on $\sigma$, while for $\ell=2$ this
follows by an argument involving the universal coefficient theorem
from the assumption that $\sigma$ induces the identity map 
on $H^\ast(BG;\Z/4)$ and the assumption that $H^\ast(BG)$
is concentrated in even degrees. Thus we have the short
exact sequence \eqref{eq:mappingtorusses}.
The assumption that $H^\ast(BG)$ 
is concentrated in even degrees
implies that the map $\iota^\ast$ in
\eqref{eq:mappingtorusses} has a unique 
section $\alpha$. This section 
is a ring homomorphism commuting
with all even-degree Steenrod operations.
Moreover, when $\sigma$ induces the identity map on 
$H^\ast(BG;\Z/\ell^2)$,
chasing the diagram
\[\xymatrix@C+1em@R-1ex{
	\cdots 
	\ar[r]
	&
	H^\ast(BG_{h\sigma};\Z/\ell^2)
	\ar[r]^-{\iota^\ast}
	\ar[d]
	&
	H^\ast(BG;\Z/\ell^2)
	\ar[r]^{\sigma^\ast - \id}
	\ar[d]
	&
	H^\ast(BG;\Z/\ell^2)
	\ar[r]^{}
	&
	\cdots
	\\
	\cdots 
	\ar[r]
	&
	H^\ast(BG_{h\sigma};\Z/\ell)
	\ar[r]^-{\iota^\ast}
	\ar[d]_\beta
	&
	H^\ast(BG;\Z/\ell)
	\ar[r]^{\sigma^\ast - \id}
	\ar[d]_\beta
	&
	H^\ast(BG;\Z/\ell)
	\ar[r]^{}
	&
	\cdots
	\\
	&
	H^{\ast+1}(BG_{h\sigma};\Z/\ell)
	&
	H^{\ast+1}(BG;\Z/\ell)
}\]
with exact rows and columns where the rows are
instances of \eqref{eq:mappingtorusles}
shows that $\beta\alpha(x_i) = 0$ for all $i=1,\ldots,n$.
Here $\beta$ denotes the Bockstein operation.
Consequently, under this assumption, the section $\alpha$
commutes with all Steenrod operations.
The existence of a class $x$
with the properties asserted in 
claims~(\ref{it:existsuniquex}) and (\ref{it:steenrodopcompat})
now follows from 
Lemma~\ref{lm:structurediso}
and 
Theorem~\ref{thm:realization}.

Claim~(\ref{it:extalgs})
follows from 
\cite[Prop.~3]{KK10}
and the existence part of claim~(\ref{it:existsuniquex}) 
already proven.
To prove claim~(\ref{it:formulaforx})
and the uniqueness part of claim~(\ref{it:existsuniquex}),
assume $x\in H^d(BG^{h\sigma})$
is such that \eqref{eq:bulletx}
is an isomorphism of rings.
By \cite[Proof of Cor.~4.2]{KM19},
the element
$s^d(\bbOne) \in H^d(LBG)$ is given by
a product
\[
	s^d(\bbOne) = y_1\cdots y_n \in H^d(LBG)
\]
where $y_1,\ldots,y_n \in H^\ast(LBG)$ are 
certain elements such that 
\[
	H^\ast(LBG) 
	= 
	\Lambda_{H^\ast(BG)}(y_1,\ldots,y_n)
	=
	H^\ast(BG) \tensor \Lambda(y_1,\ldots,y_n).
\]
Consequently, the element $s^d(\bbOne)$ can be characterized
uniquely up to a nonzero scalar multiple
in terms of the ring structure on $H^\ast(LBG)$ as the 
generator of the annihilator ideal of the nilradical
\[
	\sqrt{H^\ast(LBG)} = (y_1,\ldots,y_n) \subset H^\ast(LBG).
\]
It follows that $x$, the image of $s^d(\bbOne)$ under
\eqref{eq:bulletx}, is similarly characterized
in terms of the ring structure of $H^\ast(BG^{h\sigma})$. The 
claimed uniqueness and formula for $x$ follow.
\end{proof}

\subsection{Polynomial rings with generators also in odd degrees and proof of Theorem~\ref{thm:polyexamples}}
\label{subsec:polysimplyconn2cptgrps}
For odd $\ell$, any polynomial cohomology ring is necessarily concentrated
in even degrees, so  Theorem~\ref{thm:polytezukaelaboration} 
in particular covers all $\ell$--compact groups with polynomial cohomology 
for $\ell$ odd. In this subsection, our focus is on $2$--compact groups
with polynomial cohomology rings, not necessarily in even degrees. 
This analysis, combined with the previous subsection, will allow us to 
prove Theorem~\ref{thm:polyexamples}.

\begin{thrm}
\label{thm:polysimplyconn2cptgrps}
Suppose $\ell=2$ and 
$BG$ is a simply connected $2$--compact group
with polynomial mod $2$ cohomology.
Then for all $q\in \Z_2^\times$ there exists
an element $x \in H^d(BG(q))$
such that 
the composite 
\begin{equation}
	H^\ast(LBG)
	\xto{\ s^{-d}\ }
	\bbH^\ast(LBG)
	\xto{\ \stringprod s^{-d} (x)\ }
	\bbH^\ast(BG(q))	
	\xto{\ s^{d}\ }
	H^\ast(BG(q))
\end{equation}
is an isomorphism of algebras
over the Steenrod algebra,
except when $BG$ contains a $B\SU(n)\twocom$--summand
for some $n\geq 3$,
in which case the same holds for all 
$q \in \Z_2^\times$ satisfying
$q \equiv 1$ modulo~$4$.
\end{thrm}

Example~\ref{ex:slnfq} demonstrates that
the restriction on $B\SU(n)\twocom$--summands in 
Theorem~\ref{thm:polysimplyconn2cptgrps}
and the extra assumption in 
Theorem~\ref{thm:polytezukaelaboration}
on the behavior of $\sigma$ on mod $4$ 
cohomology in the case $\ell=2$
are necessary.

The following proposition is the key point in establishing Theorem~\ref{thm:polysimplyconn2cptgrps}.

\begin{prop}
\label{prop:mod2gen}
Suppose $\ell = 2$. 
Let $BG$ be a simple simply connected $2$--compact
group with polynomial mod $2$ cohomology.
Then for all 
$q\in \Z_2^\times$,
there exists an element $x \in H^d(BG(q))$
such that 
the composite 
\begin{equation}
	H^\ast(LBG)
	\xto{\ s^{-d}\ }
	\bbH^\ast(LBG)
	\xto{\ \stringprod s^{-d} (x)\ }
	\bbH^\ast(BG(q))	
	\xto{\ s^{d}\ }
	H^\ast(BG(q))
\end{equation}
is an isomorphism of algebras
over the Steenrod algebra,
except when 
$BG \homot BSU(n)\twocom$
for some $n\geq 3$,
in which case the same holds for all 
$q \in \Z_2^\times$ satisfying
$q \equiv 1$ modulo $4$.
\end{prop}

\begin{proof}%
Theorems~1.1, 1.4 and 1.5 of \cite{AG09} 
show that any simple simply connected $2$--compact group
with polynomial mod $2$ cohomology is homotopy equivalent
to one of the following:
$BSU(n)\twocom$ for some $n \geq 3$; 
$B\Sp(n)\twocom$ for some $n \geq 1$;
$B\Spin(n)\twocom$ for $n=7,8$, or $9$;
$(BG_2)\twocom$;
$(BF_4)\twocom$;
or 
$BDI(4)$.
In the cases $BG \homot BSU(n)\twocom$
and $BG \homot B\Sp(n)\twocom$
the claim follows from Theorem~\ref{thm:polytezukaelaboration}
together with Proposition~\ref{prop:idmappoly}.
Assume that we are in one of the remaining cases.
In view of Theorem~\ref{thm:realization}, 
it is sufficient to show that there exists 
an abstract isomorphism $H^\ast(LBG) \isom H^\ast(BG(q))$ of 
$H^\ast(BG)$--algebras commuting with the Steenrod
algebra action. 
By Theorem~\ref{thm:closedsubgrphelp}, the cohomology $H^\ast(BG(q))$
as an $H^\ast(BG)$--algebra only depends on the closed subgroup
of $\Out(BG)$ generated by $[\psi^q] \in \Out(BG)$. Since
by Corollary~\ref{cor:ladicunitstooutbgcont}
the homomorphism $\Z_2^\times \to \Out(BG)$, $q \mapsto [\psi^q]$,
is continuous, it follows that 
$H^\ast(BG(q))$ as an $H^\ast(BG)$--algebra depends only on the closed
subgroup of $\Z_2^\times$ generated by $q$. As every closed topologically 
cyclic subgroup of $\Z_2^\times$ is generated by a suitable power of a prime number,
we may assume that $q\in\Z_2^\times$ is a prime power.
The cases $BG = B\Spin(n)\twocom$, $n=7,8,9$;
$BG = (BG_2)\twocom$;
$BG = (BF_4)\twocom$;
and 
$BG = BDI(4)$
now follow from isomorphisms $H^\ast(LBG) \isom H^\ast(BG(q))$
proven by Kishimoto and Kono \cite[Thm.~20]{KK10}
and Kaji \cite[Thm.~1.1]{KajiMod2total};
an inspection of the proofs 
of the cited theorems 
reveals that they do provide isomorphisms of 
$H^\ast(BG)$--algebras.
\end{proof}

\begin{proof}[Proof of Theorem~\ref{thm:polysimplyconn2cptgrps}]
 Proposition~\ref{prop:mod2gen} establishes the claim when $BG$ is
 furthermore assumed simple. 
In the general case, by \cite[Thm.~1.2 and Prop.~8.12]{AG09},
$BG$ splits as a product of simple simply connected $2$--compact
groups. Thus the claim follows
from the simple case and  Proposition~\ref{prop:pairingforproduct}.
\end{proof}

\begin{proof}[Proof of Theorem~\ref{thm:polyexamples}]
We need to show that under the assumptions set out in the
theorem, it is possible to find an element 
$x \in H^\ast(\BtGq)$ such that the composite 
\begin{equation}
	H^\ast(LBG^{h\muet})
	\xto{\ s^{-d}\ }
	\bbH^\ast(LBG^{h\muet})
	\xto{\ \stringprod s^{-d} (x)\ }
	\bbH^\ast(\BtGq)	
	\xto{\ s^{d}\ }
	H^\ast(\BtGq)
\end{equation}
is an isomorphism of $H^\ast(BG)$--algebras and of $\calA'$--modules.
Consider first the case where $\ell$ is odd.
By Theorem~\ref{thm:poly-cohom}(\ref{poly-torsion-free}) and (\ref{polyodd}),
under the assumptions made, 
the cohomology $H^\ast(BG)$ is a polynomial ring concentrated in even degrees.
By Proposition~\ref{prop:ghksplit}, it follows that 
$H^\ast(BG^{h\muet})$ is also a polynomial ring concentrated in even degrees.
Moreover, by Proposition~\ref{prop:idmappoly}, 
the map $\psi^{q'} \colon BG^{h\muet} \to BG^{h\muet}$
induces the identity map on cohomology. Thus the claim for $\ell$ odd 
follows from Theorem~\ref{thm:polytezukaelaboration}. 
It remains to prove the claim when $\ell=2$. 
Propositions~\ref{prop:ghksplit} and \ref{prop:bghtfornonexceptional2cptgrps}
and Theorem~\ref{thm:poly-cohom}(\ref{poly2-pi1torsionfree})
ensure that $BG^{h\muet}$ in this case 
is a simply connected $2$--compact group with 
polynomial cohomology, so the claim follows from
Theorem~\ref{thm:polysimplyconn2cptgrps}. 
\end{proof}

\begin{rem}\label{rem:uniqueclass}
Unlike in Theorem~\ref{thm:polytezukaelaboration},
the element $x$ in Proposition~\ref{prop:mod2gen} 
is not always unique, although that frequently happens to be the case.
We already know from 
Theorem~\ref{thm:polytezukaelaboration}
that $x$ is unique for $BG \homot B\Sp(n)\twocom$, $n\geq 1$
(for any $q\in\Z_2^\times$) and for $BG \homot BSU(n)\twocom$, $n\geq 3$
(for any $q\in \Z_2^\times$ satisfying $q\equiv 1$ mod $4$).
Moreover, computer computations show that 
for the remaining simple simply connected $2$--compact
groups with polynomial cohomology,
the class $x$ is also unique
(for all $q\in \Z_2^\times$)
except for $BG \homot B\Spin(7)\twocom$,
$BG \homot B\Spin(8)\twocom$
and $BG \homot B\Spin(9)\twocom$, in which cases
there are precisely $2$, $4$, and $2$ possible choices for $x$,
respectively (for all $q\in \Z_2^\times$).
In general, we expect the element $x$ to be unique up to 
an automorphism of $BG(q)$, although we have not 
verified this. See Question~\ref{qu:thequestion}(\ref{q3}).
\end{rem}

The following example shows that the 
restriction $q \equiv 1$ mod $4$ in 
Proposition~\ref{prop:mod2gen}
in the case $BG = BSU(n)\twocom$
cannot be removed. (It also demonstrates that 
\cite[Cor.~ 16 and Thm.~18]{KK10} are not quite correct as stated for
the prime $2$.)

\begin{example}
\label{ex:slnfq}
Suppose $n\geq 2$ and let $q$ be an odd prime power.
It follows from Quillen's computation of the cohomology of 
general linear groups over finite fields \cite{quillen72} that
\[
	H^\ast(SL_n(\F_q);\F_2)
	\isom
	\begin{cases}
    \F_2[c_2,\ldots,c_n] \tensor \Lambda(e_2,\ldots,e_n)
    &
    \text{if $q\equiv 1$ mod $4$} 
    \\
    \F_2[c_2,\ldots,c_n,e_2,\ldots,e_n] / I	
    &
    \text{if $q \equiv 3$ mod $4$}
    \end{cases}
\]
where $\deg(c_j) = 2j$, $\deg(e_j) = 2j-1$, and 
$I$ is the ideal generated by the elements
\[
	e_j^2 + \sum_{0 \leq a < j} c_a c_{2j-1-a},
	\qquad
	j = 2,\ldots,n
\]
where we interpret $c_0 = 1$, $c_1 = 0$ and $c_j= 0$ for $j>n$.
By Propositions~\ref{prop:polycollapse} and \ref{prop:idmappoly}, 
the $\bbH^\ast(LBSU(n);\F_2)$--module
structure on $H^\ast(SL_n(\F_q);\F_2)$ is free of rank $1$
for any odd prime power $q$, and when $q \equiv 1$ mod $4$,
$H^\ast(SL_n(\F_q);\F_2)$ is isomorphic as an algebra
over the Steenrod algebra to $H^\ast(LBSU(n);\F_2)$
by Proposition~\ref{prop:mod2gen}.
On the other hand, when $q\equiv 3$ mod $4$,
for $n\geq 3$, 
$H^\ast(LBSU(n);\F_2)$ and $H^\ast(SL_n(\F_q);\F_2)$
are  isomorphic neither as rings nor as modules over $\calA_2$
or its even-degree subalgebra, as follows e.g.\ by considering 
the squaring map $\Sq^8$ on $H^8$.

This example in particular demonstrates that it is possible for 
the $\bbH^\ast(LBG)$--module structure on $H^\ast(BG^{h\sigma})$
to be free of rank $1$ without
$H^\ast(LBG)$ and $H^\ast(BG^{h\sigma})$
being isomorphic as rings even when $BG$ is a simple and simply 
connected $\ell$--compact group.
\end{example}

\appendix

\section{\texorpdfstring{$\Z_\ell$}{Z\_l}--root data and untwisting of finite groups of Lie type}
\label{app:pcg}

As explained in the introduction, the $\ell$--local structure of finite
groups of Lie type can be understood in terms of
$\ell$--compact groups and their homotopy fixed points via
the homotopy equivalences of equations \eqref{lie} and \eqref{lietype}.
In this appendix, we describe the background
on $\ell$--compact groups, their homotopy fixed points, and root data 
necessary for this,
and in the process provide a slightly stronger and more precise versions of
some results in the literature.

The appendix is structured as follows: 
We first recall basic facts about homotopy fixed points in
Section~\ref{subsec:homotopyfix}, followed by facts about $\ell$--compact groups in
Section~\ref{subsec:pcg-recol}. In Section~\ref{subsec:unadams}, we explain how
unstable Adams operations act on polynomial cohomology rings.
In Section~\ref{subapp:homotopy-act}, we prove some general results
about actions on $\ell$--compact groups and their maximal tori. 
In Section~\ref{subapp:rootdata},
we then describe the $\ell$--compact group
$BG^{h\langle \tau\rangle}$
for $\tau \in \Out(\bbD_G)$ of finite order prime to $\ell$,
summarized in Theorem~\ref{thm:root-datum}.
In Section~\ref{subapp:untwisting}, we prove the
untwisting theorem, Theorem~\ref{thm:untwisting-intro}, as a consequence of the more general Theorem~\ref{thm:class-lcglie}
classifying $\ell$--local finite groups of Lie type.
Finally, in Section~\ref{subsec:topology}, we end the appendix with a discussion
of the topology on $\Out(BG)$ used in Section~\ref{sec:funclass2}.

Throughout the appendix, we will work in the generality of connected
$\ell$--compact groups. Contrary to our usual convention of assuming that 
$BG$ is semisimple, \emph{in this appendix, except where indicated otherwise, 
$BG$ will therefore denote a fixed \emph{connected} $\ell$--compact group 
of dimension $d$.}

\subsection{Homotopy fixed-points}\label{subsec:homotopyfix}
For any self-map $\sigma\co X \to X$, we defined in the introduction the homotopy
fixed-points of $\sigma$ to be the space 
\begin{equation}
\label{eq:xhsigmadef}
X^{h\sigma} = \{ \alpha\co I \to X \mid \alpha(0) =
\sigma(\alpha(1))\} 
\end{equation}
i.e., the homotopy equalizer of $\sigma\co X \to X$ and the identity (see \cite{vogt73} and \cite[Ch.~XI\S8]{bk}).
The intuition is that instead of requiring $x$ and $\sigma(x)$
to be equal, we look at all paths between them.
The following
proposition summarizes some basic properties of this construction.

\newcounter{myCounter}
\begin{prop}\label{prop:fixed-point}
For any self-map $\sigma\co X \to X$, the homotopy fixed point space $X^{h\sigma}$
of \eqref{eq:xhsigmadef} is
\begin{enumerate}[(a)]
\item \label{it:htpypb}
	homotopy equivalent to the homotopy pullback of the diagram 
	$\xymatrix@1{X \ar[r]^-{(\sigma,1)} & X\times X & \ar[l]_-{\Delta} X}$,
\item \label{it:n0equivariantmaps}
	homeomorphic to the space $\map_{\bbN_0}(\R_{\geq 0},X)$ of 
    $\bbN_0$--equivariant maps, where $\bbN_0$ is a monoid under addition and 
    $1\in \bbN_0$ acts by addition of $1$ on $\R_{\geq 0}$ and by $\sigma$ on $X$, and
\item \label{it:n0fixedpoints} 
	homotopy equivalent to the homotopy fixed point space
	 $X^{h\bbN_0} = \holim_{\bbN_0} X$ of the action of $\bbN_0$ on $X$ given by $\sigma$.
\setcounter{myCounter}{\value{enumi}}
\end{enumerate}
In particular, the homotopy type of $X^{h\sigma}$ only depends on the free homotopy class of $\sigma$. Finally, when $\sigma$ is a homeomorphism, $X^{h\sigma}$ is furthermore
\begin{enumerate}[(a)]
\setcounter{enumi}{\value{myCounter}}
\item \label{it:zfixedpoints}
	homotopy equivalent to the homotopy fixed point space 
	$X^{h\Z}$ of the action of $\Z$ on $X$ given by $\sigma$.
\end{enumerate}
\end{prop}

\begin{proof} 
Part (\ref{it:htpypb}) follows by noting that $X^{h\sigma}$ agrees with the 
pullback of the diagram 
\[
	\xymatrix@C+2em{X \ar[r]^-{(\sigma,1)} & X\times X & \ar[l]_-{(\ev_0,\ev_1)} \map(I,X)}
\]
obtained by replacing the map $\Delta\co X \to X \times X$ in the diagram
of part (\ref{it:htpypb}) by the fibration $\map(I,X) \to X \times X$ 
given by evaluation at the endpoints of $I$.

For part (\ref{it:n0equivariantmaps}), note that the space $X^{h\sigma}$ 
is homeomorphic to $ \{ \alpha\co I \to X \, |\, \alpha(1) =
\sigma(\alpha(0))\}$ by flipping the interval, and that this latter space in turn
identifies with the space 
\[
    \{ \alpha\co \bbR_{\geq 0} \to X \,|\, \alpha(t+1) 
    = 
    \sigma(\alpha(t)) \mbox{ for all } t\geq 0 \} = \map_{\bbN_0}(\bbR_{\geq 0},X),
\] 
with $\bbN_0$ acting on $\bbR_{\geq 0}$ by addition and on $X$ by sending $1$ to
$\sigma$.

For part (\ref{it:n0fixedpoints}), we first recall that 
for any monoid $M$ acting on a space $X$, the
homotopy fixed-point space is defined as $X^{hM} = \holim_MX = \map_M(EM,X)$, where
$M$ is viewed as a category with one object, and $EM$ is the geometric
realization of the overcategory of the sole object in $M$ (i.e., the
category with objects $m \in M$ and morphism between
$m$ and $m'$ the $n \in M$ such that $n \cdot m = m'$), see e.g.\ \cite[Ch.~XI\S8]{bk}.
The space $E\bbN_0$ identifies with the infinite simplex
$\Delta^\infty$, which equivariantly deformation retracts on its spine
(subspace given by the path from $e_0$ to $e_1$ to $e_2$,...), using
that $\Delta^\infty$ is convex, and this spine again equivariantly
identifies with $\R_{\geq 0}$. Thus the claim follows from 
part~(\ref{it:n0equivariantmaps}).

That the homotopy type of $X^{h\sigma}$ only depends on the free homotopy 
class of $\sigma$ follows from part (\ref{it:htpypb}) (or part (\ref{it:n0fixedpoints}))
and the homotopy invariance of homotopy limits.

Finally, part (\ref{it:zfixedpoints}) follows from the fact that the space
$E\Z$ is $\Z$--equivariantly homotopy equivalent to $\R$: the said equivariant
homotopy equivalence induces a homotopy equivalence
$X^{h\Z} = \map_{\Z}(E\Z,X) \homot \map_{\Z}(\R,X)$
where the last space agrees with $X^{h\sigma}$ by an argument analogous to the 
one used to prove part (\ref{it:n0equivariantmaps}).
See also  \cite[\S4]{BrotoMoellerOliver}.
\end{proof}

\subsection{Basic recollections on \texorpdfstring{$\ell$}{l}--compact groups and
  \texorpdfstring{$\Z_\ell$}{Z\textunderscore l}--root data}\label{subsec:pcg-recol}
  In this subsection we recall the basics of $\ell$--compact groups and $\Z_\ell$--root data, elaborating on 
the introduction. We refer to the survey \cite{grodal10} as well as
\cite{DW94,AGMV08,AG09}, and their combined references, for much more information.
  
An $\ell$--compact group is a pointed connected space $BG$ 
which is $\F_\ell$--local, i.e., local with respect to homology with
coefficients in $\F_\ell$ \cite{bousfield75} or here equivalently Bousfield--Kan
$\ell$--complete \cite[\S11]{DW94}, and whose based loop space $G = \Omega BG$ has
finite mod $\ell$ cohomology. 
The $\ell$--compact group $BG$ is called \emph{connected} 
if $G$ is connected as a space, 
and \emph{semisimple} if furthermore $\pi_1(G)$ is finite.
The \emph{outer automorphism group} $\Out(BG)$ of $BG$ is by definition
the group of free homotopy class of self homotopy equivalences of $BG$
under composition.

It was shown in \cite{DW94} that any 
$\ell$--compact group has an essentially unique \emph{maximal torus}, which
is defined to be a map
$BT \to BG$, subject to an ``injectivity'' and ``maximality''
condition, where $BT$ is a space homotopy equivalent to the $\ell$--completion
of a finite product of copies of $BS^1 \homot \CP^\infty$.
The number of copies of $\CP^\infty$ involved is called the \emph{rank} of $BG$.

It turns out that the existence of maximal tori
enables the construction of an ``$\ell$--adic root datum $\bbD_G$ of
$BG$'', analogous the the standard construction for Lie groups, but with $\Z$
replaced by $\Z_\ell$.
Furthermore, by \cite[Thm.~1.2]{AG09},  connected $\ell$--compact groups $BG$ are classified by their
$\ell$--adic root data $\bbD_G$ and there exist an isomorphism
$\Out(BG) \isom \Out(\bbD_G)$.
To explain this correspondence, we first recall the definition of a $\Z_\ell$--root datum.
Just as with
$\Z$--root data, 
they can be defined in several
equivalent ways: In \cite[\S 8]{AG09} and \cite[Def.~1.1]{grodal10},
a $\Z_\ell$--root datum is viewed as a triple $(W,L,\{\Z_\ell
b_\sigma\})$ consisting of a $\Z_\ell$--lattice $L$, a finite
$\Z_\ell$--reflection group $W
\leq \Aut_{\Z_\ell}(L)$ acting on $L$, and a collection of
`coroots' $\Z_\ell b_\sigma$, given as rank one $\Z_\ell$--submodules
of $L$, subject to certain conditions. Equivalently, 
one can use a triple $(W,L,L_0)$ where $L_0 =
\sum_\sigma \Z_\ell b_\sigma$ is the \emph{coroot lattice}, again
subject to certain natural conditions, like in the case over $\Z$; see
\cite[\S1]{AGMV08} and \cite[\S1]{AG08auto}. We will use the latter
viewpoint below. An isomorphism of root data $\bbD = (W,L,L_0) \to
\bbD' = (W',L',L_0')$ is an isomorphism of $\Z_\ell$--lattices $L
\xrightarrow{\cong} L'$ sending $W$ isomorphically to $W'$ and $L_0$ to $L_0'$. In
particular, this defines the automorphism group $\Aut(\bbD)$ of $\bbD$, 
and we define the \emph{outer automorphism group} $\Out(\bbD)$ of $\bbD$
to be the quotient $\Out(\bbD) = \Aut(\bbD)/W$.
The root datum $\bbD_G$ of a connected $\ell$--compact group $BG$
is obtained  by
fixing a maximal torus  $BT \to BG$  which is a fibration and 
setting $L = \pi_2(BT)$, $L_0 = \ker(L \to \pi_2(BG))$, and $W =
\pi_0(\cW_G(T))$, where $\cW_G(T)$ denotes the Weyl space, 
i.e., the monoid of homotopy equivalences $BT \to BT$ over $BG$.
The action of $W$ on $L$ is induced by the evident action of 
$\cW_G(T)$ on $BT$.

The rational cohomology of a connected
 $\ell$--compact group is recovered from the action of $W$ on $L$ via
\begin{equation}
\vcenter{\xymatrix{
     H^*(BG;\Z_\ell)\otimes \Q  
     \ar@{-}[d]_\isom
     \ar[r]^\isom  
     &  
     (H^*(BT;\Z_\ell)\otimes \Q)^{W}  
     \ar@{-}[d]^\isom
     \\
     \Q_\ell[x_1,\ldots, x_r] 
     \ar[r]^\isom 
     & 
     \Q_\ell[L]^{W}
}}
\end{equation}
where the top horizontal map is induced by the map $BT\to BG$, and $r$ is the rank 
of $BG$ (see \cite[5.11 and Thm.~9.7]{DW94}).
We have $|x_i| = 2d_i$ for 
some $d_i > 0$, and the integers
\begin{equation}
   d_1, d_2 , \ldots, d_r
 \end{equation}
are called the \emph{degrees} of the reflection group $(W,L)$.

The \emph{dimension} $d$ of an $\ell$--compact group is defined as the
maximal non-trivial dimension of $H^*(G;\F_\ell)$. It can also described
purely in terms of the reflection grroup as
\begin{equation}
  d = \sum_i (2d_i-1) = 2|\{\mbox{reflections in $W$}\}|-r
  \end{equation}
    (see \cite[Lem.~3.8]{dw:split} and \cite[Thm.~7.2.1]{benson93}).

For later use in Section~\ref{subapp:rootdata} we will also record the \emph{rational homotopy type} of $BG$:
 The rationalization $BG_\Q$ splits as a product
\begin{equation}
  BG_\Q \isom \prod_iK(\Q_\ell,2d_i)
\end{equation}
(see e.g.\ \cite[pf.~of Lem.~1.2]{ABGP04}). 
In particular, $H^*(BG;\Z_\ell) \otimes \Q$
is isomorphic to the $\Q_\ell$--polynomial ring
on the graded $\Q_\ell$--vector space
$\pi_*(BG) \otimes \Q$:
\begin{equation}
  H^*(BG;\Z_\ell) \otimes \Q \isom \Q_\ell [\pi_*(BG) \otimes \Q].
\end{equation}
and consequently, 
\begin{equation} \label{eq:rationalpiBG}
  \pi_*(BG) \otimes \Q 
  \isom 
  Q(H^*(BG;\Z_\ell) \otimes \Q)
  \isom
  \oplus_i \Q_\ell x_i
  \end{equation}
where $Q(-)$ denotes indecomposable elements. (Beware that the choice of generators $x_i$ is not canonical.)

\subsection{Torsion in the cohomology of \texorpdfstring{$G$}{G}}
It is a consequence of the classification of $\ell$--compact groups
that most $\ell$--compact groups have $\ell$--torsion-free cohomology and have
classifying spaces with polynomial cohomology rings. We summarize this in the next theorem, used in Section~\ref{sec:funclass1}.
\begin{thrm}[Classification of $\ell$--compact groups with polynomial
  cohomology ring] %
\label{thm:poly-cohom}
Let $BG$ be a connected $\ell$--compact group.
\begin{enumerate}
    
\item \label{poly-torsion-free}
$H^*(BG)$ is a polynomial ring concentrated in even degrees iff
$H^*(G;\Z_\ell)$ is $\ell$--torsion free iff $H^*(BG;\Z_\ell)
\xrightarrow{\cong} H^*(BT;\Z_\ell)^W$. (The ``even degrees''
assumption is automatic if $\ell$ is odd.)

\item \label{polyodd} If $\ell$ is odd, then $H^*(G;\Z_\ell)$ is $\ell$--torsion free
  iff $\pi_1(G)$
  is $\ell$--torsion free and the universal cover of $G$ does not
  contain $F_4$,
  $E_6$, $E_7$, or $E_8$ when $\ell=3$, or $E_8$ when $\ell =5$, as
  $\ell$--compact group summands.

\item \label{poly2-pi1torsionfree} Suppose $\ell =2$ and $\pi_1(G)$ is
  $2$--torsion free. Then $H^*(BG)$ is a polynomial ring iff the
 universal cover of $G$ does not contain 
 $E_6$,  $E_7$, $E_8$, or $\Spin(n)$ for $n \geq 10$, as $\ell$--compact
  group summands.
  Moreover,  $H^*(G;\Z_\ell)$
   is $\ell$--torsion free if furthermore the universal cover
  of $G$ does not contain $G_2$, $F_4$, $DI(4)$, or $\Spin(n)$ for
$n \geq 7$, as $\ell$--compact group summands.

\end{enumerate}
\end{thrm}
\begin{proof} Point \eqref{poly-torsion-free} is classical and summarized in
  \cite[Thms.~12.1]{AGMV08}. Point
  \eqref{polyodd} follows from \cite[Thms.~12.1 and 12.2]{AGMV08}
  together with the classification of $p$--compact groups for odd primes \cite[Thm.~1.1]{AGMV08}.
 Point \eqref{poly2-pi1torsionfree} follows from \cite[Thms.~1.1 and 1.4]{AG09}.
 \end{proof}
\begin{rem}\label{rem:poly-cohom}
  For $\ell=2$, there are examples where $BG$
  has polynomial cohomology despite $\pi_1(G)$ containing $2$--torsion,
  $G=\SO(n)$ being one.
A summary of what is known is given in \cite[Rem.~7.1]{AG09}, which
includes a complete list in the case where $G$ is assumed simple.
   \end{rem}

\subsection{Effect of unstable Adams operations on cohomology}\label{subsec:unadams}
  Recall that an unstable Adams operation $\psi^q$, for $q \in
  \Z_\ell^\times$, is the operation $\psi^q \in  \Out(BG)$ corresponding
  to multiplication by $q$ on the lattice $L$ of the root datum
  $\bbD$. We here describe how $\psi^q$ acts on $H^*(BG)$ when $BG$
  has polynomial cohomology ring. Without this assumption the question
  is harder; see the
  discussion after Question~\ref{qu:thequestion}.
  
\begin{prop}
\label{prop:idmappoly} If
$H^*(BG)$ is a polynomial ring concentrated in even degrees then
$\psi^q$ acts as $q^n$ on $H^{2n}(BG;\Z_\ell)$, and in particular acts the identity on  $H^{*}(BG;\Z/\ell^r)$  if $q \equiv 1\ \mathrm{mod}\ \ell^r$.

If $\ell=2$ and $H^*(BG)$ is a polynomial ring (not necessarily concentrated in
even degrees) then $\psi^q$ still acts as the identity on
$H^*(BG)$ for all $q\in \Z_2^\times$.
\end{prop}
\begin{proof}
  If $H^*(BG)$ is concentrated in even degrees, then the statement
  about the action on $H^*(BG;\Z_\ell)$ is
  clear from Theorem~\ref{thm:poly-cohom}\eqref{poly-torsion-free}, as we
  have natural identification $H^2(BT) \cong
  \Hom_{\Z_\ell}(L,\Z_\ell)$. The statement with $\Z/\ell^r$
  coefficients now follow from the universal coefficient theorem.
  
Now, suppose that $\ell =2$. Here $H^*(BG) \to H^*(BT)$ need not be
injective, unless $G$ is $2$--torsion free (see
\cite[Thm.~12.1]{AGMV08}). However, by the theory of unstable modules over the
Steenrod algebra (see e.g.\ \cite[Prop.~7.2]{AG09}), there exists an elementary abelian
$2$--subgroup $V$, which, up to conjugation, contains every other elementary
abelian $2$--subgroup.  
In particular, we can choose a representative which contains ${}_2T$,
the elements of order $2$ in the maximal torus. Let $W_0$ denote the pointwise stabilizer of ${}_2T$ in $W$, which by e.g.\ 
\cite[Lem.~11.3]{AGMV08} is an elementary abelian $2$--subgroup of
$W$. Now, recall Tits' model for $N_G(T)$ from \cite{tits66},
elaborated and extended to $2$--compact groups in \cite{DW05} and \cite{AG08auto}:
$N_G(T)$ can be constructed from the root datum by first constructing
the reflection extension $1 \to \Z[\Sigma] \to \rho(W) \to W \to
1$, where $\Sigma$ is the set of reflections in $W$, and then
constructing $N_G(T)$ as a push-forward along a $W$--map $f\co
\Z[\Sigma] \to T$ sending each reflection $\sigma$ to a certain element of order two
$h_\sigma$ of $T$; see \cite[\S\S2-3]{AG08auto}. Let $\rho_0$ denote
the preimage of $W_0$ in $\rho(W)$, and consider the subgroup $A$
of $N_G(T)$ generated by ${}_2T$  and the image of $\rho_0$ under
$\rho(W) \to N_G(T)$, the map to the push-forward. By construction
$A$ is an abelian subgroup of $N_G(T)$. Likewise by construction it
will contain $V$.
We hence just have to see that $\psi^q$ acts trivially on $A$. However,
this is a consequence of \cite[Thms.~B and C]{AG08auto}, which explain
exactly how $\psi^q$ acts on $N_G(T)$, namely as a quotient of a map
which multiplies by $q$ on $T$ and is the identity on $\rho(W)$ (see
Step~2 of the proof of Thm.~B in \cite{AG08auto} for the definition of
the homomorphism $s\co \Out(\bbD_G) \to \Out(N_G(T))$).
\end{proof}

\subsection{Homotopical actions on
  \texorpdfstring{$BG$}{BG}}\label{subapp:homotopy-act}
We start by briefly recalling homotopical (or proxy) actions; see e.g.\ \cite[\S10]{DW94} for more details. 
A homotopical action of a discrete group $K$
 on a space $X$ is a free
homotopy class of maps 
 $BK\to B\aut(X)$, where $\aut(X)$ denotes the
 group-like topological monoid of homotopy self-equivalences of $X$. By the classification of fibrations, homotopical actions are in bijection with equivalence classes of fibrations with base $BK$ and fiber $X$. The bijection is given by pulling back the universal fibration with fiber $X$ and base $B\aut(X)$ along $BK \to B\aut(X)$. 
 
A homotopical action of $K$ on $X$ produces a free $K$--action on a space homotopy equivalent to $X$ by pulling back the fibration with fiber $X$ and base $BK$ along 
the map $EK \to BK$. The total space $X'$ of the pulled-back fibration is a free $K$--space homotopy equivalent to $X$. Conversely, an actual action of $K$ on $X$, i.e., a homomorphism $K \to \aut(X)$, obviously determines a homotopical action by taking classifying spaces, and the associated fibration is the Borel fibration $X_{hK} \to BK$.
Homotopical actions of $K$ on $X$ are also, via this correspondence, in bijection with free $K$--actions on spaces homotopy equivalent to $X$, up to $K$--homotopy equivalence. (Alternatively, one can drop the freeness assumption and change the equivalence relation to being induced by $K$--maps which are underlying homotopy equivalences, as these induce the same fibrations over $BK$.)
Homotopical actions hence `are' actual actions on a space homotopy equivalent to $X$ (hence the alternative name `proxy'), and thus also induce actions on homology groups, etc.

If $X$ is a pointed space (pointed meaning well-pointed), one can also consider pointed homotopical actions, defined to be 
free homotopy classes of maps $BK \to B\aut_*(X)$, where $\aut_*(X)$ is the
 group-like topological monoid of pointed homotopy self-equivalences of $X$, which again by the classification of fibrations correspond to fibrations over $BK$ with fiber $X$ together with a section, and pointed homotopical $K$--actions on $X$ correspond to $K$--actions on pointed $K$--spaces pointed homotopy equivalent to $X$, modulo the equivalence given by pointed $K$--maps which are underlying pointed homotopy equivalences.

For a homotopy action of $K$ on $X$, we can define the homotopy fixed point space $X^{hK}$ 
as the space of sections of the associated fibration with fiber $X$ and base $BK$, which 
again, up to homotopy, identifies with $\map_K(EK,X')$, where $X'$ is the free $K$-space 
homotopy equivalent to $X$ as above. Similarly, for a homotopy action of $K$ on $X$, we take
the homotopy orbit space $X_{hK}$ to mean the total space in the fibration over $BK$ with
fiber $X$, obtained by pulling the universal fibration back along $BK \to B\aut(X)$.
 
 In this subsection, we will relate homotopical actions of $K$ on $BG$
to homomorphisms $K \to \Out(\bbD_G)$,
building on \cite{AG08auto}, \cite{BM07}, etc.
Note that if $K = \Z$, then $BK \homot S^1$, so $[BK,B\aut(X)] = [S^1,
B\aut(X)] = \Rep(\Z, \Out(\bbD_G))$, i.e., conjugacy classes of
elements of $\Out(\bbD_G)$, but for $K$ finite the story is a priori more complicated.

Let $BT\to BG$ be a maximal torus which is a fibration, and
write $\aut(BT \to BG)$ for
the topological monoid of homotopy self-equivalences of the fibration
$BT \to BG$, that is, the topological monoid of pairs 
$(f,g)$ of homotopy equivalences (not necessarily
preserving the basepoint) fitting into a commutative square
\[\xymatrix{
	BT 
	\ar[r]^f_\homot
	\ar[d]
	&
	BT
	\ar[d]
	\\
	BG
	\ar[r]^g_\homot
	&
	BG
}\]
The center $B\cZ(G)$ of $BG$
is given by %
$B\cZ(G) \simeq \aut_1(BG)$ where $\aut_1(BG)$ denotes
the submonoid of $\aut(BG)$ 
consisting of self-homotopy equivalences of  $BG$ homotopic
to the identity;
see again e.g.\ \cite{AG09,grodal10} (or the original \cite{DW95}) for more information.

\begin{prop}[Homotopy actions on $BT$ and $BG$ versus actions on the
    root datum $\bbD_G$]\label{autBTtoBG}
Let $BG$ be a connected $\ell$--compact group with
root datum $\bbD_G$.  Then we have a commutative diagram
$$\xymatrix { \ast \ar[r] \ar[d] & B\cW_G(T) \ar[r]^\homot\ar[d] & BW \ar[d]\\
   B^2\cZ(G) \ar[r] \ar@{=}[d] & B\aut(BT \to BG) \ar[r]\ar[d]\pb &
   B\Aut(\bbD_G) \ar[d]\\
   B^2\cZ(G) \ar[r] & B\aut(BG) \ar[r] & B\Out(\bbD_G)}$$
with rows and columns fibre sequences
and the bottom right square a homotopy pullback square.
In particular, when $K$ is a finite $\ell'$--group, or 
the product of a finite
$\ell'$--group and $\Z$, we have bijections
\begin{equation}
\label{eq:outrepbij}
 \Rep(K,\Out(\bbD_G)) \isom \Rep(K,\Out(BG)) \xrightarrow{\ \isom\ }
 [BK,B\Out(BG)] \xleftarrow{\ \isom\ } [BK,B\aut(BG)]   \mbox{\,\,\, and }
\end{equation}
\begin{equation}
\label{eq:autrepbij}
\Rep(K,\Aut(\bbD_G)) \xrightarrow{\ \isom\ } [BK, B\Aut(\bbD_G)]
\xleftarrow{\ \isom\ } [BK,B\aut(BT \to BG)]
\end{equation}
where $\Rep(K,\Gamma)$ for a group $\Gamma$ denotes the set of 
conjugacy classes of homomorphisms from $K$ to $\Gamma$.

Furthermore, when $K$ is a finite $\ell'$--group, 
we have a bijection
\begin{equation}
\label{eq:autpt}
[BK,B\aut_*(BG)] \xrightarrow{\ \isom\ } [BK,B\aut(BG)] 
\end{equation}
where $\aut_*(BG)$ denotes the space of pointed homotopy self-equivalences of $BG$.
\end{prop}
    
\begin{proof} 
The rightmost column is the canonical
fibre sequence arising from the definition of $\Out(\bbD_G)$
as $\Aut(\bbD_G)/W$, and the middle column  
is a fibre sequence by the definition of the Weyl space as self-maps
of $BT$ over $BG$.
The bottom right horizontal map is induced by the composite
\[
	\aut(BG) \longto \Out(BG) \xrightarrow{\ \isom\ } \Out(\bbD_G)
\]
and the middle right-hand horizontal map is induced by the composite
\[
	\aut(BT \to BG) \longto \aut(BT)  \longto \Out(BT) \xto{\ \isom\ } \GL(L),
\]
whose image
lands in $\Aut(\bbD_G)$ and which makes the bottom right hand square
commute; see \cite[Thm.~1.2 and  Rec.~8.2]{AG09}.
That the induced map on fibres at the top right hand side 
is a homotopy equivalence
follows from the standard fact that the Weyl space
is homotopy discrete;
see \cite[Prop.~8.10 and Thm.~9.7]{DW94}. 
It follows that the bottom right square is a  homotopy pullback square.
The bottom row is a fibre sequence by the isomorphism
      $\Out(BG) \isom \Out(\bbD_G)$ and the fact that
      $\aut_1(BG) \homot B\cZ(G)$; see \cite[Thm.~1.2]{AG09}. As the
bottom right square is a homotopy pullback square,  it follows that
the middle row is a fiber sequence as well,
establishing the whole diagram.

We now establish the bijections in \eqref{eq:outrepbij}
and \eqref{eq:autrepbij}.
The first bijection
in \eqref{eq:outrepbij} is a
consequence of the isomorphism $\Out(BG) \isom \Out(\bbD_G)$ already
discussed, and the second is elementary homotopy theory. For the
third, consider the pullback of the fibration
$$B^2{\mathcal Z}G \longto B\aut(BG) \longto B\Out(BG)$$
along
$\phi \in [BK,B\Out(BG)]$.
It is enough to show  that this pullback
has a unique section, which we will do 
using obstruction theory. The existence and uniqueness obstructions lie
in 
$H^j(K;\pi_i(B^2{\mathcal Z}G))$, where $i>0$ and $j=i+1$ or $j=i$, with local coefficients induced by $\phi$.
Let $R$ denote the finite part of $K$.
By the Lyndon--Hochschild--Serre spectral sequence, and
the fact that $R$ is a finite group of order prime to $\ell$,
we have
$H^j(K;\pi_i(B^2{\mathcal Z}G)) \isom H^j(K/R;
(\pi_i(B^2{\mathcal Z}G))^{R})$. But $K/R$ is either
infinite cyclic or trivial, and $B^2{\mathcal Z}G$ is simply
connected, so this group is trivial if $j=i$ or $i+1$.
Hence we have a bijection
$[BK,B\aut(BG)] \xrightarrow{\isom} [BK,B\Out(BG)]$
as wanted. (See also \cite[Thm.~A]{AG08auto}.)
The bijections in \eqref{eq:autrepbij}
follow similarly, with the second isomorphism following
from  exactly the same obstruction theory argument as above.

Finally, for \eqref{eq:autpt} note that we have a fibration sequence
$$BG \longto B\aut_*(BG) \longto B\aut(BG)$$
and that $\map(BK,BG)$ is contractible, as $K$ is a finite $\ell'$--group and $BG$ is $\ell$--complete.
\end{proof}

      \begin{rem}
Note that by \cite[Thm.~A]{AG08auto} maps $B\Gamma \to B\Out(\bbD_G)$ in fact always
lift to $B\aut(BG)$, for any discrete group $\Gamma$, but in general the lift may not be unique. 
        \end{rem}

\subsection{Homotopy fixed point \texorpdfstring{$\ell$}{l}--compact groups
  \texorpdfstring{$BG^{h\langle\tau\rangle}$}{BG\textasciicircum h<t>} and their \texorpdfstring{$\Z_\ell$}{Zl}--root data}
\label{subapp:rootdata}

Recall that the homotopy fixed-point space for a homotopical group action of $K$ on a space $X$ is defined as the space of sections of the fibration $X_{hK} \to BK$,
or $\map_K(EK,X')$ where $X'$ is the total space of the fibration pulled back along $EK \to BK$; see Section~\ref{subapp:homotopy-act}.
The goal of this subsection
is to prove 
the following theorem describing the homotopy fixed-point
$\ell$--compact group $BG^{h\langle \tau \rangle}$ in terms of the
root datum $\bbD_G$ of $BG$ when $\tau \in \Out(\bbD_G)$ is of
finite order prime to $\ell$. The theorem can be seen as a
generalization of the Lie theoretic construction of e.g.\  
$F_4$ inside $E_6$ as $\Z/2$--fixed-points.

\begin{thrm}[The root datum of $BG^{h\langle \tau\rangle}$] \label{thm:root-datum}
 Let $\bbD = (W, L, L_0)$ be a $\Z_\ell$--root datum of a connected $\ell$--compact group
 $BG$ with maximal torus $BT$, and  $\tau \in \Out(\bbD)$ an element of finite order prime to
 $\ell$.
Let $\oldphi \in \Aut(\bbD)$ be a lift of $\tau$ with $\ell \nmid
|\oldphi|$ such that the $\oldphi$--fixed-point lattice
$L^{\oldphi}$ has maximal rank among all such lifts.
Then the homotopy fixed point space $BG^{h\langle \tau
  \rangle}$ is an $\ell$--compact group 
  with maximal torus
  $BT^{h\langle \oldphi \rangle} \to BG^{h\langle \oldphi \rangle}
  \simeq BG^{h\langle \tau \rangle}$ (with $\langle \oldphi
  \rangle$--action via Proposition~\ref{autBTtoBG}),
and root datum given by
$$\bbD_{G^{h\langle \tau
    \rangle}} = (N_W(L^{\oldphi})/C_{W}(L^{\oldphi}), L^{\oldphi}, L_0^{\oldphi
}),$$
where $N_W(L^\oldphi) = \{w \in W \mid wL^\oldphi =L^\oldphi \}$ and
$C_W(L^\oldphi) = \{ w \in W \mid wx =x \mbox{ for all } x \in L^\oldphi\}$.

In particular
$$\pi_*( BG^{h\langle \tau \rangle}) \otimes \Q = (Q(\Q_\ell[L]^W))^\tau$$
and
$$H^*(BG^{h\langle \tau \rangle};\Z_\ell) \otimes \Q \cong \Q_\ell[L^\oldphi]^{N_W(L^{\oldphi})/C_{W}(L^{\oldphi})} \cong \Q_\ell[(Q(\Q_\ell[L]^W))^\tau].$$
\end{thrm}
Our first step towards a proof of Theorem~\ref{thm:root-datum} is to show that the homotopy fixed point space $BG^{hK}$ is
again an $\ell$--compact group when $K$ is an $\ell'$--group.
\begin{prop}
\label{prop:ghksplit}
For $BG$ a connected $\ell$--compact group and $K \leq \Out(BG)$ a
finite $\ell'$--group, there exists a canonical pointed homotopical $K$--action on $G = \Omega BG$ together with an isomorphism
\begin{equation}
\label{iso:pighk} 
	\pi_*(G^{hK}) \xrightarrow{\ \isom\ } \pi_*(G)^K
\end{equation}
on homotopy groups 
and a weak equivalence
\begin{equation}
\label{eq:ghksplitting}
	G \xrightarrow{\ \homot\ } G^{hK} \times G/G^{hK}	
\end{equation}
of spaces where
$G/G^{hK}$ denotes the fibre of the map $BG^{hK} \to BG$ as usual.
In particular, $BG^{hK}$ is again an $\ell$--compact group of
at least the connectivity of $BG$,  and semisimple if $BG$ is.
The $\Z_\ell$--cohomology of $G^{hK}$ is $\ell$--torsion
free (equivalently, the cohomology of the $BG^{hK}$
is a polynomial ring on even degree generators)
if the same holds for $G$.
\end{prop}
\begin{proof}
This is a consequence of \cite[Thm.~5.2]{BM07}, but let us give
the argument for convenience.
The subgroup $K \leq \Out(BG)$ determines a unique base-point preserving homotopical action of $K$ on $BG$ 
by Proposition~\ref{autBTtoBG}, which we can treat as an actual base-point preserving action, after potentially replacing $BG$ by a homotopy equivalent space. 
We then have an evident induced $K$--action on $\loops BG$, and 
$(\Omega BG)^{hK} \homot \Omega (BG^{hK})$ as both spaces are homotopy equivalent
to the space of basepoint-preserving $K$--equivariant maps 
$EK_+ \smashprod S^1 \to BG$. Also, as $K$
is a finite $\ell'$--group, and $BG$ is $\ell$--complete, the homotopy
fixed point spectral sequence collapses and $\pi_*(BG^{hK})
\isom \pi_*(BG)^K$.  The existence of isomorphism \eqref{iso:pighk} follows.

As we have a basepoint-preserving $K$--action on $G =\Omega BG$ we can define a ``norm'' map $\rho: G \to G$ by $f \mapsto \prod_{k \in K}kf$. Here 
$kf$ denotes the action of $k \in K$ on $f \in G =  \Omega BG$, and the
product is taken with respect to the loop product on $G$,
for some
fixed ordering of $K$. 
Then  $\rho_*: \pi_*(G) \to \pi_*(G)$ is given
by $\rho_* = N_K = \sum_{k \in K} k_\ast$.

Write 
$\rho G = \hocolim(G \xrightarrow{\rho} G \xrightarrow{\rho} \cdots)$.
We claim that we have a weak equivalence
\[
	j \co G^{hK} \xto{\ \homot\ } \rho G
\]
given by the composite $G^{hK} \to G \to \rho G$ of the canonical maps.

Under \eqref{iso:pighk} and the isomorphism
\[
\colim (\pi_*(G) \xrightarrow{N_K} \pi_*(G) \xrightarrow{N_K} \cdots)
\xto{\ \isom\ }
\pi_\ast(\rho G),
\]
the map induced by $j$ corresponds to the composite
\[
	\pi_\ast(G)^K 
	\longincl 
	\pi_\ast(G) 
	\longto 
	\colim_n(\pi_*(G) \xrightarrow{N_K} \pi_*(G) \xrightarrow{N_K} \cdots)
\]
which is an isomorphism since $\pi_*(G)$ is $\ell$--local and $K$ is an $\ell'$--group.
Thus $j$ is a weak equivalence, as claimed.
Consequently, we have weak equivalences
\[
	G 
	\xrightarrow{\ \homot\ } 
	\rho G \times G/G^{hK}
	\xleftarrow[\ \homot\ ]{\ j\times \id\ }
	G^{hK} \times G/G^{hK}
\]
as the zigzag $G \to \rho G \xleftarrow{\homot} G^{hK}$ induces a splitting of the
long exact homotopy sequence of the fibration $G^{hK} \to G \to
G/G^{hK}$.

It remains to deduce the `in particular' part. 
It is now clear that $BG^{hK}$ is an $\ell$--compact group, as it is 
$\ell$--complete and the splitting \eqref{eq:ghksplitting} implies that 
its loop space has finite mod $\ell$ cohomology.
From \eqref{eq:ghksplitting},
it is also clear that $BG^{hK}$  has at least the
connectivity of $BG$, that $BG^{hK}$ is semisimple if $BG$ is, and that $G^{hK}$ has torsion free
$\Z_\ell$--cohomology if $G$ has.
The parenthetical statement that $G$ has torsion free 
$\Z_\ell$--cohomology if and only if $H^*(BG;\F_\ell)$ 
is a polynomial ring on
even degree generators is classical and recorded in Theorem~\ref{thm:poly-cohom}\eqref{poly-torsion-free}.
\end{proof}
\begin{rem}
A companion result for fixed points of $\ell$--group actions is given in \cite[Lem.~5.9]{DW94}.
\end{rem}

\begin{proof}[Proof of Theorem~\ref{thm:root-datum}]
First note that $BG^{h\langle \tau \rangle}$ is an $\ell$--compact group of rank $\dim_{\Q_\ell} (Q(\Q_\ell[L]^W))^\tau
  = \dim_{\Q_\ell} (\pi_*(BG) \otimes \Q)^\tau$ by Proposition~\ref{prop:ghksplit} and \eqref{eq:rationalpiBG}. This number equals 
 $\rk_{\Z_\ell}(L^{\oldphi})$ by a result of Springer
\cite[Thm.~6.2(i)]{springer74} (with $d=1$, and $\sigma = \tau$, and
after noting that the rank remains unchanged if we extend scalars to
$\bbC$, and that the maximum rank can always be realized by an element
$\oldphi$ of order
prime to $\ell$, by writing the element as a product of an $\ell$--part and
an $\ell'$--part).
Next, let 
 $\oldphi \in \Aut(\bbD)$ be a lift as in then statement of the theorem. Then we get a homotopy action of $\langle
\oldphi\rangle$  on $BT \to BG$ by
Proposition~\ref{autBTtoBG}. By taking homotopy fixed points, 
we get a canonical map
$BT^{h\langle \oldphi \rangle} \to BG^{h\langle \oldphi \rangle}
\simeq  BG^{h\langle \tau \rangle}$,
which has to be a maximal
torus, as it is a monomorphism since $BT \to BG$ is (see
\cite[Thm.~7.3]{DW94}) and of maximal rank by the first part.
Hence $BG^{h\langle \tau \rangle}$ has lattice $L^\oldphi$.
It is also
clear that $L_0^\oldphi$ is a coroot lattice, as we have an exact
sequence
$$0 \longto L_0^\oldphi \longto L^\oldphi \longto \pi_1(G)^\oldphi \longto 0$$
because taking $\langle
\oldphi \rangle$--fixed-points is exact, and  $\pi_1(G^{h\langle
  \tau\rangle}) \xrightarrow{\cong} \pi_1(G)^\tau$ by Proposition~\ref{prop:ghksplit}.

To identify the root datum of $BG^{h\langle \tau \rangle}$,
it remains to establish the claimed description of the Weyl group.
First note that we have 
$$W_{G^{h\langle\tau\rangle}} \leq
N_{W}(L^\oldphi)/C_{W}(L^\oldphi)$$
as subgroups of $\Aut_{\Z_\ell}(L^\oldphi)$.
Namely, with
$BT^{h\langle \oldphi \rangle}
\xrightarrow{\ \iota\ } BG^{h\langle \oldphi \rangle} \xrightarrow{\ f\ }
BG$, we have
\[
    W_{G^{h\langle \oldphi \rangle}} \cong \{ w \in
    \Aut_{\Z_\ell}(L^\oldphi) \mid  \iota \circ B^2w \simeq  \iota\}
\]
and
\[
    N_{W}(L^\oldphi)/C_{W}(L^\oldphi) \cong \{ w \in
    \Aut_{\Z_\ell}(L^\oldphi) \mid  f \circ \iota \circ B^2w \simeq  f \circ \iota \}
\]
where $B^2 w: BT^{ \langle\oldphi\rangle} \simeq B^2L^\oldphi \to
B^2L^\oldphi \simeq BT^{ \langle\oldphi\rangle}$ 
is the map induced by
$w$, by \cite[\S 7]{DW95}.

We now claim that the two groups have the same order.
For this, consider the inclusion
$$
\Q_\ell[L^\oldphi]^{N_{W}(L^\oldphi)/C_{W}(L^\oldphi)}
\subseteq 
\Q_\ell[L^\oldphi]^{W_{G^{h\langle\tau\rangle}}}.
$$
Both sides are polynomial algebras on the graded $\Q_\ell$--vector
space $Q(\Q_\ell[L]^W)^\tau$, i.e., they have the same degrees;
for the right-hand side this was observed above, and for the left-hand
side it follows from a result in invariant theory
\cite[Thm.~5.1]{LS99} (see also \cite[Thm.~12.20]{LT09}).
As the order of a reflection group is the product of its degrees
(cf.\ \cite[Thm.~4.14(i)]{LT09}), this shows equality.

The claims in the `in particular' part were established above as part of the proof.
 \end{proof} 

    \begin{rem}
Theorem~\ref{thm:root-datum}
should be compared to the construction of the
$\Phi_e$--torus in a finite group of Lie type \cite{BM92}
\cite[\S25]{MT11} \cite{EM18}, which can be
seen as passing from a $\Z$--root datum to a $\Z[\zeta_e]$--root
datum (whereas for $\Z_\ell$--root data $\Z_\ell[\zeta_e] = \Z_\ell$).
    \end{rem}

\begin{rem} Note that over $\bar \Q_\ell$ we can choose a basis $\{\bar x_1, \ldots \bar x_r\}$ for  $Q(\bar \Q_\ell[L]^W)$ such that $\tau$ acts on $\bar x_i$ as a root of unity, and in such a basis we can then read off a $\bar \Q_\ell$-basis for $\pi_*(BG^{h\langle \tau\rangle}) \otimes_{\Z_\ell} \bar \Q_\ell$ as given by those $\bar x_i$ on which $\tau$ acts trivially.
\end{rem}

\subsection{Classification of spaces of the form \texorpdfstring{$\BtGq$}{BtG(q)} 
and untwisting}
\label{subapp:untwisting}
In this subsection we will prove the untwisting theorem, Theorem~\ref{thm:untwisting-intro} from the introduction.
We will do this by presenting a more general classification of spaces of the form $\BtGq$, with the notation of \eqref{eq:lietypegrp}. Spaces of this form are sometimes called $\ell$--local finite groups of Lie type.
The result is a strengthening of \cite[Thm.~E]{BM07} and implies the equivalences of fusion systems of \cite[Thm.~A]{BrotoMoellerOliver}. 
\begin{thrm}[Classification of $\ell$--local finite groups of Lie type]\label{thm:class-lcglie}
Let $BG$ be a connected $\ell$--compact group, $q \in
  \Z_\ell^\times$, and let $\tau \in \Out(BG) \isom \Out(\bbD_G)$ be
  of finite order. 
 Make the factorizations
\begin{itemize}
\item    $q = \zeta_e q'$  in $\Z_\ell^\times$
  where 
$q' \equiv 1 (\ell)$ if $\ell$ odd and $q' \equiv 1 (4)$ if $\ell=2$, and $\zeta_e$ is a primitive $e$--th root of unity; and
\item $\tau\psi^{\zeta_e} = \tau'\tau_\ell$ in $\Out(BG)$, where $\tau'$ has order prime to $\ell$ and $\tau_\ell$ has $\ell$--power order
 \end{itemize}
so that in particular 
 $$ \tau \psi^q = \tau' \tau_\ell \psi^{q'}.$$
 Then
\begin{enumerate}
\item \label{item:Hislcg}
    The finite $\ell'$--group $\langle \tau' \rangle \leq \Out(BG)$ has a
    canonical homotopical action on $BG$. Its homotopy fixed point space 
    $BH = BG^{h\langle \tau' \rangle}$ is a connected $\ell$--compact group, semisimple or 
    simply connected if $BG$ is, and has a canonical homotopical 
    action of $\tau_\ell \psi^{q'}$.
\item \label{item:untwisting}
    With the actions from \eqref{item:Hislcg},
    we have an ``untwisting" homotopy equivalence
    \begin{equation*}%
    \BtGq \xleftarrow{\ \homot\ } B{}^{\tau_\ell}\!H(q').
    \end{equation*}
    In particular, we can completely untwist $\BtGq$ if $\ell \nmid |\tau|$.
\item\label{it:threepiecesofdata}
	The homotopy type of $\BtGq$ is  determined by the following three pieces of data:
    \begin{enumerate}
    \item  \label{item-twroot} The root datum $\bbD_{H}$ of $BH$ 
      (obtained from $\bbD_G$ and $\tau' \in \Out(\bbD_G)$ via 
      Theorem~\ref{thm:root-datum}).
     \item \label{item-ltw} The twisting $\tau_\ell \in \Out(\bbD_{H})$ of finite $\ell$--power order.
      \item \label{item-val}The $\ell$-adic valuation of $q'-1$.
    \end{enumerate}
\end{enumerate}
\end{thrm}
Our proof of Theorem~\ref{thm:class-lcglie}  is a modification of the proof of \cite[Thm.~E(i)]{BM07} (which e.g.\ involved a case-by-case check \cite[Rem.~6.3]{BM07} left to the reader; compare also with
\cite[Thm.~4.2]{BrotoMoellerOliver}). Before embarking on the proof, let us just unravel the notation to see how Theorem~\ref{thm:untwisting-intro} follows:
\begin{proof}[Proof of Theorem~\ref{thm:untwisting-intro} from Theorem~\ref{thm:class-lcglie}]
If $\ell$ is odd, the claim follows directly from Theorem~\ref{thm:class-lcglie}(\ref{item:Hislcg}) and (\ref{item:untwisting}), as $\tau_\ell = 1$ if $\ell \nmid |\tau|$, and the definitions of $\tau'$ and $q'$ across the two theorems agree. The same is true if $\ell=2$
and $q$ is congruent to $1$ mod $4$.

If $\ell =2$ and $q$ is congruent to $3$ mod $4$ the theorem also follows, but there is a slight difference in how we write things, due to different definitions of $q'$ in the two theorems: Theorem~\ref{thm:class-lcglie}\eqref{item:untwisting} tells us that  $\BtGq \xleftarrow{\ \homot\ } B{}^{\tau_2}\!H(-q)$, with $\tau_2 =-1$ and $BH = BG^{h\langle \tau'\rangle}$. But $B{}^{\tau_2}\!H(-q)$ is just another name for $BH(q)$, so Theorem~\ref{thm:untwisting-intro} also follows in that case. (See Remark~\ref{rem:mod4cong} for an explanation of this notation.)
\end{proof}

For the proof of Theorem~\ref{thm:class-lcglie} we need a lemma.
  \begin{lemma}\label{nilpot-lemma}
  Suppose $q \in \Z_\ell^\times$ satisfies 
  $q \equiv 1\ \mathrm{mod}\ \ell$.
  For any connected $\ell$--compact group $BG$, the corresponding
  homomorphism $\Z \to \Out(BG)$, $ 1 \mapsto \psi^q$, extends to a central homomorphism $\Z_\ell \to \Out(BG)$.
  
 In particular, if $F$ is any homotopy invariant functor from spaces
to $\F_\ell$--vector spaces with $F(BG)$ finite dimensional (e.g.\ 
 $\pi_n(\,\cdot\, ) \otimes
\F_\ell$, $H_n(\,\cdot\,;\F_\ell)$ or
$H_n(\Omega( \,\cdot\,);\F_\ell)$), then
$\psi^q$ acts
on  $F(BG)$
as an element of order $\ell^s$ for some $s\geq 0$.
\end{lemma}

\begin{proof} 
The map $\Z \to \Out(BG)$, $1 \mapsto \psi^q$, factors as 
$$\Z \xrightarrow{1 \mapsto q} \Z_\ell^\times \xrightarrow{t \mapsto \psi^t}\Out(BG)$$
 By
assumption, $q$ is in the kernel of $\Z_\ell^\times \to
\F_\ell^\times$, which is isomorphic to $\Z_\ell$ if $\ell$ is odd and
$\Z_2 \times \langle \pm 1\rangle$ if $\ell =2$.
In both cases, the image of the first map hence lies in an $\ell$--complete
abelian group, so by the universal property of $\ell$--completion we have a factorization
\begin{equation}\label{eq:factorization}
\begin{tikzcd}
\Z \ar[r,"1 \mapsto q"]  \ar[d,hook] & \Z_\ell^\times  \ar[r,"t \mapsto \psi^t"] & \Out(BG)\\
\Z_\ell \ar[ur] 
\end{tikzcd}
\end{equation}
Moreover, the unstable Adams operations are central in $\Out(BG) \cong \Out(\bbD_G)$.

For the second claim, first note that as $BG$ is simply connected we do not have to distinguish between pointed and unpointed homotopy classes of maps, and $\Out(BG)$ does indeed act on $F(BG)$.
By the first part, the action of $\Z$ on  $F(BG)$ via $1 \mapsto (\psi^q)_*$ extends to an action of $\Z_\ell$.
Now, since $F(BG)$ is finite, the action factors through
$\Z/\ell^s$ for some $s$.
  \end{proof}

\begin{proof}[Proof of Theorem~\ref{thm:class-lcglie}] Set $\muet = \langle \tau'\rangle \leq \Out(BG)$ for short.
By Proposition~\ref{autBTtoBG}, $\muet$ has a canonical homotopical action on $BG$ and by Proposition~\ref{prop:ghksplit}, $BH =BG^{h\muet}$ is a connected $\ell$--compact group, semisimple or simply connected if $BG$ is. We would like to see that $BH$ has an action of  $\tau_\ell\psi^{q'}$ and describe the homotopy fixed points of this action. 
Start by noting that $\tau'$ and $\tau_\ell$ commute by elementary group theory, and both commute with the central elements $ \psi^{t}$, $t \in \Z_\ell^\times$, in $\Out(BG) \cong \Out(\bbD_G)$.
Hence we can consider the product subgroup
$$A = \langle \tau'\rangle \times \langle \tau_\ell \psi^{q'}\rangle  \leq
\Out(BG).$$ 
By Proposition~\ref{autBTtoBG},
 the inclusion of $A$ into $\Out(BG)$ also lifts to a unique
homotopy action of $A$ on $BG$.
In particular $\tau_\ell\psi^{q'}$ has a residual action on  $BH =BG^{h\muet}$ (see e.g. \cite[Lem.~10.5]{DW94}), justifying~(\ref{item:Hislcg}).

For (\ref{item:untwisting}), start by noting that
by definition of homotopy fixed points,
$$
BG^{hA} \homot (BG^{h\langle \tau'\rangle})^{h\langle \tau_\ell \psi^{q'} \rangle} = (BH)^{h\langle \tau_\ell \psi^{q'} \rangle}
=  B{}^{\tau_\ell}\!H(q').$$
The canonical inclusion of subgroups $\langle \tau \psi^q
\rangle \to A = \langle \tau'\rangle \times \langle \tau_\ell\psi^{q'}\rangle$,  $\tau \psi^q \mapsto (\tau',
\tau_\ell\psi^{q'})$ defines a map on homotopy fixed points
$$f\co (BH)^{h\langle \tau_\ell\psi^ {q'}\rangle}  \homot  BG^{hA} \longto BG^{h \langle \tau \psi^q
  \rangle}  = \BtGq$$
and the claim is that this map is a homotopy equivalence. As
$A$ is also generated by $ \sigma = (\tau',\psi^{q'})$ and
$(\tau',1)$, we have $BG^{hA} \homot ((BG)^{h\muet})^{h\langle \sigma
  \rangle}$, and the homotopy fiber of $f$ identifies with $(G/G^{h\muet})^{h\langle
  \sigma \rangle}$, which we want to show is contractible. 
We will do so by showing that the $E_2$--page of its homotopy fixed point
spectral sequence vanishes.
 
The space $(G/G^{h\muet})^{h\muet} \homot \Fibre
((BG^{h\muet})^{h\muet} \to  (BG)^{h\muet})$ is contractible, as it is the
fiber of a homotopy equivalence. Hence 
also $(\pi_*(G/G^{h\muet}))^{\muet} = 0$ 
as the homotopy fixed point spectral sequence
degenerates onto the vertical axis, since $\ell \nmid |\muet|$.
Again since $\ell \nmid |\muet|$, this implies that also
$\pi_*(G/G^{h\muet})_{\muet} = 0$ and that $\pi_*(G/G^{h\muet}) \xrightarrow{1-(\tau')_*}
\pi_*(G/G^{h\muet})$ is an isomorphism.
So  $\pi_*(G/G^{h\muet})\otimes \F_\ell \xrightarrow{1-(\tau')_*}
\pi_*(G/G^{h\muet}) \otimes \F_\ell$ is likewise an isomorphism.

For each $n$, $\psi^{q'}$ acts as an element of $\ell$--power order
$\ell^s$ on $\pi_n(G/G^{h\muet})\otimes \F_\ell$  by Lemma~\ref{nilpot-lemma}. As $\tau_\ell$ has $\ell$--power order, and commutes with $\psi^{q'}$, we also have $x^{\ell^s} =1$ with $x = (\tau_\ell\psi^{q'})_*$ for some $s \geq 0$,
Hence $(1-\sigma_*)^{\ell^s} =(1-\sigma_*^{\ell^s}) =
(1-((\tau')_*x)^{\ell^s}) = (1-(\tau')_*^{\ell^s}) =
(1-(\tau')_*)^{\ell^s}$
in $\End(\pi_n(G/G^{h\muet})\tensor \F_\ell)$.
Since
$\pi_n(G/G^{h\muet})\otimes \F_\ell \xrightarrow{1-(\tau')_*}
\pi_n(G/G^{h\muet}) \otimes \F_\ell$
is an isomorphism, it follows that
$\pi_n(G/G^{h\muet})\otimes \F_\ell \xrightarrow{1-\sigma_*}
\pi_n(G/G^{h\muet}) \otimes \F_\ell$ 
also is.  Since $\pi_n(G/G^{h\muet})$ is a
finitely generated $\Z_\ell$--module, 
it follows by Nakayama's lemma that
$\pi_n(G/G^{h\muet}) \xrightarrow{1-\sigma_*}
\pi_n(G/G^{h\muet})$ is an isomorphism as well.
But this means that the $E_2$--page for the homotopy limit spectral sequence for the equalizer is identically zero, so
 $(G/G^{h\muet})^{h\langle
  \sigma \rangle}$ is contractible, as desired.

To see (\ref{it:threepiecesofdata}), i.e., which data determines the homotopy type of $\BtGq \simeq B{}^{\tau_\ell}H(q') = BH^{h\tau_\ell\psi^{q'}}$, we make the following obeservations: By the classification of $\ell$--compact groups, the homotopy type of $BH$ is determined by root datum $\bbD_H$. Moreover, by
Theorem~\ref{thm:closedsubgrphelp}, the homotopy type of 
$BH^{h\tau_\ell\psi^{q'}}$ only depends on the closure of $\langle \tau_\ell \psi^{q'}
\rangle$ in $\Out(BH)$. (We remark that Theorem~\ref{thm:closedsubgrphelp} relies on
\cite[Thm.~2.4]{BrotoMoellerOliver}, an abstraction of
\cite[Thm.~E]{BM07}; we also invite the reader to look up the short
proof of that reference.) By Proposition~\ref{prop:topcomparison},
 $\Out(BG)$ is a profinite group. As any group homomorphism from a finitely generated pro--$\ell$--group to a profinite group is automatically continuous (see e.g.\cite[Cor.~1.21]{DixonSautoyMannSegal}), it follows that the closure of $\langle \tau_\ell \psi^{q'}
\rangle$ in $\Out(BH)$ agrees with the image of the constructed extension of \eqref{eq:factorization} to $\Z_\ell$. But this image is clearly determined by $\tau_\ell$ and the $\ell$--adic valuation of $q'-1$, as wanted.
\end{proof}
  
  \begin{samepage}
  \begin{rem} 
 \nopagebreak[4]
\par\noindent
\begin{itemize}
\item
 Different data in \eqref{item-twroot}--\eqref{item-val} of Theorem~\ref{thm:class-lcglie} 
of generally produce non-isomorphic $\BtGq$. 
 We will not pursue a precise statement here. 
 The dependence on the 
  the $\ell$--adic valuation of $q'-1$ in   \eqref{item-val} is illustrated by looking at the case of $BG$ a torus, and 
the story for `integral' finite groups of Lie type is explained in \cite[Rem.~24.9]{MT11}.

\item The Tezuka conjecture implies the prediction that the 
\emph{cohomology} of $\BtGq$ is independent of \eqref{item-val} in Theorem~\ref{thm:class-lcglie}.

\item
In \eqref{item-ltw} of Theorem~\ref{thm:class-lcglie}, if $\bbD$ is simple,  $\tau_\ell \neq 1$ can only occur for $\ell = 2, 3$. More precisely, it may occur only for $\bbD$ of type $A_n$ $(n\geq2)$, $D_n$ $(n\geq4)$,  $E_6$, and $G_2$ for $\ell =2$ and $D_4$ for $\ell =3$. Moreover, only in the subset of cases $D_{2n}$ $(n\geq2)$ and $G_2$ at $\ell =2$ and $D_4$ at $\ell=3$ is this due to an element in $\Out(\bbD)$ which is not just $-1$. See Proposition~\ref{prop:twistingclassification} and Remark~\ref{rem:mod4cong}.
\end{itemize}
\end{rem}
\end{samepage}
  
In continuation of the last remark, let us further elaborate on Theorem~\ref{thm:class-lcglie} by using the classification of $\ell$--compact groups to describe all possible twistings $\tau$.
The reference \cite[\S8.4]{AG09} describes $\Out(\bbD)$ for an arbitrary
$\Z_\ell$--root datum $\bbD$ in terms of simple
simply connected root data, and \cite[Thm.~13.1]{AGMV08} then tabulates $\Out(\bbD)$
for  simple simply connected root data $\bbD$ (building on work of Brou\'e--Malle--Michel
\cite{BMM99} over the complex numbers). Theorem~\ref{thm:class-lcglie} allows to reduce to a situation when any twisting $\tau \in \Out(\bbD)$ is of $\ell$--power order and Lemma~\ref{lem:permute} helps us reduce to a case where $\bbD$ is simple. 
 The following proposition lists all possible twistings of $\ell$--power order for $\bbD$ simple simply connected (and hence for all simple):
   \begin{prop}[Classification of $\ell$--compact twistings
  of $\ell$--power order] \label{prop:twistingclassification} 
Let $\bbD$ be a simple simply connected $\Z_\ell$--root datum. Then $\Out(\bbD)/\langle \Z_\ell^\times\rangle$ is finite and tabulated in \cite[Thm.~13.1]{AGMV08}.
In particular it is of order prime to $\ell$ except 
in the following four cases:
\begin{enumerate}
\item
$\ell=2$ and 
 $\bbD \cong \bbD_{D_{2n}} \otimes \Z_2 $ ($n\geq2$), in which case $\Out(\bbD) \cong \Z_2^\times/\langle -1\rangle \times \Gamma$ 
with $\Gamma$ the graph automorphisms of $D_{2n}$ i.e., $\Gamma \cong {\mathfrak S}_3$ for $n=2$ and $C_2$ when $n \geq 3$. 
\item  
$\ell=2$ and 
$\bbD \cong \bbD_{G_2} \otimes \Z_2$ in which case $\Out(\bbD) \cong \Z_2^\times/\langle -1 \rangle \times C_2$.
\item  
$\ell=3$ and 
$\bbD \cong \bbD_{D_4} \otimes \Z_3$ in which case  $\Out(\bbD) \cong  \Z_3^\times/\langle -1\rangle \times {\mathfrak S}_3$.
\end{enumerate}
Here
$\bbD_{D_{2n}}$ and $\bbD_{G_2}$
are simply connected root data of the indicated type over $\Z$.

The kernel of $\Z_\ell^\times \to \Out(\bbD)$ is tabulated in \cite[Prop.~2.2 and Table~1]{andersen99}.  In particular, this reference
 lists when $-1$ is in the kernel,
implying $BG(q) \cong BG(-q)$.  
  \end{prop}
  \begin{proof} This is an inspection of the cases in \cite[Thm.~13.1]{AGMV08}.
  \end{proof}

\begin{rem}[Twistings over $\Z$ and $\Z_\ell$]\label{rem:twistingcomparison}
It is interesting to compare the list in Proposition~\ref{prop:twistingclassification} to the `integral' Lie twistings tabulated in e.g., \cite[\S22]{MT11} (following Steinberg's classic work \cite{steinberg68}). The $\Z_\ell$--description becomes simpler than the $\Z$--description for two reasons: First, as $-q$ is a bona-fide $\ell$--adic unit if $q$ is, so twistings given by $-1$ gets absorbed in $q$ (or even better, do not matter at all in the cases where $-1 \in W$). Second, the ``very twisted" groups just become ordinarily twisted over $\Z_\ell$, see \cite[Elaboration~13.10]{AGMV08}.
\end{rem}

  \begin{rem}[The mod $4$ congruence at $\ell=2$]\label{rem:mod4cong}
  Note that the mod $4$ congruence when $\ell =2$ in Theorem~\ref{thm:class-lcglie}  instructs us to view $q$ as $(-1)(-q)$ when $q$ is congruent to 3 modulo 4. This is due to the structure of the $2$-adic units $\Z_2^\times \cong \Z/2 \times \Z_2$, where topologically cyclic closed subgroups are parametrized by whether a generator is non-zero on the first factor or not, and the $2$--adic valuation of the second factor.  For instance $BE_6(3)\twocom$ is equivalent to $B\,{}^2\!E_6(-3)\twocom$ 
which is again equivalent to $B\,{}^2\!E_6(5)\twocom$, as $\nu_2(-3-1) = \nu_2(5-1) = 2$, i.e., we replace the group by the its twisted version to obtain the wanted congruence. In many other cases $-1$ will be inner, and we simply have that $BG(q)$ is equivalent to $BG(-q)$, so the congruence is automatic. 
   \end{rem}

\subsection{The topology on \texorpdfstring{$\Out(\bbD)$}{Out(D)} and 
\texorpdfstring{$\Out(BG)$}{Out(BG)}}\label{subsec:topology}

We end this section by discussing and comparing topologies on $\Out(\bbD_G)$
and $\Out(BG)$ when $BG$ is a connected $\ell$--compact group.
We remind the reader that $\Out(BG)$ is 
equipped with the topology induced by the actions
of $\Out(BG)$ on the cohomology rings $H^\ast(BG; \Z/\ell^k)$, $k\geq 1$,
so that a map from a space into $\Out(BG)$
is continuous
if and only if
its composite with the homomorphisms $\Out(BG) \to \Aut(H^\ast(BG;\Z/\ell^k))$
is continuous for all $k\geq 1$ 
where the groups $\Aut(H^\ast(BG;\Z/\ell^k))$
are equipped with the discrete topology.
See \cite[p.~7]{BrotoMoellerOliver}.
On the other hand, given a root datum $\bbD$,
we equip $\Out(\bbD)$ with the topology 
it obtains as a quotient of
a closed subgroup of $\GL_{\Z_\ell}(L)$
by a finite group, where $L$ denotes
the underlying finitely generated free $\Z_\ell$--module
of $\bbD$. See \cite[p.~388]{AG09}.
This topology makes 
$\Out(\bbD)$ into a profinite group,
and indeed, by Proposition~\ref{prop:outDstructure2} below,
it is the \emph{only} topology on $\Out(\bbD)$
with this property. The main result of the section is

\begin{prop}
\label{prop:topcomparison} For any connected $\ell$--compact group $BG$,
the group isomorphism $\Out(BG) \xrightarrow{\cong}\Out(\bbD_G)$ of
\cite[Thm.~1.2]{AG09} is a
homeomorphism under the topologies on $\Out(BG)$ and $\Out(\bbD_G)$
introduced above.
\end{prop}

For reference, we note the following corollary of Proposition~\ref{prop:topcomparison}.
\begin{cor}
\label{cor:ladicunitstooutbgcont}
For any connected $\ell$--compact group $BG$, the homomorphism
\[
	\phi\colon \Z_\ell^\times \longto \Out(BG), \qquad q \longmapsto [\psi^q]
\]
is continuous.
\end{cor}
\begin{proof}
The corresponding homomorphism $\Z_\ell^\times \to \Out(\bbD_G)$ is evidently
continuous.
\end{proof}

In the proof of Proposition~\ref{prop:topcomparison},
we will rely on the following statement about $\Out(\bbD)$, which we also
use in the paper, and which should be of independent interest.

\begin{prop}\label{prop:outDstructure2}
Given a root datum $\bbD$, the group $\Out(\bbD)$ admits
a unique topology making it into a profinite group. 
In this topology, $\Out(\bbD)$ has an open normal 
subgroup $M$ which is a topologically finitely generated 
pro--$\ell$--group. Moreover, $M$ can be chosen so that 
it is isomorphic as a topological group to 
a product $\Gamma_s \times (\Z_\ell)^t$
for some $s$ and $t$ 
where $\Gamma_s$ is the principal congruence subgroup
\[
	\Gamma_s = \ker(\GL_n(\Z_\ell) \to \GL_n(\Z/\ell^s))
\]
for $n = \dim_{\Q_\ell}(\pi_1(\bbD)\tensor \Q)$.
\end{prop}

The exact value of $t$ can be easily read off from the proof 
of Proposition~\ref{prop:outDstructure2}.
In the proof of both Proposition~\ref{prop:topcomparison}
and Proposition~\ref{prop:outDstructure2},
we will make use of the following lemma.

\begin{lemma}
\label{lm:homzlmodtooutbg3} 
Consider a group homomorphism $\phi\colon H \to K$
where $H$ is a topological group
containing an open subgroup
which is a topologically finitely generated pro--$\ell$--group,
and where $K$ is a topological group 
whose topology is induced from a family of finite groups
in the sense that there exist
finite discrete groups $F_i$ and homomorphisms 
$f_i \colon K \to F_i$, $i\in I$,
such that a homomorphism from a topological group 
into $K$ is continuous if and only if its composite with 
each $f_i$ is continuous.
Then $\phi\colon H \to K$ is continuous.
If furthermore $H$ is compact and $K$ is Hausdorff,
and $\phi$ is a bijection, then $\phi$ is a homeomorphism.
\end{lemma}

\begin{proof}
Since a group homomorphism is continuous if and only if 
its restriction to some open subgroup of the domain is 
continuous, the continuity of $\phi$ follows from 
a theorem of Serre \cite[Thm.~1.17]{DixonSautoyMannSegal}
which states that any group homomorphism from 
a topologically finitely generated pro--$\ell$--group
into a finite group is continuous. 
The last claim now follows from the 
point-set topological fact that a continuous bijection 
from a compact space to a Hausdorff space is a homeomorphism.
\end{proof}

\begin{proof}[Proof of Proposition~\ref{prop:outDstructure2}]
It suffices to construct on $\Out(\bbD)$ \emph{some}
profinite topology in which $\Out(\bbD)$ has a subgroup $M$
with the prescribed properties. That any profinite
topology on $\Out(\bbD)$ coincides with the one 
constructed then follows by applying  Lemma~\ref{lm:homzlmodtooutbg3}
to the identity map of $\Out(\bbD)$. 

We start by constructing the topology on $\Out(\bbD)$.
By \cite[Thm.~8.13 and Prop.~8.15]{AG09}, the group $\Out(\bbD)$ embeds as a 
finite-index subgroup 
\[
	\Out(\bbD) \leq \Out(\bbD')
\]
for a root datum $\bbD'$ splitting as a product
$\bbD' = \prod_{i=0}^k \bbD_i^{m_i}$ %
where $\bbD_0$ is the trivial root datum with lattice $\Z_\ell$
and $\bbD_1$, \ldots, $\bbD_k$ are pairwise nonisomorphic 
irreducible root data with non-trival Weyl groups.
By \cite[Prop.~8.14]{AG09}, there is an isomorphism
of groups
\begin{equation}
\label{eq:bbdprimedesc}
	\Out(\bbD') 
	\isom
	\GL_{m_0}(\Z_\ell) 
    \times 
    \prod_{i=1}^k (\Out(\bbD_i)\wr \fS_{m_i})
\end{equation}
where in view of \cite[Lemma~8.9]{AG09}
we have $m_0 = n = \dim_{\Q_\ell}(\pi_1(\bbD)\tensor \Q)$.
We will henceforth use this isomorphism to identify 
$\Out(\bbD')$ with the right hand side 
of~\eqref{eq:bbdprimedesc}. 
Equipping $\GL_{m_0}(\Z_\ell)$ with its natural topology
and giving each $\Out(\bbD_i)$ the topology 
indicated at the beginning of the section,
this identification yields on $\Out(\bbD')$ a topology  
making  $\Out(\bbD')$ into a topological group.
Finally, we topologize $\Out(\bbD)$ 
as a subgroup of $\Out(\bbD')$.

We proceed to construct the subgroup $M$,
which we will do by constructing for each
factor $F$ on the right hand side of~\eqref{eq:bbdprimedesc}
a subgroup of $F \cap \Out(\bbD)$ which is 
a topologically finitely generated pro--$\ell$--group
which is 
a finite-index closed subgroup of $F$ 
and
a normal subgroup of $\Out(\bbD)$. 
Let us first consider the factors $\Out(\bbD_i)\wr \fS_{m_i}$
of $\Out(\bbD')$.
It can be read off from
\cite[Thm.~13.1]{AGMV08} that for each $i=1,\ldots,k$,
the image of the homomorphism $\Z_\ell^\times \to \Out(\bbD_i)$ sending 
$q\in\Z_\ell^\times$ to the multiplication-by-$q$ 
map is infinite and of finite index in $\Out(\bbD_i)$. 
This homomorphism is continuous, and hence closed 
since $\Z_\ell^\times$ is compact and $\Out(\bbD_i)$
is Hausdorff.
As $\Z_\ell^\times$ contains
$\Z_\ell$ as a finite-index closed subgroup
and the image of the homomorphism $\Z_\ell^\times \to \Out(\bbD_i)$
lies in the center of $\Out(\bbD_i)$,
we conclude that 
$\Out(\bbD_i)$ 
contains a finite-index closed normal subgroup isomorphic to
$\Z_\ell$ 
for all $i$.
It follows that $\Out(\bbD_i)\wr \fS_{m_i}$
contains a finite-index closed normal subgroup $V_i$ 
isomorphic to 
$\Z_\ell^{m_i}$ for all $i$. 
Let $W_i$ be the intersection $W_i = V_i \cap \Out(\bbD)$.
Then $W_i$ is a normal subgroup of $\Out(\bbD)$.
Moreover, as $W_i$ is of finite index in the topologically finitely generated
pro--$\ell$--group $V_i$, 
by \cite[Thm.~1.17]{DixonSautoyMannSegal}
$W_i$ is open and hence closed in $V_i$. 
It follows that $W_i$ 
is a finite-index closed subgroup of $\Out(\bbD_i)\wr \fS_{m_i}$.
Moreover, as closed subgroups
of $V_i$ are $\Z_\ell$--submodules and $\Z_\ell$ is a PID,
it follows that $W_i$ is again isomorphic to $V_i \isom \Z_\ell^{m_i}$ for all $i$. 
In particular, $W_i$ is a topologically finitely generated pro--$\ell$--group.

Let us next consider the factor $\GL_{m_0}(\Z_\ell)$ of $\Out(\bbD')$.
Let $\epsilon = 1$ if  $\ell = 2$ and 
let $\epsilon = 0$ otherwise.
By \cite[Thm.~5.2]{DixonSautoyMannSegal}, the subgroup 
$\Gamma_{1+\epsilon} \leq \GL_{m_0}(\Z_\ell)$ 
is a topologically finitely generated pro--$\ell$--group, 
so again by \cite[Thm.~1.17]{DixonSautoyMannSegal},
the finite-index subgroup $\Gamma_{1+\epsilon} \cap \Out(\bbD)$
of $\Gamma_{1+\epsilon}$ is open in $\Gamma_{1+\epsilon}$.
Since the subgroups $\Gamma_s \leq \GL_{m_0}(\Z_\ell)$ form a
neighborhood basis for $\GL_{m_0}(\Z_\ell)$ at the identity 
(see the beginning of Section~5 of \cite{DixonSautoyMannSegal}),
it follows that $\Gamma_s \leq \Out(\bbD)\cap \Gamma_{1+\epsilon}$
for some $s \geq 1+\epsilon$. 
We note that $\Gamma_s$ is an open and hence a finite-index closed subgroup of
the compact group
$\GL_{m_0}(\Z_\ell)$;
that it is normal in $\Out(\bbD')$ and hence 
also in $\Out(\bbD)$; and that as an open subgroup of the 
topologically finitely generated pro--$\ell$--group
$\Gamma_{1+\epsilon}$ 
it is a topologically finitely generated pro--$\ell$--group.
See \cite[Props.~1.7 and 1.11(i)]{DixonSautoyMannSegal}.
Now the subgroup 
$M = \Gamma_s \times \prod_{i=1}^{k} W_i \leq \Out(\bbD)$
has the desired properties. That $M$ is open in $\Out(\bbD)$
follows by observing that $M$ is closed and of finite index in $\Out(\bbD')$
and hence also in $\Out(\bbD)$.
Finally, since $\Out(\bbD)$ admits a profinite finite-index closed subgroup 
(namely $M$), it is also profinite.
\end{proof}

To finish the proof of Proposition~\ref{prop:topcomparison},
we need an additional lemma.

\begin{lemma}
\label{lm:outbghausdorff}
The topology on $\Out(BG)$ is Hausdorff for all connected $\ell$--compact groups $BG$.
\end{lemma}
\begin{proof}
To prove the claim, it suffices to show that 
the action of $\Out(BG)$ on
$H^*(BG;\Z_\ell)$
is faithful; see \cite[p.~7]{BrotoMoellerOliver}.
By \cite[Thm.~9.7(iii)]{DW94}, 
we have 
\[
    H^*(BG;\Z_\ell) \otimes_{\Z_\ell} \Q_\ell 
    \isom 
    (H^*(BT;\Z_\ell)\otimes_{\Z_\ell} \Q_\ell)^{W_G},
\]
and under the isomorphism $\Out(BG) \isom \Out(\bbD_G)$
of \cite[Thm.~1.2]{AG09}, the induced action of $\Out(BG)$
on the right hand side corresponds to the
action of the group $\Out(\bbD_G) = \Aut(\bbD_G)/W_G$ on 
the ring ${\Q_\ell[L]^{W_G}}$
induced by the faithful action of $\Aut(\bbD_G)$ on $L$,
where $L$ denotes the underlying finitely generated
free $\Z_\ell$--module of $\bbD_G$.
Let $H \leq \Aut(\bbD_G)$
be the subgroup 
consisting of all elements
fixing the subring
${\Q_\ell[L]^{W_G}} 
\subset 
{\Q_\ell[L]}$
pointwise. To prove the claim, it is enough to show that $H = W_G$.
Clearly $W_G \leq H$. To prove the reverse containment,
write ${\Q_\ell(L)}$
for the fraction field of 
${\Q_\ell[L]}$.
Writing $\Aut(E/F)$ for the group of 
automorphisms of a field $E$ fixing a subfield $F$,
we then have an embedding 
$H\to \Aut(\Q_\ell(L) / \Q_\ell(L)^H)$.
Clearly $\Q_\ell(L)^H \subset \Q_\ell(L)^{W_G}$.
Moreover, as $W_G$ is finite,
given an element $f/g \in  \Q_\ell(L)^{W_G}$,
by multiplying $f$ and $g$ by the product of 
all elements of the form $w\cdot g$ for $w\in W_G, w\neq 1$,
we see that $f/g$ can be expressed as a quotient
of two elements of $\Q_\ell[L]^{W_G}$,
wherefore $f/g \in \Q_\ell(L)^H$ by the choice of $H$.
Thus $\Q_\ell(L)^H = \Q_\ell(L)^{W_G}$, and hence
$\Aut(\Q_\ell(L) / \Q_\ell(L)^H) = \Aut(\Q_\ell(L) / \Q_\ell(L)^{W_G})$.
But as $W_G$ is finite, we have 
$\Aut(\Q_\ell(L) / \Q_\ell(L)^{W_G}) = W_G$;
see e.g. \cite[Thm.~VI.1.8]{Lang}.
We conclude that $H$ embeds into $W_G$, and the claim follows.
\end{proof}

\begin{proof}[Proof of Proposition~\ref{prop:topcomparison}]
By \cite[Thm.~4.2]{ACFJS},
the cohomology ring
$H^\ast(BG;\,\Z/\ell^k)$ is Noetherian,
so by \cite[Thm.~13.1]{Matsumura} it is
a finitely generated $\Z/\ell^k$--algebra.
Thus the automorphism groups
$\Aut(H^\ast(BG;\,\Z/\ell^k))$
featuring in the definition of the topology on $\Out(BG)$
are finite for all $k$.
In view of 
Proposition~\ref{prop:outDstructure2}
and
Lemma~\ref{lm:outbghausdorff},
Lemma~\ref{lm:homzlmodtooutbg3}
now applies to show that
the inverse 
$\Out(\bbD) \xrightarrow{\isom} \Out(BG)$ 
of the isomorphism from \cite[Thm.~1.2]{AG09}  
is a homeomorphism. 
\end{proof}

\printbibliography %

\end{document}